\documentclass[A4,11pt,twoside]{amsart}
\usepackage{amsmath,amssymb,amsfonts}
\usepackage[francais,american]{babel}
\usepackage[latin1]{inputenc}
\usepackage[T1]{fontenc}
\usepackage{ae,aecompl}
\usepackage{units}

\title[Exploration of Kac algebras and lattices of intermediate subfactors]
{Exploration of finite dimensional Kac algebras\\
  and lattices of intermediate subfactors\\ of irreducible inclusions}

\author{Marie-Claude DAVID}
\email{mcld@math.u-psud.fr}

\author{Nicolas M. THIÉRY}
\email{Nicolas.Thiery@u-psud.fr}

\address{Univ Paris-Sud, Laboratoire de Mathématiques d'Orsay\\
  Orsay, F-91405; CNRS, Orsay, F-91405}

\date{\today
  \ifdraft\quad$ $Id: KD.tex 433 2010-12-06 15:00:19Z nthiery $ $\fi
}



\usepackage[a4paper,asymmetric,bindingoffset=1.5cm,inner=2.5cm,textwidth=15cm,textheight=23cm,top=3.5cm]{geometry}
\usepackage{color}
\usepackage{graphicx}
\usepackage{rotating}

\usepackage{here}

\usepackage{enumerate}
\usepackage{ifthen}

\usepackage{xspace}
\usepackage{hyperref}
\usepackage[all]{hypcap} %
\usepackage{tikz}

\definecolor{green}{RGB}{0,150,0}

\newtheorem{prop}{Proposition}
\newtheorem{thm}{Theorem}
\newtheorem{lem}{Lemma}
\newtheorem{cor}{Corollary}
\newtheorem{theorem}{Theorem}
\newtheorem{proposition}{Proposition}
\newtheorem{corollary}{Corollary}
\newtheorem{lemma}{Lemma}
\newtheorem{conjecture}{Conjecture}
\newtheorem{algorithm}{Algorithm}

\theoremstyle{definition}
\newtheorem{defn}{Definition}
\theoremstyle{remark}

\newtheorem{example}{Example}
\newtheorem{remark}{Remark}
\newtheorem{remarks}{Remarks}
\newtheorem{problem}{Problem}

\makeatletter

\def\@thm#1#2#3{%
  \ifhmode\unskip\unskip\par\fi
  \normalfont
  \trivlist
  \let\thmheadnl\relax
  \let\thm@swap\@gobble
  \thm@notefont{\fontseries\mddefault\upshape}%
  \thm@headpunct{.}%
  \thm@headsep 5\p@ plus\p@ minus\p@\relax
  \thm@space@setup
  #1%
  \@topsep \thm@preskip               %
  \@topsepadd \thm@postskip           %
  \def\@tempa{#2}\ifx\@empty\@tempa
    \def\@tempa{\@oparg{\@begintheorem{#3}{}}[]}%
  \else
    \def\@tempa{\@oparg{\@begintheorem{#3}{%
          \@currentlabel
        }}[]}%
  \fi
  \@tempa
}
\makeatother

\newboolean{draft}
\setboolean{draft}{false}
\ifdraft
\newcommand{\TODO}[2][To do: ]{\textcolor{red}{\textbf{#1#2}}}
\else
\newcommand{\TODO}[2][]{}
\fi

\newcommand{\an}{{a_{\infty}}}

\newcommand{\C}{\mathbb{C}}
\newcommand{\Q}{\mathbb{Q}}
\newcommand{\R}{\mathbb{R}}
\newcommand{\N}{\mathbb{N}}
\newcommand{\Z}{\mathbb{Z}}
\newcommand{\gap}{\texttt{GAP}\xspace}
\newcommand{\mupad}{\texttt{MuPAD}\xspace}
\newcommand{\mupadcombinat}{\texttt{MuPAD-Combinat}\xspace}
\newcommand{\sacig}{coideal subalgebra\xspace}
\newcommand{\sacigs}{coideal subalgebras\xspace}

\newcommand{\Sacigs}{Coidalgebras\xspace}
\newcommand{\aut}{{\operatorname{Aut}}}
\newcommand{\suchthat}{{\ |\ }}
\renewcommand{\ll}{\operatorname{l}}
\newcommand{\rl}{\operatorname{r}}

\newcommand{\KD}{K\!D}
\newcommand{\KQ}{K\!Q}
\newcommand{\KP}{K\!P}
\newcommand{\KA}{K\!A}
\newcommand{\KB}{K\!B}

\usepackage{verbatim}
\makeatletter
\def\Mexin@processline{>{}> \the\verbatim@line\par}
\newenvironment{Mexin} {\vspace{-.5ex}\verbatim\small\addtolength\parskip{-.5ex}\let\verbatim@processline=\Mexin@processline}{\endverbatim}
\newenvironment{Mexout}{\vspace{-1ex}\verbatim\small\addtolength\parskip{-.9ex}}{\endverbatim}
\newenvironment{Mexoutsmall}{\vspace{-.5ex}\verbatim\SMALL\addtolength\parskip{-.9ex}}{\endverbatim}
\makeatother

\def \id { {\rm id\, }}
\def \dim { {\rm dim\, }}
\def \Ad { {\rm Ad\, }}

\begin{document}
\maketitle

\makeatletter
\newskip\@bigflushglue \@bigflushglue = -100pt plus 1fil
\def\bigcenter{\trivlist \bigcentering\item\relax}
\def\bigcentering{\let\\\@centercr\rightskip\@bigflushglue%
\leftskip\@bigflushglue
\parindent\z@\parfillskip\z@skip}
\def\endbigcenter{\endtrivlist}
\makeatother
\begin{abstract}
  We study the four infinite families $\KA(n), \KB(n), \KD(n), \KQ(n)$
  of finite dimensional Hopf (in fact Kac) algebras constructed
  respectively by A.~Masuoka and L.~Vainerman: isomorphisms,
  automorphism groups, self-duality, lattices of \sacigs. We reduce
  the study to $\KD(n)$ by proving that the others are isomorphic to
  $\KD(n)$, its dual, or an index $2$ subalgebra of $\KD(2n)$.
  We derive many examples of lattices of intermediate subfactors of
  the inclusions of depth $2$ associated to those Kac algebras, as
  well as the corresponding principal graphs, which is the original
  motivation.

  Along the way, we extend some general results on the Galois
  correspondence for depth $2$ inclusions, and develop some tools and
  algorithms for the study of twisted group algebras and their
  lattices of \sacigs.  This research was driven by heavy computer
  exploration, whose tools and methodology we describe.
\end{abstract}

\begin{otherlanguage}{french}
\begin{abstract}
  Nous étudions les quatre familles $\KA(n), \KB(n), \KD(n), \KQ(n)$
  d'algèbres de Hopf (en fait de Kac) de dimension finie construites
  respectivement par A.~Masuoka et L.~Vainerman: isomorphismes,
  groupes d'automorphismes, autodualité, treillis des sous-algèbres
  coidéales.  Nous réduisons l'étude à $KD(n)$ en démontrant que les
  autres algèbres sont toutes isomorphes à $\KD(n)$, à sa duale ou à
  une sous-algèbre d'indice $2$ de $\KD(2n)$.
  Nous en déduisons de nombreux exemples de treillis de facteurs
  intermédiaires d'inclusions de profondeur $2$ associées à ces
  algèbres de Kac, ainsi que les graphes principaux correspondants.

  Au cours de cette étude, nous approfondissons des résultats généraux
  sur la correspondance de Galois pour les inclusions de profondeur
  $2$ et donnons des méthodes et algorithmes pour l'analyse des
  algèbres de groupes déformées et leurs treillis des sous-algèbres
  coidéales. Cette recherche a été guidée par une exploration
  informatique intensive, dont nous décrivons les outils et la
  méthodologie.
\end{abstract}
\end{otherlanguage}

\clearpage
\setcounter{tocdepth}{2}
\tableofcontents

\clearpage
\listoffigures

\clearpage
\section{Introduction}

The theory of Kac algebras provides a unified framework for both group
algebras and their duals. In finite dimension this notion coincides
with that of $C^*$-Hopf algebras (see~\ref{KacHopf}).  Those algebras
play an important role in the theory of inclusions of  hyperfinite
factors of type $II_1$; indeed, any irreducible finite index depth 2
subfactor is obtained as fixed point set under the action of some
finite dimensional Kac algebra on the factor.
There furthermore is a Galois-like correspondence between
the lattice of intermediate subfactors and the lattice of coideal
subalgebras of the Kac algebra.  A cocommutative (respectively commutative) Kac
algebra is the group algebra $\C[G]$ of some finite group $G$ (respectively its dual). We
call such a Kac algebra \emph{trivial}; its lattice of \sacigs is simply given by the
lattice of subgroups of $G$.

In~\cite{Watatani.1996}, Yasuo Watatani initiates the study of the
lattices of irreducible subfactors of type $II_1$ factors. He provides
general properties as well as examples. In particular, his work,
completed by Michael Aschbacher~\cite{Aschbacher.2008}, shows that
every lattice with at most six elements derives from group actions and
operations on factors.

Our aim is to study more involved examples of lattices of subfactors,
and therefore of principal graphs. It is our hope that this will
contribute to the general theory by providing a rich ground for
suggesting and testing conjectures. A secondary aim is to evaluate the
potential of computer exploration in that topic, to develop algorithms
and practical tools, and to make them freely available for future
studies.

We always assume the factors to be hyperfinite of type
$II_1$, and the inclusions to be of finite index.
Our approach is to use the Galois correspondence and study instead the
lattices of \sacigs of some non trivial examples of Kac algebras. Only
lattices of depth $2$ irreducible inclusions of type $II_1$
hyperfinite factors are obtained this way. The depth $2$ hypothesis is
not so restrictive though, since D. Nikshych and L. Vainerman showed
in~\cite{Nikshych_Vainerman.2000.2} that any finite depth inclusion is
an intermediate subfactor of a depth $2$ inclusion. There remains only
the irreducibility condition. Luckily, there still exists a Galois
correspondence for non irreducible depth $2$ inclusions, at the price
of considering the action of a finite $C^*$-quantum groupoid instead
of a Kac algebra
(see~\cite{Nikshych_Vainerman.2000.1,David.2005,David.2009}). The
study of more general examples coming from $C^*$-quantum groupoids,
like those constructed from tensor categories by C. Mével, is the
topic of subsequent work.

The first non trivial example of Kac algebra $\KP$ was constructed in
1966 (see~\cite{Kac.Paljutkin.1966}); it is the unique one of dimension
$8$. Later, M. Izumi and H. Kosaki classified
small dimension Kac algebras through factors (see~\cite{Izumi_Kosaki.2002}). In 1998, L. Vainerman
constructed explicitly two infinite families of Kac algebras, which we
denote $\KD(n)$ and $\KQ(n)$, by deformation of the group algebras of
the dihedral groups $D_{2n}$ and of the quaternion groups $Q_{2n}$
respectively (see~\cite{Vainerman.1998}\footnote{Note
  that~\cite{Vainerman.1998} is a follow up on a study initiated in
  1996 with M.~Enock~\cite{Enock_Vainerman.1996}. An analogous
  construction can be found in~\cite{Nikshych.1998}.}). In 2000, A.
Masuoka defined two other infinite families $A_{4n}$ and
$B_{4n}$ (see~\cite[def 3.3]{Masuoka.2000}), which we denote $\KA(n)$ and
$\KB(n)$ for notational consistency.

We present here a detailed study of the structure of those Kac
algebras, with an emphasis on $\KD(n)$. We get some structural
results: $\KQ(n)$ is isomorphic to $\KD(n)$ for $n$ even and to
$\KB(n)$ for $n$ odd, while $\KB(n)$ itself is an index $2$ Kac
subalgebra of $\KD(2n)$; also, $\KA(n)$ is the dual of
$\KD(n)$. Altogether, this is the rationale behind the emphasis on
$\KD(n)$. We prove that $\KD(n)$ and $\KQ(n)$ are self-dual if (and
only if) $n$ is odd, by constructing an explicit isomorphism (the
self-duality of $\KA(n)$ for $n$ odd and of $\KB(n)$ for all $n$ is
readily proved in~\cite{Calinescu_al.2004}). We also describe the
intrinsic groups and automorphism groups of $\KD(n)$ and $\KQ(n)$
(see~\cite{Calinescu_al.2004} for those of $\KA(n)$ and $\KB(n)$).

Then, we turn to the study of the lattice of \sacigs of those Kac
algebras. We describe them completely for $n$ small or prime, and
partially for all $n$. For $KD(n)$, we further obtain a conjectural
complete description of the lattice for $n$ odd, and explain how large
parts can be obtained recursively for $n$ even.
After this study, the lattice of intermediate subfactors of all
irreducible depth $2$ inclusions of index at most $15$ is known. We
derive the principal graphs of certain inclusions of factors;
reciprocally we use classification results on inclusions in some
proofs.

As the first interesting examples are of dimension $12$ or more,
calculations are quickly impractical by hand. Most of the research we
report on in this paper has been driven by computer exploration: it
led to conjectures and hinted to proofs. Furthermore, whenever
possible, the lengthy calculations required in the proofs were
delegated to the computer. Most of the tools we developed for this
research (about 8000 lines of code) are generic and integrated in the
open source \mupadcombinat package, and have readily been reused to
study other algebras.

In Section~\ref{inclusions}, we recall and refine the results
of~\cite{Nikshych_Vainerman.2000.2}, and adapt them to the duality
framework used in~\cite{David.2005}. In order to build the foundations
for our future work on $C^*$-quantum groupoids, the results are first
given in this general setting, at the price of some technical details
in the statements and proofs. The results which are used in the
subsequent sections are then summarized and further refined
in~\ref{irred} in the simpler context of Kac algebras.

In Section~\ref{section.KP}, we describe the complete lattice of
\sacigs of the smallest non trivial Kac algebra, $\KP$, using the
construction of~\cite{Enock_Vainerman.1996}.
In Section~\ref{graphe}, we gather some graphs that are obtained as
principal graphs of intermediate subfactors.

In Section~\ref{section.KD}, we give general results on $\KD(n)$.  We
describe its intrinsic group and its automorphism group; more
generally, we describe the embeddings of $\KD(d)$ into $\KD(n)$.
Along the way, we build some general tools for manipulating twisted
group algebras (isomorphism computations and proof strategies, \sacigs
induced by subgroups, efficient characterization of the algebra of the
intrinsic group). Those tools will be reused extensively in the later
sections.  We list the three dimension $2n$ \sacigs: $K_2$, $K_3$, and
$K_4$, and further prove that $K_2$ is isomorphic to
$L^\infty(D_{n})$. We also describe some \sacigs induced by Jones
projections of subgroups of $D_{2n}$, as well as how to obtain
recursively \sacigs of dimension dividing $2n$.

The properties of $\KD(n)$ depend largely on the parity of $n$. In
Section~\ref{section.KD.odd}, we concentrate on the case $n$ odd. We
show that $\KD(n)$ is then self-dual, by constructing an explicit
isomorphism with its dual.  We present a partial description of the
lattice $\ll(\KD(n))$ which is conjecturally complete. This is proved
for $n$ prime and checked on computer up to $n=51$. We conclude with
some illustrations on $\KD(n)$ for $n=3,5,9,15$.

In Section~\ref{section.KD.even}, we consider $\KD(2m)$. We prove that
$K_3$ and $K_4$ are Kac subalgebras.  In fact $K_4$ is isomorphic to
$\KD(m)$; $K_3$ is further studied in Section~\ref{section.KB}. It
follows that a large part of the lattice of \sacigs of $\KD(2m)$ can
be constructed inductively. We list the dimension $4$ \sacigs of
$\KD(2m)$, and describe the complete lattice of \sacigs for $\KD(4)$
and $\KD(6)$.

In Section~\ref{section.KQ}, we briefly study the algebra
$\KQ(n)$. Its structure mimics that of $\KD(n)$. In fact, we prove
that $\KQ(n)$ is isomorphic to $\KD(n)$ if and only if $n$ is even and
that $\KQ(n)$ is self-dual when $n$ is odd. We give the complete
lattice for $\KQ(n)$ for $n$ prime, and list the \sacigs of dimension
$2$, $4$, $n$, and $2n$, for all $n$. Those results are used in the previous
sections.

In Section~\ref{section.KB}, we show that the Kac subalgebra $K_3$ of
$\KD(2m)$ is isomorphic to $\KB(m)$. We summarize the links between
the families $\KA(n)$, $\KB(n)$, $\KD(n)$, and $\KQ(n)$, and the
\sacigs of dimension $4n$ in $\KD(2n)$. We list the \sacigs of
dimension $2$, $4$, $m$, and $2m$ of $K_3$ in $\KD(2m)$ and describe
the complete lattice of \sacigs for $\KD(8)$.

In Appendix~\ref{section.formulas}, we collect various large formulas
for matrix units, coproducts, unitary cocycles, etc. which are used
throughout the text, and we finish some technical calculations for the
proof of Theorem~\ref{self-dual}.

Finally, in Appendix~\ref{section.computerExploration}, we quickly
describe the computer exploration tools we designed, implemented, and
used, present typical computations, and discuss some exploration
strategies.

\bigskip
\centerline{\sc Acknowledgments}
\medskip

We would like to thank Vaughan Jones for his initial suggestion of
investigating concrete examples of lattices of intermediate
subfactors, and for fruitful questions. We are also very grateful to
Leonid Vainerman for his continuous support and his reactivity in
providing wise suggestions and helpful answers.

This research was partially supported by NSF grant DMS-0652641, and
was driven by computer exploration using the open-source algebraic
combinatorics package \texttt{MuPAD-Combinat}~\cite{MuPAD-Combinat}.

\newpage
\section[depth $2$ inclusions, intermediate subfactors and \sacigs]{Depth
  $2$ inclusions, intermediate subfactors and \sacigs}

\label{inclusions}

In this section, we recall and refine the general results of
D. Nikshych et
L. Vainerman~\cite{Nikshych_Vainerman.2000.2}~\cite{Nikshych_Vainerman.2000.2}
on depth $2$ inclusions and their Galois correspondence.

In order to build the foundations for our future work on finite
dimensional $C^*$-quantum groupoids, we do \emph{not} assume here the
inclusion to be irreducible. Except for some technicalities, this does
not affect the theory. For the convenience of the reader, the results
which are used in the subsequent sections are summarized and further
refined in~\ref{irred} in the simpler context of Kac algebras. This
mostly amounts to replacing $C^*$-quantum groupoids by Kac algebra,
and setting $h=1$.

For the general theory of subfactors, we refer
to~\cite{Jones.1983.Subfactors}, ~\cite{Goodman_Harpe_Jones.1989} et~\cite{Jones_Sunder.1997}.

\subsection{Depth $2$ inclusions and $C^*$-quantum groupoids}
\subsubsection{$C^*$-quantum groupoids}\label{gq}

A finite $C^*$-quantum groupoid is a \emph{weak} finite dimensional
$C^*$-Hopf algebra: the image of the unit by the coproduct is not
required to be $1\otimes1$, and the counit is not necessarily an
homomorphism (see~\cite{Boem_Nill_Szlachanyi.1999},
\cite{Nikshych_Vainerman.2000.1} or~\cite[2.1]{David.2005}). More
precisely:
\begin{defn}
  A \emph{finite $C^*$-quantum groupoid} is a finite dimensional
  $C^*$-algebra $G$ (we denote by $m$ the multiplication, by $1$ the
  unit, and by $^{*}$ the involution) endowed with an coassociative
  coalgebra structure (we denote by $\Delta$ the coproduct, $\varepsilon$ the counit
  and $S$ the antipode) such that:
  \begin{enumerate}[(i)]
    \item $\Delta$ is a $^{*}$-algebra homomorphism from $G$ to $G
      \otimes G$ satisfying:
      $$(\Delta \otimes \id)\Delta(1)=(1\otimes \Delta(1))(\Delta(1)\otimes 1)$$
    \item The counit is a linear map of $G$ to $\mathbb{C}$ satisfying:
      \begin{displaymath}
        \varepsilon(fgh)
        =\varepsilon(fg_{(1)})\varepsilon(g_{(2)}h)\qquad ((f,g,h)\in G^3)
      \end{displaymath}
      Equivalently:
      \begin{displaymath}
        \varepsilon(fgh)=\varepsilon(fg_{(2)})\varepsilon(g_{(1)}h) \qquad ((f,g,h)\in
		G^3))
      \end{displaymath}

    \item The antipode $S$ is an antiautomorphism of algebra and
      coalgebra of $G$ satisfying for all  $g$ in $G$:
      \begin{displaymath}
        m(\id \otimes S)\Delta(g)=(\varepsilon \otimes
	\id)(\Delta(1)(g \otimes 1))
      \end{displaymath}
      or equivalently:
      \begin{displaymath}
        m(S \otimes \id  )\Delta(g)=(\id \otimes \varepsilon)((1 \otimes
		g)\Delta(1))
      \end{displaymath}
  \end{enumerate}
\end{defn}

The \emph{target counital map} $\varepsilon_t$ and \emph{source counital map}
$\varepsilon_s$ are the idempotent homomorphisms defined respectively for all
$g\in G$ by:
\begin{displaymath}
  \varepsilon_t(g)=(\varepsilon \otimes \id)(\Delta(1)(g \otimes 1))
  \quad \text{and} \quad
  \varepsilon_s(g)=(\id \otimes \varepsilon)((1 \otimes g)\Delta(1))
\end{displaymath}

\subsubsection{$C^*$-quantum groupoid associated to an inclusion}
\label{action}

Let $N_{0} \subset N_{1}$ be a depth $2$ inclusion of
$\mathrm{II}_{1}$ factors of finite index $\tau^{-1}$. Consider the
tower
\begin{displaymath}
  N_0 \,\subset \, N_1 \, \stackrel{f_1}{\subset} \, N_2 \,\stackrel{f_2}{\subset} \, N_3 \subset\cdots
\end{displaymath}
obtained by basic construction.
We denote by $tr$ the unique normalized trace of the factors. The relative commutants
$A=N'_{0} \cap N_{2}$ and $B=N'_{1} \cap N_{3}$ are endowed with
dual finite regular\footnote{$S^2$ is the identity on the counital algebras.} $C^*$-quantum groupoid structures thanks to
the duality:
$$\langle a,b\rangle = [N_{1}:N_{0}]^{2} tr(ahf_{2}f_{1}hb)\quad (a\in
A, \;b\in B)$$ where $h$ is the square root of the index of the
restriction of the trace $tr$ to $N'_{1} \cap N_{2}$
(see~\cite{Nikshych_Vainerman.2000.1} or~\cite[3.2]{David.2005}).

The quantum groupoid $B$ acts on $N_2$ in such a way that $N_{3}$
is isomorphic to the crossed product $N_{2}\rtimes B$ and $N_1$
is isomorphic to the subalgebra of the points of $N_2$ fixed under the
action of $B$.  We recall that the inclusion is irreducible if and only if $B$ is a Kac
  algebra. We refer to~\cite{Nikshych_Vainerman.2000.1} for
historical notes on the problem and references to analogous results.

We have shown in~\cite{David.2005} that any finite $C^*$-quantum
groupoid $B$ is isomorphic to a regular $C^*$-quantum groupoid and
acts on the $II_1$ hyperfinite factor $R$.
If $A$ and $B$ are two dual finite regular $C^*$-quantum groupoids, we
can slightly modify the construction of~\cite{David.2005} so that the
obtained tower of $II_1$ hyperfinite factors $N_0 \subset N_1 \subset
N_2 \subset N_3$ endows the relative commutants $M'_0 \cap M_2$ and
$M'_1 \cap M_3$ (respectively isomorphic to $A$ and $B$ as
$C^*$-algebras) with dual finite regular $C^*$-quantum groupoid
structures which are respectively isomorphic to the original ones
(see~\cite{David.2009}).

Therefore we may, without loss of generality, reduce our study to
finite regular $C^*$-quantum groupoids associated to an inclusion.

In this context, there are convenient expressions for the counit, the
source and target counital maps, the antipodes, and the action:
for any element $b\in B$, they satisfy:
$$\varepsilon_{B}(b)=\tau^{-1}tr(hf_{2}hb), \qquad\qquad S_{B}(b) = hj_{2}(h^{-1})j_{2}(b)h^{-1}j_{2}(h),$$
$$\varepsilon^t_{B}(b)=\tau^{-1}E_{N'_1 \cap N_2}(bhf_{2}h^{-1}),\qquad \varepsilon^s_{B}(b)=\tau^{-1}E_{N'_2 \cap N_3}(j_2(b)h^{-1}f_{2}h),$$
where, for any $ n \in
\mathbb N$, $j_{n}$ is the antiautomorphism of $N'_{0} \cap N_{2n}$
defined by:
$$ j_{n}(x)=J_{n}x^*J_{n}\qquad \text{ for } x\in N'_{0} \cap N_{2n}\,,$$
where
 $J_{n}$ is the canonical anti-unitary involution of  $L^2(N_{n},tr)$,
space of the GNS construction of $N_{n}$~\cite[3.1]{David.2005}.

The application $S^2_B$ is the inner automorphism defined by $G=j_2(h^{-2})h^2$.

The action of $B$ on $N_2$ is defined by:

 $$b\triangleright x= \tau^{-1} E_{N_2}(bxhf_{2}h^{-1})\qquad \text{
 for } x\in N_2,b\in B\,.$$
There are analogous formulas for $A$.

The algebra $N_{0}$ is then the algebra of the points of $N_{1}$ which
are fixed under the action of the $C^*$-quantum groupoid $A$. The
application $\Theta_A$: $[x \otimes a] \longmapsto xa$ is a von
Neumann algebra isomorphism from $N_{1}\rtimes A$ to $N_{2}$. The
application $\Theta_B$ from $N_{2}\rtimes B$ to $N_{3}$ is defined
similarly.  In the sequel, we mostly consider the $C^*$-quantum
groupoid $B$, and then set $\Theta=\Theta_B$ for short.

\subsection{Galois correspondence}\label{galois}

In~\cite{Nikshych_Vainerman.2000.2}, D. Nikshych and L. Vainerman show
that the correspondence between finite index depth $2$ inclusions and
finite $C^*$-quantum groupoids is a Galois-type correspondence.

\subsubsection{Involutive left coideal subalgebra}
\label{sacig}
\begin{defn}
  An \textbf{involutive left (resp. right) coideal $*$-subalgebra} (or
  \emph{\sacig} for short) $I$ of $B$ is a unital $C^*$-subalgebra of
  $B$ such that $\Delta(I) \subset B \otimes I$ (resp. $\Delta(I)
  \subset I \otimes B$).

  As in~\cite[3.1]{Nikshych_Vainerman.2000.2}, the lattice of left
  (resp. right) \sacig of $B$ is denoted by $\ll(B)$ (resp. $\rl(B)$).
\end{defn}

\subsubsection{Cross product by a \sacig}

By the definition of a \sacig, the image of $N_2\otimes I$ by $\Theta$
is a von Neumann subalgebra $M_3$ of $N_3$, called \textit{cross
  product of $N_2$ by $I$} and denoted by $M_3=N_2\rtimes I$.

\subsubsection {Intermediate subalgebras of an inclusion}
\label{inter}

Let $\ll(N_{2}\subset N_{3})$ be the lattice of intermediate
subalgebras of an inclusion $N_{2} \subset N_{3}$, that is von Neumann
algebras $M_3$ such that $N_{2}\subset M_3\subset N_{3}$.

In theorem 4.3 of~\cite{Nikshych_Vainerman.2000.2}, D. Nikshych and
L. Vainerman establish an isomorphism between $\ll(N_{2}
\subset N_{3})$ and $\ll(B)$. More precisely, if $M_3$ is an
intermediate von Neumann subalgebra of $N_{2} \subset N_{3}$, then
$I=N'_1 \cap M_3$ is a \sacig of $B$. Reciprocally, if $I$ is a \sacig
of $B$, then $M_3= N_{2}\rtimes I$ is an intermediate von Neumann
subalgebra of $N_{2} \subset N_{3}$.

In~\cite[4.5]{Nikshych_Vainerman.2000.2}, they prove that the
intermediate von Neumann subalgebra is a factor if and only if the
\sacig is \emph{connected}, that is $Z(I)\cap B_s$ is
trivial. \emph{From now on, we only consider connected \sacigs.}

\subsubsection{Quasi-basis}\label{quasibase}

The notion of quasibasis is defined in~\cite[1.2.2]{Watatani.1990} and
made explicit in the special case of depth $2$ inclusions
in~\cite[3.3]{David.2005}.

\begin{prop}
  Let $\{c_t\suchthat t\in T\}$ be a family of matrix units of $I$, normalized
  by $tr(c^{*}_t c_t')=\delta_{t,t'}$. Then, $\{c_t h^{-1} \suchthat t\in T\}$
  is a quasibasis of $M_3 \cap N'_1$ over $N_2\cap N'_1$, and
  therefore a quasibasis of $M_3$ over $N_2$.

  The index $\tau^{-1}_I$ of $N_2$ in $M_3$ is then $\sum_{t\in T} c_t H^{-1}c ^*_t$.
\end{prop}

\begin{proof}
  By~\cite[2.4.1]{Watatani.1990}, the family $\{c_t h^{-1}\suchthat t\in T\}$
  is a quasibasis of $M_3 \cap N'_1$ over $N_2\cap N'_1$. Since $M_3$
  is generated by $N_2$ and $I$, this family is a quasibasis of $M_3$
  over $N_2$ (see also~\cite[3.3]{David.2005}). Since the trace of the
  factors is a Markov trace, the index of the conditional expectation
  of $M_3$ over $N_2$ is the scalar $\tau^{-1}_I=\sum_{t\in T} c_t
  H^{-1}c ^*_t$~\cite[2.5.1]{Watatani.1990}.
\end{proof}

\begin{cor}
  If the inclusion $N_0 \subset N_1$ is irreducible, the family
  $\{c_t\suchthat t\in T\}$ of normalized matrix units of $I$ is a
  Pimsner-Popa basis of $M_3$ over $N_2$, and the dimension of $I$
  coincides with the index of $N_2$ in $M_3$.
\end{cor}
\begin{proof}
  In this case, $h=1$. Therefore $[M_3:N_2]$ is the dimension of $I$,
  and the family $\{c_t\suchthat t\in T\}$ is a Pimsner-Popa basis of $M_3$
  over $N_2$ (see~\cite[1.3]{Pimsner_Popa.1986} and~\cite[5.2.1]{David.1996}).
\end{proof}

\subsection{The Jones projection of a \sacig}
\label{pI}

\subsubsection {Restriction of the counit to $I$}
\label{xI}

It has been noticed in~\cite{David.2005} that the counit, the Haar
projection of $B$, and the Jones projection are closely related. This
phenomenon also occurs for \sacigs.
\begin{prop}[{Proposition~\cite[3.5]{Nikshych_Vainerman.2000.2}}]
  The restriction of $\varepsilon$ to $I$ is a positive linear form on
  $I$. There therefore exists a unique positive element $x_I$ of $I$
  such that $\varepsilon(b)=tr(x_Ib)$ for all $b$ in $I$. Then,
  $$\Delta(x_I)= \sum_{t\in T} S(c_t^{*})G \otimes c_t \quad
  \text{et}\quad S(x_I)=x_IG^{-1}\,.$$
\end{prop}
Note that we use here a single trace, namely the normalized Markov
trace which is that of the factors.

\subsubsection {The Jones projection of $I$}
\label{defpI}

We now extend Proposition 4.7 of~\cite{Nikshych_Vainerman.2000.2}:
\begin{prop}
  \label{proposition.delta_pI}
  The element $p_I=\tau_I h^{-1}x_Ih^{-1}$ is a projection of $I$ of
  trace $\tau_I$ which we call \textit{Jones projection of $I$}. Furthermore:
  \begin{align*}
    p_I&=\tau_I \tau^{-1} E_I(f_2)\\
    \Delta(p_I)&= \tau_I \sum_{t\in T} h^{-1}S(H^{-1}c_t^{*})h \otimes c_t=\tau_I \sum_{t\in T} j_2(c_t^{*}h^{-1})\otimes c_t h^{-1}\\
    j_2(p_I)&=p_I \quad \text{et}\quad
    S(p_I)=G^{1/2}p_IG^{-1/2}
  \end{align*}
\end{prop}
\begin{proof}
  Note that, in the proof of Proposition 4.7
  of~\cite{Nikshych_Vainerman.2000.2}, $\varepsilon_s(H^{-1}x_I)$ is
  exactly $\tau_I^{-1}$. That $p_I$ is a projection follows from the
  same proof. The equality $p_I=\tau_I \tau^{-1} E_I(f_2)$ follows
  from the uniqueness of $x_I$, and yields the trace of $p_I$. The
  counit restricted to $I$ can then be expressed as
  $\varepsilon(y)=\tau^{-1}_I tr(yhp_Ih)$, for all $y\in I$.

  The coproduct formula is obtained using~\ref{quasibase}
  and~\ref{xI}. The formula of~\ref{action} for the antipode gives the
  second expression which also involves the quasibasis
  (see~\ref{quasibase}).

  From $S(x_I)=x_IG^{-1}$ and~\ref{action}, one deduces
  successively
  $$j_2(x_I)=h^{-1}j_2(h)x_Ih^{-1}j_2(h)\,, \quad j_2(p_I)=p_I\,, \quad\text{ and }\quad S(p_I)=G^{1/2}p_IG^{-1/2}\,.$$
\end{proof}
\begin{remark}
  \label{remark.jones_left_legs}
  The projection $p_I$ generates $I$ as a \sacig. Furthermore, the
  \sacig $I$ is the linear span of the right legs of the coproduct of
  $p_I$. In general, we denote by $I(f)$ the \sacig generated by a
  projection $f$.
\end{remark}

\subsubsection {The Jones projection of   $I$ is a Bisch projection}
\label{pbisch}

In~\cite{Bisch.1994}, D.~Bisch gives the following partial
characterization of Jones projections. Let $BP(N_1, N_2)$ be the set
of projections $q$ of $N'_{1} \cap N_{3}$ satisfying
\begin{description}
\item[BP1] $qf_2=f_2$;
\item[BP2] $E_{N_2}(q)$ is a scalar;
\item[BP3] $E_{N_2}(q f_1f_2)$ is the multiple of a projection.
\end{description}
Then, $BP(N_1, N_2)$ contains all the Jones projections of the
intermediate subfactors of $N_{1} \subset N_{2}$ and is contained in
the set of all the Jones projection of the intermediate von Neumann
subalgebra of $N_{1} \subset N_{2}$.

Using this result, D. Nikshych et L. Vainerman show that $p_I$ is the
Jones projection\footnote{This motivates
  our terminology \emph{Jones projection of $I$}.}
 of the intermediate inclusion $M_1=\{p_I\}' \cap N_2
\subset N_2$ of $N_{1} \subset
N_{2}$~\cite[Proposition~4.8]{Nikshych_Vainerman.2000.2}.

\begin{prop}(\cite[Proposition~4.8]{Nikshych_Vainerman.2000.2})
  The projection $p_I$ belongs to $BP(N_1, N_2)$. It implements the
  conditional expectation on the intermediate subfactor $M_1=\{p_I\}'\cap
  N_2$, whose index in $N_2$ is $\tau_I^{-1}$.
\end{prop}
The upcoming proof follows that of Proposition~4.8
of~\cite{Nikshych_Vainerman.2000.2}, albeit with  our notations and
duality. We identify $\Theta([z\otimes 1)]$ and $z$ for $z \in
N_2$, as well as $\Theta([1\otimes b)]$ and $b$ for $b \in B$.
\begin{proof}
  Since the source and target counits satisfy $\varepsilon_s \circ
  S=S\circ \varepsilon_t$, from
  $\varepsilon_s(H^{-1}x_I)=\tau_I^{-1}$, follows that
  $\varepsilon_t(hp_Ih^{-1})=1$. By definition of the Haar projection
  $e_B=d^{-1} hf_2h^{-1}$ of $B$, we get $hp_Ih^{-1}e_B=e_B$; this is
  property (BP1).

  We use Proposition~\ref{defpI} to compute $E_{N_2}(p_I)$, and obtain
  property (BP2); indeed:
  $$E_{N_2}(p_I)=E_{N_2}(\tau_I \tau^{-1} E_I(f_2))=\tau_I
  \tau^{-1}E_{N_2}(f_2)=\tau_I\,.$$

  From Corollaries~2.2 and~3.9 of~\cite{Nikshych_Vainerman.2000.2},
  and from the expression of the coproduct of $p_I$, we deduce that,
  for any $b \in B$, $\tau_I^{-1} b_{(1)}tr(hp_Ihb_{(2)})$ belongs to
  $I$ and coincides with $E_I(b)$. The element $E_{N_2}(p_If_1f_2)$
  belongs to $A$; using duality, \cite[3.1.3]{David.2005}, and the
  expression of $E_I$, one get:
  \begin{align*}
    \langle E_{N_2}(p_If_1f_2), b\rangle&=\tau^{-2}tr(E_{N_2}(p_If_1f_2)hf_2f_1hb)=\tau^{-1}tr(p_If_1f_2f_1hbj_2(h))\\
    &=tr(p_If_1hbj_2(h))=\tau tr(p_Ihbj_2(h))\\
    &=\tau tr(p_IhE_I(bj_2(h)))=\tau \tau_I^{-1}tr(p_Ihb_{(1)})tr(hp_Ihb_{(2)}j_2(h))\\
    &=\tau^{-1} \tau_I^{-1}\langle E_{N_2}(p_If_1f_2)h^{-1},b_{(1)}\rangle\langle h E_{N_2}(p_If_1f_2),b_{(2)}\rangle\\
    &=\tau^{-1} \tau_I^{-1}\langle E_{N_2}(p_If_1f_2)^2,b\rangle\,.
  \end{align*}
  Therefore, $\tau^{-1} \tau_I^{-1}E_{N_2}(p_If_1f_2)$ is an
  idempotent. Using $j_2(p_I)=p_I$
  and~\cite[3.1.3]{David.2005}, it is further self-adjoint:
  $$E_{N_2}(p_If_1f_2)^*=E_{N_2}(f_2f_1p_I)=E_{N_2}(p_If_1f_2)\,.$$
  By~\cite[4.2]{Bisch.1994.2}, $p_I$ implements the conditional
  expectation on the intermediate subfactor $M_1=\{p_I\}'\cap N_2$; its
  index in $N_2$ is $\tau_I^{-1}$ since $\tau_I$ is the trace of
  $p_I$.
\end{proof}
\begin{remark}

\begin{enumerate}
\item  In~\cite[Remark 4.4]{Bisch.1994}, D.~Bisch notes that BP3 can be
  replaced by BP3': $qN_2q=N_2q$. In the depth $2$ case, it is easy to
  show that one can also use instead BP3": $qAq=Aq$. Indeed, BP3"
  follows from BP3' by application of $E_{N'_0 \cap N_3}$.
  Reciprocally, if BP3" holds, then BP3' follows easily since $N_2$
  is spanned linearly by $N_1.A$ and $q$ commutes with $N_1$.
  Furthermore, for any $a\in A$, $p_Iap_I=E_{\delta(I)}(a)p_I$, where
  $\delta(I)$ is defined in~\ref{d} and calculated in~\ref{tour}~(1).
  \item Using~\ref{tour},  the projection $\tau^{-1}
    \tau_I^{-1}E_{N_2}(p_If_1f_2)$ can be described as  the Jones
    projection of $\delta(I)$; indeed, in this setting and thanks to
    ~\ref{proposition.delta_pI}, one can write:
  \begin{align*}
 \tau^{-1} \tau_I^{-1}E_{N_2}(p_If_1f_2)&=\tau^{-1} \tau_I^{-1}E_{N_2}(p_If_1E_{M_3}(f_2))=\tau_I^{-2}E_{N_2}(p_If_1p_I)\\
 &=\tau_I^{-2}E_{N_2}(E_{M_1}(f_1)p_I)=\tau_I^{-1}E_{\delta(I)}(f_1)=p_{\delta(I)}\,.
  \end{align*}
\end{enumerate}
\end{remark}

\subsubsection {Distinguished projection of $I$}
\label{distingue}

Following Definition 3.6 of~\cite{Nikshych_Vainerman.2000.2}, we
denote by $e_I$ the support of the restriction of $\varepsilon$ on
$I$, and call it the \emph{distinguished projection}\footnote{It could be called \textit{Haar projection}, since it  satisfies all the relevant properties} of
$I$. It is the minimal projection with the property $\varepsilon(e_Iye_I)=\varepsilon(y)$ for all $y$ of $I$. Furthermore, $e_I$ satisfies: $x_Ie_I=e_Ix_I=x_I$, $\varepsilon$
is faithful on $e_IIe_I$, and for all $b\in I$ one has
$be_I=\varepsilon_t(b)e_I$~\cite[Proposition~3.7]{Nikshych_Vainerman.2000.2}.

\begin{prop}
  The distinguished projection of $I$ is $e_I= E_{M_1}(H^{-1}) hp_Ih$.
\end{prop}
\begin{proof}
  Since $h$ belongs to the center of $N'_1 \cap N_2$, one has:
  $E_{M_1}(h^{-1})=E_{M_1}(h)^{-1}$.  Let us show first some
  properties of $e=E_{M_1}(H^{-1}) hp_Ih$. Since $p_I$ implements the
  conditional expectation on $M_1$, $e$ is a projection which
  satisfies $\varepsilon(eye)=\varepsilon(y)$ for  $y$ in $I$ and
  $x_Ie=ex_I=x_I$. From the formula for $\varepsilon_t$ we derive:
  \begin{align*}
    \varepsilon_t(e)&=\tau^{-1}E_{N'_1\cap N_2}(E_{M_1}(H)^{-1}hp_IHf_2h^{-1})\\
    &=\tau^{-1}E_{N'_1\cap N_2}(E_{M_1}(H)^{-1}hp_IHp_If_2h^{-1})\\
    &=\tau^{-1}E_{N'_1\cap N_2}(hf_2h^{-1})=1\,,
  \end{align*}
  as well as, for all $y \in I$:
  \begin{align*}
    \varepsilon_t(y)e&=\tau^{-1}E_{N'_1\cap N_2}(yhf_2h^{-1}E_{M_1}(H^{-1}) h)p_IE_{M_1}(H^{-1})h\\
    &=\tau^{-1}E_{N'_1\cap N_2}(yhf_2p_I)p_IE_{M_1}(H^{-1})h\\
    &=\tau^{-1} \tau_I yhE_I(f_2)p_IE_{M_1}(H^{-1})h\\
    &= yE_{M_1}(H^{-1})hp_Ih=ye\,.\\
  \end{align*}

  On the other hand, from $e_I=\varepsilon_t(e_I)e_I$ we deduce
  successively:
  \begin{displaymath}
    x_I=\varepsilon_t(e_I)x_I, \qquad
    \varepsilon_t(H^{-1}x_I)=\tau_I^{-1}, \qquad\text{and}\qquad \varepsilon_t(e_I)=1\,.
  \end{displaymath}
  Since $ee_I=\varepsilon_t(e)e_I=e_I$, we can write
  $$e_I=ee_I=e_Ie= \varepsilon_t(e_I)e=e\,.\qedhere$$
\end{proof}

\subsection{The inclusion associated with a  \sacig and its Jones tower}

\subsubsection{The antiisomorphism $\delta$}
\label{d}

In~\cite[Proposition~3.2]{Nikshych_Vainerman.2000.2}, D. Nikshych and
L. Vainerman define a lattice isomorphism from $\ll(B)$ to $\rl(B)$ by
$I \longmapsto \tilde{I}= G^{-1/2}S_B(I)G^{1/2}$.  This isomorphism is
induced by $j_2$; indeed, $\tilde{I}$ is simply $j_2(I)$. Note that
$j_2(I)$ is the linear span of the left legs of $\Delta(p_I)$, as well as the right \sacig generated by $p_I$.

Following~\cite[Proposition~3.3]{Nikshych_Vainerman.2000.2}, consider
$B$ as embedded into $A\rtimes B= N'_0 \cap N_3$, and for $I$ in
$\ll(B)$, denote by $\delta(I)$ the \sacig $j_2(I)'\cap A$ of $A$. Since
$j_2(I)$ is the right \sacig generated by $p_I$, $\delta(I)$ is the
commutant of $p_I$ in $A$.

The map $\delta$ is a lattice antiisomorphism from $\ll(B)$ to
$\ll(A)$~\cite[Proposition~3.3]{Nikshych_Vainerman.2000.2}.  In
particular, $\ll(B)$ is self-dual whenever $B$ is. In this case, we
identify $\delta(I)$ and its image in $B$ via some isomorphism from
$A$ to $B$. Note that this is slightly abusive, as this is only
defined up to an automorphism of $B$.

\subsubsection{The Jones tower}\label{tour}

We keep the notations of~\ref{pI} and study the inclusion $M_1 \subset
N_2$ whose Jones projection is $p_I$.

\begin{prop}\

\begin{enumerate}
  \item The algebra $M_1=\{p_I\}' \cap N_2$ is the subalgebra of $N_2$
    obtained by cross product of $N_1$ by $\delta(I)=N'_0 \cap M_1$;
  \item Let $N_2^I$ be the subalgebra of the points of $N_2$ which are
    fixed under the action of $I$; then, $M_1=h^{-1}N_2^Ih$;
  \item $M_1=N_1\rtimes\delta(I) \subset N_2=N_1\rtimes A \subset
    M_3=N_2\rtimes I$ is the basic construction. The relative
    commutant $M'_1 \cap M_3$ is $I\cap j_2(I)$.

  \item The indices satisfy the following relations:

  $[N_2\rtimes I:N_2]=[N_2:N_2^I]$ and
  $[N_1\rtimes \delta(I):N_1][N_2\rtimes I:N_2]=[N_2\rtimes B:N_2]$.
  \item For any two \sacigs $I_1$ and $I_2$,
  $$I_1 \subset I_2\quad \Leftrightarrow \quad p_{I_2} \leq p_{I_1}\,.$$
  \end{enumerate}
\end{prop}
\begin{proof}\

\begin{enumerate}
  \item The element of $\ll(A)$ associated to $M_1$ by the Galois
    correspondence (see~\ref{inter}) is $\delta(I)$; indeed:
    $$N'_0 \cap M_1= N'_0 \cap \{p_I\}' \cap N_2=\{p_I\}' \cap A= \delta(I)\,.$$
  \item Let $y= hzh^{-1}$, with $z \in M_1$.  Then, for $b\in I$,
    \begin{displaymath}
      \begin{array}{lclcl }
        b\triangleright y & = &\tau^{-1} E_{N_2}(bhzf_{2}h^{-1})
        &  =& \tau^{-1} E_{N_2}(bhzp_If_{2}h^{-1}) \\
        &=& \tau^{-1} E_{N_2}(bhp_Izf_{2}h^{-1})
        &=& \tau^{-1}\tau_I E_{N_2}(bx_Ih^{-1}zf_{2}h^{-1}) \\
        & =&\tau^{-1}\tau_I E_{N_2}(be_Ix_Ih^{-1}zf_{2}h^{-1})
        & =&\tau^{-1}\tau_I E_{N_2}(\varepsilon^t_B(b)e_Ix_Ih^{-1}zf_{2}h^{-1}) \\
        &=& \tau^{-1}\tau_I \varepsilon^t_B(b)E_{N_2}(e_Ix_Ih^{-1}zf_{2}h^{-1})
        &= &\tau^{-1} \varepsilon^t_B(b)E_{N_2}(hp_Izf_{2}h^{-1}) \\
        &=& \tau^{-1} \varepsilon^t_B(b)E_{N_2}(hzf_{2}h^{-1}) &=&\varepsilon^t_B(b)\triangleright y\,.
      \end{array}
    \end{displaymath}
    Therefore,  $M_1$ is contained in $h^{-1}N_2^Ih$.\\
    Take reciprocally $y \in h^{-1}N_2^Ih$; then, the following
    sequence of equality holds:
    \begin{align*}
      \label{}
      p_Ih^{-1}\triangleright hy h^{-1} & =  \varepsilon^t_B(p_Ih^{-1}) \triangleright (hy h^{-1})\\
      \tau^{-1} E_{N_2}(p_Iyf_{2})h^{-1}  & = \tau^{-1} E_{N'_1\cap N_2}(p_If_{2})yh^{-1}\\
      \tau^{-1} E_{N_2}(p_Iyp_If_{2})& =  y\\
      \tau^{-1} E_{M_1}(y)E_{N_2}(p_If_{2})& =  y\\
      E_{M_1}(y)& =  y
    \end{align*}
    Therefore, $y$ belongs to $M_1$ and (2) is proved.
  \item The algebra $M_3=\langle N_2,p_I \rangle$ obtained by basic
    construction from $M_1\subset N_2$ is, using (1):
     $$J_2 M'_1 J_2=J_2 (N_1 \cup \delta(I))' J_2=J_2 N'_1
    \cap \delta(I)'J_2=(J_2 N'_1J_2)
    \cap (J_2\delta(I)'J_2)$$
    On one hand $J_2 N'_1 J_2$ is $N_3$, and on the other hand,
    $\delta(I)'$ contains both $N'_2$ and $\delta(I)'\cap B$.
    Therefore, by~\cite[3.3]{Nikshych_Vainerman.2000.2}, $M_3$
    contains both $N_2$ and $I$, while being also contained in $N_2
    \rtimes I$. Hence $M_3=N_2 \rtimes I$. Furthermore,
    $$I\cap \tilde{I}=I\cap j_2(I)=(N'_1\cap M_3) \cap (M'_1\cap N_3)= M'_1 \cap M_3\,.$$

  \item Follows from (2), (3), and Proposition 2.1.8 of~\cite{Jones.1983.Subfactors}.

  \item Y. Watatani shows at the beginning of~\cite[Part
    II]{Watatani.1996} that if $M$ and $P$ are two intermediate
    subfactors of $N_1\subset N_2$, then $M$ is contained in $P$ if and
    only if $e^{N_2}_{M}$ is dominated by $e^{N_2}_{P}$. Let us apply
    this result to $P=J_2(N_2\rtimes I_1)'J_2$ and $M=J_2(N_2\rtimes
    I_2)'J_2$, where $I_1$ and $I_2$ are two \sacigs; then, $p_{I_2}
    \leq p_{I_1}$ is equivalent to $M \subset P$ and therefore to
    $I_1\subset I_2$ since $\delta$ is an antiisomorphism.
  \end{enumerate}
\end{proof}

\subsection{The principal graph of an intermediate inclusion}
\label{grapheprincipal}

The principal graph of an inclusion is  obtained from
the Bratelli diagram of the tower of relative commutants or from the
equivalence classes of simple bimodules
(see~\cite{Goodman_Harpe_Jones.1989}, and~\cite{Jones_Sunder.1997}). It is
an invariant of the inclusion.

By~\cite[5.9]{Nikshych_Vainerman.2000.2}, the principal graph of the
inclusion $N_2 \subset N_2\rtimes I$ is the connected component of the
Bratteli diagram of $\delta(I) \subset A$ containing the trivial
representation of $A$.
All the principal graphs we obtained
in our examples are gathered in Section~\ref{graphe}.

\subsection{Depth $2$ intermediate inclusions}

\subsubsection{$C^*$-quantum subgroupoid}\label{sgq}

A \sacig $I$ is a \emph{$C^*$-quantum subgroupoid} of $B$ if it is
stabilized by the coproduct and antipode of $B$; then, it becomes a
$C^*$-quantum groupoid for the structure induced from $B$. The
following proposition provides some equivalent characterizations.
\begin{prop}
  For $I$ a \sacig of $B$, the following are equivalent:
  \begin{enumerate}[(i)]
  \item $I$ is a  $C^*$-quantum subgroupoid of $B$;
  \item $I$ is stabilized by the antipode $S$;
  \item $I$ is stabilized by the antiautomorphism $j_2$;
  \item $I$ is stabilized by $\Delta$: $\Delta(I)\subset I\otimes I$;
  \item $\Delta^{\varsigma}(p_I)=(h^{-1}j_2(h) \otimes 1)\Delta(p_I)(h^{-1}j_2(h) \otimes 1)$.
  \end{enumerate}
\end{prop}
\begin{proof}
  By definition, (i) is equivalent to (ii) and (iv) together.

  (ii) $\Longleftrightarrow$ (iii): It is sufficient to remark that if
  $I$ is stabilized by $S$ or $j_2$, then $j_2(h)$ belongs to $I$. It
  remains to use the relation between $S$, $h$ and $j_2$
  (see~\ref{action}).

  (iii) $\Longrightarrow$ (v): Assume that $I$ is stabilized by
  $j_2$. Then $\{j_2(c^*_t)\suchthat t\in T\}$ is a family of normalized
  matrix units of $I$; using it, one can write the coproduct of $p_I$
  and check the claimed relation between $\Delta(p_I)$ et
  $\Delta^{\sigma}(p_I)$.

   (v) $\Longrightarrow$ (iii): Assume that
   $\Delta^{\sigma}(p_I)=(h^{-1}j_2(h) \otimes 1)\Delta(p_I)(h^{-1}j_2(h) \otimes 1)$. Then the left and right legs  of $\Delta(p_I)$ span the same subspace, which is $I=j_2(I)$.

 (ii) $\Longrightarrow$ (iv):
   $$\Delta(I)=(S^{-1} \otimes S^{-1})\Delta^{\sigma}(S(I)) \subset S^{-1}(I)\otimes I=I\otimes I\,,$$

  (iv) $\Longrightarrow$ (iii): Assume $\Delta(I)\subset I\otimes
  I$. Then, $\Delta(p_I)$ is in $I\otimes I$, and by
  Proposition~\ref{proposition.delta_pI}, $I$ is stabilized by
  $j_2$.
\end{proof}

\subsubsection{$C^*$-quantum subgroupoid and intermediate inclusions}\label{prof2}

From~\ref{tour}~(3) and~\ref{sgq}, a \sacig $I$ is a $C^*$-quantum
subgroupoid of $B$, if and only if $I=M'_1 \cap M_3$. If $I$ is a
$C^*$-quantum groupoid, the inclusion $N_2 \subset N_2 \rtimes I$ is
of depth $2$. %
In~\ref{irredprof2}, we will show that the reciprocal holds as soon as
the inclusions are irreducible, while the following example shows that
it can fail without the irreducibility condition.

\subsubsection{An example}
  As in~\cite[2.7]{Nikshych_Vainerman.2002} and~\cite[5]{David.2005},
  we consider the Jones subfactor of index $4\cos^2\frac{\pi}{5}$ in $R$. We
  denote by $P_0\subset P_1$ this inclusion whose principal graph is
  $A_4$ and write
  $$P_0\subset P_1\subset P_2\subset P_3\subset P_4\subset P_5\subset P_6\subset P_7\subset P_8\subset P_9$$
  the tower obtained by basic construction. The relative commutants
  are minimals; they are the Temperley-Lieb
  algebras. From~\cite[4.1]{Nikshych_Vainerman.2000.2}, the inclusions
  $P_0\subset P_2$ and $P_0\subset P_3$ are of depth $2$. Consider
  $P_8$ as intermediate subfactor of $P_6\subset P_9$; it is obtained by
  cross product of $P_6$ by a \sacig $I=P'_3\cap P_8$ of $P'_3 \cap
  P_9$. The algebra $P'_3 \cap P_9$ is endowed with a $C^*$-quantum
  groupoid structure since the inclusion $P_0\subset P_3$ is of depth
  $2$. On the other hand, since the inclusion $P_0\subset P_2$ is of
  depth $2$, this is also the case for $P_6\subset P_8$. However, by
  dimension count, $P'_4 \cap P_8$ is strictly contained in $I$.

\subsection{The $C^*$-quantum groupoids of dimension $8$}

The smallest non trivial $C^*$-quantum groupoids are of dimension
$8$. They are described in ~\cite[3.1]{Nikshych_Vainerman.2000.3} and
in~\cite[7.2]{Nikshych_Vainerman.2000.1}. L. Vainerman pointed to us
that, since their representation categories are the same as for
$\mathbb{Z}_2$, their intermediate subfactors lattices are trivial.
Together with the Kac-Paljutkin algebra $\KP$, this covers all the non
trivial $C^*$-quantum groupoids of dimension~$8$.

\subsection{Special case of irreducible inclusions}
\label{irred}

In this article, we concentrate on Kac algebras, and therefore on
irreducible inclusions of finite index. In this case, any intermediate
inclusion is of integral index (see~\cite{Bisch.1994.2}); in fact, as
we showed in~\ref{quasibase}, the index of $N_2 \subset N_2 \rtimes I$
is the dimension of $I$. Furthermore, the lattice of intermediate
subfactors is finite (see~\cite{Watatani.1996}).

\subsubsection{Kac algebras and $C^*$-Hopf algebras}
\label{KacHopf}

This is a summary of~\cite[6.6]{Enock_Schwartz.1992} which relates Kac
algebras and $C^*$-Hopf algebras. Let $B$ be a $C^*$-algebra of finite
dimension\footnote{Recall that any $C^*$-algebra of finite dimension
  is semi-simple.}. Set $B=\oplus_j M_{n_j}(\C)q_j$ and denote by
$Tr_j$ the canonical trace\footnote{$Tr_j(q_j)=n_j$.} on the factor
$M_{n_j}(\C)q_j$. The normalized canonical trace of $B$ is given by
the formula $tr(x)=\frac{1}{\dim B}\sum_j n_j Tr_j(xq_j)$.

If $(B, \Delta, \varepsilon, S)$ is a $C^*$-Hopf algebra, then $(B,
\Delta, S, tr)$ is a Kac algebra. Reciprocally, if $(B, \Delta, S,
tr)$ is a Kac algebra, then, from \cite[6.3.5]{Enock_Schwartz.1992},
there exist a projection $p$ of dimension $1$ in the center of $B$
such that $\varepsilon$, defined by $\varepsilon(x)=tr(xp)$, is the
counit of $B$, and $(B, \Delta, \varepsilon, S)$ is a  $C^*$-Hopf
algebra.

The coinvolution, or antipode, of a Kac algebra constructed from a
depth $2$ irreducible inclusion coincides with $j_2$
(voir~\cite{David.1996}). From~\cite[III.3.2]{Kassel.1995}, the
antipode of a $C^*$-algèbre de Hopf is uniquely determined by the rest
of the structure. Also, in the examples we consider, the coinvolution
is not deformed. Therefore, we often skip in the sequel the obvious
checks on the coinvolution.

\subsubsection{}
\label{resumep}

The following proposition summarizes the properties of \sacigs and
their Jones projections in the irreducible case:
\begin{prop}
  \label{prop.resumep}
  Assume that the inclusion $N_0 \subset N_1$ is irreducible, and
  consider a \sacig $I$ in $\ll(B)$. Then, $I$ is connected, and
  \begin{enumerate}
  \item $\tau_I^{-1} = tr(p_I)^{-1}=[N_2\rtimes I:N_2]=\dim I$, and
    $\dim I$ divides $\dim B$.
  \item The Jones projection $p_I$ dominates $f_2$, belongs to the
    center of $I$, and satisfies:
    \begin{displaymath}
      Ip_I=\C p_I \quad \text{ and } \quad \varepsilon(b)=\dim I
      \; tr(p_Ib), \text{ for all } b\in I\,.
    \end{displaymath}
  \item The \sacig $I$ is the subspace spanned by the right legs of
    $p_I$, and
    $$\Delta(p_I)= (\dim I)^{-1}\sum_{t\in T} S(c_t^{*}) \otimes c_t\,.$$
  \item $\displaystyle\Delta(p_I)(1\otimes p_I)=(p_I\otimes p_I)\,.$

  \end{enumerate}
\end{prop}
\begin{proof}
 (1) follows from corollary~\ref{quasibase} and~\ref{tour}~(4).

 (2) The counital algebras are trivial, the element $h$ is $1$, the
 projections $e_I$ and $p_I$ coincide with $(\dim I)^{-1} x_I$.  Take
 $y\in I$. From~\cite[3.7]{Nikshych_Vainerman.2000.2} (see
 also~\ref{distingue}), one has $yp_I=\varepsilon(y)p_I$, and it
 follows successively that $p_Iy^*=\overline{\varepsilon(y)}p_I$, and
 $p_Iy=\overline{\varepsilon(y^*)}p_I=\varepsilon(y)p_I$. Therefore
 $p_I$ is in the center of $I$.

  (3) follows from~\ref{defpI}.

  (4) Since $Ip_I$ is of dimension $1$, one can set
  $c_{t_0}=\tau_I^{-1/2}p_I$ for some $t_0\in T$, and $c_tp_I=0$ for all
  $t\in T\backslash\{t_0\}$. The desired equation follows easily.
\end{proof}
An equivalent statement of Proposition~\ref{prop.resumep} (4) can be
found in~\cite{Baaj_Blanchard_Skandalis.1999}; namely, in their
setting, the image of $p_I$ in $L^2(B,tr)$ is a presubgroup.

The following remark is useful in practice when searching for \sacigs.
\begin{remark}
  \label{remark.resumep}
  Let $p$ be a projection of trace $1/k$, for some divisor $k$ of
  $\dim B$, which dominates $f_2$ and generates a \sacig $I$ of
  dimension $k$. Then, $p$ is the Jones projection $p_I$ of $I$;
  indeed, $p$ dominates $p_I$ and has the same trace.
\end{remark}

\subsubsection{Intermediate inclusions of depth $2$}
\label{irredprof2}

The following proposition characterize the depth $2$ intermediate
inclusions in the case of irreducible inclusions.

\begin{prop}
  Let $I$ be a \sacig of $B$. Then, the following are equivalent:
  \begin{enumerate}[(i)]
  \item The inclusion $N_2 \subset N_2\rtimes I$ is of depth $2$;
  \item $I$ is a Kac subalgebra of $B$;
  \item the coproduct of $p_I$ is symmetric;
  \item $I=M'_1 \cap M_3$.
  \end{enumerate}
\end{prop}
\begin{proof}
  In the irreducible case, Proposition~\ref{sgq} implies that $(ii)
  \Leftrightarrow (iii) \Leftrightarrow (iv)$; indeed,
  by~\ref{tour}~(3), (iv) is equivalent to $I=j_2(I)$.

  Furthermore (ii) implies (i).  Assume reciprocally (i): $N_2 \subset
  N_2\rtimes I$ of depth $2$. Then $M'_1 \cap M_3$ is a Kac algebra of
  dimension $[M_3:N_2]=\dim I$ contained in $I$; therefore $M'_1 \cap
  M_3=I\cap j_2(I)$ coincides with $I$, that is (iv).

\end{proof}

\subsubsection{Special case of group algebras}
\label{section.group_algebras}
\cite[4.7]{Goodman_Harpe_Jones.1989},
\cite[A.4]{Jones_Sunder.1997} et
\cite[2]{Hong_Szymanski.1996}
\label{groupe}

We now describe the special case of group algebras and their duals.

Let $G$ be a finite group. We denote by $(\mathbb C[G],\Delta_s)$ the
symmetric Kac algebra of the group, and by $L^{\infty}(G)$ the
commutative Kac algebra of complex valued functions on the
group\footnote{Recall that any finite dimensional symmetric
  (resp. commutative) Kac algebra is the algebra of a group
  (resp. the algebra of functions on a
  group)\cite{Enock_Schwartz.1992}}. Those two algebras are in
duality.

If $G$ acts outerly on a $II_1$ factor $N$, then
$N^G \subset N\subset N\rtimes G \subset (N\rtimes G)\rtimes
L^{\infty}(G)$ is the basic construction. The principal graph of $N^G
\subset N$ (resp. $N \subset N\rtimes G$) is the Bratelli diagram of
$\mathbb C \subset \mathbb CG$ (resp. $\mathbb C \subset
L^{\infty}(G)$).

The intermediate subfactors of $N \subset N\rtimes G$ are of the form
$N\rtimes H$, where $H$ is a subgroup of $G$; indeed, since the
coproduct $\Delta_s$ is symmetric, any \sacig of $\C[G]$ is a
symmetric Kac subalgebra, hence the algebra of a subgroup $H$ of $G$.
The Jones projection of $\C[H]$ is $\frac{1}{\vert H
  \vert}\sum_{h\in H} h$. Remark that $H$ is normal if and only if
$p_H$ belongs to the center of $\mathbb C[G]$.

If $H$ is commutative, from~\cite[page 714]{Vainerman.1998}, the
matrix units of $\mathbb C[H]$ are the projections
 $P_h=\frac{1}{|H|}\sum_{g\in H}\langle h,g\rangle \lambda (g)$, for $h\in \widehat{H}$,
with standard coproduct:
$$\Delta_s(P_h)=\sum_{g\in \widehat{H}} P_g\otimes P_{g^{-1}h}\,.$$
More precisely, for the Jones projection, one get using
Proposition~\ref{resumep} (3):
 $$\Delta_s(\frac{1}{|H|}\sum_{g\in H} \lambda (g))=\sum_{h\in \widehat{H}} S(P_h) \otimes P_h\,.$$

The intermediate subfactors of $N\rtimes G \subset (N\rtimes G)\rtimes
L^{\infty}(G)$ are of the form $(N\rtimes G)\rtimes I$, where $I$ is a
\sacig of $L^{\infty}(G)$; they are
obtained by basic construction from the inclusion $N\rtimes H\subset
N\rtimes G$, where $\delta(I)=\mathbb C[H]$. The \sacig $I$ is then
the algebra $L^{\infty}(G/H)$ of functions which are constant on the
right $H$-cosets of $G$. Indeed $L^{\infty}(G/H)$ is an involutive
subalgebra, and by definition of the coproduct $\Delta(f)(s,t)=f(st)$,
it is a \sacig which is furthermore a Kac subalgebra if and only if
$H$ is normal. Its Jones projection is $\sum_{h\in H} \chi_h$, where
$\chi_h$ is the characteristic function of $h$.

\subsubsection{Normal \sacigs}
\label{normal}

Generalizing from group algebras and their dual (see~\ref{groupe}),
we define normal \sacigs as follow (see also~\cite{Kac.Paljutkin.1966}
and~\cite[2.2]{Nikshych.1998}).
\begin{defn}
  A \sacig $I$ of $B$ is \textbf{normal} if $I$ is a Kac subalgebra of
  $B$, and $p_I$ belongs to the center of $B$.
\end{defn}

The following characterization of normal \sacigs results
from~\ref{irredprof2} and~\ref{grapheprincipal}.
\begin{prop}
  A \sacig $I$ of $B$ is normal if and only if the inclusions $N_1 \subset
  M_1$ and $M_1 \subset N_2$ are of depth $2$ if and only if $p_I$ belongs to the center of $B$ and the coproduct of $p_I$ is symmetric.
\end{prop}

\begin{example}
In $\KD(2m)$, the \sacigs $K_j$, for $j=1,\dots,4$, are normal (see~\ref{section.KD.even}).
\end{example}

\subsubsection{Intrinsic group}
\label{intrin}

The intrinsic group $G(K)$ of a Kac algebra $K$ is the subset of
invertible elements $x$ of $K$ satisfying $\Delta(x)=x \otimes
x$~\cite[1.2.2]{Enock_Schwartz.1992}. It is in correspondence
  with the subset of characters of the dual Kac algebra of
$K$~\cite[3.6.10]{Enock_Schwartz.1992}.

The group algebra $\C[G(K)]$ of $G(K)$ is a Kac subalgebra of
$K$ which contains any symmetric Kac subalgebra of $K$. In particular,
since any inclusion of index $2$ results from a cross product by
$\mathbb{Z}_2$, any \sacig of dimension $2$ is contained in
$\C[G(K)]$.

\newpage
\section{The Kac-Paljutkin algebra}
\label{section.KP}

In this section, we study the Kac-Paljutkin algebra
$\KP$~\cite{Kac.Paljutkin.1966}. This is the unique non trivial Kac
algebra of dimension $8$, making it the smallest non trivial Kac
algebra. The study of the lattice of \sacigs of $\KP$, which does not
require heavy calculations, gives a simple illustration of the results
of~\ref{inclusions}. Furthermore, $\KP$ will often occur as Kac
subalgebra in the sequel.

\subsection{The Kac-Paljutkin algebra $\KP$}

We use the description and notations of~\cite{Enock_Vainerman.1996}.
The Kac-Paljutkin algebra is the $C^*$-algebra
\begin{displaymath}
  \KP=\C e_1 \oplus \C e_2\oplus  \C e_3 \oplus
  \C e_4 \oplus M_2(\C),
\end{displaymath}
endowed with its unique non trivial Kac algebra structure. The matrix
units of the factor $M_2(\C)$ are denoted by $e_{i,j}$ ($i=1,2$ and
$j=1,2$).
\noindent From~\ref{KacHopf}, the trace is given by: $\forall\;
i=1,2,3,4\;\; tr(e_i)=\frac{1}{8}$ and
$tr(e_{1,1})=tr(e_{2,2})=\frac{1}{4}$.

The Kac algebra structure can be constructed by twisting the group
algebra of
$$G=\{1, a, b, ab=ba, c,ac=cb,bc=ca,abc=cab\}\,.$$
We denote by $\lambda$ the left regular representation of $G$:
\begin{align*}
  \lambda(a)&=e_1 -e_2+e_3 -e_4-e_{1,1}+e_{2,2}\\
  \lambda(b)&=e_1 -e_2+e_3 -e_4+e_{1,1}-e_{2,2}\\
  \lambda(c)&=e_1 +e_2-e_3 -e_4+e_{1,2}+e_{2,1}
\end{align*}
and by $\Delta$ the coproduct of $\KP$ given explicitly
in~\cite{Enock_Vainerman.1996}.

\subsection{The \sacigs of $\KP$}

We now determine $\ll(\KP)$, using the results of~\ref{resumep}. Note
first that the Jones projection of the non trivial \sacigs are
projections of trace either $\frac{1}{2}$ or $\frac{1}{4}$ which
dominate $e_1$, the support of the counit of $\KP$ hence the Jones
projection of $\KP$.

\subsubsection{}
Since the coproduct is not deformed on
$\C[(\mathbb{Z}_2)^2]$, $\ll(\KP)$ contains the Kac subalgebras:
\begin{itemize}
\item $J_0=\C(e_1+e_3) \oplus \C(e_2+e_4) \oplus
  \C e_{1,1} \oplus \C e_{2,2}$ spanned by $1$,
  $\lambda(a)$, $\lambda(b)$, and $\lambda(ab)$;
\item $I_0=\C(e_1 +e_2+e_3 +e_4)\oplus \C(e_{1,1}
  +e_{2,2})$ spanned by $1$ and $\lambda(ab)$;
\item $I_1=\C(e_1 +e_3 +e_{1,1})\oplus \C(e_2
  +e_4+e_{2,2})$ spanned by $1$ and $\lambda(b)$;
\item $I_2=\C(e_1 +e_3 +e_{2,2})\oplus  \C(e_2
  +e_4+e_{1,1})$ spanned by  $1$ and $\lambda(a)$.
\end{itemize}
The Jones projections of those subalgebras are therefore respectively
\begin{displaymath}
  p_0=e_1+e_3,\quad p'_0=e_1 +e_2+e_3 +e_4,\quad p'_1=e_1 +e_3
  +e_{1,1},\quad \text{and} \quad p'_2=e_1 +e_3 +e_{2,2}\,.
\end{displaymath}
Since the intrinsic group of $\KP$ is $J_0$, by~\ref{intrin}, $I_0$,
$I_1$, and $I_2$ are the only dimension $2$ \sacigs.

\subsubsection{}
Since the Kac-Paljutkin algebra is self-dual, $\ll(\KP)$ is also
self-dual. We therefore can obtain the other elements of $\ll(\KP)$ by
using the antiisomorphism $\delta$.  Since $J_0$ is the group algebra
of $(\mathbb{Z}_2)^2$, the inclusion $N_2\subset N_2\rtimes J_0$ is of
index $4$, and its principal graph is $D_4^{(1)}$
(see~\ref{grapheD1}); from~\ref{grapheprincipal}, this graph is
the connected component of the Bratteli diagram of $\delta(J_0)
\subset \KP$ which contains $\C e_1$. Therefore, the Jones projection
of $\delta(J_0)$ is $e_1 +e_2+e_3 +e_4$, and $\delta(J_0)$ is $I_0$.

\subsubsection{}
Since $\delta$ is a lattice antiisomorphism, $\delta(I_1)$ and
$\delta(I_2)$ contain $I_0=\delta(J_0)$; furthermore, their Jones
projection (which dominate $e_1$) are of trace $\frac{1}{4}$. Only two
remaining projections are possible: $p_2=e_1+e_2 \quad \text{et}\quad
p_4=e_1+e_4$.

Denote by $q_k$ the projection $\frac{1}{2}\begin{pmatrix}
1&\rm{e}^{-\mathrm{i}k\pi/4}\\
\rm{e}^{\mathrm{i}k\pi/4}&1
\end{pmatrix}$ of the factor $M_2(\C)$.  We describe the
\sacig $J_2$ generated by $p_2$. From the equality
\begin{align*}
\Delta(p_{2})=(e_1+e_2)&\otimes(e_1+e_2)+(e_3+e_4)\otimes(e_3+e_4)\\&+\frac{1}{2}(e_{1,1}+e_{2,2})\otimes(e_{1,1}+e_{2,2})
+\frac{1}{2}(e_{1,2}-i e_{2,1}) \otimes (e_{1,2}+ie_{2,1})\,,
\end{align*}
we deduce that $J_2$ contains $e_3+e_4$, $e_{1,1}+e_{2,2}$ (which
we already know since $J_2$ contains $p_2$ and $I_0$) and
$u=e_{1,2}+ie_{2,1}$. It therefore contains the symmetry
$s=\rm{e}^{-i\pi/4} u$, and the projection
$q_1=\frac{1}{2}(e_{1,1}+e_{2,2}+s)$. More precisely:
\begin{displaymath}
  \Delta(p_{2})=(e_1+e_2)\otimes(e_1+e_2)+(e_3+e_4)\otimes(e_3+e_4)
  +q_7 \otimes q_1 +q_3 \otimes q_5
\end{displaymath}
It follows that:
\begin{displaymath}
  J_2=\C(e_1+e_2) \oplus \C(e_3+e_4) \oplus  \C q_1 \oplus \C q_5\,.
\end{displaymath}
Similarly, $J_4$ can be described as follow:
\begin{displaymath}
  J_4=\C(e_1+e_4) \oplus \C(e_2+e_3) \oplus  \C q_3 \oplus \C q_7\,.
\end{displaymath}

From~\ref{grapheprincipal}, the principal graph of the inclusions
$N_2\subset N_2\rtimes J_2$ and $N_2\subset N_2\rtimes J_4$
is the connected component of the Bratteli diagram of $I_j \subset
\KP,\,j=1,2$ containing $\C e_1$; in both case, it is
$D_6^{(1)}$ (see~\ref{grapheD1}).

\subsection{The lattice $\mathrm{l}(\KP)$}
\label{KP}
We have now completed the construction of $\ll(\KP)$.
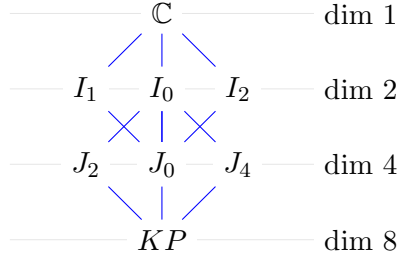
\begin{figure}[h]
  \centering
  \begin{tikzpicture}[baseline=(current bounding box.east),yscale = -1]
  \newcommand{\dimline}[2]{\draw[color=black!10,very thin](0,#1) -- (4,#1); \node[right] at (4,#1) {dim $#2$};}
  \dimline{1}{1};
  \dimline{2}{2};
  \dimline{3}{4};
  \dimline{4}{8};
  
  \tikzstyle{every node}=[fill=white]

  \node(C)  at (2,1) {$\C$ };
                   
  \node(I1) at (1,2) {$I_1$};
  \node(I0) at (2,2) {$I_0$};
  \node(I2) at (3,2) {$I_2$};
                   
  \node(J1) at (1,3) {$J_2$};
  \node(J0) at (2,3) {$J_0$};
  \node(J2) at (3,3) {$J_4$};

  \node(KP) at (2,4) {$KP$};

  \tikzstyle{every path}=[blue]

  \draw (I1) -- (C);
  \draw (I0) -- (C);
  \draw (I2) -- (C);

  \draw (J0) -- (I1);
  \draw (J0) -- (I0);
  \draw (J0) -- (I2);

  \draw (I0) -- (J1);
  \draw (I0) -- (J0);
  \draw (I0) -- (J2);
  
  \draw (J1) -- (KP);
  \draw (J0) -- (KP);
  \draw (J2) -- (KP);
\end{tikzpicture}
  \caption{The lattice of \sacigs of  $\KP$}
  \label{figure.KP}
\end{figure}
\begin{prop}
  The lattice of \sacigs of the Kac-Paljutkin algebra is as given in
  Figure~\ref{figure.KP}.
\end{prop}
\begin{proof}
  The Jones projections of the \sacigs of dimension $4$ are necessarily
  of the form $e_1+e_j$ with $j= 2,3$ or $4$, and the \sacigs of
  dimension $2$ are all contained in $J_0$.
\end{proof}

Thanks to the Galois correspondence, we can deduce the non trivial
intermediate subfactors of the inclusion $N_2 \subset N_2 \rtimes \KP$:
\begin{itemize}
\item Three factors $N_2\rtimes I_i$, $i=0,1,2$, isomorphic to $N_2
  \rtimes \mathbb{Z}_2$;
\item The factor $N_2\rtimes J_0=N_2 \rtimes (\mathbb{Z}_2)^2$ with
  principal graph $D^{(1)}_4$;
\item Two factors $N_2\rtimes J_i$, $i=1,2$, with principal graph $D^{(1)}_6$.
\end{itemize}

\subsection{Realization of $\KP$ by composition of subfactors}
\label{KPIK}

In~\cite[VIII]{Izumi_Kosaki.2002}, Izumi and Kosaki consider an
inclusion of principal graph $D_6^{(1)}$, from which they construct a
depth $2$ inclusion isomorphic to $R\subset R\rtimes \KP$. We can now
describe explicitly all the ingredients of their construction (see also~\cite{Popa.1990}).
\begin{prop}
  The inclusion $N_2 \subset N_2 \rtimes J_2$, which is of principal graph $D_6^{(1)}$, can
  be put under the form $M^{(\alpha, \Z_2)} \subset M \rtimes_\beta \Z_2$ as follows.
  Set $M=N_2 \rtimes I_0$, $v=(e_1+e_2)-(e_3+e_4)+s$ and $\beta=\Ad
  v$. Let $\alpha$ be the automorphism of $M$ which fixes $N_2$ and
  changes $\lambda(ab)$ into its opposite (this is the dual action of
  $\mathbb{Z}_2$).  Then, $\alpha$ and $\beta$ are involutive
  automorphisms of $M$ such that the outer period of $\beta \alpha$ is
  $4$, and
  \begin{displaymath}
    (\beta \alpha)^4          = \Ad \lambda(ab)  , \quad
    \beta \alpha(\lambda(ab)) = -\lambda(ab)     , \quad
    M^{\alpha}=N_2                                , \quad   \text{and }\quad
    N_2\rtimes J_2=M\rtimes_{\beta}\mathbb{Z}_2   \,.
  \end{displaymath}
\end{prop}
\begin{proof}
  As in~\ref{action} (see~\cite{David.2005}), we identify $N_3$ and
  $N_2\rtimes \KP$. Therefore, for all $x \in N_2$, one has
  $\beta(x)=vxv^*=(v_{(1)}\triangleright x)v_{(2)}v^*$.
  From $\Delta(v)=v\otimes [(e_1+e_2)-(e_3+e_4)] +v'\otimes s $, (with
  $v'=(e_1+e_4)-(e_2+e_3)+\mathrm{e}^{\mathrm{i}\pi/4}e_{1,2}+\mathrm{e}^{-\mathrm{i}\pi/4}e_{2,1}$),
  we get, for $x\in N_2$,
  $$\beta(x)= 8 [E_{N_2}(vxe_1)(e_1+e_2+e_3+e_4)+E_{N_2}(v'xe_1)(e_{1,1}+e_{2,2})]\,,$$
and  deduce that $\beta$ normalizes $M$; by a straightforward calculation, the two automorphisms  satisfy the claimed relations. Then, $N_2\rtimes J_2$ is indeed the
cross product of $M$ by $\beta$, since $\lambda(ab)$ and $v$ generate the subalgebra $J_2$.
\end{proof}

\newpage
\section{Principal graphs obtained in our examples}
\label{graphe}

Using~\ref{grapheprincipal}, the examples we study in the sequel
yield several families of non trivial principal graphs of inclusions
which we collect and name here.

\subsection{Graphs $D_{2n+1}/\mathbb{Z}_2$}
\label{grapheDimpair}

In~\ref{Himpair},~\ref{sacig4impair},~\ref{sacig4KQimpair}, and~\ref{grapheK0}, we
obtain the graphs $D_{2n+1}/\mathbb{Z}_2$, where $n$ is the number of
vertices in the spine of the graph (altogether $n+4$ vertices). The
inclusions are then of index $2n+1$. The graph
$D_{11}/\mathbb{Z}_2$ is depicted in Figure~\ref{figure.D11Z2}.
\begin{figure}[h]
  \centering
  \scalebox{.9}{\pgfdeclarelayer{background}
\pgfdeclarelayer{nodes}
\pgfsetlayers{background,main,nodes}
\begin{tikzpicture}[yscale=-1]
  \tikzstyle{every node}=[fill=white]
  \begin{pgfonlayer}{nodes}
    \node(bullet)	at (3,1) {$\bullet$};
    \node(bulleu)	at (3,2) {$\bullet$};
    \node(bullev)	at (1,3) {$\bullet$};
    \node(bullew)	at (2,3) {$\bullet$};
    \node(bullex)	at (3,3) {$\bullet$};
    \node(bulley)	at (4,3) {$\bullet$};
    \node(bullez)	at (5,3) {$\bullet$};
    \node(bullfa)	at (3,4) {$\bullet$};
    \node(bullfb)	at (3,5) {$\bullet$};
  \end{pgfonlayer}
  \draw (bullev) -- (bullew);
  \draw (bullew) -- (bullet);
  \draw (bullew) -- (bulleu);
  \draw (bullew) -- (bullex);
  \draw (bullew) -- (bullfa);
  \draw (bullew) -- (bullfb);
  \draw (bulley) -- (bullet);
  \draw (bulley) -- (bulleu);
  \draw (bulley) -- (bullex);
  \draw (bulley) -- (bullfa);
  \draw (bulley) -- (bullfb);
  \draw (bulley) -- (bullez);
\end{tikzpicture}}
  \caption{The graph $D_{11}/\mathbb{Z}_2$}
  \label{figure.D11Z2}
\end{figure}
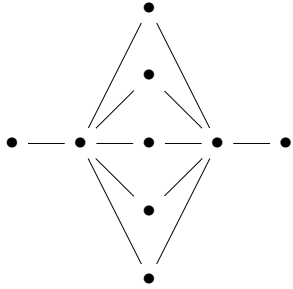

Those graphs are the principal graphs of the inclusions $R\rtimes
\mathbb{Z}_2 \subset R\rtimes D_{2n+1}$
\cite{deBoer.Goeree.1991}. Furthermore, any inclusion admitting such a
principal graph is of the form $R\rtimes H \subset R\rtimes G$ where
$G$ is a group of order $2(2n+1)$ obtained by semi-direct product of
an abelian group and $H \equiv
\mathbb{Z}_2$~\cite{Hong_Szymanski.1996}.

\subsection{Graphs $D_{2n+2}/\mathbb{Z}_2$}
\label{grapheDpair}

In~\ref{Himpair} and~\ref{grapheK0}, we obtain the graphs
$D_{2n+2}/\mathbb{Z}_2$, where $n$ is the number of vertices in the
spine (altogether $n+6$ vertices). The inclusion is then of index
$2(n+1)$. The graph $D_{12}/\mathbb{Z}_2$ is depicted in
Figure~\ref{figure.D12Z2}. Those graphs are the principal graphs of
the inclusions $R\rtimes \mathbb{Z}_2 \subset R\rtimes D_{2n+2}$
\cite{deBoer.Goeree.1991}.
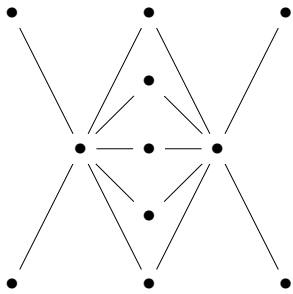
\begin{figure}[h]
  \centering
  \scalebox{.9}{\pgfdeclarelayer{background}
\pgfdeclarelayer{nodes}
\pgfsetlayers{background,main,nodes}
\begin{tikzpicture}[yscale=-1]
  \tikzstyle{every node}=[fill=white]
  \begin{pgfonlayer}{nodes}
    \node(bullet)	at (1,1) {$\bullet$};
    \node(bulleu)	at (3,1) {$\bullet$};
    \node(bullev)	at (5,1) {$\bullet$};
    \node(bullew)	at (3,2) {$\bullet$};
    \node(bullex)	at (2,3) {$\bullet$};
    \node(bulley)	at (3,3) {$\bullet$};
    \node(bullez)	at (4,3) {$\bullet$};
    \node(bullfa)	at (3,4) {$\bullet$};
    \node(bullfb)	at (1,5) {$\bullet$};
    \node(bullfc)	at (3,5) {$\bullet$};
    \node(bullfd)	at (5,5) {$\bullet$};
  \end{pgfonlayer}
  \draw (bullet) -- (bullex);
  \draw (bullfb) -- (bullex);
  \draw (bullex) -- (bulleu);
  \draw (bullex) -- (bullew);
  \draw (bullex) -- (bulley);
  \draw (bullex) -- (bullfa);
  \draw (bullex) -- (bullfc);
  \draw (bullez) -- (bulleu);
  \draw (bullez) -- (bullew);
  \draw (bullez) -- (bulley);
  \draw (bullez) -- (bullfa);
  \draw (bullez) -- (bullfc);
  \draw (bullez) -- (bullev);
  \draw (bullez) -- (bullfd);
\end{tikzpicture}}
  \caption{The graph $D_{12}/\mathbb{Z}_2$}
  \label{figure.D12Z2}
\end{figure}

\subsection{Graphs $D^{(1)}_n$}
\label{grapheD1}

We obtain the graphs $D^{(1)}_n$, with $n+1$ vertices, for subfactors
of index $4$. These graphs $D^{(1)}_n$ were readily obtained
in~\cite{Popa.1990}.
The graph $D^{(1)}_6$ (Figure~\ref{figure.Dn1}) is the principal graph
of two intermediate subfactors of $R\subset R\rtimes \KP$
(see~\ref{section.KP}).
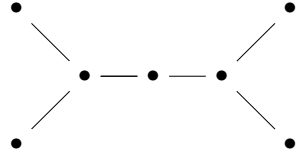
\begin{figure}[h]
  \centering
  \scalebox{.9}{\pgfdeclarelayer{background}
\pgfdeclarelayer{nodes}
\pgfsetlayers{background,main,nodes}
\begin{tikzpicture}[yscale=-1]
  \tikzstyle{every node}=[fill=white]
  \begin{pgfonlayer}{nodes}
    \node(bullet)	at (1,1) {$\bullet$};
    \node(bulleu)	at (5,1) {$\bullet$};
    \node(bullev)	at (2,2) {$\bullet$};
    \node(bullew)	at (3,2) {$\bullet$};
    \node(bullex)	at (4,2) {$\bullet$};
    \node(bulley)	at (1,3) {$\bullet$};
    \node(bullez)	at (5,3) {$\bullet$};
  \end{pgfonlayer}
  \draw (bullet) -- (bullev);
  \draw (bulley) -- (bullev);
  \draw (bullev) -- (bullew);
  \draw (bullev) -- (bullew);
  \draw (bullex) -- (bullew);
  \draw (bullex) -- (bullez);
  \draw (bullex) -- (bulleu);
\end{tikzpicture}}
  \caption{The graph $D^{(1)}_6$}
  \label{figure.Dn1}
\end{figure}
The graph $D^{(1)}_8$ is the principal graph of two intermediate
subfactors of $R\subset R\rtimes \KD(3)$ (see~\ref{KD3section}). The
graph $D^{(1)}_{10}$ is the principal graph of $R\rtimes
\delta(K_{2k})$ (with $k \in \{1,\dots,4\}$) intermediate in $R\rtimes
\widehat{\KD(4)}$ (see~\ref{KD4K2}).
The graphs $D^{(1)}_n$ also occur as principal graphs of $R\rtimes
J_k$ (with $k \in \{1,\dots,m\}$) intermediate in $R\rtimes \KD(2n+1)$
(see~\ref{sacig4impair}) and of $R\rtimes J_k$ (with $k \in
\{1,\dots,m\}$) intermediate in $R\rtimes \KQ(2n+1)$
(see~\ref{sacig4KQimpair}).

\subsection{Graphs $QB_n$}
\label{grapheQB}

In~\ref{grapheK0}, we obtain a family of graphs which we denote
$QB_n$. They have $4n+5$ vertices; $n$ of type $\star$ in each of the
three groups, and $n-1$ of type $\square$. The inclusion is then of index
$8n$. The graph $QB_2$ is depicted in Figure~\ref{figure.QB2}.
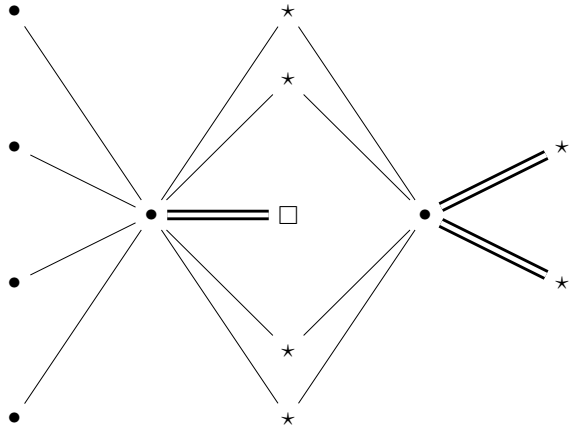
\begin{figure}[h]
  \centering
  \scalebox{.9}{\pgfdeclarelayer{background}
\pgfdeclarelayer{nodes}
\pgfsetlayers{background,main,nodes}
\begin{tikzpicture}
         \tikzstyle{every node}=[fill=white]
         \tikzstyle{emph}=[very thick, double distance=2pt]
    \node(00)  at (0,0)  {$\bullet$};
    \node(02)  at (0,2)  {$\bullet$};
    \node(04)  at (0,4)  {$\bullet$};
     \node(06)  at (0,6)  {$\bullet$};
    \node(23)  at (2,3)  {$\bullet$};
     \node(63)  at (6,3)  {$\bullet$};
     \node(46)  at (4,6)  {$\star$};
     \node(45)  at (4,5)  {$\star$};
     \node(84)  at (8,4)  {$\star$};
     \node(82)  at (8,2)  {$\star$};
      \node(43)  at (4,3)  {$\square$};
      \node(41)  at (4,1)  {$\star$};
      \node(40)  at (4,0)  {$\star$};

   \draw      (00) -- (23);
  \draw      (02) -- (23);
  \draw      (06) -- (23);
  \draw      (04) -- (23);
  \draw      (46) -- (23);
  \draw      (45) -- (23);
  \draw      (41) -- (23);
  \draw      (40) -- (23);
  \draw      (46) -- (63);
  \draw      (45) -- (63);
  \draw      (41) -- (63);
 \draw      (40) -- (63);
    \tikzstyle{every path}=[emph] 
\draw      (84) -- (63);
 \draw      (82) -- (63); 
 \draw      (43) -- (23);
     \end{tikzpicture}
   }
  \caption{The graph $QB_2$}
  \label{figure.QB2}
\end{figure}

\subsection{Graphs $DB_n$}
\label{grapheDB}

In~\ref{grapheK0}, we find a family of graphs which we denote by
$DB_n$. They have $4n+7$ vertices; $n$ of type $\star$ in each of the
three groups, and $n+1$ of type $\square$. The inclusion is then of
index $8n+4$. The graph $DB_2$ is depicted in Figure~\ref{figure.DB2}.
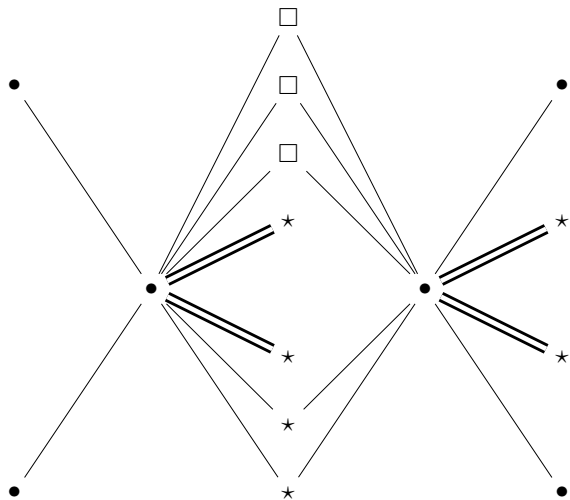
\begin{figure}[h]
  \centering
  \scalebox{.9}{\pgfdeclarelayer{background}
\pgfdeclarelayer{nodes}
\pgfsetlayers{background,main,nodes}
\begin{tikzpicture}
         \tikzstyle{every node}=[fill=white]
         \tikzstyle{emph}=[very thick, double distance=2pt]
    \node(00)  at (0,0)   {$\bullet$};
    \node(80)  at (8,0)   {$\bullet$};
    \node(86)  at (8,6)   {$\bullet$};
    \node(06)  at (0,6)   {$\bullet$};

    \node(23)  at (2,3)   {$\bullet$};
    \node(63)  at (6,3)   {$\bullet$};
    \node(47)  at (4,7)   {$\square$};
    \node(46)  at (4,6)   {$\square$};
    \node(45)  at (4,5)   {$\square$};
    \node(84)  at (8,4)   {$\star$};
    \node(82)  at (8,2)   {$\star$};
    \node(44)  at (4,4)   {$\star$};
    \node(42)  at (4,2)   {$\star$};
    \node(41)  at (4,1)   {$\star$};
    \node(40)  at (4,0)   {$\star$};

    \draw      (00) -- (23);
    \draw      (06) -- (23);
    \draw      (80) -- (63);
    \draw      (86) -- (63);
    \draw      (47) -- (23);
    \draw      (47) -- (63);
    \draw      (46) -- (23);
    \draw      (45) -- (23);
    \draw      (41) -- (23);
    \draw      (40) -- (23);
    \draw      (46) -- (63);
    \draw      (45) -- (63);
    \draw      (41) -- (63);
    \draw      (40) -- (63);
     \tikzstyle{every path}=[emph] 
    \draw(84) -- (63);
    \draw(82) -- (63); 
    \draw(42) -- (23);
    \draw(44) -- (23);
  \end{tikzpicture}}
  \caption{The graph $DB_2$}
  \label{figure.DB2}
\end{figure}

\newpage
\section{The Kac algebras $\KD(n)$ and twisted group algebras}
\label{section.KD}

In this section, we recall the definition of the family of Kac
algebras $(\KD(n))_n$, describe completely $\KD(2)$ (\ref{KD2}),
determine the automorphism group of $\KD(n)$, and obtain some general
results on the lattice of their \sacigs. The rest of the structure
depends much on the parity of $n$, and are further investigated
in~\ref{section.KD.odd} for $n$ odd and in~\ref{section.KD.even}
and~\ref{section.KB} for $n$ even.

Along the way, we build some general tools for manipulating twisted
group algebras: isomorphism computations
(Algorithm~\ref{algo.isomorphism}) and proofs
(Proposition~\ref{proposition.KacAlgebraIsomorphism}), \sacigs induced
by the group (Theorem~\ref{theorem.sacigG}), efficient
characterization of the algebra of the intrinsic group
(Corollary~\ref{corollary.intrinsicGroup}). Those tools will be reused
extensively in the later sections.

\subsection{Twisted group algebras}
\label{groupedef}

In~\cite{Enock_Vainerman.1996} and~\cite{Vainerman.1998}, M. Enock and
L. Vainerman construct non trivial finite dimensional Kac algebras by
twisting the coproduct of a group algebra by a 2-(pseudo)-cocycle
$\Omega$. More precisely, $\Omega$ is obtained from a
2-(pseudo)-cocycle $\omega$ defined on the dual $\hat{H}$ of a
commutative subgroup $H$ of $G$. We refer to those articles for the
details of the construction.

For example, the Kac-Paljutkin algebra $\KP$ studied in~\ref{KP} can
be constructed this way (see~\cite{Enock_Vainerman.1996}). Similarly,
the Kac algebras $\KD(n)$ (resp. $\KQ(n)$) of dimension $4n$ are
obtained in~\cite[6.6]{Vainerman.1998} by twisting of the algebra of
the dihedral group $\C[D_{2n}]$ (resp. of the quaternion group).

In~\cite[end of page 714]{Vainerman.1998}, L. Vainerman gives a
condition on the cocycle $\omega$ (to be counital) for $\Omega$ itself
to be counital (see definition in \cite[Lemma
6.2]{Vainerman.1998}). Then, the coinvolution exists in the twisted
algebra and the counit is well defined, and in fact left
unchanged~\cite[lemma 6.2]{Vainerman.1998}. \emph{This condition is
  always satisfied in our examples}.

\subsection{Notations}
\label{q}

Let $n$ be a positive integer, and $D_{2n}=\langle a, b \suchthat
a^{2n}=1, b^2=1, ba = a^{-1}b \rangle$ be the dihedral group of order
$4n$.  It contains a natural commutative subgroup of order $4$:
$H=\langle a^n,b\rangle$.
We denote by $\KD(n)$ the Kac algebra constructed
in~\cite[6.6]{Vainerman.1998} by twisting the group algebra
$\C[D_{2n}]$ along a $2$-cocycle of $H$. For notational convenience we
include $\KD(1)$ in the definition, even though the twisting is trivial
there since $H=D_2$. We refer to~\cite[6.6]{Vainerman.1998} for the
full construction, and just recall here the properties that we
need. Some further formulas are listed in
Appendix~\ref{section.formulas}.

The algebra structure  and the trace are left unchanged; the block matrix structure of $\KD(n)$ is given by:
\begin{displaymath}
  \KD(n)=\C e_1 \oplus \C e_2 \oplus \C e_3\oplus \C e_4\oplus   \bigoplus_{k=1}^{n-1}\KD(n)^k\,,
\end{displaymath}
where $\KD(n)^k$ is a factor $M_2(\C)$ whose matrix units we denote by
$e^{k}_{i,j}$ ($i=1,2$ and $j=1,2$).  From~\ref{KacHopf}, the trace on
$\KD(n)$ is given by: $tr(e_i)=1/4n$ and $tr(e^{k}_{j,j})=1/2n$.

We use the following matrix units in $\KD(n)^k$ :
$$r^{k}_{1,1}=\frac{1}{2}\begin{pmatrix}1 & -\mathrm{i} \\ \mathrm{i}& 1\end{pmatrix},
 \;r^{k}_{1,2}=\frac{1}{2}\begin{pmatrix}\mathrm{i} & -1 \\ -1&-\mathrm{i}\end{pmatrix},\;
 r^{k}_{2,1}=\frac{1}{2}\begin{pmatrix}-\mathrm{i} & -1 \\ -1&\mathrm{i}\end{pmatrix},\;
 r^{k}_{2,2}=\frac{1}{2}\begin{pmatrix}1 & \mathrm{i} \\ -\mathrm{i}& 1\end{pmatrix}\,,
 $$%
and the projections
\begin{displaymath}
  p^k_{1,1}=\frac{1}{2}(e^k_{1,1}+e^k_{1,2}+e^k_{2,1}+e^k_{2,2})
  \qquad \text{ and }\qquad
  p^k_{2,2}=\frac{1}{2}(e^k_{1,1}-e^k_{1,2}-e^k_{2,1}+e^k_{2,2})\,,
\end{displaymath}
as well as, following L. Vainermann, their odd and even sums:
\begin{displaymath}
  p_i = \sum_{k=1,\,k \text{ odd}}^{n-1} p^k_{i,i}, \quad i=1,2\,
  \qquad \text{ and }\qquad
  q_i = \sum_{k=1,\,k \text{ even}}^{n-1} p^k_{i,i}, \quad i=1,2\,.
\end{displaymath}
We finally consider the projections $q(\alpha, \beta,k)$ of
the $k$-th factor $M_2(\C)$ given by
\begin{displaymath}
  \frac{1}{2}\begin{pmatrix}
  1-\sin \alpha  &   \rm{e}^{-\mathrm{i}\beta}\cos \alpha \\
  \rm{e}^{\mathrm{i}\beta}\cos \alpha  & 1+\sin \alpha
\end{pmatrix}\,,
\end{displaymath}
for $\alpha$ and $\beta$ in $\R$. Then,
\begin{displaymath}
  q(\alpha, \beta,k)+q(-\alpha,\beta+\pi,k)=e^k_{1,1}+e^k_{2,2}\,.
\end{displaymath}
For the new coalgebra structure, the coproduct is obtained by twisting the symmetric
coproduct $\Delta_s$ of $\C[D_{2n}]$:
$$\Delta(x)=\Omega \Delta_s(x)\Omega^*\,,$$
where $\Omega$ is the unitary of $\C[H] \otimes \C[H]$ given
in~\ref{omega}. We set $J_0=\C[H]$.

The coinvolution and the counit are left unchanged\footnote{In most proofs, checking the properties of the coinvolution and the counit is straightforward and we omit them.}.
The left regular representation $\lambda$ of $\C[D_{2n}]$ and the
2-cocycle $\Omega$ are given explicitly in Appendix~\ref{lambda}
and~\ref{omega}, respectively. The coproduct expansions relevant to
the description of the \sacigs are computed in Appendix~\ref{cop} for
general formulas and by computer for examples
(see Appendix~\ref{subsubsection.mupad.sacig.K2}).

\subsection{The Kac algebra $\KD(2)$ of dimension $8$}
\label{KD2}

We now describe completely the Kac algebra $\KD(2)$. As in
Appendix~\ref{subsection.demo.sacig}, we check on computer that the
coproduct of $\KD(2)$ is symmetric and that its intrinsic group is
$\langle c, \lambda(b) \rangle$, where
$c=e_1-e_2+e_3-e_4-ie_{1,1}+ie_{2,2}$ is of order $4$ with
$c^2=\lambda(a^2)$.

Therefore, $\KD(2)$ is isomorphic to $\C[D_4]$, and the lattice of the
\sacigs of $\KD(2)$, depicted in Figure~\ref{figure.KD2}, coincides with
that of the subgroups of $D_4$ (see~\ref{groupe}).

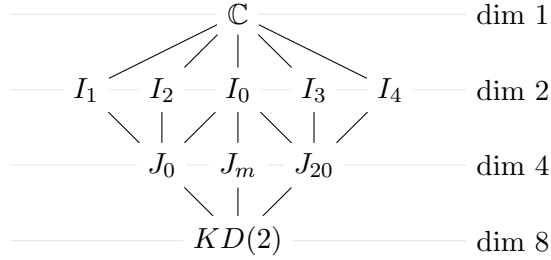
\begin{figure}[h]
  \centering
  \pgfdeclarelayer{background}
\pgfdeclarelayer{nodes}
\pgfsetlayers{background,main,nodes}
\begin{tikzpicture}[yscale=-1]
  \tikzstyle{every node}=[fill=white]
  \begin{pgfonlayer}{nodes}
    \node(C)	at (3,1) {$\C$};
    \node(I1)	at (1,2) {$I_1$};
    \node(I2)	at (2,2) {$I_2$};
    \node(I0)	at (3,2) {$I_0$};
    \node(I3)	at (4,2) {$I_3$};
    \node(I4)	at (5,2) {$I_4$};
    \node(J0)	at (2,3) {$J_0$};
    \node(J2)	at (3,3) {$J_m$};
    \node(J1)	at (4,3) {$J_{20}$};
    \node(KD2)	at (3,4) {$KD(2)$};
  \end{pgfonlayer}
  \draw (C) -- (I1);
  \draw (C) -- (I2);
  \draw (C) -- (I0);
  \draw (C) -- (I3);
  \draw (C) -- (I4);
  \draw (J0) -- (I1);
  \draw (J0) -- (I2);
  \draw (J0) -- (I0);
  \draw (J2) -- (I0);
  \draw (J1) -- (I0);
  \draw (J1) -- (I3);
  \draw (J1) -- (I4);
  \draw (J1) -- (KD2);
  \draw (J0) -- (KD2);
  \draw (J2) -- (KD2);
  \begin{pgfonlayer}{background}
    \newcommand{\dimline}[2]{\draw[color=black!10,very thin](0,#1) -- (6,#1); \node[right] at (6,#1) {dim $#2$};}
    \dimline{1}{1}
    \dimline{2}{2}
    \dimline{3}{4}
    \dimline{4}{8}
  \end{pgfonlayer}
\end{tikzpicture}
  \caption{The lattice of \sacigs of $\KD(2)$}
  \label{figure.KD2}
\end{figure}
The \sacigs of dimension $2$ are:
\begin{itemize}
\item $I_0=\C(e_1 +e_2+e_3 +e_4)\oplus \C(e_{1,1}+e_{2,2})$ generated
  by $\lambda(a^2)$;
\item $I_1=\C(e_1 +e_4+p_2)\oplus \C(e_2 +e_3+p_1)$ generated by
  $\lambda(ba^2)$;
\item $I_2=\C(e_1 +e_4+p_1)\oplus \C(e_2 +e_3+p_2)$ generated by
  $\lambda(b)$;
\item $I_3=I(e_1 +e_2+r^1_{2,2})$ generated by $\lambda(b)c$;
\item $I_4=I(e_1 +e_2+r^1_{1,1})$  generated by $\lambda(b)c^3$.
\end{itemize}

The \sacigs of dimension $4$ are:
\begin{itemize}
\item $J_0=\C(e_1+e_4) \oplus \C(e_2+e_3) \oplus \C p_1 \oplus \C p_2$
  generated by $\lambda(a^2)$ and $\lambda(b)$;
\item $J_{20}=\C(e_1+e_2) \oplus \C(e_3+e_4) \oplus \C(r^{1}_{1,1})\oplus
  \C(r^1_{2,2})$ generated by $\lambda(a^2)$ and $\lambda(b)c$;
\item $J_m=\C(e_1+e_3) \oplus \C(e_2+e_4) \oplus \C(e^1_{1,1}) \oplus
  \C(e^1_{2,2})$ generated by $c$.
\end{itemize}
The rationale behind the notations will become clearer later on.

\subsection{\Sacigs induced by subgroups}

\subsubsection{\Sacigs containing $H$}\label{sacigG}

\begin{theorem}
  \label{theorem.sacigG}
  Let $K$ be a finite dimensional Kac algebra obtained by twisting a
  group algebra $\C[G]$ with a $2$-(pseudo)-cocycle $\Omega$ of a
  commutative subgroup $H$. Then, the \sacigs of $K$ containing $H$
  are in correspondence with the subgroups of $G$ containing
  $H$. Namely, the \sacig $I_{G'}$ corresponding to the subgroup $G'$
  is obtained by twisting the group algebra $\C[G']\subset\C[G]$ with
  $\Omega$; in particular, its Jones projection is still $p_G = \frac1{|G'|}
  \sum_{g'\in G'} \lambda(g')$.
\end{theorem}
\begin{proof}
  Recall that the twisted coproduct of $K$ is defined by:
  $$\Delta(x)=\Omega \Delta_s(x)\Omega^*\,,$$
  where $\Delta_s(x)$ is the standard coproduct of the group algebra
  $\C[G]$; recall also that $\Omega$ is a unitary of $H\otimes H$.
  Let $G'$ be a subgroup of $G$ containing $H$, and $I=\C[G']$. Then,
  $\Delta(I) \subset \Omega(I\otimes I) \Omega^* = I\otimes I$.
  Therefore $I$ is the Kac subalgebra of $K$, obtained by twisting
  $\C[G']$ with the cocycle $\Omega$. In particular, its trace,
  counit, and Jones projections are the same as for $\C[G']$
  (see~\ref{groupedef}, and~\ref{section.group_algebras} for the formula).

  Reciprocally, let $I$ be a \sacig of $K$ containing $H$. We may
  untwist the coproduct on $I$
  $$\Delta_s(x)=\Omega^* \Delta(x)\Omega\,.$$
  Then, $\Delta_s(I) \subset K \otimes I$, making $I$ into a
  \sacig of $\C[G]$. Therefore, $I$ is the algebra of some subgroup
  $G'$ of $G$.

 \end{proof}

\begin{corollary}
  \label{corollary.intrinsicGroup}
  Let $K$ be a finite dimensional Kac algebra obtained by twisting a
  group algebra $\C[G]$ with a $2$-(pseudo)-cocycle $\Omega$ of a
  commutative subgroup $H$.  The intrinsic group $G(K)$ of $K$
  contains $H$. Therefore, its group algebra coincides with the
  algebra of the subgroup $G'$ of those elements of $G$ such that
  $\Delta(g)$ is symmetric.
\end{corollary}
\begin{remark}
  This gives a rather efficient way to compute the order of the
  intrinsic group of $K$ (complexity: $|G|$ coproduct computations,
  that is $O(|G||H|^4)$). Beware however that $G'$ need not coincide
  nor even be isomorphic to $G(K)$; consider for example the Kac
  algebras $\KD(2)$ and $\KQ(2)$ which are isomorphic (see~\ref{iso.KDKQ}).
  Therefore, to compute the actual intrinsic group, and not only its
  group algebra, one still needs to compute the minimal idempotents of
  the (commutative) dual algebra of $(\C[G'], \Delta)$.
\end{remark}

\subsubsection{Embedding of $\KD(d)$ into $\KD(n)$}
\label{plonge}

We now apply the previous results to $\KD(n)$.
\begin{corollary}
  \label{corollary.plonge}
  Let $n$ be a positive integer and $d$ a divisor of $n$. Then $\KD(d)$
  embeds as a Kac subalgebra into $\KD(n)$ via the morphism defined by
  $\varphi_k(\alpha)=a^k$ and $\varphi_k(\beta)=b$, where $k=\frac n
  d$, and $\alpha$ and $\beta$ denote the generators of $\KD(d)$.

  Furthermore, all the \sacigs of $\KD(n)$ containing $H$ are obtained
  this way.
\end{corollary}
\begin{proof}
  The subgroups of $D_{2n}$ containing $H$ are the $G_k =\langle a^k,
  b\rangle$ for $k$ as above. Furthermore, by
  Theorem~\ref{theorem.sacigG}, $(\C[G_k], \Delta)$ is a Kac
  subalgebra of $\KD(n)$; it is obviously isomorphic to $\KD(d)$ (the
  unitary $\Omega$ is constructed from the same elements $a^n=(a^k)^d$
  and $b$ and the same cocycle).
\end{proof}
Note that $\KD(d)$ could alternatively be embedded into $\KD(n)$ by
choosing any $k$ of the form $\frac n d (1+2k')$; this gives the same subalgebra of
$\KD(n)$ and amounts to compose the embedding with some automorphism of
$\KD(n)$ (see~\ref{intrinseque}). In fact, under the assumption that
$n\geq3$, the automorphism group of $\KD(n)$ stabilizes the embedding
of $\KD(d)$.

\begin{example}
  When $n=2m$, $\KD(n)$ contains the Kac subalgebras $K_0=\C[a^m,b]$
  isomorphic to $\KD(2)$, and $K_4=\C[a^2,b]$ isomorphic to
  $\KD(m)$. Those subalgebras are studied in~\ref{section.KD.even}.
\end{example}

\subsubsection{Intrinsic group of $\KD(n)$ and its dual}
\label{intrinseque}

The intrinsic group of the dual Kac algebra $\widehat{\KD(n)}$ of
$\KD(n)$ is $\mathbb{Z}_2 \times
\mathbb{Z}_2$~\cite[p.718]{Vainerman.1998}. We give here the intrinsic
group of $\KD(n)$ itself.
\begin{prop}
  The intrinsic group of $\KD(n)$ is $\mathbb{Z}_2 \times \mathbb{Z}_2$
  if $n$ is odd and $D_4$ if $n$ is even.
\end{prop}
\begin{proof}
  Applying Corollaries~\ref{corollary.intrinsicGroup}
  and~\ref{corollary.plonge}, the algebra of the intrinsic group of
  $\KD(n)$ is isomorphic to $\KD(k)$ for some $k$ dividing $n$.
  Furthermore $\KD(k)$ is cocommutative if and only if $k=1$ or $k=2$.
  Finally, $\KD(1)$ and $\KD(2)$ are respectively isomorphic to
  $\C[D_2]$ and $\C[D_4]$.
\end{proof}

\subsection{Automorphism group}
\label{subsection.automorphisms}

In this subsection, we describe the automorphism group of $\KD(n)$. The
point is that each automorphism of $\KD(n)$ induce a symmetry in the
lattice $\ll(\KD(n))$ of \sacigs of $\KD(n)$. Namely, it maps a \sacig
on a \sacig (mapping the Jones projection of the first to the Jones
projection of the second).

\subsubsection{Automorphisms induced by automorphisms of $D_{2n}$ fixing $\C[H]$}
\label{conjugue}

\begin{proposition}
  \label{proposition.automorphisms.KD.ak}
  Let $k$ be coprime with $2n$, and $\Theta_k$ be the algebra morphism
  defined by:
  \begin{displaymath}
    \Theta_k(\lambda(a)) = \lambda(a^k)
    \qquad \text{and} \qquad
    \Theta_k(\lambda(b))=\lambda(b)\,.
  \end{displaymath}
  Then, $\Theta_k$ is a Kac algebra automorphism of $\KD(n)$ which
  fixes $H$.
\end{proposition}
\begin{proof}
  Such a $\Theta_k$ is an automorphism of the group $D_{2n}$, and
  therefore an automorphism of its Kac algebra $\C[D_{2n}]$. Since $H$
  is fixed, so is $\Omega$, so that the coproduct $\Delta$ of $\KD(n)$
  is preserved as well.
\end{proof}

\begin{example}
  \label{example.conjugue}
  Take $\Theta=\Theta_{-1}$. It is an involutive automorphism of
  $\KD(n)$; it can be alternatively described as the conjugation by
  $\lambda(b)$.
\end{example}

\subsubsection{An automorphism which does not fix $\C[H]$}

We now construct an automorphism $\Theta'$ which acts non trivially on
$\C[H]$.
\begin{proposition}
  \label{proposition.automorphism.KD.ThetaPrime}
  The following formulas define an involutive Kac algebra automorphism
  of $\KD(n)$:
  \begin{displaymath}
    \Theta'(\lambda(a))=\lambda(a)-\frac{1}{2}(\lambda(a)-\lambda(a^{-1}))(1+\lambda(a^n))
    \qquad \text{and} \qquad
    \Theta'(\lambda(b))=\lambda(a^nb)\,.
  \end{displaymath}
\end{proposition}

\subsubsection{Tips and tricks to prove an isomorphism }

The following (easy) proposition gives a conveniently small list of
properties to be checked to be sure that an homomorphism defined by
the images of some generators extends to a proper Kac-algebra isomorphism.
\begin{proposition}
  \label{proposition.KacAlgebraIsomorphism}
  Let $K$ and $K'$ be two Kac algebras of identical finite dimension,
  and $A$ and $A'$ be sets of algebra generators of $K$ and $K'$
  respectively. Fix $\Phi(a)\in K'$ for $a\in A$. Then, $\Phi$
  extends to a Kac-algebra isomorphism from $K$ to $K'$ if and only
  if:
  \begin{enumerate}[(i)]
  \item The elements $\{\Phi(a) \suchthat a\in A\}$ satisfy the
    relations of $A$, and therefore define a (non necessarily unital)
    algebra morphism $\Phi$ from $K$ to $K'$;
  \item $\Phi$ preserves the involution on $A$:
    $\Phi(a^*) = \Phi(a)^*$, for all $a\in A$;
  \item $\Phi$ preserves the coproduct on $A$:
    $(\Phi\otimes\Phi)(\Delta(a)) = \Delta(\Phi(a))$, for all $a\in A$;
  \item All elements $a'\in A'$ have a preimage by $\Phi$ in $K$.
  \end{enumerate}
\end{proposition}
\begin{proof}
  Using (i), $\Phi$ extends to a (possibly non unital) algebra
  morphism from $K$ to $K'$. By (iv), $\Phi$ is bijective, and by the
  uniqueness of the unit it is an unital algebra isomorphism.

  Similarly, by (iii) $\Phi$ is a coalgebra isomorphism.  The counit
  and the antipode are preserved, as they are unique in a Hopf algebra
  (see~\cite[III.3.2]{Kassel.1995}).

  By semi-simplicity of the algebra, the trace is preserved; indeed,
  the central idempotents are unique (up to permutation among those of
  same rank), and the normalization of the trace on each corresponding
  matrix block is determined by the rank.  The coinvolution coincides
  with the antipode~\ref{KacHopf} and is preserved.  Finally, the
  involution being an anti-algebra morphism, (ii) ensures it is
  preserved as well.
\end{proof}

\subsubsection{Automatic checks on computer}

Checking the properties listed in
Proposition~\ref{proposition.KacAlgebraIsomorphism} often boils down
to straightforward, pointless, and tedious calculations. We therefore
wish to delegate them to the computer whenever possible. To obtain
Proposition~\ref{proposition.automorphism.KD.ThetaPrime}, we shall in
principle prove those properties for all values of $n$; however it
turns out that, in most cases, it is sufficient to check them only for
small values of $n$, which is easier to automatize. We introduce here
just a bit of formalism that justifies this approach.  For the sake of
simplicity, we do that for the special case of $\KD(n)$, but we later
reuse straightforward analogues in other isomorphism proofs.

The idea is to consider $n$ as a formal parameter, which can be though
of as letting $n$ go to infinity. Namely, let $D_\infty$ be the
(infinite) group generated by $a,b,\an$, subject to the relations
$b^2=\an^2=1$, $\an b=b\an$, and $ab=ba^{-1}$. We denote by $\Pi_{D,
  n}$ the canonical quotient map from $D_\infty$ to $D_{2n}$ which
sends $\an$ to $a^n$ (its kernel is the subgroup generated by $\an
a^{-n}$). Let $\C[D_\infty]$ be the group algebra of $D_\infty$.  The
quotient map $\Pi_{D,n}$ extends to a surjective algebra morphism from
$\C[D_\infty]$ to $\C[D_{2n}]$.  It is further injective whenever
restricted to the subspace of $\C[D_\infty]$ spanned by $a^i \an^j
b^k$ with $2|i| <n$.  Define the degree of an algebraic expression $A$
involving $a$, $b$ and $\an$ as the maximal $d$ such that $a^d$ or
$a^{-d}$ occurs in it.
\begin{proposition}
  \label{proposition.equationAtInfinity}
  An algebraic expression $A$ of degree $d$ which vanishes in some
  $\KD(N)$, $N>2d$, vanishes in $\KD(n)$ for any $n\in\N$.
   \end{proposition}
\begin{proof}
  Indeed, it can be lifted up via $\Pi_{D,N}$ to $\C[D_\infty]$ and
  then projected down to $\C[D_n]$ for any other $n\in\N$ via
  $\Pi_{D,n}$.
\end{proof}
Note that the bound $N>2d$ is tight because of the cancellation $a^d -
\an a^{-d}=0$ if $n=2d$.

As illustrated in the following example, we often use implicitly
straightforward variants of this proposition.
\begin{example}
  The coproduct $\Delta(a) = \Omega (a\otimes a)\Omega^*$ is of degree
  $1$. One can read off the general formula for its expansion from the
  computation for $N=3$; this expansion has $64$ terms, and is non
  symmetric. However, for $N=2$ there are cancellations: the expansion
  degenerates to $16$ terms and becomes symmetric.
\end{example}

\subsubsection{Technical lemma}

\begin{lemma}
  \label{lemma.automorphism.KD.ThetaPk}
  For all $k\in \Z$,
  \begin{displaymath}
    \Theta'(\lambda(a))^k=\lambda(a^k)-\frac{1}{2}(\lambda(a^k)-\lambda(a^{-k}))(1+\lambda(a^n))\,,
  \end{displaymath}
  and $\Theta'$ preserves the involution.
\end{lemma}
\begin{proof}
  Write
  $f_k=\lambda(a^k)-\frac{1}{2}(\lambda(a^k)-\lambda(a^{-k}))(1+\lambda(a^n))$.
  Since everything commutes and $\lambda(a^n)$ is its own inverse,
  one checks easily that $f_k f_1 = f_{k+1}$. Specially we have $f_{-1}=f_1^*=f_1^{-1}$, so $\Theta'(\lambda(a)^*)=\Theta'(\lambda(a))^*$. The lemma follows by induction.
\end{proof}

\subsubsection{Proof of Proposition~\ref{proposition.automorphism.KD.ThetaPrime}}
\begin{proof}%
  We check that $\Theta'$ satisfies the properties listed in
  Proposition~\ref{proposition.KacAlgebraIsomorphism}.
  \begin{enumerate}[(i)]
  \item Thanks to Lemma~\ref{lemma.automorphism.KD.ThetaPk},
    $\Theta'(\lambda(a))^{2n}=1$. The relation $\Theta'(\lambda(b))^2=1$ is obvious. The remaining relation
    $\Theta'(\lambda(b))\Theta'(\lambda(a))=\Theta'(\lambda(a))^{-1})\Theta'(\lambda(b))$
   is of degree $1$. Using
    Proposition~\ref{proposition.equationAtInfinity}, it is sufficient to check
    on computer that it holds in $\KD(3)$.
  \item Follows from Lemma~\ref{lemma.automorphism.KD.ThetaPk}
  \item Using a direct extension of
    Proposition~\ref{proposition.equationAtInfinity} for
    $\C[D_{2n}]\otimes\C[D_{2n}]$, we check on computer that, in
    $\KD(3)$ and for $x=a,b$, $(\Theta'\otimes
    \Theta')(\Delta(\lambda(x)))=\Delta(\Theta'(x))$.
  \item We prove that $\Theta'$ is an involution (and therefore an
    isomorphism) by checking on computer that the equation
    $\Theta'(\Theta'(\lambda(a)))=\lambda(a)$ holds in $\KD(3)$. The
    equation $\Theta'(\Theta'(\lambda(b)))=\lambda(b)$ is obvious.\qedhere

  \end{enumerate}
\end{proof}

\subsubsection{The automorphism group of $\KD(n)$}
\label{subsubsection.automorphismGroupKD}

\begin{theorem}
  \label{theorem.automorphisms.KD}
  For $n\geq 3$, the automorphism group $\aut(\KD(n))$ of $\KD(n)$ is
  given by:
  \begin{displaymath}
    A_{2n} = \{\  \Theta_k, \ \Theta_k\Theta'\  \suchthat k\wedge 2n = 1 \}\,.
  \end{displaymath}
  In particular, it is of order $2\varphi(2n)$ (where $\varphi$ is the
  Euler function), and isomorphic to $\Z_{2n}^* \rtimes \Z_2$, where
  $\Z_{2n}^*$ is the multiplicative group of units of $\Z_{2n}$.
\end{theorem}

\begin{proof}
  From Propositions~\ref{proposition.automorphisms.KD.ak}
  and~\ref{proposition.automorphism.KD.ThetaPrime}, $\aut(\KD(n))$ contains $A_{2n}$.
  Let us prove the converse.

  Let $\Phi$ be an automorphism of $\KD(n)$. It induces an automorphism
  $\sigma$ of the intrinsic group $G(\KD(n))$ (respectively $H$ for $n$
  odd, and $D_4$ for $n$ even). Note that $(\Phi\otimes\Phi)(\Omega)$
  is completely determined by $\sigma$. Furthermore, we can use it to
  "untwist" the coproduct of $\KD(n)$ into a cocommutative coproduct by
  defining for all $y$ in $\KD(n)$:
  \begin{displaymath}
    \Delta_\sigma(y) = (\Phi\otimes\Phi)(\Omega)^* \Delta(y) (\Phi\otimes\Phi)(\Omega)\,.
  \end{displaymath}
  This coproduct is indeed cocommutative because
  \begin{displaymath}
    \Delta_\sigma(y) = (\Phi\otimes\Phi)(\Omega)^* \Delta(\Phi(x)) (\Phi\otimes\Phi)(\Omega)
    = (\Phi\otimes \Phi) ( \Omega^* \Delta(x) \Omega )
    = (\Phi\otimes \Phi) ( \Delta_s(x) ) \,
  \end{displaymath}
  where $x=\Phi^{-1}(y)$, and $\Delta_s$ is the usual cocommutative
  coproduct of $\C[D_{2n}]$.

  We claim that, for $n\geq 3$, this later property rules out all the
  automorphism of $G(\KD(n))$ but $\sigma=\id$ and $\sigma$ defined by
  $\sigma(a^n)=a^n$ and $\sigma(b) =a^nb$. Consider indeed the case
  $n$ odd (resp. even), and take $\tau$ one of the $4$ (resp. $6$)
  remaining automorphisms of $G(\KD(n))=H$ (resp $=D_4$). The
  expression $\Delta_\tau(\lambda(a))$ is of degree $1$, so it is
  sufficient to check on computer that it is non symmetric for $n=3$
  (resp $n=4$).

  If $\sigma=\id$, then $\Delta_\sigma = \Delta_s$, so $\Phi$ is an
  automorphism of $\C[D_{2n}]$ fixing $H$.  Therefore $\Phi=\Theta_k$
  for some $k$ coprime to $2n$. Otherwise, $\Phi\Theta'$ fixes $H$, so
  $\Phi$ is of the form $\Theta_k\Theta'$.
\end{proof}

\subsubsection{Computing isomorphisms}

The reasoning developed in the previous proof turns into an algorithm
to compute Kac algebra isomorphisms. All the isomorphisms in this
article were conjectured from the application of this algorithm on
small examples (see~\ref{subsubsection.automorphismsK3},
\ref{subsubsection.demo.self-dual},
and~\ref{subsubsection.demo.KDKQ}).
\begin{algorithm}
  \label{algo.isomorphism}
  Let $K=K(\C[G],\,H,\,\Omega)$ be a finite dimensional Kac algebra
  obtained by twisting a group algebra $\C[G]$ with a
  $2$-(pseudo)-cocycle $\Omega$ of a commutative subgroup $H$, and let
  $K'$ be any finite dimensional Kac algebra.  The following algorithm
  returns all the Kac algebra isomorphisms $\phi$ from $K$ to $K'$. In
  particular, the result is non empty if and only if the two algebras
  are isomorphic.
  \begin{itemize}
  \item Compute the intrinsic group $H'$ of $K'$;
  \item For each embedding $\rho$ of $H$ into $H'$:
    \begin{itemize}
    \item Construct $K''=K(K',\,\rho(H),\,\Omega^*)$ whose
      coproduct $\Delta''$ is obtained by (un)twisting the coproduct
      $\Delta'$ of $K'$ by the inverse cocycle $\Omega^*$;
    \item If the untwisted coproduct $\Delta''$ is cocommutative:
      \begin{itemize}
      \item Compute the intrinsic group $G''$ of $K''$;
      \item Compute the group isomorphisms from $G$ to $G''$ which are
        compatible with $\rho$.
      \end{itemize}
    \end{itemize}
  \end{itemize}
\end{algorithm}
\begin{proof}
  If there exist a Kac algebra isomorphism $\phi$ from $K$ to $K'$,
  then the Kac algebra $K'' = K(K',\,\phi(H),\,\Omega^*)$ is
  isomorphic to $\C[G]$.
\end{proof}
This algorithm is not blazingly fast, but we are not
aware of any other algorithms to test systematically the isomorphism
of two Kac algebras (although their definitely should exist one). It
also has some nice theoretical consequences: first there are finitely
many isomorphisms, and second the existence of an isomorphism does not
depend on the actual ground field (typically $\C$ or
$\Q(i,\epsilon)$).

\subsection{The three \sacigs $K_2$, $K_3$, $K_4$ of dimension $2n$}
\label{sacig2n}

We come back to the study of the lattice $\ll(\KD(n))$. The action of
$G(\widehat{\KD(n)})$ on $\KD(n)$ produces three distinct exterior
actions of $\mathbb{Z}_2$ on $N_3$, and the corresponding three fixed
point algebras are three subfactors of index $2$ in $N_3$.

Therefore, $\KD(n)$ admits three \sacigs of dimension $2n$, with Jones
projections $e_1+e_2$, $e_1+e_3$ et $e_1+e_4$ respectively (those are
the only possible projections of trace $1/2n$). We denote by $K_i$ the
\sacig $I(e_1+e_i)$ for $i=1,2,3$.  We postpone the study of $K_3$ and
$K_4$ whose structure depend on the parity of $n$.

\subsection{The Kac subalgebra $K_2=I(e_1+e_2)$ of dimension $2n$}
\label{e1e2}
\label{section.K2}

In this section, we describe completely the \sacig $K_2$ and its
\sacigs. The following result was suggested by computer exploration
(see Appendix~\ref{mupad.K2=Dn}).

\begin{prop}
  \label{proposition.K2}
  The \sacig $K_2=I(e_1+e_2)$ is isomorphic to the Kac algebra of
  functions on the group $D_n$. Its lattice of \sacigs is the dual of
  the lattice of subgroups of $D_n$.
\end{prop}

\begin{proof}
  First, $K_2=I(e_1+e_2)$ is indeed a Kac subalgebra of $\KD(n)$
  because $\Delta(e_1+e_2)$ is symmetric (see Appendix~\ref{cop}).
  Looking further at the expression of this coproduct yields the
  following block matrix structure for $K_2$:
  \begin{displaymath}
    K_2=\C(e_1 +e_2)\oplus  \C(e_3 +e_4) \oplus \bigoplus_{j=1,\,j \text{ odd}}^{n-1}
    \left(\C r^j_{1,1}\oplus \C r^j_{2,2}\right)\oplus \bigoplus_{j=1,\,j \text{ even}}^{n-1} \left(\C e^j_{1,1}\oplus \C e^j_{2,2}\right)\,.
  \end{displaymath}
  All the blocks are trivial, so $K_2$ is commutative. Therefore it is
  the Kac algebra $L^\infty(G)$ of functions on some group $G$. The
  minimal projections are the characteristic functions $\chi_g$ of the
  elements of $g\in G$. To start with, $\chi_1=e_1+e_2$, since the
  later is the support of the counit of $I(e_1+e_2)$. We denote the
  others as follows: $e_3+e_4=\chi_{\beta_0}$,
  \begin{align*}
    r^{2j-1}_{1,1}&=\chi_{\beta_j},&&r^{2j-1}_{2,2}=\chi_{\beta'_j},&&\text{for } j=1,\dots,m\,,\\
    e^{2j}_{1,1}&=\chi_{\alpha_j},&&e^{2j}_{2,2}=\chi_{\alpha_{-j}},&&\text{for
    } j=1,\dots,m'\,.
  \end{align*}
  with $m'=\lfloor(n-1)/2\rfloor$. For short, we also write $\alpha=\alpha_1$.

  The group law on $G$ is determined by the coproducts of those
  central idempotents: $\Delta(\chi_g)= \sum_{hk=g} \chi_h \otimes
  \chi_k$ (the expressions of those coproducts are computed in
  Appendix~\ref{cop}). In particular, $\Delta(e_1+e_2)$ gives the
  inverses: the $\beta_j$ are idempotents, while $\alpha_j$ is the
  inverse of $\alpha_{-j}$. To obtain the remaining products, we need
  to distinguish between $n$ even and odd.

  \textbf{Case $n=2m$:} From the expression of $\Delta(e^2_{1,1})$,
  one get:
  \begin{itemize}
  \item for $j>2$ even, $\Delta(e^2_{1,1})(e^{j-2}_{2,2}\otimes
    e^{j}_{1,1})=e^{j-2}_{2,2}\otimes e^{j}_{1,1}$. Therefore, for
    $k=2,\dots,m-1$, $\alpha_{k}=\alpha_{k-1}\alpha_1$, and by
    induction $\alpha_{k}=\alpha^k$ and $\alpha_{-k}=\alpha^{-k}$.

  \item $\Delta(e^2_{1,1})((e_3+e_4)\otimes
    e^{n-2}_{2,2})=(e_3+e_4)\otimes e^{n-2}_{2,2}$. Therefore,
    $\beta_0=\alpha\alpha_{m-1}=\alpha^m$ and $\alpha$ is of order
    $n$.

  \item For $j>2$ odd, $\Delta(e^2_{1,1})(r^{j-2}_{1,1}\otimes
    r^{j}_{1,1})=r^{j-2}_{1,1}\otimes r^{j}_{1,1}$ and
    $\Delta(e^2_{1,1})(r^{n-j+2}_{2,2}\otimes r^{n-j}_{2,2}
    )=r^{n-j+2}_{2,2}\otimes r^{n-j}_{2,2}$. It follows that for
    $k=1,\dots,m-1$, $\beta_{k+1}=\beta_{k}\alpha$ and
    $\beta'_{k+1}=\alpha\beta'_{k}$ and by induction
    $\beta_{k+1}=\beta_1\alpha^k$ et $\beta'_{k+1}=\alpha^k\beta'_1$.
  \item $\Delta(e^2_{1,1})(r^1_{2,2}\otimes
    r^{1}_{1,1})=r^1_{2,2}\otimes r^{1}_{1,1}$. Therefore,
    $\beta'_1=\alpha\beta_1$.
  \end{itemize}
  The group $G$ is generated by $\alpha$ of order $n$ and
  $\beta=\beta_1$ of order $2$. From the expression $\Delta(e_3+e_4)$
  given in~\ref{cop}, one get: $\Delta(e_3+e_4)(r^{2m-1}_{1,1}\otimes
  r^{1}_{2,2})=(r^{2m-1}_{1,1}\otimes r^{1}_{2,2})$, that is
  $\beta'_1=\beta_{m}\beta$, which gives the relation
  $\alpha\beta=\beta\alpha^{-1}$. Therefore, $G$ is the dihedral group
  $D_n$.

  \textbf{Case $n=2m+1$:} From the expression of $\Delta(e^2_{1,1})$
  one get:
  \begin{itemize}
  \item for $j>2$ even, $\Delta(e^2_{1,1})(e^{j-2}_{2,2}\otimes
    e^{j}_{1,1})=e^{j-2}_{2,2}\otimes e^{j}_{1,1}$. Therefore, for
    $k=2,\dots,m$ $\alpha_{k}=\alpha_{k-1}\alpha_{1}$, and by
    induction $\alpha_{k}=\alpha^k$ and $\alpha_{-k}=\alpha^{-k}$.

  \item $\Delta(e^2_{1,1})(e^{n-1}_{2,2}\otimes
    e^{n-1}_{2,2})=e^{n-1}_{2,2}\otimes e^{n-1}_{2,2}$. Therefore,
    $\alpha$ is of order $n$.
  \end{itemize}
  The expression of the coproduct of $e_3+e_4$ gives the following
  relations for $j=1,\dots,m$, $\beta_{m+1-j}=\beta_0\alpha_{-j}$ and
  $\beta'_{m+1-j}=\beta_0\alpha_{j}$.

  Since $\beta_{m+1-j}$ is of order $2$, for all $k=0,\dots,n-1$, the
  following commutation relation holds:
  $\beta_0\alpha^{-k}=\alpha^{k}\beta_0$ .  Therefore, $G$ is
  generated by $\alpha$ of order $n$ and $\beta=\beta_0$ of order $2$
  with the relation $\alpha\beta=\beta\alpha^{-1}$: it is again the
  dihedral group $D_n$.
\end{proof}

\subsection{The \sacigs of $K_2$}\label{sacigsK2}
Each \sacig of $K_2$ is the subspace of invariant functions on the
right cosets of some subgroup of $D_n$.  In
Appendix~\ref{subsubsection.mupad.sacig.K2} we show how to get all the
Jones projection on computer for $n$ small. We do it here explicitly
for some examples for all $n$:
\begin{itemize}
\item Take $m$ such that $n=2m$ or $n=2m+1$. The $2m+1$ subgroups of
  order $2$ of $D_n$ induce $2m+1$ \sacigs of dimension $n$. Their
  Jones projection are respectively: $e_1+e_2+e_3+e_4$,
  $e_1+e_2+r_{1,1}^{2j+1}$ and $e_1+e_2+r_{2,2}^{2j+1}$ with
  $j=0,\dots,m-1$.
\item The subgroup generated by $\alpha$ induces the \sacig
  $I_0=I(e_1+e_2+q_1+q_2)$ of dimension $2$.
\item If $n$ is even, the subgroup generated by $\alpha^2$ induces a
  \sacig $J_{20}$ of dimension $4$ contained in $K_0$ (see~\ref{sacigK0}).
\end{itemize}

\subsection{The \sacig $K_1=I(e_1+e_2+e_3+e_4)$ of dimension $n$}
\label{K1}

The intersection $K_1$ of the three \sacigs $K_2$, $K_3$ et $K_4$ is
the image by $\delta$ of the algebra of the intrinsic group of
$\widehat{\KD(n)}$. Its dimension is $n$; its Jones projection is
$e_1+e_2+e_3+e_4$ since it is of trace $1/n$ and dominates $e_1+e_2$,
$e_1+e_3$, and $e_1+e_4$.  The connected component of $e_1$ in the
Bratelli diagram of $I(e_1+e_2+e_3+e_4) \subset \KD(n)$ is the
principal graph of the inclusion $N \subset N\rtimes \mathbb{Z}_2
\times \mathbb{Z}_2$. The structure of $K_1$ depends on the parity of
$n$ and is further elucidated in~\ref{K1pair}
and~\ref{K1impair}.

\subsection{\Sacigs of dimension dividing  $2n$}
\label{sacign}

The following proposition allows for recovering many, if not all, of
the \sacigs recursively from those of $K_2$, $K_3$, and $K_4$.
\begin{prop}
  If the dimension of a \sacig $I$ of $\KD(n)$ divides $2n$ (and not
  just $4n$), then either its Jones projection $p_I$ dominates
  $e_1+e_2+e_3+e_4$ and then $I$ is contained in $K_1$, or there
  exists a unique $i \in \{2,3,4\}$ such that $p_I$ dominates
  $e_1+e_i$ and $I$ is a sub\sacig of $K_i$. In particular, if $n$ is
  a power of $2$, then any \sacig is contained in one of the $K_i$'s.
\end{prop}
\begin{proof}
  The Jones projection $p_I$ of a \sacig $I$ dominates $e_1$. Given
  the block matrix structure of $\KD(n)$, it is the sum of $x$
  projections of trace $1/4n$ (with $x\geq 1$) et and $y$ projections
  of trace $1/2n$. Its trace $tr(p_I)=(\dim I)^{-1}$, is therefore
  $(x+2y)/4n$ (see~\ref{resumep}). If the dimension of $I$ divides
  $2n$ (which is always the case when $n$ is a power of $2$), $x$ must
  be even, and therefore $p_I$ takes one of the two following forms:
  \begin{itemize}
  \item $e_1+e_2+e_3+e_4+s$, with $s$ projection and $tr(s)=(\dim I)^{-1}-n^{-1}$;
  \item $e_1+e_i+s'$, with $i=2$, $3$, or $4$, and $s'$ projection of
    trace $(\dim I)^{-1}-(2n)^{-1}$.
  \end{itemize}
  The proposition follows from~\ref{tour}~(5).
\end{proof}

In~\ref{sacignimpair}, we will establish an even more precise result
when $n$ is odd.

\subsection{\Sacigs induced by Jones projections of subgroups}
\label{abelien}

We have seen in~\ref{plonge} that the algebras of the subgroups
containing $H$ are \sacigs of $\KD(n)$. The algebras of the other
subgroups are usually not \sacigs in $\KD(n)$; yet, in many cases,
their Jones projections remain Jones projections in $\KD(n)$.
\begin{prop}
  \label{prop.abelien}
  The Jones projections of the following subgroups of $D_{2n}$ are
  Jones projections of \sacigs of $\KD(n)$ of dimension the order of
  the subgroup. If the subgroup is commutative, then so is the \sacig.
  \begin{itemize}
  \item $H_r=\{1,  \lambda(a^n), \lambda(ba^r), \lambda(ba^{r+n}) \}$
  for $r=0,\dots,n-1$;
  \item $A_k$, subgroup generated by $a^k$, for $k$ such that $2n=kd$;
  \item $B_k$, subgroup generated by $a^k$ and $b$, for $k$ such that $2n=kd$.
  \end{itemize}
\end{prop}
For example, in $\KD(9)$, the Jones projection of $\C[A_3]$ is that of
$M_2$, and that of $\C[A_6]$ generates $L_0$ (see
Figure~\ref{KD9}). We refer to~\ref{K4impair} and~\ref{sacig4impair}
for further applications, and to~\ref{lattice_K3_even} and~\ref{KD6-8}
for examples of Jones projections given by dihedral subgroups.

Not every subgroup, even commutative, gives a Jones projection in the
twisted Kac algebra; for example, in $\KD(n)$ with $n\geq 3$, there
are only three \sacigs of dimension $2$ whereas $D_{2n}$ has $2n+1$
subgroups of order $2$. Reciprocally, not all Jones projections are
given by subgroups; for example, in $\KD(3)$ there are three \sacigs
of dimension $3$ whereas $D_6$ has a single subgroup of order $3$.

\begin{problem}
  Let $K$ be a Kac algebra by twisting a group algebra
  $\C[G]$. Characterize the subgroups of $G$ whose Jones projections
  remain Jones projections in $K$.
\end{problem}

We now turn to the proof of the proposition, starting with a general
lemma.
\begin{lemma}
  \label{lemma.abelien}
  Let $K$ be a Kac algebra and $p$ be a projection which dominates the
  Jones projection $f_2$ of $K$, and such that $\Delta(p) = \sum_{i=1}^d Q_i\otimes P_i$, where:
  \begin{itemize}
  \item the families $(P_i)_{i=1,\dots,d}$ and $(Q_i)_{i=1,\dots,d}$
    are linearly independent;
  \item $d = 1/tr(p)$;
  \item $I=\bigoplus_{i=1}^d \C P_i$ is a unital involutive
    subalgebra of $K$.
  \end{itemize}
  Then, $I$ is a \sacig of Jones projection $p$. In particular, $Q_i =
  S(P_i^*)$.
\end{lemma}
\begin{proof}
  This is a consequence of the coassociativity of $\Delta$. Indeed,
  \begin{equation}
    \sum_{i=1}^d Q_i \otimes \Delta(P_i) =
    (\id \otimes \Delta)\Delta(p) =
    (\Delta \otimes \id)\Delta(p) =
    \sum_{i=1}^d \Delta(Q_i) \otimes P_i\,;
  \end{equation}
  applying $\widehat Q_j \otimes \id \otimes \id$ on both sides of
  this equation, where $(\widehat Q_i)_{i\in I}$ is a family of linear
  forms such that $\widehat Q_i(Q_j)=\delta_{i,j}$, yields that
  $\Delta(P_j)\in K\otimes J$. Therefore, $\Delta(J)\subset K\otimes
  J$, and the result follows by Proposition~\ref{resumep} and
  Remark~\ref{remark.resumep}.
\end{proof}

\begin{proof}[Proof of proposition~\ref{prop.abelien}]
  We shall see in~\ref{J4} that the coproduct of the Jones
  projection $Q_r$ of $\C[H_r]$ is of the form
  $\Delta(Q_r)=\sum_{i=1}^4 S(P_i)\otimes P_i$, where $P_1=Q_r$ and
  $1=P_1+P_2+P_3+P_4$ is a decomposition of the identity into
  orthogonal projections. Then, $J_r=\sum_i \C P_i$ is an involutive
  commutative subalgebra of dimension $4$, while $tr(Q_r)=1/4$.  The
  result follows by applying Lemma~\ref{lemma.abelien}.

  If $d$ is even, the Jones projection of $\C [A_k]$ generates a
  commutative \sacig, which is the image of $K_2$ in $\KD(d/2)$
  through the embedding $\varphi_k$ in $\KD(n)$ (see~\ref{plonge} and
  the expression of $e_1+e_2$ in~\ref{form2}).  For example, the Jones
  projection of $\C[A_1]$ is $e_1+e_2=p_{K_2}$.
  Otherwise, $k$ is even, and the Jones projection of $\C[A_k]$
  generates a commutative \sacig which is the image of $K_1$ in
  $\KD(d)$ through the embedding $\varphi_{k/2}$ in $\KD(n)$.  For
  example, the Jones projection of $\C[A_2]$ is
  $e_1+e_2+e_3+e_4=p_{K_1}$.

  If $k$ divises $n$, $\mathbb{C}[B_k]$ is isomorphic to $KD[d/2]$
  (see~\ref{plonge}). If $k=2k'$ does not divide $n$, the Jones
  projection of $\C[B_k]$ generates a \sacig which is the image of
  $K_4$ of $\KD(d)$ through the embedding $\varphi_{k/2}$ in $\KD(n)$.
 \end{proof}

\newpage
\section{The Kac algebras $\KD(n)$ for $n$ odd}
\label{section.KD.odd}

The study of $\ll(\KD(n))$ for $n$ odd led us to conjecture, and later
prove, that $\KD(n)$ is self-dual in that case. This property is not
only interesting by itself. First, this is a useful tool for
constructing new \sacigs by duality (see
e.g. Appendix~\ref{subsubsection.demo.delta}), and unraveling the
self-dual lattice $\ll(\KD(n))$; in particular, this is a key
ingredient to get the full lattice when $n$ is prime. Last but not
least, this allows us to describe completely the principal graph for
some inclusions $N_2 \subset N_2 \rtimes I$ with $I$ \sacig of $\KD(n)$
(which requires information on $\delta(I)$).

After establishing the self-duality, we describe the symmetric
\sacigs, as well as those of dimension $2n$, of odd dimension, and of
dimension $4$. Putting everything together, we get a partial
description of $\ll(\KD(n))$ (Theorem~\ref{theorem.lattice.nodd}) which
is conjecturally complete (Conjecture~\ref{conj.lattice.nodd}). This
is proved for $n$ prime (Corollary~\ref{corollary.lattice.nprime}) and
checked on computer up to $n=51$. We conclude with some illustrations
on $\KD(n)$ for $n=3,5,9,15$.

\subsection{Self-duality}\label{self-dualKD}

In Appendix~\ref{subsubsection.demo.self-dual}, the self-duality of
$\KD(n)$ for $n$ odd is checked for $n\leq 21$ by computing an
explicit isomorphism. From this exploration, we infer the definition
of $\psi$, candidate as Kac isomorphism from $\KD(n)$ to its dual.

The dual of $\C[D_{2n}]$ is the algebra $L^{\infty}(D_{2n})$ of
functions on $D_{2n}$. We denote by $\chi_g$ the characteristic
function of $g$, so that $\{\chi_g, g \in G\}$ is the dual basis of
$\{\lambda(g), g\in G\}$.
We denote by $x \mapsto \hat{x}$ the vector-space isomorphism from
$\C[D_{2n}]$ to its dual $L^{\infty}(D_{2n})$ which extends
$\lambda(g)\mapsto \chi_g$ by linearity.

\label{self-dual}

The following theorem is the main result of this section:
\begin{thm}
  \label{theorem.self-dual}
  Let $n$ be odd. Then $\KD(n)$ is self-dual, via the Kac algebra
  isomorphism $\psi$:
  \begin{displaymath}
    a \mapsto n(2\widehat{e^{n-1}_{1,1}}+ \widehat{e^1_{2,2}}- (\widehat{e^1_{1,1}}+\widehat{e^{n-1}_{1,2}}+\widehat{e^{n-1}_{2,1}})) \qquad\qquad
    b \mapsto 4n\widehat {e_4}\,.
  \end{displaymath}
\end{thm}

We prove this theorem in the sequel of this section. We start by
describing the block matrix decomposition of $\widehat{\KD(n)}$. We use
it to define a $C^*$-algebra isomorphism $\psi$ on the matrix units of
$\KD(n)$, calculate its expression on $\lambda(a)$ and $\lambda(b)$,
and conclude by proving that $\psi$ is a Kac algebra isomorphism.
Some heavy calculations are relegated to~\ref{psicoproduit}.

For an alternative non constructive proof of self-duality, one can use
the isomorphism with $\KA(n)$ (see Theorem~\ref{theorem.isomorphisms})
and the self-duality of the later for $n$ odd which is mentioned
in~\cite[p. 776]{Calinescu_al.2004}.

\subsubsection{Block matrix decomposition of $\widehat{\KD(n)}$}  \label{block}

The product on $L^{\infty}(D_{2n})$ is the point wise (Hadamard)
product; the coproduct is defined by duality:
$$\Delta(\chi_g)=\sum_{s\in D_{2n}}\chi_s\otimes \chi_{s^{-1}g}\,.$$
For $f\in L^{\infty}(D_{2n})$ and $s\in D_{2n}$, the involution is
given by $f^{*}(s)=\overline{f(s)}$, the coinvolution $S$ by
$S(f)(s)=f(s^{-1})$, and $tr(\chi_s)$ by $1/2n$.

The coproduct of $\KD(n)$ is that of $\C[D_{2n}]$ twisted by $\Omega$;
Therefore by duality, the product of $\widehat{\KD(n)}$ is twisted from
that of $L^{\infty}(G)$. From~\ref{omega}, it can be written as:
\begin{displaymath}
  \Delta(\lambda(g))=\sum_{i,j,r,s=1}^4 c_{i,j}c_{s,r}
  \lambda(h_igh_r) \otimes \lambda(h_jgh_s)\,,
\end{displaymath}
with $h_1=1$, $h_2=a^n$, $h_3=b$, $h_4=ba^n$.  By duality, we obtain
the twisted product on $L^{\infty}(D_{2n})$:
\begin{displaymath}
  \chi_{g_1} \odot \chi_{g_2}=\sum_{i,j,r,s=1}^4
  c_{i,j}c_{s,r}\chi_{h_ig_1h_r}\chi_{h_jg_2h_s}\,.
\end{displaymath}

Let us study this new product. If one of the $g_k$ is in $H$ and the
other is not, then the product $\chi_{g_1} \odot \chi_{g_2}$ is
zero. If both are in $H$, the product is left unchanged since the
coproduct is unchanged on $H$. Therefore, $\chi_1$, $\chi_{a^n}$,
$\chi_{b}$ and $\chi_{ba^n}$ are central projections of
$\widehat{\KD(n)}$.

Consider now $s$ et $t$ in $D_{2n}-H$. The product $\chi_s \odot
\chi_t$ is non zero if and only if there exists $h$ and $h'$ in $H$
such that $t=hsh'$, that is only if $t$ belongs to the double coset
modulo $H$ of $s$: $H_s=\{s,s^{-1}, sa^n, s^{-1}a^n,bs, bs^{-1},
bsa^n, bs^{-1}a^n\}$; then $\chi_s \odot \chi_t$ belongs to the
subspace spanned by $\{\chi_r, r\in H_s\}$.

\begin{proposition}
  The $C^*$-algebra $\widehat{\KD(n)}$ decomposes as an algebra as the
  direct sum of $\C\chi_h$ for $h\in H$ and of $\C H_{a^k}$ for $k$
  odd integer in $\{1,\dots,n-1\}$, with
  $$H_{a^k}=\{a^k, a^{2n-k}, a^{n+k}, a^{n-k},ba^k, ba^{2n-k}, ba^{n+k}, ba^{n-k}\}\,.$$
  Furthermore, each non trivial block $\C H_{a^k}$ of
  $\widehat{\KD(n)}$ is isomorphic to $\C H_{a}$ in
  $\widehat{\KD(3)}$.
\end{proposition}
\begin{proof}
  The $H_{a^k}$ are disjoints: indeed $k$ (resp. $n-k$) is odd
  (resp. even) and strictly lower than $n$, and $2n-k$ (resp. $n+k$)
  is odd (resp. even) and strictly larger than $n$. Therefore, the
  blocks $\C H_{a^k}$ and the $\C\chi_h$ are in direct sum and by
  dimension, they span $\widehat{\KD(n)}$.

  Each non trivial block $\C H_{a^k}$ of $\widehat{\KD(n)}$ is stable
  under product; in fact, it is trivially isomorphic to $\C H_{a}$ in
  $\widehat{\KD(3)}$ because the relations between the eight elements
  of $H_s$ do not depend on $s=a^k$. From previous considerations,
  products of elements in different blocs are zero. Putting everything
  together, they form a (non minimal) matrix block decomposition of
  $\widehat{\KD(n)}$.
\end{proof}

We now define eight matrix units in each non trivial block $\C
H_{a^k}$. Those definitions are easily suggested and checked by
computer exploration in $\KD(3)$.
\begin{defn}
  For any odd integer $k$ in $\{1,\dots,n-1\}$, set:
  \begin{align*}
    E^k_{1,1}&=\chi_{ba^{n+k}}+\chi_{ba^{k}}\\
    E^k_{2,2}&=\chi_{ba^{2n-k}}+\chi_{ba^{n-k}}\\
    E^k_{1,2}&=-\frac{1}{2} \rm{i}(-\chi_{a^k}+\chi_{a^{2n-k}}+\chi_{a^{n+k}}-\chi_{a^{n-k}})+\frac{1}{2}(\chi_{ba^k}+\chi_{ba^{2n-k}}-\chi_{ba^{n+k}}-\chi_{ba^{n-k}})\\
    E^k_{2,1}&=\frac{1}{2}\rm{i}(-\chi_{a^k}+\chi_{a^{2n-k}}+\chi_{a^{n+k}}-\chi_{a^{n-k}})+\frac{1}{2}(\chi_{ba^k}+\chi_{ba^{2n-k}}-\chi_{ba^{n+k}}-\chi_{ba^{n-k}})\\
    E^{n-k}_{1,1}&=\chi_{a^{n-k}}+\chi_{a^{2n-k}}\\
    E^{n-k}_{2,2}&=\chi_{a^k}+\chi_{a^{n+k}}\\
    E^{n-k}_{1,2}&=\frac{1}{2}(-\chi_{a^k}-\chi_{a^{2n-k}}+\chi_{a^{n+k}}+\chi_{a^{n-k}})+\frac{1}{2}\rm{i}(\chi_{ba^k}-\chi_{ba^{2n-k}}-\chi_{ba^{n+k}}+\chi_{ba^{n-k}})
    \\
    E^{n-k}_{2,1}&=\frac{1}{2}(-\chi_{a^k}-\chi_{a^{2n-k}}+\chi_{a^{n+k}}+\chi_{a^{n-k}})-\frac{1}{2}\rm{i}(\chi_{ba^k}-\chi_{ba^{2n-k}}-\chi_{ba^{n+k}}+\chi_{ba^{n-k}})
  \end{align*}
\end{defn}

\subsubsection{The Kac algebra isomorphism $\psi$ from $\KD(n)$ to its dual}

Using definition~\ref{block}, we can construct an isomorphism $\psi$
from $\KD(n)$ to its dual that preserves the involution and the
trace. As we noticed just before, we can check others properties of
$\psi$ on computer for $n=3$ to obtain the following proposition :
\begin{proposition}
  The application $\psi$ defined by
  \begin{gather*}
    \psi(e_1)=\chi_1,\qquad
    \psi(e_2)=\chi_{a^n},\qquad
    \psi(e_3)=\chi_{ba^n},\qquad
    \psi(e_4)=\chi_b,\\
    \psi(r^k_{i,j})=E^k_{i,j}, \qquad
    \psi(e^{n-k}_{i,j})=E^{n-k}_{i,j}\,,
  \end{gather*}
  for $k$ odd in $\{1,\dots,n-1\}$, $i=1,2$, and $j=1,2$ extends by
  linearity into a $C^*$-algebra from $\KD(2m +1 )$ to
  $\widehat{\KD(2m+1)}$ which preserves the trace $tr$.
\end{proposition}

\begin{proof}[Proof of Theorem~\ref{self-dual}]
  In~\ref{psicoproduit}, we calculate explicitly $\psi(\lambda(a))$
  and $\psi(\lambda(b))$, and check that the corresponding coproducts
  are preserved by $\psi$. It follows that the isomorphism $\psi$ is a
  $C^*$-algebra isomorphism which preserves the trace and the
  coproducts of the generators of $\KD(n)$. It therefore also preserves
  the coinvolution (see~\ref{KacHopf}), and is a Kac algebra
  isomorphism from $\KD(n)$ to its dual.
\end{proof}

\subsection{The symmetric Kac subalgebras}\label{Himpair}

The lattice $\ll(\KD(n))$ contains the algebra $J_0$ of the intrinsic
group of $\KD(n)$ (see~\ref{intrinseque}) and its Kac subalgebras:
\begin{itemize}
\item $J_0=\C(e_1+q_1) \oplus \C(e_2+q_2) \oplus \C(e_3+p_2) \oplus
  \C(e_4+p_1)$ spanned by $1$, $\lambda(a^n)$, $\lambda(b)$ and
  $\lambda(ba^n)$;
\item $I_0=\C(e_1 +e_2+q_1+q_2)\oplus \C(e_3 +e_4+p_1+p_2)$ spanned
  by $1$ and $\lambda(a^n)$;
\item $I_1=\C(e_1 +e_3+q_1+p_2)\oplus \C(e_2 +e_4+p_1+q_2)$ spanned
  by $1$ and $\lambda(ba^n)$;
\item $I_2=\C(e_1 +e_4+p_1+q_1)\oplus \C(e_2 +e_3+p_2+q_2)$ spanned
  by $1$ and $\lambda(b)$.
\end{itemize}
Using~\ref{intrin}, the only \sacigs of dimension $2$ of $\KD(n)$ are $I_0$, $I_1$, and $I_2$.

Since $\KD(n)$ is self-dual, we can describe the principal graph of an
inclusion $N_2 \subset N_2\rtimes J$ as soon as $\delta(J)$ is known
after identification of $\KD(n)$ and its dual.
Considering the Bratteli diagram of $I_i \subset \KD(n)$ (with
$i=0,1,2$) shows that the inclusion $N_2 \subset N_2\rtimes
\delta^{-1}(I_0)$ is the only one of depth $2$. Therefore,
by~\ref{irredprof2}, $\delta(K_2)$ has to be identified with $I_0$.
Then, the principal graph of $N_2 \subset N_2\rtimes K_3$ and of $N_2
\subset N_2\rtimes K_4$ is $D_{2n}/\mathbb{Z}_2$ (see~\ref{grapheDpair}).

Since $J_0$ contains the three \sacigs of dimension $2$, $\delta(K_1)$
is $J_0$, the intersection of the three \sacigs of dimension $2n$. The
graph of the inclusion $N_2 \subset N_2\rtimes K_1$ is therefore
$D_n/\mathbb{Z}_2$ (see~\ref{grapheDimpair}). Furthermore,
using~\cite{Hong_Szymanski.1996}, the inclusion is of the form
$R\rtimes \mathbb{Z}_2 \subset R\rtimes G$, where $G$ is the semi-direct
product of an abelian group of order $n$ with $\mathbb{Z}_2$.

\subsection{The other  \sacigs of dimension $2n$: $K_3$ and $K_4$}
\label{K3K4impair}
\label{K3K4nimpair}

\subsubsection{The \sacigs $K_3$ and $K_4$ are isomorphic}
\label{iso-n-impair}

\begin{proposition}
  Let $n$ be odd. Then, the automorphism $\Theta'$ of $\KD(n)$
  (see Proposition~\ref{proposition.automorphism.KD.ThetaPrime})
  exchanges $e_3$ and $e_4$. Therefore, the \sacigs $K_3=I(e_1+e_3)$
  and $K_4=I(e_1+e_4)$ of $\KD(n)$ are exchanged by $\Theta'$ and thus
  isomorphic.
\end{proposition}
\begin{proof}
  Using~\ref{form2}, and dropping the $\lambda$'s for clarity, we
  have:
  \begin{displaymath}
    \Theta'(e_1+e_3)
    = \frac{1}{2n} \sum_{k=0}^{n-1}a^{2k} +ba^{2k+1+n}
    - \frac{1}{4n} \sum_{k=0}^{n-1}[a^{2k}-a^{-2k}+ba^{2k+1}-ba^{-2k-1}](1+a^n)\,.
  \end{displaymath}
  Since $n$ is odd, the first sum is $e_1+e_4$, while the second is
  obviously null.
\end{proof}

\subsubsection{Decomposition of $K_3$ and $K_4$}
\label{K3K4.blocks.impair}

From the expressions of the coproducts of the projections $e_1+e_3$
and $e_1+e_4$ for $n$ odd, one deduces:
$$ K_3 =\C(e_1+e_3)\oplus \C(e_2+e_4) \oplus \bigoplus_{j=1, j \text{ odd}}^{n-1} K_3^j\,,$$
$$K_4 =\C(e_1+e_4)\oplus \C(e_2+e_3) \oplus \bigoplus_{j=1, j \text{ odd}}^{n-1} K_4^j\,,$$
where each $K_3^j$ and $K_4^j$ is a factor $M_2(\C)$ with matrix units:
\begin{displaymath}
  (e^j_{1,1}+r^{n-j}_{1,1}), \ (e^j_{2,2}+r^{n-j}_{2,2}), \
  (e^j_{1,2}+r^{n-j}_{1,2}), \text{ and } (e^j_{2,1}+r^{n-j}_{2,1})\,,
\end{displaymath}
and
\begin{displaymath}
  (e^j_{1,1}+r^{n-j}_{2,2}),\ (e^j_{2,2}+r^{n-j}_{1,1}),\
  (e^j_{1,2}-r^{n-j}_{2,1}), \text{ and } (e^j_{2,1}-r^{n-j}_{1,2})\,,
\end{displaymath}
respectively.
From~\ref{tour}~(5), $I_1$ is contained in $K_3$ and $I_2$ in $K_4$.

\subsection{The \sacig $K_1=I(e_1+e_2+e_3+e_4)$ of dimension $n$}
\label{K1impair}

\begin{proposition}
  For $n=2m+1$, the matrix structure of
  $K_1=I(e_1+e_2+e_3+e_4)=K_2\cap K_3 \cap K_4$ is given by
  \begin{displaymath}
    K_1 =\C(e_1+e_2+e_3+e_4) \oplus \bigoplus_{j=1, j \text{ odd}}^{n-1} \C(r^j_{1,1}+e^{n-j}_{1,1}) \oplus \C(r^j_{2,2}+e^{n-j}_{2,2})\,.
  \end{displaymath}
  In $K_2\equiv D_n$ it is the subalgebra of constant functions on the
  right cosets w.r.t. $\{1, \beta\}$. Using the notations
  of~\ref{e1e2}, the matrix structure can be written as:
  \begin{displaymath}
    K_1 = \bigoplus_{j=0, \dots ,n-1} \C (\chi_{\alpha^j} + \chi_{ \alpha^j \beta})\,.
  \end{displaymath}
  Each \sacig $I$ of $K_1$ is the subalgebra of constant functions on
  the right cosets of some subgroup of $D_n$ containing
  $\{1,\beta\}$. Such a subgroup is a dihedral group generated by
  $\beta$ and $\alpha^k$ for some $k$ divisor of $n$. Then, the Jones
  projection of $I$ is $e_1+e_2+e_3+e_4+\sum_{h=1}^{(n-1)/2k}
  e^{2hk}_{1,1}+e^{2hk}_{2,2}+r^{n-2hk}_{1,1}+r^{n-2hk}_{2,2}$.
\end{proposition}
\begin{proof}
  Straightforward using~\ref{K1} and Appendix~\ref{cop}
\end{proof}

\subsection{The \sacigs of dimension dividing $n$}
\label{sacignimpair}

From~\ref{sacign}, the other \sacigs of dimension dividing $2n$ are
contained in one of the \sacigs of dimension $2n$. We now show a more
precise result for those of dimension dividing $n$ (or equivalently of
odd dimension).
\begin{prop}
  \label{prop.sacignimpair}
  When $n$ is odd, the \sacigs of odd dimension of $\KD(n)$ are those
  of $K_2$; in particular, they are isomorphic to the algebra of
  functions on $D_n$ which are constant on right cosets w.r.t. one of
  its subgroups. There are exactly $n$ \sacigs of dimension $n$, whose
  Jones projections are given in~\ref{sacigsK2}.
\end{prop}
\begin{proof}
  Let $I$ be a \sacig of odd dimension $k$. Then, $k$ is an odd
  divisor of $2n$. Set $k'=2n/k$ which is even. From~\ref{sacign},
  either $I$ is contained in $K_1$ (and therefore in $K_2$, or its
  Jones projection is of the form $e_1+e_i+s$, where $i=2,3$ or $4$
  and $s$ is a projection of trace $1/k-1/2n=(k'-1)/2n$ distinct from
  $e_j+e_{j'}$. However, in $K_3$ (resp.  $K_4$), and except for
  $e_2+e_4$ (resp. $e_2+e_3$), the minimal projections are of trace
  $1/n$, so it is not possible to combine them to get a projection of
  trace $(k'-1)/2n$ with $k'$ even. Therefore, $p_I$ is of the form
  $e_1+e_2+s$, and by~\ref{tour}~(5), $I\subset K_2$.
\end{proof}

\subsection{The \sacigs of $K_3$ and $K_4$}\label{K4impair}

Since $K_3$ and $K_4$ are isomorphic by $\Theta'$, their lattice of
\sacigs are isomorphic. We focus on the description of that of
$K_4$. From~\ref{sacignimpair}, the \sacigs of odd dimension of $K_4$
are those of $K_1$, and therefore contained in $K_2$. For the others,
we have the following partial description.
\begin{prop}
  \label{prop.K4impair}
  Let $n=2m+1$, and $k$ a divisor of $n=kd$. Let $B_{2d}$ be the
  subgroup generated by $a^{2d}$ and $b$ as in
  Proposition~\ref{prop.abelien}. Then,
  \begin{equation*}
    \label{equation.pK4k}
    e_1+e_4+\sum_{j=1}^{d -1} p^{jk}_{1,1}=\frac{1}{2k} \sum_{g\in B_{2d}} \lambda(g)\,.
  \end{equation*}
  is the Jones projection of a \sacig of dimension $2k$ of $K_4$,
  image by $\varphi_{d}$ of the $K_4$ of $\KD(k)$.

  In general, every \sacig of dimension $2k$ of $K_4$ contains
  $I_2$. Its Jones projection is of the form:
  \begin{equation*}
    e_1+e_4+\sum_{t\in T} p^{2t}_{1,1}+p^{n-2t}_{1,1}\,,
  \end{equation*}
  where $T$ is a subset of $\{1,2,\dots,m\}$ of size
  $\frac12(d-1)$. In particular, the Jones projection for $B_{2d}$
  above is obtained when $T$ is the set of multiples of $k$ in
  $\{1,\dots,m\}$.
\end{prop}
\begin{proof}
  The first part is the third item of Proposition~\ref{prop.abelien}.

  Let $k$ be a non trivial divisor of $n$ and $L$ a \sacig of
  dimension $2k$ of $K_4$. Then, $\delta(L)$ is of dimension $2d$,
  and by~\ref{sacign} it is contained in exactly one of $K_2$, $K_3$,
  and $K_4$.  Therefore, $L$ contains exactly one of $I_0$, $I_1$, and
  $I_2$, and from the expressions of $p_{I_i}$, $i=0,1,2$
  (see~\ref{Himpair}), this is necessarily $I_2$.

  The Jones projection $p_L$ of $L$ is of trace $1/2k$, dominates
  $p_{K_4}=e_1+e_4$, and is dominated by
  $p_{I_2}=e_1+e_4+q_1+p_1$. Using the matrix units of $K_4$
  (see~\ref{K3K4.blocks.impair}), it has to be of the given form.
\end{proof}

\subsection{The \sacigs of dimension $4$}
\label{sacig4impair}

The following proposition was suggested by the computer exploration of
Appendix~\ref{subsubsection.demo.delta}, where we use the explicit
isomorphism between $\KD(n)$ and its dual (see~\ref{self-dualKD}) to
derive the Jones projections of the \sacigs of dimension $4$ from the
\sacigs of dimension $n$ of $K_2$.
\begin{prop}
  When $n$ is odd, there are $n$ \sacigs $(J_k)_{k=0,\dots,n-1}$ of
  dimension $4$ in $\KD(n)$. The Jones projection of $J_k$ is $p_{J_k}
  = e_1+ \sum_{j=1}^m q(0,\frac{2kj\pi}{n},2j)$. The principal graph
  of the inclusions $N_1 \subset N_1\rtimes\delta(J_k)$ and $N_2
  \subset N_2\rtimes J_k$ are respectively $D_n/\mathbb{Z}_2$
  (see~\ref{grapheDimpair}) and $D_{2n+2}^{(1)}$.
 \end{prop}
\begin{proof}
  By self-duality, the \sacigs of dimension $4$ are mapped by $\delta$
  to the \sacigs of dimension $n$ of $K_2$, and there are $n$ of them
  (see~\ref{sacigsK2}). The Jones projections $p_{J_k}$ are given in
  Proposition~\ref{J4} together with the central projections of $J_k$;
  the principal graph of $N_1 \subset N_1\rtimes\delta(J_k)$ follows.
  The principal graph of $N_2 \subset N_2\rtimes J_k$ follows
  from~\ref{sacigsK2}.
\end{proof}

\subsection{The lattice $\mathrm{l}(\KD(n))$}
\label{resumeimpair}

The following theorem summarizes the results of this section, and
describe most, if not all (see
Conjecture~\ref{conj.lattice.nodd}), of the lattice
$\mathrm{l}(\KD(n))$ of \sacigs of $\KD(n)$.
\begin{theorem}
  \label{theorem.lattice.nodd}
  When $n$ is odd, the lattice $\ll(\KD(n))$ is self-dual. It has:
  \begin{itemize}
  \item 3 \sacigs of dimension $2$: $I_0$, $I_1$, and $I_2$,
    contained in the intrinsic group $J_0$;
  \item 3 \sacigs of dimension $2n$:
    \begin{itemize}
    \item $K_2$, isomorphic to $\mathcal{L}^{\infty}(D_n)$;
    \item $K_3$ and $K_4$, isomorphic via $\Theta'$, but non Kac subalgebras;
    \end{itemize}
  \item $n$ \sacigs of dimension $n$ contained in $K_2$:
    $K_1=I(e_1+e_2+e_3+e_4)$, and, for $j=0,\dots,m-1$,
    $I(e_1+e_2+r_{1,1}^{2j+1})$ and $I( e_1+e_2+r_{2,2}^{2j+1})$;

  \item $n$ \sacigs of dimension $4$:
   \begin{equation*}
     I(e_1+ \sum_{j=1}^m q(0,\frac{2kj\pi}{n},2j)), \quad\text{ for $k=0,\dots,n-1$.}
  \end{equation*}
  \end{itemize}
  When $n$ is not prime, $\ll(\KD(n))$ contains further:
  \begin{itemize}
  \item the \sacigs of odd dimension dividing strictly $n$, contained
    in $K_2$, and associated to the dihedral subgroups of $D_n$;
  \item the \sacigs of dimension a multiple of $4$, image of the
    previous ones by $\delta$;
  \item the \sacigs of even dimension dividing strictly $2n$:
    \begin{itemize}
    \item the \sacigs of $K_2$ corresponding to the subgroups of
      $<\alpha>$ in $D_n$;
    \item the \sacigs of $K_3$ containing $I_1$; they are the images
      by $\Theta'$ of those of $K_4$ containing $I_2$ below;
    \item \sacigs of $K_4$ containing $I_2$:
      \begin{itemize}
      \item for each $k$ non trivial divisor of $n=kd$, the image of
        the \sacig $K_4$ of dimension $2k$ of $\KD(k)$ by the
        embedding $\varphi_{d}$;
      \item a set $\mathcal{I}$ of other \sacigs, controlled by
        Proposition~\ref{prop.K4impair}, and empty for $n\leq 51$.
      \end{itemize}
    \end{itemize}
  \end{itemize}
\end{theorem}
\begin{proof}
  Follows from Theorem~\ref{theorem.self-dual}, and
  Propositions~\ref{sacign},~\ref{iso-n-impair},~\ref{K1impair},~\ref{prop.sacignimpair},~\ref{prop.K4impair},
  and~\ref{sacig4impair}.
  The emptyness of $\mathcal{I}$ for $n\leq 51$ was checked on
  computer, using that, by Proposition~\ref{prop.K4impair}) there are
  only finitely many possible Jones projections for \sacigs in
  $\mathcal{I}$. This check can be further reduced by using $\delta$
  which exchanges \sacigs of dimension $2k$ and $2d$.
\end{proof}

\begin{cor}
  \label{corollary.lattice.nprime}
  For $n$ odd prime, the non trivial \sacigs of $\KD(n)$ are of
  dimension $2$, $4$, $n$, and $2n$ and the graph of the lattice of
  \sacigs of $\KD(n)$ is similar to Figure~\ref{figure.KD5} with $n$
  \sacigs of dimension $4$ and $n$.
\end{cor}
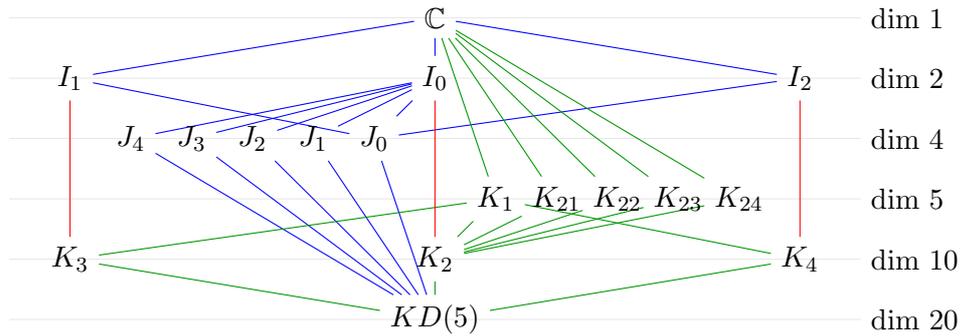
\begin{figure}[h]
  \centering
  \pgfdeclarelayer{background}
\pgfdeclarelayer{nodes}
\pgfsetlayers{background,main,nodes}
\begin{tikzpicture}[yscale=-0.8,xscale=0.8]
  \begin{pgfonlayer}{nodes}
    \node(C)	at (7,1) {$\C$};
    \node(I1)	at (1,2) {$I_1$};
    \node(I0)	at (7,2) {$I_0$};
    \node(I2)	at (13,2) {$I_2$};
    \node(J4)	at (2,3) {$J_4$};
    \node(J3)	at (3,3) {$J_3$};
    \node(J2)	at (4,3) {$J_2$};
    \node(J1)	at (5,3) {$J_1$};
    \node(J0)	at (6,3) {$J_0$};
    \node(K1)	at (8,4) {$K_1$};
    \node(Kp22)	at (9,4) {$K_{21}$};
    \node(Kp23)	at (10,4) {$K_{22}$};
    \node(Kp24)	at (11,4) {$K_{23}$};
    \node(Kp25)	at (12,4) {$K_{24}$};
    \node(K3)	at (1,5) {$K_3$};
    \node(K2)	at (7,5) {$K_2$};
    \node(K4)	at (13,5) {$K_4$};
    \node(KD5)	at (7,6) {$KD(5)$};
  \end{pgfonlayer}
  \begin{pgfscope}
    \tikzstyle{every path}=[blue]
    \draw (C) -- (I1);
    \draw (C) -- (I0);
    \draw (C) -- (I2);
    \draw (I1) -- (J0);
    \draw (I2) -- (J0);

    \draw (I0) -- (J4);
    \draw (I0) -- (J3);
    \draw (I0) -- (J2);
    \draw (I0) -- (J1);
    \draw (I0) -- (J0);
    \draw (KD5) -- (J4);
    \draw (KD5) -- (J3);
    \draw (KD5) -- (J2);
    \draw (KD5) -- (J1);
    \draw (KD5) -- (J0);

    \tikzstyle{every path}=[green]
    \draw (C) -- (K1);
    \draw (C) -- (Kp22);
    \draw (C) -- (Kp23);
    \draw (C) -- (Kp24);
    \draw (C) -- (Kp25);
    \draw (K3) -- (K1);
    \draw (K4) -- (K1);
    \draw (K2) -- (K1);
    \draw (K2) -- (Kp22);
    \draw (K2) -- (Kp23);
    \draw (K2) -- (Kp24);
    \draw (K2) -- (Kp25);
    \draw (K3) -- (KD5);
    \draw (K2) -- (KD5);
    \draw (K4) -- (KD5);

    \tikzstyle{every path}=[red]
    \draw (I1) -- (K3);
    \draw (I2) -- (K4);
    \draw (I0) -- (K2);
  \end{pgfscope}
  \begin{pgfonlayer}{background}
    \newcommand{\dimline}[2]{\draw[color=black!10,very thin](0,#1) -- (14,#1); \node[right] at (14,#1) {dim $#2$};}
    \dimline{1}{1}
    \dimline{2}{2}
    \dimline{3}{4}
    \dimline{4}{5}
    \dimline{5}{10}
    \dimline{6}{20}
  \end{pgfonlayer}
\end{tikzpicture}
  \caption[The lattice of \sacigs of $\KD(5)$]{The lattice of \sacigs of $\KD(5)$}
  \label{figure.KD5}
\end{figure}

The computer exploration results mentioned in
Theorem~\ref{theorem.lattice.nodd} suggests right away the following
conjecture when $n$ is not prime.
\begin{conjecture}
  \label{conj.lattice.nodd}
  Let $n$ be odd. Then, the description of the lattice $\ll(\KD(n))$
  of Theorem~\ref{theorem.lattice.nodd} is complete: $\mathcal I$ is
  empty.
\end{conjecture}

\subsection{The Kac algebra $\KD(3)$ of dimension $12$}

\subsubsection{The lattice  $\mathrm{l}(\KD(3))$}
\label{KD3section}
In this section, we describe the lattice of $\KD(3)$ in detail to
illustrate Corollary~\ref{corollary.lattice.nprime}.

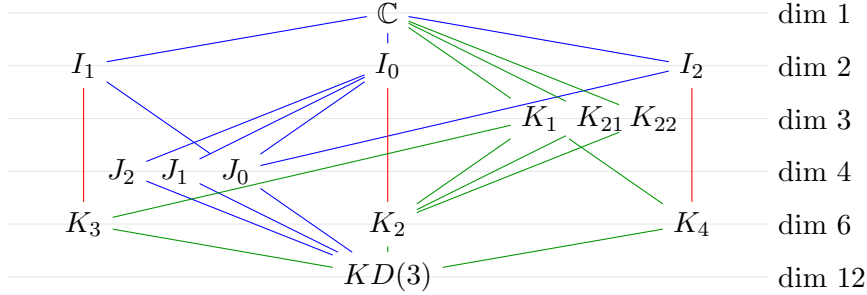
\begin{figure}[h]
  \centering
  \pgfdeclarelayer{background}
\pgfdeclarelayer{nodes}
\pgfsetlayers{background,main,nodes}
\begin{tikzpicture}[xscale=1,yscale=-0.7]
  \begin{pgfonlayer}{nodes}
    \node(C)	at (5,1) {$\C$};
    \node(I1)	at (1,2) {$I_1$};
    \node(I0)	at (5,2) {$I_0$};
    \node(I2)	at (9,2) {$I_2$};
    \node(K1)	at (7,3) {$K_1$};
    \node(L1)	at (7.8,3) {$K_{21}$};
    \node(L2)	at (8.5,3) {$K_{22}$};
    \node(Jp2)	at (1.5,4) {$J_2$};
    \node(Jp1)	at (2.2,4) {$J_1$};
    \node(J0)	at (3,4) {$J_0$};
    \node(K3)	at (1,5) {$K_3$};
    \node(K2)	at (5,5) {$K_2$};
    \node(K4)	at (9,5) {$K_4$};
    \node(KD3)	at (5,6) {$KD(3)$};
  \end{pgfonlayer}
\begin{pgfscope}
  \tikzstyle{every path}=[blue]
  \draw (C) -- (I1);
  \draw (C) -- (I0);
  \draw (C) -- (I2);
  \draw (I1) -- (J0);
  \draw (I2) -- (J0);
  \draw (I0) -- (J0);
  \draw (I0) -- (Jp2);
  \draw (I0) -- (Jp1);
  \draw (Jp2) -- (KD3);
  \draw (Jp1) -- (KD3);
  \draw (J0) -- (KD3);
  \tikzstyle{every path}=[red]
  \draw (I0) -- (K2);
  \draw (I1) -- (K3);
  \draw (I2) -- (K4);
  \tikzstyle{every path}=[green]
  \draw (C) -- (K1);
  \draw (C) -- (L1);
  \draw (C) -- (L2);
  \draw (K1) -- (K2);
  \draw (K1) -- (K3);
  \draw (K1) -- (K4);
  \draw (L1) -- (K2);
  \draw (L2) -- (K2);
  \draw (K3) -- (KD3);
  \draw (K2) -- (KD3);
  \draw (K4) -- (KD3);
  \end{pgfscope}
  \begin{pgfonlayer}{background}
    \newcommand{\dimline}[2]{\draw[color=black!10,very thin](0,#1) -- (10,#1); \node[right] at (10,#1) {dim $#2$};}
    \dimline{1}{1}
    \dimline{2}{2}
    \dimline{3}{3}
    \dimline{4}{4}
    \dimline{5}{6}
    \dimline{6}{12}
  \end{pgfonlayer}
\end{tikzpicture}
  \caption{The lattice of \sacigs of $\KD(3)$}
  \label{figure.KD3}
\end{figure}
\begin{prop}
  \label{prop.lattice.KD3}
In $\KD(3)$,
\begin{itemize}
\item The \sacigs of dimension $2$ are $I_0$, $I_1$, and $I_2$;
\item The \sacigs of dimension $3$ are contained in $K_2$; they are
  the algebras of functions constant w.r.t. the three subgroups of
  order $2$:
\begin{align*}
  K_1&=\C(e_1+e_2+e_3+e_4)\oplus \C(r^1_{1,1}+e^2_{1,1}) \oplus \C(r^1_{2,2}+e^2_{2,2})\,;    \\
  K_{21}&=\C(e_1+e_2+r^1_{1,1}) \oplus \C(e_3+e_4+e^2_{2,2})   \oplus \C(r^1_{2,2}+e^2_{1,1})\,;  \\
  K_{22}&=\C(e_1+e_2+r^1_{2,2}) \oplus \C(e_3+e_4+e^2_{1,1})   \oplus \C(r^1_{1,1}+e^2_{2,2})\,.
\end{align*}
\item The \sacigs of dimension $4$, images of those of dimension $3$
  by $\delta$, are:
\begin{align*}
J_0&=\C(e_1+q_1) \oplus \C(e_2+q_2) \oplus \C(e_3+p_2) \oplus
  \C(e_4+p_1)\,;\\
  J_1&=\C(e_1+q(0,\frac{2\pi}{3},2  )) \oplus \C(e_2+q(0,\frac{5\pi}{3},2    )) \oplus
        \C(e_3+q(  \frac{ \pi}{3},0,1)) \oplus \C(e_4+q( -\frac{ \pi}{3},\pi,1))\,;
        \\
  J_2&= \C(e_1+q(0,\frac{4\pi}{3},2))\oplus \C(e_2+q(0,\frac{\pi}{3},2))\oplus \C(e_3+q(-\frac{\pi}{3},0,1))\oplus \C(e_4+q(\frac{\pi}{3},\pi,1))\,.
\end{align*}
\item The three \sacigs of dimension $6$ are $K_2$, $K_3$, and $K_4$.
\end{itemize}

With the notations of~\ref{graphe}, the inclusion $N_2\subset N_2\rtimes J$ has for principal graph
\begin{itemize}
\item  $D^{(1)}_8$ for $J= J_1$ or $J_2$.
\item  $A_5$ for $J=K_1$, $K_{21}$ or $K_{22}$,
\item  $D_6/\mathbb{Z}_2$ for $J=K_3$ or $K_4$.
\end{itemize}
For $J= I_0$, $I_1$, $I_2$, $J_0$ or $K_2$, it is of depth $2$.

 The lattice of \sacigs  de $\KD(3)$ is as given in
  Figure~\ref{figure.KD3}.
 \end{prop}

\subsubsection{Realization of $N_2 \subset N_2 \rtimes J_1$ by composition of subfactors}\label{KD3IK}

As in~\ref{KPIK}, the inclusion $N_2\subset N_2\rtimes J_i$ of
principal graph $D^{(1)}_8$ can be interpreted as $M^{(\alpha, \Z_2)}
\subset M \rtimes_\beta \Z_2$ (see also~\cite{Popa.1990}).
\begin{prop}
  The inclusion $N_2 \subset N_2 \rtimes J_1$, which is of principal graph $D_8^{(1)}$, can
  be put under the form $M^{(\alpha, \Z_2)} \subset M \rtimes_\beta
  \Z_2$ as follows. Take the following basis of $J_1$:
  \begin{displaymath}
    B_1= e_1+q(0,\frac{2\pi}{3},2), \; B_2=e_2+q(0,\frac{5\pi}{3},2), \; B_3=e_3+q(\frac{ \pi}{3},0,1), \; B_4=e_4+q( -\frac{ \pi}{3},\pi,1)\,,
  \end{displaymath}
  and set $ v=B_1-B_2+B_3-B_4$. Recall that $\lambda(a^3)=B_1+B_2-B_3-B_4$.

  Set $M=N_2 \rtimes I_0$, and let $\alpha$ be the automorphism of $M$
  which fixes $N_2$ and changes $\lambda(a^3)$ into its opposite (this
  is the dual action of $\mathbb{Z}_2$) and $\beta=\Ad v$.  Then,
  $\alpha$ and $\beta$ are involutive automorphisms of $M$ such that
  the period of $\beta \alpha$ is $6$, while $M^{\alpha}=N_2$ and
  $N_2\rtimes J_1=M\rtimes_{\beta}\mathbb{Z}_2$.
\end{prop}
\begin{proof}
  As in~\ref{action} (see~\cite{David.2005}), we identify $N_3$ and
  $N_2\rtimes \KD(3)$. Therefore, for all $x \in N_2$, one has
  $\beta(x)=vxv^*=(v_{(1)}\triangleright x)v_{(2)}v^*$.
  From $\Delta(v)=v\otimes (B_1-B_2) +w\otimes (B_3-B_4)$,
  (with
  $w=(e_1+q(0,-\frac{2\pi}{3},2))-(e_2+q(0,\frac{\pi}{3},2))+(e_3+q(-\frac{\pi}{3},0,1))-(e_4+q(\frac{ \pi}{3},\pi,1))$),
  we get, for $x\in N_2$,
  $$\beta(x)= (v\triangleright x)(B_1+B_2)+(w\triangleright x)(B_3+B_4)\,,$$
  and deduce that $\beta$ normalizes $M$. By a straightforward
  calculation, we obtain $(wv)^3=(vw)^3=1$ and
  $(\beta\alpha)^6=\id$. Then, $N_2\rtimes J_1$ is indeed the cross
  product of $M$ by $\beta$, since $\lambda(a^3)$ and $v$ generate the
  subalgebra $J_1$.
\end{proof}

\subsection{The algebra $\KD(9)$ of dimension $36$}
\label{KD9}

We illustrate, on $\KD(9)$, Theorem~\ref{theorem.lattice.nodd} for $n$
not prime.
\begin{corollary}
  The lattice of \sacigs of $\KD(9)$ is given by
  Figure~\ref{figure.KD9}. Namely:
  \begin{itemize}
  \item The \sacigs of dimension $2$ are $I_0$, $I_1$ et $I_2$.
  \item The \sacigs of dimension $3$ are contained in $K_2$:
    \begin{itemize}
    \item The \sacig of constant functions modulo $<\alpha^3,\beta>$ is \\$L_0=I(e_1+e_2+e_3+e_4+r_{1,1}^3+r_{2,2}^3+e_{1,1}^6+e_{2,2}^6)$
    \item The \sacig of constant functions modulo $<\alpha^3,\beta \alpha>$ is \\$L_1=I(e_1+e_2+r_{2,2}^1+r_{1,1}^5+e_{1,1}^6+e_{2,2}^6+r_{2,2}^7)$
    \item The \sacig of constant functions modulo $<\alpha^3,\beta \alpha^2>$ is \\$L_2=I(e_1+e_2+r_{1,1}^1+r_{2,2}^5+e_{1,1}^6+e_{2,2}^6+r_{1,1}^7)$
    \end{itemize}
  \item The nine \sacigs of dimension $4$ are the $J_k=I(e_1+
    \sum_{j=1}^4 q(0,\frac{2kj\pi}{9},2j))$, for $k=0,\dots,8$;
  \item The  \sacigs of dimension $6$ are:
    \begin{itemize}
    \item The Kac subalgebra of  constant functions modulo $<\alpha^3>$ in $K_2$:\\
      $M_{2}=I(e_1+e_2+e_{1,1}^6+e_{2,2}^6)$ which contains all the
      $L_i$, $i=0,1,2$;
    \item The \sacig $M_{3}=I(e_1+e_3+p^{3}_{2,2}+p^{6}_{1,1})$
      contained in $K_3$;
    \item The \sacig $M_{4}=I(e_1+e_4+p^{3}_{1,1}+p^{6}_{1,1})$
      contained in $K_4$.
    \end{itemize}

  \item The nine \sacigs of dimension $9$ are contained in $K_2$:\\
    $K_1=I(e_1+e_2+e_3+e_4)$,  $K_{21}=e_1+e_2+r_{1,1}^{3}$, and
    $K_{22}=e_1+e_2+r_{2,2}^{3}$ whose intersection is $L_0$; \\
    $K_{23}=e_1+e_2+r_{2,2}^{1}$, $K_{24}=e_1+e_2+r_{1,1}^{5}$, and
    $K_{25}=e_1+e_2+r_{2,2}^{7}$ whose intersection is $L_1$; \\
    $K_{23}=e_1+e_2+r_{2,2}^{1}$, $K_{24}=e_1+e_2+r_{1,1}^{5}$, and
    $K_{25}=e_1+e_2+r_{2,2}^{7}$ whose intersection is $L_1$; \\
    $K_{26}=e_1+e_2+r_{1,1}^{1}$, $K_{27}=e_1+e_2+r_{2,2}^{5}$, and
    $K_{28}=e_1+e_2+r_{1,1}^{7}$ whose  intersection is $L_2$.

  \item The three \sacigs of dimension $12$, images by $\delta$ of
    $L_i$, $i=0,1,2$. Their Jones projections are $e_1+p_{1,1}^6$ for
    $\delta(L_0)$, and $e_1+ q(0,\frac{2\pi}{3},6))$ and $e_1+
    q(0,\frac{4\pi}{3},6))$ for the two others.

  \item The three \sacigs of dimension $18$ are $K_2$, $K_3$, and $K_4$.
  \end{itemize}
\end{corollary}

\begin{proof}
  Most of this proposition follows from Theorem~\ref{theorem.lattice.nodd}. To derive
  the formulas for the Jones projections of $\delta(L_i)$, we noticed
  that they all contain $M_2$ as well as three $J_k$'s, and looked on
  computer through the projectors of trace $1/12$ which dominates the
  projections of those \sacigs. Since $L_0$ is contained in $K_1$,
  $\delta(L_0)$ contains $J_0$.
\end{proof}
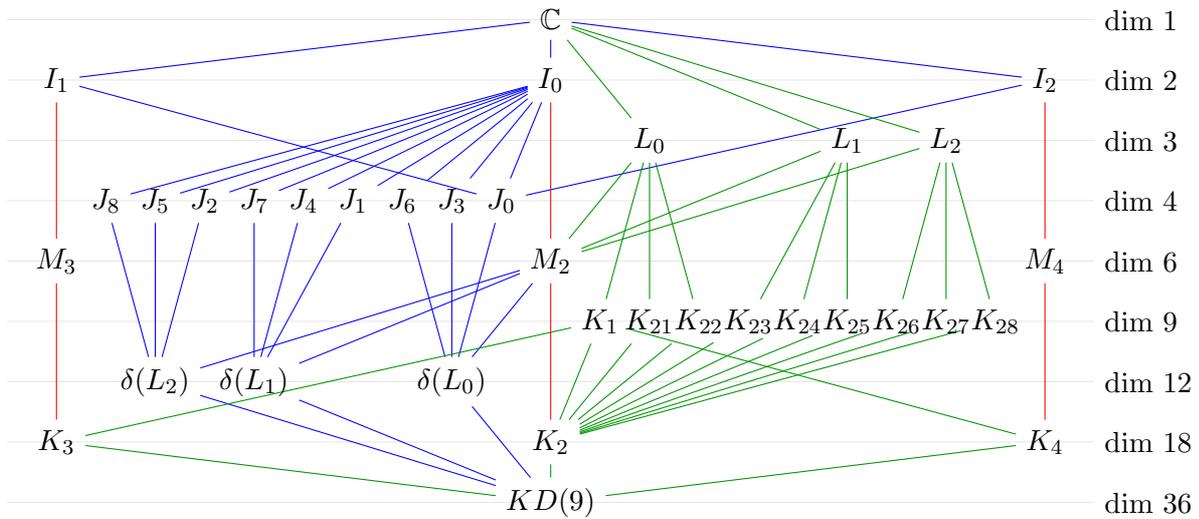
\begin{figure}[h]
    \centering
    \begin{bigcenter}
      \pgfdeclarelayer{background}
\pgfdeclarelayer{nodes}
\pgfsetlayers{background,main,nodes}
\begin{tikzpicture}[yscale=-0.8,xscale=0.65]
  \begin{pgfonlayer}{nodes}
    \node(C)	at (10,0) {$\C$};
    \node(I1)	at (0,1) {$I_1$};
    \node(I0)	at (10,1) {$I_0$};
    \node(I2)	at (20,1) {$I_2$};
    \node(L1)	at (16,2) {$L_1$};
    \node(L0)	at (12,2) {$L_0$};
    \node(L2)	at (18,2) {$L_2$};    
     \node(J8)	at (1,3) {$J_8$};
    \node(J7)	at (4,3) {$J_7$};
    \node(J6)	at (7,3) {$J_6$};
    \node(J5)	at (2,3) {$J_5$};
    \node(J4)	at (5,3) {$J_4$};
    \node(J3)	at (8,3) {$J_3$};
    \node(J2)	at (3,3) {$J_2$};
    \node(J1)	at (6,3) {$J_1$};
    \node(J0)	at (9,3) {$J_0$};
    \node(M4)	at (20,4) {$M_{4}$};
    \node(M3)	at (0,4) {$M_{3}$};
    \node(M2)	at (10,4) {$M_{2}$};
    \node(K1)	at (11,5) {$K_1$};
    \node(K21)	at (12,5) {$K_{21}$};
    \node(K22)	at (13,5) {$K_{22}$};
    \node(K23)	at (14,5) {$K_{23}$};
    \node(K24)	at (15,5) {$K_{24}$};
    \node(K25)	at (16,5) {$K_{25}$};
    \node(K26)	at (17,5) {$K_{26}$};
    \node(K27)	at (18,5) {$K_{27}$};
    \node(K28)	at (19,5) {$K_{28}$};
    \node(d2)	at (2,6) {$\delta(L_2)$};
    \node(d1)	at (4,6) {$\delta(L_1)$};
    \node(d0)	at (8,6) {$\delta(L_0)$};
    \node(K3)	at (0,7) {$K_3$};
    \node(K2)	at (10,7) {$K_2$};
    \node(K4)	at (20,7) {$K_4$};
    \node(KD9)	at (10,8) {$KD(9)$};
  \end{pgfonlayer}
  \begin{pgfscope}
    \tikzstyle{every path}=[blue]
    \draw (C) -- (I1);
    \draw (C) -- (I0);
    \draw (C) -- (I2);
    \draw (I1) -- (J0);
    \draw (I2) -- (J0);
    \draw (I0) -- (J8);
    \draw (I0) -- (J7);
    \draw (I0) -- (J6);
    \draw (I0) -- (J5);
    \draw (I0) -- (J4);
    \draw (I0) -- (J3);
    \draw (I0) -- (J2);
    \draw (I0) -- (J1);
    \draw (I0) -- (J0);
    \draw (d2) -- (J8);
    \draw (d1) -- (J7);
    \draw (d0) -- (J6);
    \draw (d2) -- (J5);
    \draw (d1) -- (J4);
    \draw (d0) -- (J3);
    \draw (d2) -- (J2);
    \draw (d1) -- (J1);
    \draw (d0) -- (J0);
    \draw (d1) -- (M2);
    \draw (d2) -- (M2);
    \draw (d0) -- (M2);
    \draw (KD9) -- (d0);
    \draw (KD9) -- (d1);
    \draw (KD9) -- (d2);

    \tikzstyle{every path}=[green]
    \draw (C) -- (L1);
    \draw (C) -- (L0);
    \draw (C) -- (L2);
    \draw (L0) -- (K1);
    \draw (L0) -- (K22);
    \draw (L0) -- (K21);
    \draw (L0) -- (M2);
    \draw (L1) -- (K23);
    \draw (L1) -- (K24);
    \draw (L1) -- (K25);
    \draw (L1) -- (M2);
    \draw (L2) -- (K26);
    \draw (L2) -- (K27);
    \draw (L2) -- (K28);
    \draw (L2) -- (M2);
    \draw (K3) -- (K1);
    \draw (K4) -- (K1);
    \draw (K2) -- (K1);
    \draw (K2) -- (K22);
    \draw (K2) -- (K23);
    \draw (K2) -- (K24);
    \draw (K2) -- (K25);
    \draw (K2) -- (K21);
    \draw (K2) -- (K26);
    \draw (K2) -- (K27);
    \draw (K2) -- (K28);
    \draw (K3) -- (KD9);
    \draw (K2) -- (KD9);
    \draw (K4) -- (KD9);

    \tikzstyle{every path}=[red]
    \draw (I1) -- (M3);
    \draw (M3) -- (K3);
    \draw (I2) -- (M4);
    \draw (M4) -- (K4);
    \draw (I0) -- (M2);
    \draw (M2) -- (K2);
  \end{pgfscope}
  \begin{pgfonlayer}{background}
    \newcommand{\dimline}[2]{\draw[color=black!10,very thin](-1,#1) -- (21,#1); \node[right] at (21,#1) {dim $#2$};}
    \dimline{0}{1}
    \dimline{1}{2}
    \dimline{2}{3}
    \dimline{3}{4}
    \dimline{4}{6}
    \dimline{5}{9}
    \dimline{6}{12}
    \dimline{7}{18}
    \dimline{8}{36}
  \end{pgfonlayer}
\end{tikzpicture}
    \end{bigcenter}
    \caption{The lattice of \sacigs of $\KD(9)$}
  \label{figure.KD9}
\end{figure}

\subsection{The algebra $\KD(15)$ of dimension $60$}
\label{KD15}

Figure~\ref{figure.KD15} illustrates, on $\KD(15)$,
Theorem~\ref{theorem.lattice.nodd} for $n$ not the square of a prime.

\begin{figure}[h]
  \begin{bigcenter}
    \scalebox{0.5}{\input{Fig/KD15.tikz}}
  \end{bigcenter}
  \caption{The lattice of \sacigs of $\KD(15)$}
  \label{figure.KD15}
\end{figure}

\newpage
\section{The Kac algebras $\KD(n)$ for $n$ even}
\label{section.KD.even}

In this section we assume that $n$ is even: $n=2m$. The structure of
$\KD(n)$ is then very different from the odd case. The intrinsic group
is $D_4$, and the algebra $\KD(2m)$ is never self-dual. On the other
hand, the \sacigs $K_3$ and $K_4$ are non isomorphic Kac
subalgebras. Using that $K_2$ is isomorphic to $L^\infty(D_n)$, $K_4$
to $\KD(m)$ and $K_3$ to $\KB(m)$ (i.e. $\KQ(m)$ for $m$ odd), we
recover a large part of the lattice of \sacigs by induction. For
example the lattices of \sacigs of $\KD(4)$ (see~\ref{KD4}) and
$\KD(8)$ (see~\ref{KD8}) and a large part of the lattice of $\KD(6)$
(see~\ref{KD6}) can be constructed this way.

\subsection{The algebra of the intrinsic group $K_0$}
\label{section.K0}

We start with the algebra $K_0$ of the intrinsic group of $\KD(n)$.
From~\ref{plonge} and~\ref{intrinseque}, it is generated by
$\lambda(a^{m})$ and $\lambda(b)$, and isomorphic to $\KD(2)\equiv
\C(D_4)$.

\subsubsection{The isomorphism between $\KD(2)$ and $K_0$ using matrix units}
\label{section.KD2.K0}

We have seen in~\ref{plonge} that the application $\varphi_m$ from
$\KD(2)$ to $K_0$ which sends $\lambda(b)$ to $\lambda(b)$ and
$\lambda(a)$ to $\lambda(a^m)$ can be extended into a Kac algebra
isomorphism. We now give the expression of $\varphi_m$ on the matrix
units of $\KD(2)$ (which will be marked with a $'$ to avoid
confusion). Taking the conditions on the index $j$ modulo $4$, and
using the formulas~\ref{lambda} in $\KD(2)$ and $K_0$, we obtain:\\
For $n=2m=4m'$:
  \begin{align*}
    \varphi_m(e'_1) &=e_1+e_4 + \sum_{j \equiv 0} p^j_{1,1}\,,&&
    \varphi_m(e'_2) =e_2+e_3 + \sum_{j \equiv 0} p^j_{2,2}\,,\\
    \varphi_m(e'_3) &= \sum_{j \equiv 2} p^j_{2,2}\,,&&
    \varphi_m(e'_4) =\sum_{j \equiv 2} p^j_{1,1}\,.
  \end{align*}
For $n=2m=4m'+2$:
  \begin{align*}
    \varphi_m(e'_1) &=e_1 + \sum_{j \equiv 0} p^j_{1,1}\,,&&
    \varphi_m(e'_2) =e_2+ \sum_{j \equiv 0} p^j_{2,2}\,,\\
    \varphi_m(e'_3) &= e_3 + \sum_{j \equiv 2} p^j_{2,2}\,,&&
    \varphi_m(e'_4) =e_4 + \sum_{j \equiv 2} p^j_{1,1}\,.
  \end{align*}
In both cases:
  $$\varphi_m(e'_{1,2})=\sum_{j \equiv 1} e^j_{1,2}+\sum_{j \equiv 3} e^j_{2,1}\,.$$

\subsubsection{The lattice $\mathrm{l}(K_0))$}
\label{sacigK0}

We obtain the central projections and the lattice of \sacigs of $K_0$
by combining the results of~\ref{KD2} and~\ref{section.KD2.K0}. All
the \sacigs are Kac subalgebras. We keep the same naming convention
for the \sacigs in $\KD(2)$ and their images by $\varphi_m$ in $K_0$.

The algebra $K_0$ contains $J_0=\C[H]$ and all
its \sacigs (note that by~\ref{tour} (5), those are also \sacigs of
$K_4$):
\begin{itemize}
  \item $J_0=\C(e_1+e_4+q_1) \oplus \C(e_2+e_3+q_2) \oplus  \C p_2 \oplus \C p_1$;%

  \item $I_0=\C(e_1 +e_2+e_3 +e_4+q_1+q_2)\oplus  \C(p_1+p_2)$;%

  \item $I_1=\C(e_1 +e_4+q_1+p_2)\oplus  \C(e_2 +e_3+p_1+q_2)$;%

 \item $I_2=\C(e_1 +e_4+p_1+q_1)\oplus  \C(e_2 +e_3+p_2+q_2)$. %
\end{itemize}

The basis of the other \sacigs depends on the parity of $m$:
\begin{description}
\item[For $n=4m'$] The Jones projection of $K_0$ is $e_1+e_4
  +\sum_{j=1}^{m'-1}p^{4j}_{1,1}$. The other central projection in the
  same connected component of the Bratelli diagram of $K_0 \subset
  \KD(n)$ is $e_2+e_3 + \sum_{j \equiv 0} p^j_{2,2}$. The \sacigs not
  contained in $J_0$ are:
  \begin{itemize}
  \item $I_3=I(e_1+e_2+e_3+e_4+ \sum_{j \equiv 0}(e^j_{1,1}+e^j_{2,2})+\sum_{j \equiv 1}r^j_{2,2}+\sum_{j \equiv 3}r^j_{1,1})$;
  \item $I_4=I(e_1+e_2+e_3+e_4+ \sum_{j \equiv 0}(e^j_{1,1}+e^j_{2,2})+\sum_{j \equiv 1}r^j_{1,1}+\sum_{j \equiv 3}r^j_{2,2})$;
  \item $J_{20}=I(e_1+e_2+e_3+e_4+ \sum_{j \equiv 0}(e^j_{1,1}+e^j_{2,2}))$;
  \item $J_m=I(e_1+e_4+ \sum_{j \equiv 0}p^j_{1,1}+\sum_{j \equiv 2}p^j_{2,2})$.
  \end{itemize}
  $K_0$ is a Kac subalgebra of $K_4$ and $J_{20}$ is contained in $K_1$.

\item[For $n=4m'+2$] The Jones projection of $K_0$ is
  $e_1+\sum_{j=1}^{m'}p^{4j}_{1,1}$.  The other central projection in
  the same connected component of the Bratelli diagram of $K_0 \subset
  \KD(n)$ is $e_2+ \sum_{j \equiv 0} p^j_{2,2}$.
The \sacigs not contained in $J_0$ are:
  \begin{itemize}
  \item $I_3=I(e_1+e_2+ \sum_{j \equiv 0}(e^j_{1,1}+e^j_{2,2})+\sum_{j \equiv 1}r^j_{2,2}+\sum_{j \equiv 3}r^j_{1,1})$;
  \item $I_4=I(e_1+e_2+ \sum_{j \equiv 0}(e^j_{1,1}+e^j_{2,2})+\sum_{j \equiv 1}r^j_{1,1}+\sum_{j \equiv 3}r^j_{2,2})$;
  \item $J_{20}=I(e_1+e_2+ \sum_{j \equiv 0}(e^j_{1,1}+e^j_{2,2}))$;
  \item $J_m=I(e_1+e_3+ \sum_{j \equiv 0}p^j_{1,1}+\sum_{j \equiv 2}p^j_{2,2})$.
  \end{itemize}
  From~\ref{tour}~(5), $J_{m}$ is contained in $K_3$.
\end{description}
In both cases, $J_{20}$ is the \sacig of functions which are constant on
the right cosets for $ <\alpha^2>$ in $K_2$
(see~\ref{section.K2}). Recall that, by~\ref{intrin}, all the \sacigs
of dimension $2$ of $\KD(n)$ are in $K_0$.

\subsubsection{Principal graphs coming from $K_0$ and its \sacigs}
\label{grapheK0}

Since $\KD(n)$ is not self-dual when $n$ is even, the Bratelli diagram
of an inclusion $I\subset \KD(n)$ yields by~\ref{grapheprincipal} the
principal graph of an intermediate factor in $\widehat{\KD(n)}$. With
the notations of~\ref{graphe}, here are the principal graphs coming
from the \sacigs of $K_0$:
\begin{itemize}
\item The inclusions $N_1 \subset N_1\rtimes \delta(I_0)$ and
  $R\subset R\rtimes \delta(J_{20})$ are of depth $2$;
\item The principal graph of $N_1 \subset N_1\rtimes \delta(I_1)$ and
  of $N_1 \subset N_1\rtimes \delta(I_2)$ is $D_{2n}/\mathbb{Z}_2$;
\item The principal graph of $N_1\subset N_1\rtimes \delta(I_3)$ and
  of $N_1\subset N_1\rtimes \delta(I_4)$ is $QB_{m'}$ for $n=4m'$
  and $DB_{m'}$ for $n=4m'+2$;
\item The principal graph of $N_1 \subset N_1\rtimes\delta(J_0)$ and
  of $N_1\subset N_1\rtimes \delta(J_m)$ is $D_n/\mathbb{Z}_2$
  (see~\ref{sacig4pair});
\item The principal graph of $N_1\subset N_1\rtimes \delta(K_0)$ is
  $D_m/\mathbb{Z}_2$.
\end{itemize}

\subsection{The Kac subalgebra $K_4=I(e_1+e_4)$}

As in~\ref{plonge}, consider the Kac subalgebra of $\KD(n)$ isomorphic
to $\KD(m)$ and generated by $\lambda(a^2)$ and $\lambda(b)$. Its
dimension is $2n$, and it contains $e_1+e_4$
(see~\ref{form2}). Therefore, it coincides with $K_4$.

From the formulas of~\ref{form2} or~\ref{cop}:
$$K_4=\C(e_1+e_4) \oplus \C(e_2+e_3)\oplus  \C p^m_{1,1} \oplus \C p^m_{2,2} \oplus \bigoplus_{j=1}^{m-1}K_4^j\,,$$
where $K_4^j$ is the factor $M_2(\C)$ with matrix units
$e^j_{1,1}+e^{n-j}_{2,2}$, $e^j_{1,2}+e^{n-j}_{2,1}$,
$e^j_{2,1}+e^{n-j}_{1,2}$, and $e^j_{2,2}+e^{n-j}_{1,1}$ (note: $j$
and $n-j$ have the same parity).

\subsection{The Kac subalgebra $K_3=I(e_1+e_3)$}
\label{section.KD.even.K3}

Since $\Delta(e_1+e_3)$ is symmetric (see~\ref{cop}),
by~\ref{irredprof2}, $K_3$ is a Kac subalgebra of $\KD(n)$.

\label{K3pair}

From the expression of the coproduct of the Jones projection, we deduce:\\
if $m$ is odd:
  $$K_3=I(e_1+e_3)=\C(e_1+e_3) \oplus \C(e_2+e_4)\oplus \C e^m_{1,1}\oplus \C e^m_{2,2} \oplus   \bigoplus_{j=1}^{m-1}K_3^j\,,$$
if $m$ is even:
  $$K_3=I(e_1+e_3)=\C(e_1+e_3) \oplus \C(e_2+e_4)\oplus \C r^m_{1,1}\oplus \C r^m_{2,2} \oplus   \bigoplus_{j=1}^{m-1}K_3^j\,,$$
where in both cases $K_3^j$ is the factor  $M_2(\C)$ with matrix
units:\\
for $j$ even:
\begin{displaymath}
  e^j_{2,2}+e^{n-j}_{1,1}, \quad e^j_{2,1}-e^{n-j}_{1,2},
  e^j_{1,2}-e^{n-j}_{2,1}, \quad \text{ and } \quad e^j_{1,1}+e^{n-j}_{2,2}\,;
\end{displaymath}
for $j$ odd:
\begin{gather*}
  r^j_{2,2}+r^{n-j}_{1,1},\quad r^j_{2,1}-r^{n-j}_{1,2}, \quad r^j_{1,2}-r^{n-j}_{2,1}, \quad\text{ and }\quad r^j_{1,1}+r^{n-j}_{2,2}\,,\\
  \text{ or }\\
  e^j_{1,1}+e^{n-j}_{1,1},\quad e^j_{2,1}-e^{n-j}_{2,1}, \quad e^j_{1,2}-e^{n-j}_{1,2}, \quad\text{ and }\quad e^j_{2,2}+e^{n-j}_{2,2}\,.
\end{gather*}

$K_3$ is further studied in Section~\ref{section.KB}, in connexion
with the families of Kac algebras $\KQ(m)$ and $\KB(m)$.

\subsection{The Kac subalgebra $K_1=I(e_1+e_2+e_3+e_4)$}
\label{K1pair}

\begin{proposition}  For $n=2m$, $K_1=I(e_1+e_2+e_3+e_4)=K_2\cap K_3 \cap K_4$ is a
  commutative Kac subalgebra isomorphic to $L^\infty(D_m)$.  Its matrix
  structure is given by
  \begin{displaymath}
    K_1=\C(e_1+e_2+e_3+e_4) \oplus
    \bigoplus_{j=1, \, j \text{  odd}}^{n-1} \C(r^j_{1,1}+r^{n-j}_{2,2})
    \oplus \bigoplus_{j=1,\, j \text{ even}}^{n-1} \C(e^j_{1,1}+e^{n-j}_{2,2})
  \end{displaymath}
  Its lattice of \sacigs is the dual of the lattice of subgroups of
  $D_m$.

  In $K_2\equiv L^\infty(D_{2m})$ it is the subalgebra of constant
  functions on the right cosets of $\{1, \alpha^m\}$. In $K_4\equiv
  \KD(m)$, $K_1$ plays the role of $K_2$.
\end{proposition}
\begin{proof}
  The commutative subalgebra $K_1$ has to play the role of $K_2$ in $K_4$ because $K_4\equiv
  \KD(m)$ has a single commutative \sacig of this dimension. The rest
  follows from~\ref{e1e2},~\ref{K1}, and~\ref{cop}.
\end{proof}

\subsection{The \sacigs of dimension $4$}
\label{sacig4pair}

Recall that, from~\ref{sacign}, any \sacig of dimension $4$ is
contained in some $K_i$. In this section, we use this fact to
inductively describe %
all the \sacigs of dimension $4$ of $\KD(2m)$. %
As in the case $n$ odd (see~\ref{J4}), we exhibit $2m$ projections
that generate $2m$ \sacigs $J_k$ of dimension $4$:
\begin{lemma}
  For $k=0,\dots,n-1$, define
  \begin{displaymath}
    p_{J_k} = e_1+e_i + \sum_{j=1}^{m-1} q(0,\frac{2kj\pi}{n},2j)\,,
  \end{displaymath}
  where $i=3$ (resp. $i=4$) for $k$ odd (resp. even).  The projection
  $p_{J_k}$ is the Jones projection of a \sacig $J_k$ of dimension $4$
  of $K_3$ (resp. $K_4$) for $k$ odd (resp. even). In particular,
  $J_0$ and $J_m$ are contained in $K_0$.

  Furthermore, the principal graph of the inclusion $N_1 \subset
  N_1\rtimes\delta(J_k)$ is $D_n/\mathbb{Z}_2$
  (see~\ref{grapheDpair}).
\end{lemma}

\begin{prop}
  The \sacigs of dimension $4$ of $\KD(2m)$ are:
  \begin{itemize}
  \item in $K_3$: the $m$ \sacigs $J_k$, $k$ odd;
  \item in $K_4$: the $m$ \sacigs $J_k$, $k$ even;
  \item in $K_2$:
    \begin{itemize}
    \item when $m$ is odd, the unique \sacig $J_{20}$ of dimension $4$;
    \item when $m$ is even, the five \sacigs of dimension $4$.
    \end{itemize}
  \end{itemize}
\end{prop}

\begin{proof}
  \textbf{Case $m$ odd:} $K_3$ has $m$ \sacigs of dimension $4$ since
  it is isomorphic to $\KQ(m)$ (see~\ref{theorem.lattice.nprime});
  $K_4$ has $m$ \sacigs of dimension $4$ since it is isomorphic to
  $\KD(m)$ (see~\ref{sacig4impair}).
  The unique \sacig of dimension $4$ of $K_2$ is $J_{20}$; it
  corresponds to the unique subgroup of order $m$ of $D_n$
  (see~\ref{section.K2}).

  \textbf{Case $m=2m'$ even:} $D_n$ has five subgroups of order $m$,
  and the associated \sacigs of $K_2$ can be made explicit from the
  results of~\ref{sacigsK2}, which we will do in the examples. We
  prove in~\ref{latticeKB} that the \sacigs of dimension $4$ of $K_3$
  are either in $K_1$, or are the $J_k$ for $k$ odd.
  Since $K_4$ is isomorphic to $\KD(2m')$, we can use induction: we
  shall see in~\ref{KD4} that the proposition holds for $m=2$. Assume
  that the proposition holds for $m'$; then the dimension $4$ \sacigs
  of $K_4$ are:
  \begin{itemize}
  \item the \sacigs $J_k$ of $\KD(2m')$ which give the $J_{2k}$ in $K_4$;
  \item those of the $K_2$ of $\KD(2m')$ which itself is the subalgebra $K_1$ of $K_2$ of $\KD(2m)$.
  \end{itemize}
  So the proposition holds for $m$.
\end{proof}

\subsection{The Kac algebra $\KD(4)$ of dimension $16$}
\label{KD4}

We now illustrate the general study of $\KD(2m)$ with $m$ even on the
Kac algebra $\KD(4)$ of dimension $16$. This algebra (and its dual
$\widehat{\KD(4)}$ with underlying algebra $\C^8\oplus M_2(\C) \oplus
M_2(\C)$) are described in~\cite[XI.15]{Izumi_Kosaki.2002}.

\begin{prop}
  The lattice of \sacigs of $\KD(4)$ is as given in Figure~\ref{figure.KD4}.
\end{prop}
\begin{figure}[h]
  \centering
  \pgfdeclarelayer{background}
\pgfdeclarelayer{nodes}
\pgfsetlayers{background,main,nodes}
\begin{tikzpicture}[yscale=-1]
  \begin{pgfonlayer}{nodes}
    \node(C)	at (5,1) {$\C$};
    \node(I1)	at (2,2)  [text=blue] {$I_1$};
    \node(I2)	at (3,2)  [text=blue] {$I_2$};
    \node(I0)	at (5,2)  [text=blue]{$I_0$};
    \node(I3)	at (7,2) [text=blue] {$I_3$};
    \node(I4)	at (8,2)  [text=blue] {$I_4$};
    \node(J0)	at (1,3)  [text=blue] {$J_0$};
    \node(J2)	at (2,3)  [text=blue] {$J_2$};
    \node(J3)	at (3,3) [text=orange] {$J_1$};
    \node(J4)	at (4,3) [text=orange] {$J_3$};
    \node(J1)	at (5,3) [text=blue] {$J_{20}$};
    \node(J5)	at (6,3) [text=green]{$J_{21}$};
    \node(J6)	at (7,3) [text=green]{$J_{22}$};
    \node(J7)	at (8,3) [text=green]{$J_{23}$};
    \node(J8)	at (9,3) [text=green]{$J_{24}$};
    \node(K4)	at (3,4) [text=blue] {  $K_4$};
    \node(K3)	at (5,4) [text=orange] {$K_3$};
    \node(K2)	at (7,4) [text=green] {$K_2$};
    \node(KD4)	at (5,5) {$KD(4)$};
  \end{pgfonlayer}
  \begin{pgfscope}
   \draw (K4) -- (KD4);
  \draw (K3) -- (KD4);
  \draw (K2) -- (KD4);
  \tikzstyle{every path}=[red]
  \draw (C) -- (I0);
   \draw (I0) -- (J1);
    \draw (I3) -- (J1);
  \draw (I4) -- (J1);
    \draw (C) -- (I3);
  \draw (C) -- (I4);
  \tikzstyle{every path}=[blue]
  \draw (C) -- (I1);
  \draw (C) -- (I2);
  \draw (I1) -- (J0);
  \draw (I2) -- (J0);
  \draw (I0) -- (J0);
  \draw (I0) -- (J2);
  \draw (J0) -- (K4);
  \draw (J2) -- (K4);
  \draw (J1) -- (K4);
  \tikzstyle{every path}=[orange]
    \draw (I0) -- (J3);
  \draw (I0) -- (J4);
   \draw (J3) -- (K3);
  \draw (J4) -- (K3);
  \draw (J1) -- (K3);
  \tikzstyle{every path}=[green]
   \draw (J1) -- (K2);
  \draw (J5) -- (K2);
  \draw (J6) -- (K2);
  \draw (J7) -- (K2);
  \draw (J8) -- (K2);
  \draw (I3) -- (J5);
  \draw (I3) -- (J6);
  \draw (I4) -- (J7);
  \draw (I4) -- (J8);
  \end{pgfscope}
  \begin{pgfonlayer}{background}
    \newcommand{\dimline}[2]{\draw[color=black!10,very thin](0,#1) -- (10,#1); \node[right] at (10,#1) {dim $#2$};}
    \dimline{1}{1}
    \dimline{2}{2}
    \dimline{3}{4}
    \dimline{4}{8}
    \dimline{5}{16}
  \end{pgfonlayer}
\end{tikzpicture}
  \caption{The lattice of \sacigs of $\KD(4)$}
  \label{figure.KD4}
\end{figure}
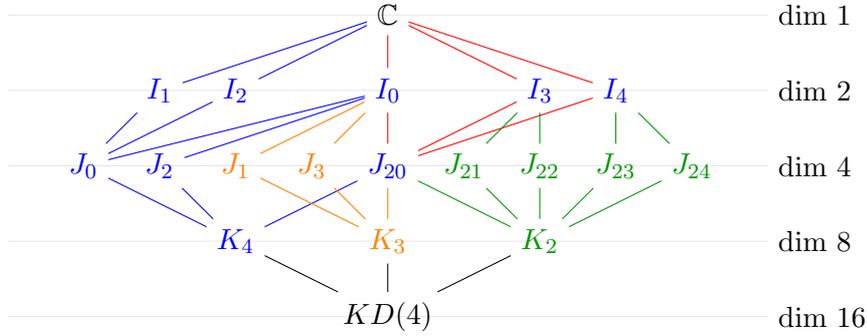
\begin{proof}
  From~\ref{sacign}, any \sacig of $\KD(4)$ is inductively a \sacig of
  one of the three Kac subalgebras of dimension $8$: $K_2$, $K_3$, and
  $K_4$. In the sequel of this section, we study them in turn.
\end{proof}
The coproduct expressions which we use in this section are collected
in~\ref{cop}.
\subsubsection{The \sacigs of $K_4$}
\label{D4}

The Kac subalgebra $K_4=I(e_1+e_4)$ is isomorphic to $\KD(2)$:
$$K_4=\C(e_1+e_4) \oplus \C(e_2+e_3)\oplus  \C q_1 \oplus \C q_2 \oplus M_2(\C)\,,$$
where the matrix units of the factor $M_2(\C)$ are:
\begin{displaymath}
  e^1_{1,1}+e^{3}_{2,2},\quad e^1_{1,2}+e^{3}_{2,1},\quad
  e^1_{2,1}+e^{3}_{1,2},\quad \text{ and } \quad e^1_{2,2}+e^{3}_{1,1}\,.
\end{displaymath}

As a special feature of $\KD(4)$, $K_4$ coincides with the group
algebra $K_0$ of the intrinsic group studied in~\ref{sacigK0}. The element
$c=e_1 -e_2-e_3
+e_4-i(e^1_{1,1}-e^{1}_{2,2})-(q_1-q_2)+i(e^3_{1,1}-e^{3}_{2,2})$ of
$K_4$ is group-like and of order $4$, and its square is
$\lambda(a^4)$. The intrinsic group $D_4$ of $\KD(4)$ is therefore
generated by $c$ and $\lambda(b)$.

The \sacigs of $K_4$ are the algebras of the subgroups of $D_4$
(see~\ref{sacigK0}). In dimension $2$, they are:
\begin{itemize}
\item $I_0=\C(e_1 +e_2+e_3 +e_4+q_1+q_2)\oplus  \C(p_1+p_2)$ generated by  $\lambda(a^4)$;
\item $I_1=\C(e_1 +e_4+q_1+p_2)\oplus  \C(e_2 +e_3+p_1+q_2)$ generated by  $\lambda(ba^4)$;
\item $I_2=\C(e_1 +e_4+p_1+q_1)\oplus  \C(e_2 +e_3+p_2+q_2)$ generated by $\lambda(b)$;
\item $I_3=I(e_1 +e_2+e_3 +e_4+r^1_{2,2}+r^3_{1,1})$ generated by $\lambda(b)c$;
\item $I_4=I(e_1 +e_2+e_3 +e_4+r^1_{1,1}+r^3_{2,2})$ generated by $\lambda(b)c^3$.
\end{itemize}
Since $K_4=K_0$, those give all the \sacigs of dimension $2$ of $\KD(4)$.

In dimension $4$, the \sacigs are:
\begin{itemize}
\item $J_0=\C(e_1+e_4+q_1) \oplus \C(e_2+e_3+q_2) \oplus  \C p_2 \oplus \C p_1$
  generated by $\lambda(a^4)$ and $\lambda(b)$;
\item$J_{20}=\C(e_1+e_2+e_3+e_4)  \oplus \C(q_1+q_2) \oplus \C(r^{1}_{1,1}+r^{3}_{2,2})\oplus \C(r^1_{2,2}+r^{3}_{1,1})$
  generated by $c^2$ and $\lambda(b)c$;
\item$J_2=\C(e_1+e_4+q_2) \oplus \C(e_2+e_3+q_1) \oplus  \C(e^1_{1,1}+e^3_{2,2}) \oplus \C(e^1_{2,2}+e^3_{1,1})$
  generated by $\lambda(c)$.
\end{itemize}
\subsubsection{The \sacigs of $K_3$}

From~\ref{K3pair},
$$K_3=I(e_1+e_3)=\C(e_1+e_3) \oplus \C(e_2+e_4)\oplus  \C r^2_{1,1} \oplus \C r^2_{2,2} \oplus M_2(\C)\,,$$
where the matrix units of the factor $M_2(\C)$ are
\begin{displaymath}
  e^1_{1,1}+e^{3}_{1,1},\quad e^1_{1,2}-e^{3}_{1,2},\quad
  e^1_{2,1}-e^{3}_{2,1},\quad \text{ and }\quad e^1_{2,2}+e^{3}_{2,2}\,.
\end{displaymath}

Since the coproduct of $r^2_{1,1}$ is not symmetric (checked on
computer), $K_3$ is a non trivial Kac algebra of dimension $8$, and
therefore isomorphic to $\KP$. Its \sacigs are $I_0$, $I_3$, and $I_4$
contained in $J_{20}$, and $J_1$ and $J_3$ which contains $I_0$
(see~\ref{KP}):
\begin{align*}
  J_1=I(e_1+e_3+r^2_{1,1})&=\C(e_1+e_3+r^2_{1,1}) \oplus \C(e_2+e_4+r^2_{2,2})\\
  &\oplus \C(q(\pi/4, 0,1)+q(\pi/4, \pi,3))\oplus \C(q(-\pi/4, \pi,1)+q(-\pi/4, 0,3))\,, \\
  J_3=I(e_1+e_3+r^2_{2,2})&=\C(e_1+e_3+r^2_{2,2}) \oplus \C(e_2+e_4+r^2_{1,1})\\
  &\oplus \C(q(-\pi/4, 0,1)+q(-\pi/4, \pi,3))\oplus \C(q(\pi/4, \pi,1)+q(\pi/4, 0,3))\,.&
\end{align*}
The principal graph of $N_1 \subset N_1\rtimes\delta(J_i)$, $i=1,3$ is $D_6^{(1)}$.

\subsubsection{The \sacigs of $K_2$}\label{KD4K2}
From~\ref{e1e2}, the commutative Kac subalgebra $K_2$ is the algebra $L^\infty(D_4)$
of  functions on the group $D_4$, its basis is given by:
$$K_2=I(e_1+e_2)=\C(e_1+e_2) \oplus \C(e_3+e_4) \oplus  \C r^1_{1,1} \oplus \C r^1_{2,2}\oplus \C e^2_{1,1} \oplus \C e^2_{2,2}\oplus \C r^3_{1,1} \oplus \C r^3_{2,2}\,$$
and its lattice of \sacigs is in correspondence with the dual of the lattice of the subgroups of $D_4$.
There are three \sacigs of dimension $2$ contained in $J_{20}$ and four \sacigs of dimension $4$:
\begin{align*}%
  J_{21}&=\C(e_1+e_2+r^1_{1,1}) \oplus \C(e_3+e_4+r^3_{2,2})\oplus \C(e^2_{2,2}+r^3_{1,1})\oplus \C(r^1_{2,2}+e^2_{1,1});\\%
  J_{22}&=\C(e_1+e_2+r^1_{2,2}) \oplus \C(e_3+e_4+r^3_{1,1})\oplus \C(e^2_{1,1}+r^3_{2,2})\oplus \C(r^1_{1,1}+e^2_{2,2});\\%
  J_{23}&=\C(e_1+e_2+r^3_{2,2}) \oplus \C(e_3+e_4+r^1_{1,1})\oplus \C(e^2_{1,1}+r^3_{1,1})\oplus \C(r^1_{1,1}+e^2_{2,2});\\%
  J_{24}&=\C(e_1+e_2+r^3_{1,1}) \oplus \C(e_3+e_4+r^1_{2,2})\oplus \C(e^2_{2,2}+r^3_{2,2})\oplus \C(r^1_{1,1}+e^2_{1,1}).%
\end{align*}
For $j=21,\dots,24$, the principal graph of $N_1 \subset N_1\rtimes\delta(J_{j})$ is $D_{10}^{(1)}$.

\subsection{The Kac algebra $\KD(6)$ of dimension $24$}\label{KD6}

The study of the lattice of the \sacigs of $\KD(6)$ illustrates how a
large part of it (but not all!) can be constructed inductively.
\begin{thm}
  The lattice of \sacigs of $\KD(6)$ is as given in Figure~\ref{figure.KD6}.
\end{thm}
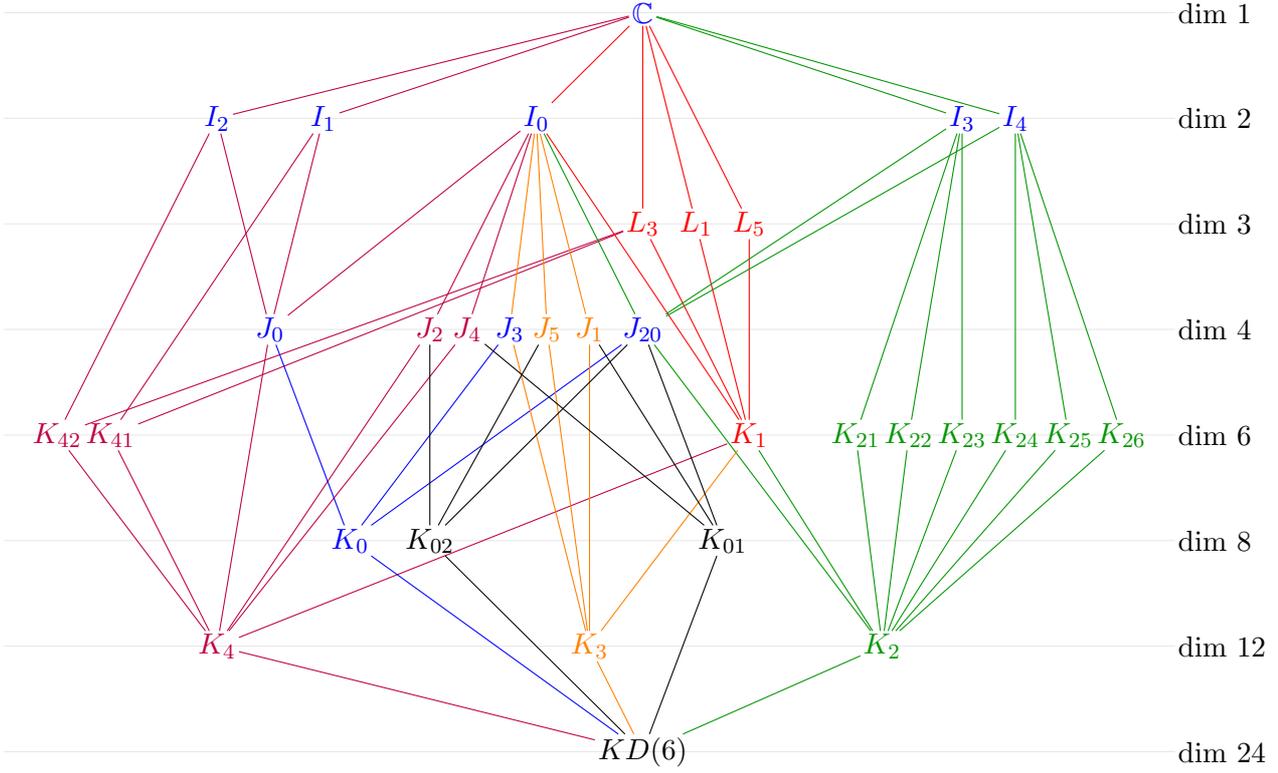
\begin{figure}
  \centering
  \begin{bigcenter}
    \pgfdeclarelayer{background}
\pgfdeclarelayer{nodes}
\pgfsetlayers{background,main,nodes}
\begin{tikzpicture}[xscale=.7,yscale=.7]
  \tikzstyle{every node}=[inner sep=1pt]
  \begin{pgfonlayer}{nodes}
    \node(C)	at (6,8)[text=blue] {$\C$};
    \node(I2)	at (-2,6) [text=blue] {$I_2$};
    \node(I1)	at (0,6) [text=blue] {$I_1$};
    \node(I0)	at (4,6) [text=blue] {$I_0$};
    \node(I3)	at (12,6) [text=blue] {$I_3$};
    \node(I4)	at (13,6) [text=blue] {$I_4$};
    \node(L3)	at (6,4) [text=red] {$L_3$};
    \node(L1)	at (7,4) [text=red] {$L_1$};
    \node(L5)	at (8,4) [text=red] {$L_5$};
    \node(J0)	at (-1,2) [text=blue] {$J_0$};
    \node(J5)	at (2,2) [text=purple] {$J_2$};
    \node(J6)	at (2.7,2) [text=purple] {$J_4$};
    \node(J3)	at (4.2,2) [text=orange] {$J_5$};
    \node(J2)	at (3.5,2) [text=blue] {$J_3$};
    \node(J4)	at (5,2) [text=orange] {$J_1$};
    \node(J1)	at (6,2) [text=blue] {$J_{20}$};
    \node(Kp3)	at (-5,0) [text=purple] {$K_{42}$};
    \node(Kp4)	at (-4,0) [text=purple] {$K_{41}$};
    \node(K1)	at (8,0) [text=red] {$K_1$};
    \node(Kp21)	at (10,0) [text=green] {$K_{21}$};
    \node(Kp22)	at (11,0) [text=green] {$K_{22}$};
    \node(Kp23)	at (12,0) [text=green] {$K_{23}$};
    \node(Kp24)	at (13,0) [text=green] {$K_{24}$};
    \node(Kp25)	at (14,0) [text=green] {$K_{25}$};
    \node(Kp26)	at (15,0) [text=green] {$K_{26}$};
    \node(K0)	at (0.5,-2) [text=blue] {$K_0$};
    \node(Kp135)at (2,-2) {$K_{02}$};
    \node(Kp146)at (7.5,-2) {$K_{01}$};
    \node(K4)	at (-2,-4) [text=purple] {$K_4$};
    \node(K3)	at (5,-4) [text=orange] {$K_3$};
    \node(K2)	at (10.5,-4) [text=green] {$K_2$};
    \node(KD6)	at (6,-6) {$KD(6)$};
  \end{pgfonlayer}
  \tikzstyle{every path}=[purple]
  \draw (C) -- (I2);
  \draw (C) -- (I1);
  \draw (I2) -- (J0);
  \draw (I1) -- (J0);
  \draw (I0) -- (J0);
   \draw (I0) -- (J6);
  \draw (I0) -- (J5);
  \draw (L3) -- (Kp3);
  \draw (I2) -- (Kp3);
  \draw (I1) -- (Kp4);
  \draw (K4) -- (Kp3);
  \draw (K4) -- (Kp4);
  \draw (L3) -- (Kp4);
  \draw (J0) -- (K4);
  \draw (J6) -- (K4);
  \draw (J5) -- (K4);
  \draw (K1) -- (K4);
  \draw (K4) -- (KD6);
  \tikzstyle{every path}=[blue]
  \draw (J0) -- (K0);
  \draw (J2) -- (K0);
  \draw (J1) -- (K0);
  \draw (K0) -- (KD6);
  \tikzstyle{every path}=[red]
   \draw (C) -- (I0);
  \draw (C) -- (L3);
  \draw (C) -- (L1);
  \draw (C) -- (L5);
  \d2aw (C) -- (I0);
  \draw (I0) -- (K1);
  \draw (L3) -- (K1);
  \draw (L1) -- (K1);
  \draw (L5) -- (K1);
    \tikzstyle{every path}=[green]
   \draw (C) -- (I3);
  \draw (C) -- (I4);
   \draw (I0) -- (J1);
   \draw (I3) -- (J1);
  \draw (I4) -- (J1);
  \draw (I3) -- (Kp21);
  \draw (I3) -- (Kp22);
  \draw (I3) -- (Kp23);
  \draw (I4) -- (Kp24);
  \draw (I4) -- (Kp25);
  \draw (I4) -- (Kp26);
  \draw (K2) -- (J1);
  \draw (K2) -- (K1);
  \draw (K2) -- (Kp21);
  \draw (K2) -- (Kp22);
  \draw (K2) -- (Kp23);
  \draw (K2) -- (Kp24);
  \draw (K2) -- (Kp25);
  \draw (K2) -- (Kp26);
  \draw (K2) -- (KD6);
  \tikzstyle{every path}=[orange]
  \draw (I0) -- (J2);
  \draw (I0) -- (J3);
  \draw (I0) -- (J4);
  \draw (J4) -- (K3);
  \draw (J3) -- (K3);
  \draw (J2) -- (K3);
  \draw (K1) -- (K3);
  \draw (K3) -- (KD6);
  \tikzstyle{every path}=[black]
  \draw (J5) -- (Kp135);
  \draw (J3) -- (Kp135);
  \draw (J1) -- (Kp135);
  \draw (J6) -- (Kp146);
  \draw (J4) -- (Kp146);
  \draw (J1) -- (Kp146);
  \draw (KD6) -- (Kp135);
  \draw (KD6) -- (Kp146);
  \begin{pgfonlayer}{background}
    \newcommand{\dimline}[2]{\draw[color=black!10,very thin](-6,#1) -- (16,#1); \node[right] at (16,#1){dim $#2$};}
    \dimline{8}{1}
    \dimline{6}{2}
    \dimline{4}{3}
    \dimline{2}{4}
    \dimline{0}{6}
    \dimline{-2}{8}
    \dimline{-4}{12}
    \dimline{-6}{24}
  \end{pgfonlayer}
\end{tikzpicture}
  \end{bigcenter}
  \caption{The lattice of \sacigs of $\KD(6)$}
  \label{figure.KD6}
\end{figure}

\begin{proof}%
  From~\ref{sacign} and~\ref{isoproj}, any \sacigs not of dimension
  $8$ is contained in one of the three Kac subalgebras of dimension
  $12$: $K_2\equiv L^\infty(D_n)$, $K_3\equiv \KQ(3)$, and $K_4\equiv
  \KD(3)$ which we study in turn in the sequel of this section.

  In~\ref{KD6-8} we construct all \sacigs of dimension $8$.
\end{proof}

\subsubsection{The \sacigs of $K_0$}

From~\ref{section.K0}, $K_0= I(e_1+p^{4}_{1,1})$ is the algebra of the
intrinsic group $D_4$ and is of dimension $8$. It contains all the
\sacigs of dimension $2$ of $\KD(6)$. It also contains the algebra
$J_0$ of the subgroup $H$. Its \sacigs are:
\begin{itemize}
\item $I_0=\C(e_1 +e_2+e_3 +e_4+q_1+q_2)$%
\item $I_1=\C(e_1 +e_4+q_1+p_2)$%
\item $I_2=\C(e_1 +e_4+q_1+p_1)$%
\item $I_3=I(e_1+e_2+ e^4_{1,1}+e^4_{2,2}+r^1_{2,2}+r^3_{1,1}+r^5_{2,2})$
\item $I_4=I(e_1+e_2+ e^4_{1,1}+e^4_{2,2}+r^1_{1,1}+r^3_{2,2}+r^5_{1,1})$
\item $J_0=\C(e_1+e_4+q_1) $%
\item $J_{20}=I(e_1+e_2+ e^4_{1,1}+e^4_{2,2})$
\item $J_3=I(e_1+e_3+ p^4_{1,1}+p^2_{2,2})$
\end{itemize}

\subsubsection{The \sacigs of $K_2$}

From~\ref{section.K2}, the Kac subalgebra $K_2=I(e_1+e_2)$ is
$L^\infty(D_6)$. Its \sacigs of dimension $6$ correspond to the seven
subgroups of order $2$ of $D_6$:
\begin{itemize}
\item $K_1=I(e_1+e_2+e_3+e_4)$ which is isomorphic to $L^\infty(D_3)$;
\item $K_{21}=I(e_1+e_2+r^1_{1,1})$;
\item $K_{22}=I(e_1+e_2+r^3_{1,1})$;
\item $K_{23}=I(e_1+e_2+r^5_{1,1})$;
\item $K_{24}=I(e_1+e_2+r^1_{2,2})$;
\item $K_{25}=I(e_1+e_2+r^3_{2,2})$;
\item $K_{26}=I(e_1+e_2+r^5_{2,2})$.
\end{itemize}
The group $D_6=\langle \alpha,\beta\rangle$ has a single subgroup of
order $3$ which gives the \sacig $J_1$ of dimension $4$, and three
subgroups of order $4$ which give three \sacigs of dimension $3$:
\begin{itemize}
\item $L_1=I(e_1 +e_2+e_3 +e_4+r^1_{1,1}+r^5_{2,2})$;
\item $L_3=I(e_1 +e_2+e_3 +e_4+r^3_{1,1}+r^3_{2,2})$;
\item $L_5=I(e_1 +e_2+e_3 +e_4+r^5_{1,1}+r^1_{2,2})$.
\end{itemize}
Finally, the \sacigs of dimension $2$ of $K_2$ are $I_0$, $I_3$ et
$I_4$ (see~\ref{tour} (5)).

\subsubsection{The \sacigs of $K_3$}
\label{K3inKD6}

The Kac algebra $K_3=I(e_1+e_3)$ is isomorphic to $\KQ(3)$, and we can
derive its \sacigs from~\ref{KQ3section}. There is a single
\sacig of dimension $6$: $K_1$. Its \sacigs of dimension $2$ and $3$
are those of $K_1$ which we treated above. The three \sacigs of
dimension $4$ are:
\begin{itemize}
\item $J_3$, which is the algebra of its intrinsic group,
\item $J_5=I(e_1+e_3+q(0,\frac{5\pi}{3},2)+q(0,\frac{4\pi}{3},4))$,
\item $J_1=I(e_1+e_3 + q(0,\frac{\pi}{3},2)+q(0,\frac{2\pi}{3},4))$.
\end{itemize}
We computed those projections by using successively the embedding of
$\KQ(3)$ into $\KQ(6)$ and the explicit isomorphism from $\KQ(6)$ to
$\KD(6)$ (see~\ref{mupad.plonge}).

\subsubsection{The \sacigs of $K_4$}
\label{K4inKD6}

The Kac subalgebra $K_4=I(e_1+e_4)$ is isomorphic to $\KD(3)$, and we
can derive its \sacigs from~\ref{KD3section}. Besides $K_1$, it
contains two other \sacigs of dimension $6$:
\begin{itemize}
\item $K_{41}=I(e_1+e_4+p(2,2,3))$ which contains $I_1$ and $L_3$,
\item $K_{42}=I(e_1+e_4+p(1,1,3))$ which contains $I_2$ and $L_3$.
\end{itemize}
The \sacigs of dimension $4$ are:
\begin{itemize}
\item $J_0$, which is the algebra of its intrinsic group,
\item $J_2=I(e_1+e_4+q(0,\frac{2\pi}{3},2)+q(0,\frac{4\pi}{3},4))$,
\item $J_4=I(e_1+e_4+q(0,\frac{4\pi}{3},2)+q(0,\frac{2\pi}{3},4))$.
\end{itemize}
We computed those projections by using the embedding of $\KD(3)$ in
$\KD(6)$ (see~\ref{mupad.plonge}).

\subsubsection{The \sacigs of dimension $4$ of $\KD(6)$ }

As predicted by Proposition~\ref{sacig4pair}, there are seven \sacigs of
dimension $4$ in $\KD(6)$: $J_{20}$ and $J_k$, with $k= 0 \dots 5$.

\subsubsection{The \sacigs of dimension $8$ of $\KD(6)$ }
\label{KD6-8}
\begin{prop}
   The \sacigs of dimension $8$ of $\KD(6)$  are the following :
  \begin{itemize}
  \item $K_0= I(e_1+p^{4}_{1,1})$, the algebra of the intrinsic group;
  \item $K_{01}=I(e_1 + q(0,\frac{2\pi}{3},4))=I(p_{\langle a^3, ba\rangle})$, which contains $J_{20}$,
    $J_1$, and $J_4$;
  \item $K_{02}=I(e_1 + q(0,\frac{4\pi}{3},4))=I(p_{\langle a^3, ab\rangle})$, which contains $J_{20}$,
    $J_2$, and $J_5$.
  \end{itemize}
  The last two are isomorphic by $\Theta$; they share the same block matrix
  structure as $\KP$, and their lattice of \sacigs are isomorphic to
  that of $\KP$. However they are not Kac subalgebras.
\end{prop}
\begin{proof}
  We first prove that any \sacig of dimension $8$ in $\KD(6)$ contains
  $J_{20}$ and that its Jones projection is of the form $e_1+s$, with
  $s$ a minimal projection of the fourth factor $M_2(\C)$ of $\KD(6)$:

  By a usual trace argument, the Jones projection of a \sacig $K$ of
  dimension $8$ is of the form $e_1+s$, with $s$ a minimal projection
  of some $j$-th factor $M_2(\C)$ (in principle, it could also be of
  the form $e_1+e_i+e_k$, but we ruled out those later projections by
  a systematic check on computer). The index of $N_1 \subset
  N_1\rtimes \delta(K)$ is $3$. Given the shape of $p_K$, its
  principal graph is $A_5$. Therefore, the center of $K$ admits a
  projection of the form $e_i+(e^j_{1,1}+e^j_{2,2}-s)$ (with $i=2,3$,
  or $4$). Then, the \sacig generated by $e_1+e_i+
  e^j_{1,1}+e^j_{2,2}$ is of dimension at most $8$ since it is
  contained in $K$. A computer check on the $15$ possible values shows
  that only $e_1+e_2+e^4_{1,1}+e^4_{2,2}$ satisfies this
  condition. Therefore, every \sacig of dimension $8$ contains
  $J_{20}=I(e_1+e_2+e^4_{1,1}+e^4_{2,2})$ and its Jones projection is
  of the form $e_1+s$, with $s$ a minimal projection of the fourth
  factor $M_2(\C)$ of $\KD(6)$.

  Using Proposition~\ref{prop.resumep} (4), we checked on computer
  that the only such Jones projections are $e_1+p^{4}_{1,1}$, $e_1 +
  q(0,\frac{2\pi}{3},4)$ and $e_1 + q(0,\frac{4\pi}{3},4)$; this
  reduced to solving a system of quadratic algebraic equations. We
  also checked that both new \sacigs have the same block matrix
  structure as $\KP$ but are not Kac subalgebras.

  The mentioned inclusions are straightforward to check using
  Proposition~\ref{tour} (5).
\end{proof}

\newpage
\section{The Kac algebras $\KQ(n)$}
\label{section.KQ}

The structure of $\KQ(n)$ follows closely that of $\KD(n)$, and we
sketch some general results. We first prove that $\KQ(2m)$ is
isomorphic to $\KD(2m)$. Then, when $n$ is odd, we establish the
self-duality of $\KQ(n)$, describe its automorphism group, list its
\sacigs of dimension $4$, $n$ and $2n$, and get a partial description
of $\ll(\KQ(n))$ (Theorem~\ref{theorem.lattice.nprime}); this
description is complete for $n$ odd prime. Finally, we complete the
systematic study of all Kac algebras of dimension at most $15$ by the
lattice $\ll(\KQ(3))$ which serves as intermediate step toward the
lattice of $\KD(6)$ in~\ref{KD6}.

\subsection{Definition and general properties}

\subsubsection{Notations}
\label{lambdaKQ}

For $n\geq 1$, we denote by $\KQ(n)$ the Kac algebra of dimension $4n$
constructed in~\cite[6.6]{Vainerman.1998} from the quaternion group
$Q_{n}=\langle a, b | a^{2n}=1, b^2=a^n, ba = a^{-1}b \rangle$. As for
$\KD(n)$, its block matrix structure is:
$$\KQ(n)=\C e'_1 \oplus \C e'_2 \oplus \C e'_3\oplus \C e'_4\oplus   \bigoplus_{k=1}^{n-1}\KQ(n)^k\,,$$
where $\KQ(n)^k$ is a factor $M_2(\C)$ whose matrix units we denote by
$e'^{k}_{i,j}$, $i=1,2$, $j=1,2$. In general, we use the same
notations for matrix units as in~\ref{q}, and put a $'$ when needed to
distinguish between the matrix units of $\KD(n)$ and $\KQ(n)$.

From~\ref{KacHopf}, the trace on $\KQ(n)$ is given by: $tr(e'_i)=1/4n
\quad tr(e_{i,j}'^{k})=1/2n$.

Setting $\epsilon_n=\mathrm{e}^{\mathrm{i}\pi/n}$, the left regular
representation of $Q_n$ is given by:
\begin{align*}
  \lambda'(a^k)&=e'_1+e'_2+(-1)^k(e'_3+e'_4)+\sum_{j=1}^{n-1}(\epsilon_n^{jk}\ e'^j_{1,1}+\epsilon_n^{-jk}\  e'^j_{2,2})\,;\\
  \lambda'(a^k b)&=e'_1-e'_2+(-1)^k(e'_3-e'_4)+\sum_{j=1}^{n-1}(\epsilon_n^{j(k-n)}\  e'^j_{1,2}+\epsilon_n^{-jk}\  e'^j_{2,1}) \quad \text{for $n$ even}\,;\\
  \lambda'(a^k b)&=e'_1-e'_2+(-1)^k \mathrm{i} (e'_3-e'_4)+\sum_{j=1}^{n-1}(\epsilon_n^{j(k-n)}\  e'^j_{1,2}+\epsilon_n^{-jk}\  e'^j_{2,1}) \quad \text{for $n$ odd}\,.
\end{align*}
The coproduct of $\C[Q_n]$ is twisted by a 2-pseudo-cocycle $\Omega'$,
given in~\ref{sacig22nKQimpair}. We refer
to~\cite[6.6]{Vainerman.1998} for details.
As in~\ref{KD2} one can show that $\KQ(1)$ and $\KQ(2)$ are respectively
the algebras of the groups $\mathbb{Z}_4$ and $D_4$.

\subsubsection{Intrinsic group}
\label{Qintrinseque}

In~\cite[p.718]{Vainerman.1998}, L. Vainerman mentions that the
intrinsic group of the dual of $\KQ(n)$ is $\mathbb{Z}_4$ if $n$ is
odd, and $\mathbb{Z}_2 \times \mathbb{Z}_2$ otherwise.  As in
\ref{intrinseque}, one can show that the intrinsic group of $\KQ(n)$ is
$\mathbb{Z}_4$ if $n$ is odd and $D_4$ otherwise.

\subsubsection{Automorphism group}
\label{section.KQ.automorphismGroup}

\begin{theorem}
  \label{theorem.automorphisms.KQ}
  For $n\geq 3$, the automorphism group $\aut(\KQ(n))$ of $\KQ(n)$ is
  given by:
  \begin{displaymath}
    A_{2n} = \{\  \Theta_k, \ \Theta_k\Theta'\  \suchthat k\wedge 2n = 1 \}\,,
  \end{displaymath}
  where $\Theta_k$ and $\Theta'$ are defined as in
  Theorem~\ref{theorem.automorphisms.KD}.  In particular, it is of
  order $2\varphi(2n)$ (where $\varphi$ is the Euler function), and
  isomorphic to $\Z_{2n}^* \rtimes \Z_2$, where $\Z_{2n}^*$ is the
  multiplicative group of $\Z_{2n}$.
\end{theorem}
\begin{proof}
  The proof follows that of Theorem~\ref{theorem.automorphisms.KD}.
  The analogue of Lemma~\ref{lemma.automorphism.KD.ThetaPk} holds
  because $\Theta'(a)$ lies in the subalgebra generated by $a$ in
  $\KQ(n)$ which is isomorphic to that in $\KD(n)$. Furthermore, the
  group automorphisms which fix $a^n$ and $b$ are the same in $D_{2n}$
  and $Q_{2n}$. Everything else is checked on computer.
\end{proof}

\subsection{The isomorphism $\phi$ from $\KD(2m)$ to $\KQ(2m)$}
\label{iso.KDKQ}
\label{isotheorem}

\begin{theorem}
  \label{theorem.isomorphism.KD.KQ}
  The Kac algebras $\KD(n)$ and $\KQ(n)$ are isomorphic if and only if
  $n$ is even. More specifically, when $n=2m$, $\phi$ defined by:
  \begin{align*}
    \phi(\lambda(a)) &= a+1/4 (a-a^{-1}) (a^{n}-1) (1 - ba^m)\,,\\
    \phi(\lambda(b)) &=1/2[b(a^n+1)+\mathrm{i}(a^n-1)]a^m\,,
  \end{align*}
  where for the sake of readability the $\lambda'$ have been omitted
  on the right hand sides, extends into a Kac algebra isomorphism from
  $\KD(n)$ to $\KQ(n)$.
\end{theorem}

Note first that $\KD(n)$ and $\KQ(n)$ are not isomorphic for $n$ odd, as
their intrinsic groups are $\Z_2^2$ and $\Z_4$, respectively
(see~\ref{intrinseque} and~\ref{Qintrinseque}).  The proof of the
isomorphism for $n$ even follows the same line as that of
Proposition~\ref{proposition.automorphism.KD.ThetaPrime}. We start
with some little remarks which simplify the calculation, give an
explicit formula for $\phi(\lambda(a))^k$, and conclude the proof with
computer checks.

\subsubsection{Technical lemmas}

\begin{remarks}
  Recall that, in $Q_n$, $ba^kb^{-1}=a^{-k}$, $b^4=1$, and $b^2=a^n$.
  Therefore:
  \begin{enumerate}[(i)]
  \item $a^n$ is in the center of $\KQ(n)$;
  \item $(a^n-1)(a^n+1)=0$;
  \item $(1+a^n)^2= 2(1+a^n)$ et $(1-a^n)^2= 2(1-a^n)$;
  \item $(a^k-a^{-k})ba^m=-ba^m(a^k-a^{-k})$;
  \item $(1+ba^m)(1-ba^m)=1-a^n$.
  \end{enumerate}
\end{remarks}

\begin{lemma}
  \label{lemma.isomorphism.KD.KQ}
  For all $k\in \Z$,
  \begin{displaymath}
    \phi(\lambda(a))^k
    =\lambda'(a^k)+\frac{1}{4}(\lambda'(a^k)-\lambda'(a^{-k}))(\lambda'(a^n)-1)(1-\lambda'(ba^m))\,.
  \end{displaymath}
  Furthermore $\phi(\lambda(a)^*)=\phi(\lambda(a))^*$.
\end{lemma}
\begin{proof}
  Write
  $f_k=\lambda'(a^k)+\frac{1}{4}(\lambda'(a^k)-\lambda'(a^{-k}))(\lambda'(a^n)-1)(1-\lambda'(ba^m))$.
  First note that $f_0=1$. One checks further that $f_k f_1 =
  f_{k+1}$ (omitting $\lambda'$):
  \begin{align*}
    f_kf_1-a^{k+1}& =\frac{a^{n}-1}{4}[(1+ba^m)(a^k- a^{-k})a+a^k(1+b a^{m})(a-a^{-1})]\\&+\frac{(a^{n}-1)^2}{16} (1+ b a^m)(a^k-a^{-k})(a-a^{-1}) (1-ba^m)\\
    &=\frac{a^{n}-1}{4}[(2a^{k+1}-a^{-k+1}-a^{k-1})+ba^m(a^{k+1}-a^{-k-1})]
    \\&-\frac{(a^{n}-1)}{8}(1+ b a^m)(1-ba^m)(a^{k+1}+a^{-k-1}-a^{k-1}-a^{-k+1})]\\
    & =  \frac{a^{n}-1}{4}(1+ba^m)(a^{k+1}-a^{-k-1})\,.
  \end{align*}
  In particular, we have $f_{-1}=f_1^*=f_1^{-1}$, that is
  $\phi(\lambda(a)^*)=\phi(\lambda(a))^*$.  The lemma follows by
  induction.
\end{proof}

\subsubsection{Proof of Theorem~\ref{theorem.isomorphism.KD.KQ}}

\begin{proof}[Proof of Theorem~\ref{theorem.isomorphism.KD.KQ}]
  We check that $\phi$ satisfies the properties listed in
  Proposition~\ref{proposition.KacAlgebraIsomorphism}.
  Most expressions below involve the element $a^m$. Therefore the
  computer checks use a close analogue of
  Proposition~\ref{proposition.equationAtInfinity} for the group algebra
  $\C[Q_{2\infty}]$ of
  \begin{displaymath}
    Q_{2\infty} =
    \langle a, b, a_\infty\,|\, b^2=a_\infty,\, ab=ba^{-1},\,a_\infty b =
    ba_\infty^{-1} \rangle\,,
  \end{displaymath}
  where $a_\infty$ plays the role of $a^m$. Then, an algebraic
  expression of degree $d$ in $a$ vanishes in $\KQ(2m)$ for all $m$
  whenever it vanishes for some $M>2d$.
  \begin{enumerate}[(i)]
  \item By Lemma~\ref{lemma.isomorphism.KD.KQ}, $\phi(\lambda(a))^{2n}=1$,
    and we check on computer that $\phi(\lambda(b))^2=1$ and
    $\phi(\lambda(b))\phi(\lambda(a))=\phi(\lambda(a))^{-1}\phi(\lambda(b))$.
  \item By Lemma~\ref{lemma.isomorphism.KD.KQ},
    $\phi(\lambda(a)^*)=\phi(\lambda(a))^*$, and check on computer
    that $\phi(\lambda(b))^*\phi(\lambda(b))=1$.
  \item Using a close analogue of
    Proposition~\ref{proposition.equationAtInfinity}, we check on
    computer that $(\phi\otimes
    \phi)(\Delta(\lambda(x)))=\Delta(\phi(x))$ for $x\in\{a,b\}$.
  \item Similarly, we check on computer the following equations:
    \begin{align*}
      \lambda'(a) &= \phi\left(\frac14 (\lambda(a)-\lambda(a)^{-1})
        (\lambda(a)^{m}-\lambda(a)^{-m}) ( \lambda(a)^{m} - \mathrm{i}
        \lambda(b) ) + \lambda(a)\right)\,,\\
      \lambda'(b) &=
      \phi\left(\frac12 \lambda(a)^m\left((\lambda(a)^n+1)\lambda(b)-(\lambda(a)^n-1)\right)\right)\,.
      \qedhere
    \end{align*}
  \end{enumerate}
\end{proof}

\subsubsection{The isomorphism $\phi$ on the central projections}

The following proposition gives the values of the isomorphism $\phi$
on the central projections of $\KD(2m)$.
\begin{prop}
  \label{prop.KDKQ.central}
  Set $n=2m$.  Then,\\
  for $m$ even:
  $$\phi(e_1)=e'_1, \qquad \phi(e_2)=e'_2, \qquad \phi(e_3)=e'_4, \qquad \phi(e_4)=e'_3\,,$$
  and for $m$ odd:
  $$\phi(e_1)=e'_1, \qquad \phi(e_2)=e'_2, \qquad \phi(e_3)=e'_3, \qquad \phi(e_4)=e'_4\,.$$
\end{prop}

\begin{proof}
  We recall the following formulas (see~\ref{form2} and~\ref{lambdaKQ}), where $\lambda$
  (resp. $\lambda'$) is the left regular representation of the
  dihedral (resp. quaternion) group:
  \begin{align*}
    e_1+e_2 &= \frac{1}{2n}\sum_{k=0}^{2n-1}\lambda(a^k) &e'_1+e'_2 &= \frac{1}{2n}\sum_{k=0}^{2n-1}\lambda'(a^k)\\
    e_3+e_4 &= \frac{1}{2n}\sum_{k=0}^{2n-1}(-1)^k\lambda(a^k) &e'_3+e'_4 &= \frac{1}{2n}\sum_{k=0}^{2n-1}(-1)^k\lambda'(a^k)\\
    e_1-e_2&=\lambda(b)(e_1+e_2) &e'_1-e'_2&=(e'_1+e'_2)\lambda'(b)\\
    e_3-e_4 &=-\lambda(b)(e_3+e_4) &e'_3-e'_4 &=(e'_3+e'_4) \lambda'(b)
  \end{align*}
  Using further the identities
  \begin{displaymath}
    \sum_{k=0}^{2n-1}(\lambda'(a^k)-\lambda'(a^{-k})) = 0
    \qquad \text{ and } \qquad
    \sum_{k=0}^{2n-1}(-1)^k(\lambda'(a^k)-\lambda'(a^{-k})) = 0\,,
  \end{displaymath}
  it follows easily that $\phi(e_1+e_2)=e'_1+e'_2$ and
  $\phi(e_3+e_4)=e'_3+e'_4$. Then, from
  \begin{displaymath}
    \lambda'(a^n)=e'_1+e'_2+e'_3+e'_4+\sum_{j=1}^{n-1}(-1)^j(e^j_{1,1}+e^j_{2,2})\,,
  \end{displaymath}
  we get
  $$(e'_1+e'_2+e'_3+e'_4)\phi(\lambda(b))= (e'_1-e'_2+e'_3-e'_4) \lambda'(a^m)= (e'_1-e'_2)+(-1)^m(e'_3-e'_4)$$
  and it follows that
  $\phi(e_1-e_2)=e'_1-e'_2$ \quad and \quad $\phi(e_3-e_4)=(-1)^{m+1}(e'_3-e'_4)$.
\end{proof}

\subsection{Self-duality for  $n$ odd}
\label{self-dualKQ}

As for $\KD(n)$, $\KQ(n)$ is self-dual for $n$ odd, and the proof
follows the same lines (see~\ref{self-dualKD}).
\begin{thm}
  The Kac algebra $\KQ(n)$ is self-dual if and only if $n$ is odd. When this is the
  case, one can take as Kac algebra isomorphism $\psi$ defined by:
  \begin{displaymath}
    a \mapsto 2n(\widehat{ e^{n-1}_{1,1}} + \frac12 (i(\widehat {e^1_{2,2}} - \widehat {e^1_{1,1}}) -\widehat {e^{n-1}_{1,2}} - \widehat {e^{n-1}_{2,1})}, \qquad
    b \mapsto 4n\widehat {e_4}\,,
  \end{displaymath}
  where the notations are as in~\ref{self-dualKD}.
\end{thm}

The following proposition gives an alternative description of $\psi$
on the matrix units.
\begin{prop}
  Take $k$ odd in $\{1,\dots,n-1$, and set:
  \begin{align*}
    E^k_{1,1}&=\chi_{ba^{n+k}}+\chi_{ba^{n-k}}\\%
    E^k_{2,2}&=\chi_{ba^{2n-k}}+\chi_{ba^{k}}\\%
    E^k_{1,2}&=\frac{1}{2}[(-\chi_{a^k}+\chi_{a^{2n-k}}+\chi_{a^{n+k}}-\chi_{a^{n-k}})+(\chi_{ba^k}-\chi_{ba^{2n-k}}+\chi_{ba^{n+k}}-\chi_{ba^{n-k}})]\\%
    E^k_{2,1}&=\frac{1}{2}[-(-\chi_{a^k}+\chi_{a^{2n-k}}+\chi_{a^{n+k}}-\chi_{a^{n-k}})+(\chi_{ba^k}-\chi_{ba^{2n-k}}+\chi_{ba^{n+k}}-\chi_{ba^{n-k}})]\\
    E^{n-k}_{1,1}&=\chi_{a^{n-k}}+\chi_{a^{2n-k}}\\
    E^{n-k}_{2,2}&=\chi_{a^k}+\chi_{a^{n+k}}\\
    E^{n-k}_{1,2}&=\frac{1}{2}[(-\chi_{a^k}-\chi_{a^{2n-k}}+\chi_{a^{n+k}}+\chi_{a^{n-k}})+(\chi_{ba^k}-\chi_{ba^{2n-k}}-\chi_{ba^{n+k}}+\chi_{ba^{n-k}})]
    \\%
    E^{n-k}_{2,1}&=\frac{1}{2}[(-\chi_{a^k}-\chi_{a^{2n-k}}+\chi_{a^{n+k}}+\chi_{a^{n-k}})-(\chi_{ba^k}-\chi_{ba^{2n-k}}-\chi_{ba^{n+k}}+\chi_{ba^{n-k}})]
  \end{align*}
  Then,
  $$\psi(e_1)=\chi_1 ,\qquad \psi(e_2)=\chi_{a^n},\qquad \psi(e_3)=\chi_{ba^n},\qquad \psi(e_4)=\chi_b,$$
  $$\psi(r^k_{i,j})=E^k_{i,j}, \qquad \psi(e^{n-k}_{i,j})=E^{n-k}_{i,j}.$$
\end{prop}

\subsection{The lattice $\mathrm{l}(\KQ(n))$ for $n$ odd}
\label{section.KQ.odd.sacig}

In the sequel of this section, we prove the following theorem, and
illustrate it on $\KQ(3)$.
\begin{theorem}
  \label{theorem.lattice.nprime}
  When $n$ is odd, $\KQ(n)$ admits:
  \begin{itemize}
  \item A single \sacig of dimension $2n$: $K'_2=I(e_1+e_2)$,
    isomorphic to $L^\infty(D_n)$;
  \item $n$ \sacigs of dimension $n$ in $K'_2$;
  \item $n$ \sacigs of dimension $4$, $J_k$ for $k=0,\dots,n-1$, with Jones projections:
    \begin{displaymath}
      p_{J_k}=e_1+\sum_{j =1, \, j\text{ even}}^{n-1} q(0,\frac{jk\pi}{n},j)\,;
    \end{displaymath}
  \item A single \sacig of dimension $2$: $I(e_1 +e_2+q_1+q_2)$.
  \end{itemize}
  If $n$ is prime there are no other \sacigs and the graph of the
  lattice of \sacigs of $\KQ(n)$ is similar to Figure~\ref{figure.KQ3}.
\end{theorem}

\subsubsection{The \sacigs of dimension $2$ and $2n$}
\label{sacig22nKQimpair}

From~\ref{Qintrinseque} and~\ref{intrin}, $\widehat{\KQ(n)}$ has a
single \sacig of dimension $2$. Therefore $\KQ(n)$ has a single \sacig
of dimension $2n$.

By Remark~\ref{groupe} its standard coproduct is
\begin{displaymath}
  \Delta_s(e_1+e_2)=(e_1+e_2)\otimes(e_1+e_2)+(e_3+e_4)\otimes(e_3+e_4)+\sum_{j=1}^{n-1} e^j_{1,1}\otimes e^{j}_{2,2}+e^{j}_{2,2}\otimes e^j_{1,1}\,.
\end{displaymath}
Beware that, for an improved consistency with $\KD(n)$, our notations
differ slightly from those of~\cite{Vainerman.1998}:
$$r_1=\frac{1}{2}\sum_{j=1,\,j \text{ odd}}^{n-1}r^j_{1,1}, \qquad r_2=\frac{1}{2}\sum_{j=1,\,j \text{ odd}}^{n-1}r^j_{2,2}\,,$$
the unitary $\Omega'$ used to twist the coproduct of $\KQ(n)$ can be
written as:
\begin{align*}
\Omega'&=(e_1+e_4+r_1+q_1)\otimes (e_1+e_4+r_1+q_1)\\
&+(e_1+q_1)\otimes(e_2+e_3+r_2+q_2)+(e_2+e_3+r_2+q_2)\otimes(e_1+q_1)\\
&+\rm{i}(e_4+r_1)\otimes(e_2-e_3+q_2-r_2)-\rm{i}(e_2-e_3+q_2-r_2)\otimes(e_4+r_1)\\
&+(e_2-\rm{i}e_3+q_2-\rm{i}r_2)\otimes(e_2+\rm{i}e_3+q_2+\rm{i}r_2)\,.
\end{align*}
Therefore:
\begin{multline*}
  \Delta(e_1+e_2)=(e_1+e_2)\otimes(e_1+e_2)+(e_3+e_4)\otimes(e_3+e_4)\\+\sum_{j=1,\,j \text{ even}}^{n-1} e^j_{1,1}\otimes e^{j}_{2,2}+e^{j}_{2,2}\otimes
  e^j_{1,1}+\sum_{j=1,\, j \text{ odd}}^{n-1} p^j_{1,1}\otimes
  p^{j}_{1,1}+p^{j}_{2,2}\otimes p^j_{2,2}\,.
\end{multline*}
The unique \sacig of dimension $2n$ of $\KQ(n)$ is therefore
$K'_2=I(e_1+e_2)$. As in~\ref{e1e2} it can be shown to be isomorphic
to $L^\infty(D_n)$. %
In particular, it admits $n$ \sacigs of dimension $n$.

\subsubsection{The \sacigs of dimension $4$ and $n$}
\label{sacig4KQimpair}
\begin{prop}
  For $n$ odd, the \sacigs of dimension $n$ of $\KQ(n)$ are contained
  in $K'_2$. They correspond to the $n$ sub-groups of order $2$ of
  $D_n$. By self-duality, $\KQ(n)$ admits $n$ \sacigs of dimension $4$,
  with Jones projections:
  $$p_{J_k}=e_1+\sum_{j=1, \, j \text{ even}}^{n-1} q(0,\frac{jk\pi}{n},j) \quad  (k=0,\dots,n-1)\,.$$
  The principal graph of the inclusion $N_1 \subset
  N_1\rtimes\delta(J_k)$ is $D_n/\mathbb{Z}_2$
  (see~\ref{grapheDimpair}), that of $N_2 \subset N_2\rtimes J_k$ is
  $D_{2n+2}^{(1)}$ (see~\ref{grapheD1}).
\end{prop}
\begin{proof}
  If $J$ is a \sacig of dimension $n$, its Jones projection is of the
  same form as described in~\ref{sacign} for $\KD(n)$.  From the
  classification of subfactors of index $4$ in~\cite{Popa.1990}, the
  principal graph of $R \subset R\rtimes \delta^{-1}(J)$ is either
  $D_4^{(1)}$ or $D_{n'}^{(1)}$. Therefore, $R \subset R\rtimes
  \delta^{-1}(J)$ admits an intermediate factor of index $2$, which
  implies that $J$ is contained in the unique \sacig $K'_2$ of
  dimension $2n$.  Using~\ref{sacig22nKQimpair}, there are exactly $n$
  \sacigs of dimension $n$.

  Since $\KQ(n)$ is self-dual, it admits exactly $n$ \sacigs of
  dimension $4$. As in~\ref{sacig4impair}, one shows that their Jones
  projections are, for $k=0,\dots,n-1$:
  \begin{displaymath}
    p_{J_k}=e_1+\sum_{j=1, \, j \text{ even}}^{n-1} q(0,\frac{jk\pi}{n},j)\,.\qedhere
\end{displaymath}
\end{proof}

\subsubsection{The Kac algebra $\KQ(3)$ of dimension $12$}
\label{KQ3section}

We illustrate Theorem~\ref{theorem.lattice.nprime} on $\KQ(3)$, making
explicit the block matrix decompositions of the \sacigs (obtained by
computer) and the principal graphs of the corresponding inclusions.

\begin{figure}[h]
  \centering
  \pgfdeclarelayer{background}
\pgfdeclarelayer{nodes}
\pgfsetlayers{background,main,nodes}
\begin{tikzpicture}[baseline=(current bounding box.east),yscale=-1]
  \begin{pgfonlayer}{nodes}
    \node(C)	at (3,1) {$\C$};
    \node(I0)	at (1,2) {$I_0$};
    \node(K1)	at (3,3) {$K_1$};
    \node(L1)	at (4,3) {$K_{21}$};
    \node(L2)	at (5,3) {$K_{22}$};
    \node(Jp2)	at (1,4) {$J_2$};
    \node(Jp1)	at (2,4) {$J_1$};
    \node(J0)	at (3,4) {$J_0$};
    \node(K2)	at (5,5) {$K_2$};
    \node(KQ3)	at (3,6) {$KQ(3)$};
  \end{pgfonlayer}
\begin{pgfscope}
\tikzstyle{every path}=[blue]
  \draw (C) -- (I0);
  \draw (I0) -- (J0);
  \draw (I0) -- (Jp2);
  \draw (I0) -- (Jp1);
  \draw (Jp2) -- (KQ3);
  \draw (Jp1) -- (KQ3);
  \draw (J0) -- (KQ3);
\tikzstyle{every path}=[green]
  \draw (C) -- (K1);
  \draw (C) -- (L1);
  \draw (C) -- (L2);
  \draw (K1) -- (K2);
  \draw (L1) -- (K2);
  \draw (L2) -- (K2);
  \draw (K2) -- (KQ3);
\tikzstyle{every path}=[red]
   \draw (I0) -- (K2);
  \end{pgfscope}
  \begin{pgfonlayer}{background}
    \newcommand{\dimline}[2]{\draw[color=black!10,very thin](0,#1) -- (7,#1); \node[right] at (7,#1) {dim $#2$};}
    \dimline{1}{1}
    \dimline{2}{2}
    \dimline{3}{3}
    \dimline{4}{4}
    \dimline{5}{6}
    \dimline{6}{12}
  \end{pgfonlayer}
\end{tikzpicture}
  \caption{The lattice of \sacigs of $\KQ(3)$}
  \label{figure.KQ3}
\end{figure}
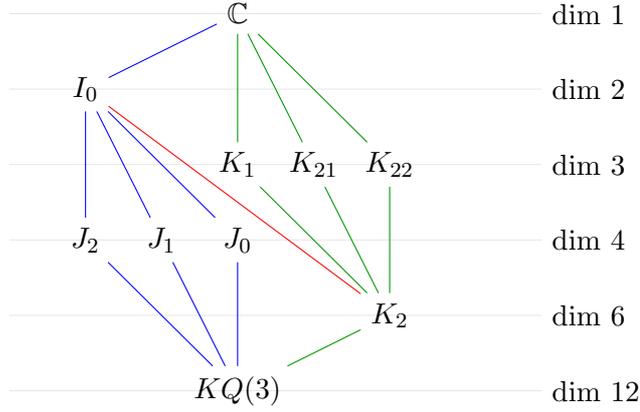
\begin{prop}
  The lattice of \sacigs de $\KQ(3)$ is as given by
  Figure~\ref{figure.KQ3}.

  The block matrix decomposition of the \sacigs is given by:
  \begin{align*}
    I_0&=\C(e_1 +e_2+q_1+q_2)\oplus  \C(e_3 +e_4+p_1+p_2);\\%
    K_1&=\C(e_1+e_2+e_3+e_4)\oplus \C(p^1_{1,1}+e^2_{2,2}) \oplus \C(p^1_{2,2}+e^2_{1,1})      ;\\
    K_{21}&=\C(e_1+e_2+p^1_{1,1}) \oplus \C(e_3+e_4+e^2_{1,1})   \oplus \C(p^1_{2,2}+e^2_{2,2});\\
    K_{22}&=\C(e_1+e_2+p^1_{2,2}) \oplus \C(e_3+e_4+e^2_{2,2})   \oplus \C(p^1_{1,1}+e^2_{1,1});\\
    J_0&=\C(e_1+q_1) \oplus \C(e_2+q_2) \oplus  \C(e_3+r^{1}_{2,2}) \oplus \C(e_4+r^{1}_{1,1});\\ %
   J_1&=\C(e_1+q(0,\frac{2\pi}{3},2))\oplus \C(e_2+q(0,\frac{5\pi}{3},2))\oplus \C(e_3+q(\frac{\pi}{3},\frac{\pi}{2},1))\oplus \C(e_4+q(-\frac{\pi}{3},\frac{3\pi}{2},1));\\
   J_2&=\C(e_1+q(0,\frac{4\pi}{3},2))\oplus \C(e_2+q(0,\frac{\pi}{3},2))\oplus \C(e_3+q(-\frac{\pi}{3},\frac{\pi}{2},1))\oplus \C(e_4+q(\frac{\pi}{3},\frac{3\pi}{2},1));\\
    K_2&=I(e_1+e_2)=\C(e_1 +e_2)\oplus  \C(e_3 +e_4)\oplus \C p^1_{1,1}\oplus \C p^1_{2,2}\oplus \C e^2_{1,1}\oplus \C e^2_{2,2}.
  \end{align*}
With the notations of~\ref{graphe}, the inclusion $N_2\subset N_2\rtimes J$ has for principal graph
\begin{itemize}
\item  $D^{(1)}_8$ for $J= J_1$, or $J_2$;
\item  $A_5$ for $J=K_1$, $K_{21}$, or $K_{22}$.
\end{itemize}
For $J= I_0$, $J_0$, or $K_2$, it is of depth $2$.
\end{prop}

\newpage
\section{The Kac algebras $\KB(n)$ and $K_3$ in $\KD(2m)$}
\label{section.KB}

When $n=2m$ is even, the Kac subalgebras $K_2$ and $K_4$ of $\KD(n)$
are respectively isomorphic to $L^{\infty}(D_m)$ and $\KD(m)$; their
structure is therefore known, or can at least be studied
recursively. There remains to describe the family of Kac subalgebras
$K_3$, as $m$ varies. This question was actually at the origin of our
interest in the family of Kac algebras $\KQ(n)$, and later in the
family $B_{4m}$ defined by A. Masuoka in~\cite[def 3.3]{Masuoka.2000},
and which we denote $\KB(m)$ for notational consistency.

In this section, we first prove that, for all $m$, $K_3$ in $\KD(2m)$
is isomorphic to $\KB(m)$, and therefore self-dual, and discuss
briefly a construction of $K_3$ by composition of
factors. %
Then, we compare the three families: $K_3$ in $\KD(n)$ is also
isomorphic to $\KQ(m)$ when $m$ is odd (this gives an alternative
proof of the self-duality of $\KQ(m)$ when $m$ is odd) and the family
$(\KD(n))_n$ contains the others, $(\KQ(n))_n$ and $(\KB(n))_n$, and
$\KP$ as subalgebras.  Using the self-duality of $\KB(m)$, we list the
\sacigs of dimension $2$, $4$, $m$, and $2m$ of $K_3$ in $\KD(2m)$ and
describe the complete lattice of \sacigs for $\KD(8)$.

\subsection{The subalgebra $K_3$ of $\KD(2m)$}
\label{section.K3.KD2m}

In $\KD(4)$, $K_3$ is isomorphic to $\KP$ (see~\ref{KD4}); in
particular, it is self-dual and can be obtained by twisting a group
algebra. In $\KD(8)$, $K_3$ is the self-dual Kac algebra described
in~\cite[Lemma XI.10]{Izumi_Kosaki.2002} (the intrinsic group of its
dual is $\mathbb{Z}_2 \times \mathbb{Z}_2$). We however found on
computer that it cannot be obtained by twisting a group algebra: we
took a generic cocycle $\omega$ on $H$ with six unknowns such that
$\omega(h,h)=1$, and $\omega(h,h')= \alpha_{h,h'}$ if $h\ne h'$, and
asked whether there existed a choice for the $\alpha_{h,h'}$ making
the untwisted coproduct $\Delta_u = \Omega\Delta\Omega^*$ is
cocommutative; this gave an algebraic system of equations of degree
$2$ which had no solution (note that we relaxed the condition
$\omega(h,h')=\overline{\omega(h',h)}$ to keep the system algebraic).

\subsubsection{A presentation of $K_3$ and isomorphism with $\KB(m)$}
\label{B4m}

We now construct explicit generators for $K_3$ satisfying the
presentation of $\KB(m)$ given in~\cite[def 3.3]{Masuoka.2000}.
\begin{lemma}
  With $n=2m$ and  $\epsilon=\rm{e}^{\rm{i}\pi/(2m)}$ set:
  \begin{align*}
    v&=(e_1+e_3)-(e_2+e_4)
    +\sum_{j=1}^{m-1} \epsilon^{-2j}e_{1,2}^{2j} +\epsilon^{2j}e_{2,1}^{2j}
    +\sum_{j=0}^{m-1}\epsilon^{-(2j+1)}r_{1,2}^{2j+1}+\epsilon^{2j+1}r_{2,1}^{2j+1}
    \\
    w&=(e_1+e_3)-(e_2+e_4)
    +\sum_{j=1}^{m-1} \epsilon^{2j}e_{1,2}^{2j} +\epsilon^{-2j}e_{2,1}^{2j}
    +\sum_{j=0}^{m-1}\epsilon^{2j+1}r_{1,2}^{2j+1}+\epsilon^{-(2j+1)}r_{2,1}^{2j+1}\\
    B_0&=(1+\lambda(a^n))/2\\
    B_1&=(1-\lambda(a^n))/2)
  \end{align*}
  The unitary elements $v$ and $w$ are contained in $K_3$ and satisfy:
  \begin{align*}
    v^2&=w^2=1\\
    (vw)^m&=\lambda(a^n)\\
    \Delta(v)=v \otimes B_0v + w\otimes B_1 v \quad &\text{et}\quad
    \Delta(w)=w \otimes B_0w + v\otimes B_1 w\\
    \varepsilon(v)&=\varepsilon(w)=1\\
    S(v)=B_0v+B_1w \quad &\text{et}\quad S(w)=B_0w+B_1v
  \end{align*}
\end{lemma}
\begin{proof}
  The non trivial part is the calculation of the coproducts which is
  carried over in~\ref{subsection.copv}.
\end{proof}

\begin{theorem}
  The Kac subalgebra $K_3$ in $\KD(2m)$ is isomorphic to the Kac algebra
  $\KB(m)$ defined by A. Masuoka in~\cite[def 3.3]{Masuoka.2000}.
  In particular, it is self-dual.
\end{theorem}
\begin{proof}
  Using Lemma~\ref{B4m}, $K_3$ contains the subalgebra $I_0=\C\Z_2$
  generated by $\lambda(a^n)$, as well as the unitary elements $v$ and
  $w$ which satisfy the desired relations. By dimension count, $K_3$
  is therefore isomorphic to $\KB(m)$. The later is
  self-dual (\cite[p. 776]{Calinescu_al.2004}).
\end{proof}

\subsubsection{Realization of $K_3$ by composition  of subfactors}

As in~\ref{KPIK} and~\ref{KD3IK}, we describe the inclusion
$N_2\subset N_2\rtimes K_3$ using group actions.

Define the group $G=D_m\rtimes \Z_2$, where $\Z_2$ acts on $D_m=
\langle \alpha, \beta \suchthat \alpha^{m}=1, \beta^2=1, \beta\alpha =
\alpha^{-1}\beta \rangle$ by the automorphism $\nu$ of $D_m$ defined
by $\nu(\alpha)=\alpha^{-1}$ and $\nu(\beta)=\alpha\beta$. Set
$M=N_2\rtimes K_1$. Consider the dual action of $D_m$ on $M$ ($K_1$ is
isomorphic to $L^\infty(D_m)$) and the action of $z$ of $\Z_2$ on $M$
by $\Ad(v)$.  As in~\cite[def II.1]{Izumi_Kosaki.2002}, it can be
easily proved that $(d,z)$ is an outer action of the matched pair
$(D_m,\Z_2)$. From $(d,z)$ arises a couple of cocycles $(\eta, \zeta)$
whose class in a certain cohomology group $H^2((D_m,\Z_2),T)$
characterizes up to an isomorphism the depth 2 irreducible inclusion
of factors $M^{(D_m,d)} \subset M \rtimes_{z} \Z_2$ (\cite[Theorem
II.5 and Remark 2 p.12]{Izumi_Kosaki.2002}) and the Kac algebra which
is associated to this inclusion (\cite[Theorem
VI.1]{Izumi_Kosaki.2002}).  The group $H^2((D_m,\Z_2),T)$ is reduced
to $\Z_2$ (\cite[VII.§5 and Proposition VII.5]{Izumi_Kosaki.2002}) and
the pair $(d,z)$ is associated to the non-trivial cocycle. Indeed, if
$m$ is even, this follows from ~\cite[Lemma~VII.6]{Izumi_Kosaki.2002},
since the intrinsic group of $K_3$ and $\widehat{K_3}$ are of order
$4$ (that of $K_3$ is $J_{20}\equiv \Z_2\times\Z_2$ and $K_3$ has four
characters).  If $m$ is odd, $K_3$ is isomorphic to $\KQ(m)$, which is
self-dual, and its intrinsic group is $\Z_4$, so one can
use~\cite[Lemma~VII.2]{Izumi_Kosaki.2002}. For any value of $m$, the
inclusion $N_2 \subset N_2\rtimes K_3$ is therefore isomorphic to the
inclusion $M^{(D_m,d)} \subset M \rtimes_{z} \Z_2$ associated to the
non trivial cocycle.

\subsection{$\KD(n)$, $\KQ(n)$, $\KA(n)$, and $\KB(n)$: isomorphisms, \sacigs of dimension $2n$}
\label{isoproj}
We collect here all the results on the isomorphisms between the four
families and their Kac subalgebras of dimension $2n$.
\begin{theorem}
  \label{theorem.isomorphisms}
  Let $n \geq 2$.
  \begin{enumerate}
  \item The Kac algebra $\KA(n)$ is isomorphic to the dual of $\KD(n)$.
  \item Assume further $n=2m$ even. Then, in $\KD(n)$,
    \begin{itemize}
    \item $K_2=I(e_1+e_2)$ is isomorphic to $L^\infty(D_n)$,
    \item $K_3=I(e_1+e_3)$ is isomorphic to $\KB(m)$,
    \item $K_4=I(e_1+e_4)$ is isomorphic to $\KD(m)$.
    \end{itemize}
  \item The Kac algebra $\KQ(n)$ is isomorphic to $\KD(n)$ for $n$ even and
  to $\KB(n)$ for $n$ odd.
  \end{enumerate}
\end{theorem}
\begin{proof}
  (1) For $n=2$, see~\cite[Remark 3.4 (1)]{Masuoka.2000}. For $n \geq
  3$, $\KA(n)$ is the unique non trivial Kac algebra obtained by
  twisting the product of $L^{\infty}(D_{n})$ by a cocycle
  (see~\cite[Theorem 4.1 (2)]{Masuoka.2000}). Its dual is therefore
  the unique non trivial Kac algebra obtained by twisting the
  coproduct of $\C[D_{n}]$\footnote{It is a classical fact in the
    folklore that twisting products or coproducts of algebras by
    cocycles are two dual constructions}. Hence, $\KA(n)$ must be
  isomorphic to $KD(n)$.

  (2) See Proposition~\ref{e1e2}, Theorem~\ref{B4m},
  Theorem~\ref{plonge} for the structure of $K_2$, $K_3$, and $K_4$
  respectively.

  (3) For $n$ even, $\KQ(n)$ is isomorphic to $\KD(n)$
  (Theorem~\ref{theorem.isomorphism.KD.KQ}). If $n$ odd, $\KQ(n)$,
  which is embedded in $\KQ(2n)$ as $K'_3$, is isomorphic to $K_3$ of
  $\KD(2n)$ by~Proposition~\ref{prop.KDKQ.central}, and therefore to
  $\KB(n)$ by Theorem~\ref{B4m}.
\end{proof}
\subsection{The lattice $\mathrm{l}(\KB(m))$}\label{latticeKB}

\subsubsection{General case}

\label{lattice_K3_even}

In~\ref{section.KQ.odd.sacig}, we partially describe the lattice
$\mathrm{l}(\KQ(m))$ when $m$ is odd; those results therefore apply to
the isomorphic algebras $\KB(m)$ and $K_3$ in $\KD(2m)$.

Let us now explore $\mathrm{l}(\KB(m))$ for $m$ even: $m=2m'$. To this
end, we consider $\KB(m)$ as $K_3$ in $\KD(2m)$. To avoid handling
degeneracies, the special cases $KB(2)=KP$ and $KB(4)$ are treated
respectively in Subsections~\ref{KP} and~\ref{KD8}, and we now assume
$m\geq 6$.

\begin{proposition}
  When $m$ is even, the \sacigs of dimension $2$, $4$, $m$, $2m$ of
  the Kac algebra $KB(m)\approx K_3\subset KD(2m)$ are, keeping the
  notations of~\ref{sacigK0} and~\ref{sacig4pair}:
  \begin{itemize}
  \item Three \sacigs of dimension $2m$:
    \begin{itemize}
    \item $K_{32}=K_1 = I(e_1+e_2+e_3+e_4)$, isomorphic to
      $L^\infty(D_m)$;
    \item $K_{33}=I(e_1+e_3+r^m_{1,1})=I(p_{\langle a^4, ba\rangle})$;
    \item $K_{34}=I(e_1+e_3+r^m_{2,2})=I(p_{\langle a^4, ab\rangle})$;
    \end{itemize}
    The \sacigs $K_{33}$ and $K_{34}$ are isomorphic through the
    involutive automorphism $\Theta_{-1}$;
  \item The \sacigs of $K_{32}$, which correspond to subgroups of
    $D_m$; in dimension $2$, $4$, $m$, they are:
    \begin{itemize}
    \item Dimension $2$: $I_0$, $I_3$, and $I_4$ (subgroups of
      order $m$);
    \item Dimension $4$: the \sacig $J_{20}$ containing $I_0$, $I_3$,
      and $I_4$; furthermore, if $m'$ is even, there are four \sacigs
      (dihedral subgroups of order $m'$), which contain exactly one of
      $I_3$ or $I_4$;
    \item Dimension $m$:
      $K_{31}=I(e_1+e_2+e_3+e_4+e^m_{1,1}+e^m_{2,2})$ (subgroup
      generated by $\alpha^{m'}$), and
      $K_{1k}=I(e_1+e_2+e_3+e_4+r^{2k-1}_{1,1}+r^{n-2k+1}_{2,2})$ for
      $k=1,\dots, m$ (subgroups generated by one reflection); each
      $K_{1k}$ contains exactly one of $I_3$ or $I_4$;
    \end{itemize}
  \item The \sacigs of $K_{33}$ and $K_{34}$:
    \begin{itemize}
    \item Dimension $2$: $I_0$, and if $m'$ is even, $I_3$ and $I_4$ ;
    \item Dimension $4$: $J_{2k+1}$ for $k=0, \dots m-1$ and, if $m'$
      is even, $J_{20}$; for $2k+1 \equiv 1$ (resp. $2k+1 \equiv 3$)
      mod $4$, $J_{2k+1}$ is contained in $K_{33}$ (resp. $K_{34}$);

    \item Dimension $m$: $K_{31}$, and if $m'$ is even, two \sacigs
      contained in $K_{33}$ and two in $K_{34}$.
    \end{itemize}

  \end{itemize}  
  In particular, $K_3$ admits exactly $m$ \sacigs of dimension $4$ not
  contained in $K_1$. Also, $K_{31}$ is the intersection of the three
  \sacigs of dimension $2m$ of $K_3$.
\end{proposition}
\begin{proof}

  Since $K_{32}=K_1=L^{\infty}(D_m)$ the description of its lattice of
  \sacigs is straightforward.

  From~\ref{sacigK0} and~\ref{sacig4pair}, we know that $K_3$ admits
  three \sacigs of dimension $2$: $I_0$, $I_3$, and $I_4$. By
  self-duality of $K_3$, there are exactly three \sacigs of dimension
  $2m$, and by trace argument the given Jones projections for
  $K_{32}$, $K_{33}$, and $K_{34}$ are the only possible ones. Their
  expressions in the group basis is straightforward.

  Then, any automorphism of $\KD(2m)$ stabilizes $K_3$ and therefore
  induces an automorphism of $K_3$.  From the expressions of the Jones
  projection in the group basis, the involutive automorphism
  $\Theta_{-1}$ exchanges $K_{33}$ and $K_{34}$ which are therefore
  isomorphic.

  The specified inclusions are derived by comparison of the Jones
  projections. Looking at the inclusion diagrams imposes that $\delta$
  maps $I_0$ to $K_1$, $I_3$ and $I_4$ to $K_{33}$ and $K_{34}$ (or
  $K_{34}$ and $K_{33}$), and $J_{20}$ to $K_{31}$, and the $J_{2k+1}$
  to the $K_{1h}$ (see~\ref{sacig4pair}).

  As in~\ref{sacign}, the \sacigs of dimension dividing $2m$ are
  either contained in $K_1$ or in exactly one of $K_{33}$ or
  $K_{34}$. Therefore, the remaining \sacigs of dimension $4$ and $m$
  of $K_3$ are in exactly one of $K_{33}$ or $K_{34}$. As $K_{33}$ ad
  $K_{34}$ are isomorphic, it is sufficient to find those in
  $K_{33}$. Any \sacig of dimension $4$ contains a \sacig of dimension
  $2$, since it is the image by $\delta$ of a \sacig of dimension $m$
  contained in one of dimensions $2m$. The next argument depends on
  the parity of $m'$ :

  If $m'$ is odd, the only \sacig of dimension $2$ contained in
  $K_{33}$ is $I_0$; therefore the \sacigs of dimension $4$ of
  $K_{33}$ are mapped by $\delta$ on the $m+1$ \sacigs of dimension $m$ of $K_1$: they are $J_{20}$ and the $J_{2k+1}$.

  If $m'$ is even, $K_{33}$ admits only $4$ projections of trace
  $1/2m$ (see~\ref{K33}); therefore, $K_{33}$ contains at most $3$
  \sacigs of dimension $m$, including $K_{31}$. Conclusion: $K_3$
  contains at most $m+5$ \sacigs of dimension $m$ and at most $m+5$
  \sacigs of dimension $4$, and by self-duality it contains exactly
  $m+5$ \sacigs de dimension $m$ and $m+5$ \sacigs de dimension $4$.

  In both cases, $K_3$ admits exactly $m$ \sacigs of dimension $4$ not
  contained in $K_1$.
\end{proof}

\subsubsection{Lattice $\mathrm{l}(\KD(8))$}
\label{KD8}

Since $n$ is a power of $2$, any \sacig is contained in one of the
$K_i$'s (see prop~\ref{sacign}). The lattice $\mathrm{l}(\KD(8))$ can
therefore be constructed from those of $K_2 \equiv L^{\infty}(D_8)$,
$K_3$, and $K_4 \equiv \KD(4)$ (see~\ref{KD4})

\begin{prop}
  The lattice of \sacigs of $\KD(8)$ is as given in Figure~\ref{figure.KD8}.
\end{prop}
\begin{figure}[h]
  \begin{bigcenter}
    \definecolor{qqzzqq}{rgb}{0,0.6,0}
\definecolor{ffqqtt}{rgb}{1,0,0.2}
\definecolor{ccqqcc}{rgb}{0.8,0,0.8}
\definecolor{cqcqcq}{rgb}{0.75,0.75,0.75}

\pgfdeclarelayer{background}
\pgfdeclarelayer{nodes}
\pgfsetlayers{background,main,nodes}
\begin{tikzpicture}[yscale=-1,xscale=0.8]
  \tikzstyle{every node}=[fill=white]
  \begin{pgfonlayer}{nodes}
    \draw [color=blue] (12,5) node {$\KD(8)$};
    \draw [color=ccqqcc] (7,4) node {$K_4$};
    \draw [color=ffqqtt] (4,2) node {$J_0$};
    \draw [color=ccqqcc] (5,2) node {$J_4$};
    \draw [color=ccqqcc] (6,2) node {$J_2$};
    \draw [color=ccqqcc] (7,2) node {$J_6$};
    \draw [color=ffqqtt] (12,2) node {$J_{20}$};
    \draw [color=qqzzqq] (13,2) node {$J_{21}$};
    \draw [color=qqzzqq] (15,2) node {$J_{22}$};
    \draw [color=qqzzqq] (17,2) node {$J_{23}$};
    \draw [color=qqzzqq] (19,2) node {$J_{24}$};
    \draw [color=ffqqtt] (5,3) node {$K_0$};
    \draw [color=ccqqcc] (7,3) node {$K_{43}$};
    \draw [color=qqzzqq] (12,3) node {$K_1$};
    \draw [color=ffqqtt] (12,1) node {$I_0$};
    \draw [color=ffqqtt] (7,1) node {$I_2$};
    \draw [color=ffqqtt] (5,1) node {$I_1$};
    \draw [color=ffqqtt] (14,1) node {$I_3$};
    \draw [color=ffqqtt] (16,1) node {$I_4$};
    \draw [color=blue] (8,2) node {$J_1$};
    \draw [color=blue] (10,2) node {$J_3$};
    \draw [color=blue] (9,2) node {$J_5$};
    \draw [color=blue] (11,2) node {$J_7$};
    \draw [color=blue] (12,4) node {$K_3$};
    \draw [color=qqzzqq] (17,4) node {$K_2$};
    \draw [color=blue] (9,3) node {$K_{33}$};
    \draw [color=blue] (10,3) node {$K_{34}$};
    \draw [color=qqzzqq] (14,3) node {$K_{21}$};
    \draw [color=qqzzqq] (15,3) node {$K_{22}$};
    \draw [color=qqzzqq] (16,3) node {$K_{23}$};
    \draw [color=qqzzqq] (17,3) node {$K_{24}$};
    \draw [color=qqzzqq] (18,3) node {$K_{25}$};
    \draw [color=qqzzqq] (19,3) node {$K_{26}$};
    \draw [color=qqzzqq] (20,3) node {$K_{27}$};
    \draw [color=qqzzqq] (21,3) node {$K_{28}$};
    \draw [color=blue] (12,0) node {$\C$};
  \end{pgfonlayer}
  \draw [color=ccqqcc] (5,1)-- (4,2);
  \draw [color=ccqqcc] (7,1)-- (4,2);
  \draw [color=ccqqcc] (12,1)-- (4,2);
  \draw [color=ccqqcc] (12,1)-- (5,2);
  \draw (5,2)-- (5,3);
  \draw (4,2)-- (5,3);
  \draw (5,3)-- (7,4);
  \draw (7,4)-- (7,3);
  \draw [color=ccqqcc] (7,4)-- (12,3);
  \draw [color=ccqqcc] (6,2)-- (7,3);
  \draw [color=ccqqcc] (7,2)-- (7,3);
  \draw [color=ccqqcc] (12,1)-- (7,2);
  \draw [color=ccqqcc] (12,1)-- (6,2);
  \draw (12,1)-- (12,2);
  \draw [color=ccqqcc] (12,2)-- (7,3);
  \draw [color=ccqqcc] (14,1)-- (12,2);
  \draw [color=ccqqcc] (16,1)-- (12,2);
  \draw (14,1)-- (13,2);
  \draw (14,1)-- (15,2);
  \draw (16,1)-- (17,2);
  \draw (16,1)-- (19,2);
  \draw (19,2)-- (12,3);
  \draw (17,2)-- (12,3);
  \draw (15,2)-- (12,3);
  \draw (13,2)-- (12,3);
  \draw (12,2)-- (12,3);
  \draw (12,3)-- (12,4);
  \draw (12,3)-- (17,4);
  \draw (7,4)-- (12,5);
  \draw (12,4)-- (12,5);
  \draw (17,4)-- (12,5);
  \draw (13,2)-- (14,3);
  \draw (13,2)-- (15,3);
  \draw (15,2)-- (16,3);
  \draw (15,2)-- (17,3);
  \draw (17,2)-- (18,3);
  \draw (17,2)-- (19,3);
  \draw (19,2)-- (20,3);
  \draw (19,2)-- (21,3);
  \draw (14,3)-- (17,4);
  \draw (17,4)-- (15,3);
  \draw (17,4)-- (16,3);
  \draw (17,4)-- (17,3);
  \draw (17,4)-- (18,3);
  \draw (17,4)-- (19,3);
  \draw [color=ccqqcc] (12,2)-- (5,3);
  \draw (20,3)-- (17,4);
  \draw (21,3)-- (17,4);
  \draw (8,2)-- (9,3);
  \draw (9,2)-- (9,3);
  \draw (10,2)-- (10,3);
  \draw (11,2)-- (10,3);
  \draw (12,1)-- (11,2);
  \draw (12,1)-- (9,2);
  \draw (12,1)-- (10,2);
  \draw (12,1)-- (8,2);
  \draw [color=ccqqcc] (12,0)-- (7,1);
  \draw [color=ccqqcc] (12,0)-- (5,1);
  \draw (12,0)-- (12,1);
  \draw [color=ccqqcc] (12,0)-- (14,1);
  \draw [color=ccqqcc] (12,0)-- (16,1);
  \draw (10,3)-- (12,4);
  \draw (9,3)-- (12,4);
  \draw (12,2)-- (10,3);
  \draw (12,2)-- (9,3);
  \begin{pgfonlayer}{background}
    \newcommand{\dimline}[2]{\draw[color=black!10,very thin](3.5,#1) -- (21.5,#1); \node[right] at (21.5,#1) {dim $#2$};}
    \dimline{0}{1}
    \dimline{1}{2}
    \dimline{2}{4}
    \dimline{3}{8}
    \dimline{4}{16}
    \dimline{5}{32}
  \end{pgfonlayer}
  %
  %
\end{tikzpicture}
  \end{bigcenter}
  \caption{The lattice of \sacigs of $\KD(8)$}
  \label{figure.KD8}
\end{figure}
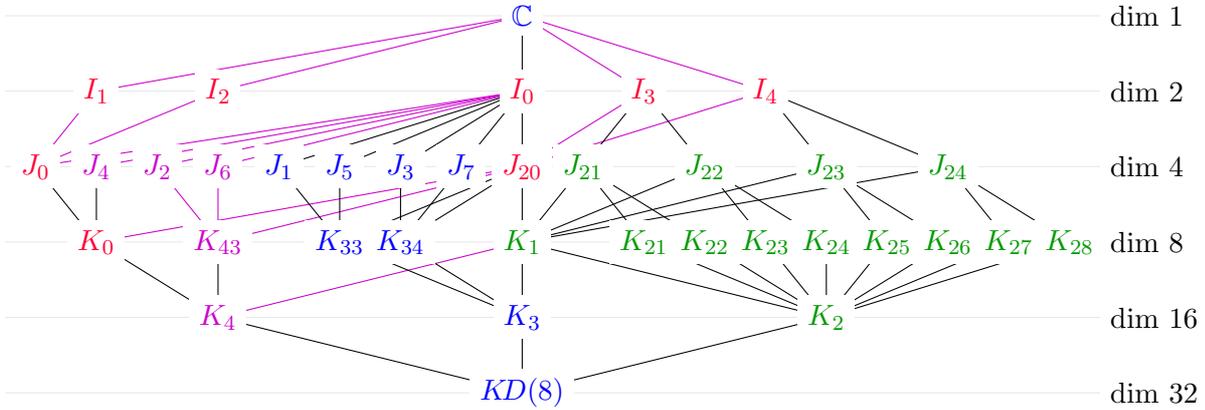
\begin{proof}
  We proceed essentially as in the general case. The algebra $K_1$
  admits, beside $J_{20}$, four \sacigs of dimension $m=4$, denoted
  $J_{21}$, $J_{22}$, $J_{23}$, and $J_{24}$ which each contains
  either $I_3$ or $I_4$. Beside $K_1$, there are two \sacigs de
  dimension $8$:
  \begin{itemize}
  \item $K_{33}=I(e_1+e_3+r^4_{1,1})$ containing $J_{20}$, $J_1$, and $J_5$;
  \item $K_{34}=I(e_1+e_3+r^4_{2,2})$ containing $J_{20}$, $J_3$, and $J_7$.
  \end{itemize}
  Both $K_{33}$ and $K_{34}$ have $\C \oplus \C \oplus \C \oplus \C
  \oplus M_2(\C)$ as matrix structure (computer calculation).
  Therefore they both admit, beside $J_{20}$, exactly two \sacigs of
  dimension $4$; those are $J_1$, $J_5$ and $J_3$, $J_7$. Their
  images by $\delta$ of $K_3$ are $J_{21}$, $J_{22}$ and $J_{23}$,
  $J_{24}$.
\end{proof}

\newpage
\appendix
\section{Collection of formulas for $\KD(n)$}
\label{section.formulas}

In this appendix, we collect all the formulas used in the study of the
algebra $\KD(n)$: left regular representation and expression of matrix
units in the group basis, standard coproducts of some special
elements, expression of $\Omega$ in terms of matrix units, method to
calculate the twisted coproducts and its expression on some elements.

\subsection{Left regular representation of $D_{2n}$}
\cite[6.6]{Vainerman.1998}\label{lambda}

Set $\epsilon_n=\mathrm{e} ^{i\pi/n}$. The left regular representation
of $D_{2n}$ in terms of matrix units is given by, for all $k$:
\begin{align*}
\lambda(a^k)&=e_1+e_2+(-1)^k(e_3+e_4)+\sum_{j=1}^{n-1}(\epsilon_n^{jk} e^j_{1,1}+\epsilon_n^{-jk} e^j_{2,2})\\
\lambda(ba^k)&=e_1-e_2-(-1)^k(e_3-e_4)+\sum_{j=1}^{n-1}(\epsilon_n^{-jk} e^j_{1,2}+\epsilon_n^{jk} e^j_{2,1})\\
\end{align*}

\subsection{Matrix units in the group basis and coinvolution}
\label{form2}

By inverting the previous formulas we get the following expressions
for the matrix units in the group basis:
\begin{align*}
&e_1 =\frac{1}{4n} \sum_{k=0}^{2n-1}(\lambda(a^k)+\lambda(ba^k))\,,
&e_2 &=\frac{1}{4n} \sum_{k=0}^{2n-1}(\lambda(a^k)-\lambda(ba^k))\,,\\
&e_3 =\frac{1}{4n} \sum_{k=0}^{2n-1}(-1)^k(\lambda(a^k)-\lambda(ba^k))\,,
&e_4 &=\frac{1}{4n} \sum_{k=0}^{2n-1}(-1)^k(\lambda(a^k)+\lambda(ba^k))\,,\\
&e^j_{1,1}=\frac{1}{2n} \sum_{k=0}^{2n-1}\mathrm{e}^{-ijk\pi/n}\lambda(a^k)\,,
&e^j_{2,2}&=\frac{1}{2n} \sum_{k=0}^{2n-1}\mathrm{e}^{ijk\pi/n}\lambda(a^k)\,,\\
&e^j_{1,2}=\frac{1}{2n} \sum_{k=0}^{2n-1}\mathrm{e}^{ijk\pi/n}\lambda(ba^k)\,,
&e^j_{2,1}&=\frac{1}{2n} \sum_{k=0}^{2n-1}\mathrm{e}^{-ijk\pi/n}\lambda(ba^k)\,.
\end{align*}
For the convenience of the reader, here are some direct consequences
of those formulas:
\begin{align*}
&e_1+e_2 =\frac{1}{2n} \sum_{k=0}^{2n-1}\lambda(a^{k})\,,
&e_3+e_4 &=\frac{1}{2n} \sum_{k=0}^{2n-1}(-1)^k\lambda(a^{k})\,,\\
&e_1+e_3 =\frac{1}{2n} \sum_{k=0}^{n-1}\lambda(a^{2k})+\lambda(ba^{2k+1})\,,
&e_1+e_4& =\frac{1}{2n} \sum_{k=0}^{n-1}\lambda(a^{2k})+\lambda(ba^{2k})\,,\\
&e_2+e_3 =\frac{1}{2n} \sum_{k=0}^{n-1}(\lambda(a^{2k})-\lambda(ba^{2k}))\,,
&e_1+e_2+e_3+e_4& = \frac{1}{n}\sum_{k=0}^{n-1}\lambda(a^{2k})\,,\\
&e^j_{1,1}+e^{n-j}_{2,2}=\frac{1}{n} \sum_{k=0}^{n-1}\mathrm{e}^{-2ijk\pi/n}\lambda(a^{2k})\,,
&e^j_{1,2}+e^{n-j}_{2,1}&=\frac{1}{n} \sum_{k=0}^{n-1}\mathrm{e}^{2ijk\pi/n}\lambda(ba^{2k})\,,\\
&e^j_{1,2}-e^{n-j}_{2,1}=\frac{1}{n} \sum_{k=0}^{n-1}\mathrm{e}^{ij(2k+1)\pi/n}\lambda(ba^{2k+1})\,.
\end{align*}

The coinvolution is the antiisomorphism defined by
$S(\lambda(g))=\lambda(g)^*$. It fixes the $e_i$ and $e_{i,j}^k$ (with
$i\neq j$) and exchanges the $e_{1,1}^j$ and $e_{2,2}^j$.

\subsection{The unitary $2$-cocycle $\Omega$}
\label{omega}

In~\cite[6.6]{Vainerman.1998} the unitary $2$-cocycle $\Omega$ used to
twist the coproduct is expressed in the group basis:
$$\Omega=\sum_{i,j=1}^4 c_{i,j} \lambda(h_i)\otimes\lambda(h_j)\,,$$
where $h_1=1$, $h_2=a^n$, $h_3=b$, $h_4=ba^n$, and the  $c_{i,j}$ are
the coefficients of the matrix:
$$ \frac{1}{8}\begin{pmatrix}
5&1&1&1\\
1&1&\bar{\nu}&\nu\\
1&\nu&1&\bar{\nu}\\
1&\bar{\nu}&\nu&1
\end{pmatrix}\quad
\text{with}\;\; \nu=-1+2\rm{i}\,.$$
Furthermore, $\Omega^{*}= \sum_{i,j=1}^4 c_{j,i} \lambda(h_i)\otimes\lambda(h_j)$.

We deduce its expression in terms of matrix units:\\
For $n$ even:
\begin{align*}
\Omega&=(e_1+e_2+e_3+e_4+q_1+q_2)\otimes(e_1+e_2+e_3+e_4+q_1+q_2)\\
&+(p_1+p_2)\otimes(e_1+e_4+q_1)+(e_1+e_4+q_1)\otimes(p_1+p_2)\\
&+i(p_1-p_2)\otimes(e_2+e_3+q_2)-i(e_2+e_3+q_2)\otimes(p_1-p_2)+(p_1+ip_2)\otimes(p_1-ip_2)\end{align*}
For $n$ odd:
\begin{align*}
\Omega&=(e_1+e_2+q_1+q_2)\otimes(e_1+e_2+q_1+q_2)\\
&+(e_3+e_4+p_1+p_2)\otimes(e_1+q_1)+(e_1+q_1)\otimes(e_3+e_4+p_1+p_2)\\
&+i(e_2+q_2)\otimes(e_3-e_4-p_1+p_2)-i(e_3-e_4-p_1+p_2)\otimes (e_2+q_2)\\
&+(e_3-ie_4-ip_1+p_2)\otimes(e_3+ie_4+ip_1+p_2)
\end{align*}

\begin{remark}
  \label{remark.omega}
  Note that the conjugation by $\Omega$ acts similarly on all factors
  $M_2(\C)$ of same parity. For example, if $x$ and $y$ are each in
  some even factor (not necessarily the same), then:
  $$\Omega (x\otimes y) \Omega^*=(q_1+q_2)\otimes (q_1+q_2)(x\otimes y) (q_1+q_2)\otimes (q_1+q_2)=x\otimes y\,.$$
  If $x$ and $y$ are each in some odd factor (not necessarily the
  same), then:
  $$\Omega (x\otimes y) \Omega^*=(-ip_1+p_2)\otimes (ip_1+p_2) (x\otimes y) (ip_1+p_2)\otimes (-ip_1+p_2)\,.$$
  In particular, for $j$ and $j'$ odd:
  $$\Omega (e^j_{1,2}\otimes e^{j'}_{1,2}) \Omega^*=r^j_{2,1}\otimes r^{j'}_{1,2}\,.$$

  In~\ref{subsection.copv}, we shall see an example of twisting of
  $x\otimes y$ for $x$ in some odd factor and $y$ in some even
  factor. In subsequent sections, we give further examples of
  calculation of twisted coproducts using this remark.
\end{remark}
\subsection{Twisted coproduct of the Jones projection of the subgroup $H_r$}
\label{J4}
Recall that, in $\C[D_{2n}]$ and for
$r=1,\dots, n-1$, the projection 
$$Q_r=\frac{1}{4}(1+\lambda(a^n)+\lambda(ba^r)+\lambda(ba^{r+n}))\, ,$$
is the Jones projection of the subalgebra $\C[H_r]$ of the subgroup%
\begin{displaymath}
  H_r=\{1, \lambda(a^n), \lambda(ba^r), \lambda(ba^{r+n}) \}\,.
\end{displaymath}
In order to illustrate Remark~\ref{remark.omega}, we prove that the twisted coproduct of $Q_r$ is of the form
\begin{displaymath}
  \Delta(Q_r)=Q_r\otimes Q_r+Q'_r\otimes Q'_r+S(\tilde{P}_r)\otimes \tilde{P}_r + S(\tilde{P'}_r)\otimes \tilde{P'}_r\,.
\end{displaymath}
where $\tilde{P}_r$ and $\tilde{P'}_r$ are defined below, depending on
the parity of $n$. This is used in the proof of
Proposition~\ref{prop.abelien} to conclude that $Q_r$ is the Jones
projection of $J_r$, spanned by $Q_r$, $Q'_r$, $\tilde{P}_r$ and
$\tilde{P'_r}$.

We recall from~\ref{groupe} that the standard coproduct of $Q_r$ is:
$$\Delta_s(Q_r)=Q_r\otimes Q_r+Q'_r\otimes Q'_r+P_r\otimes P_r + P'_r\otimes P'_r\,,$$
where $Q_r$, $Q'_r$, $P_r$ and $P'_r$ are the minimal projections of
$J_r$ (notice that all the elements of the group $H_r$ are of order $2$).

Let us start with $n$ odd, and express $Q_r,Q'_r,P_r,P'_r$ in the
matrix units basis:
\begin{align*}
  Q_r&=\frac{1}{4}(1+\lambda(a^n)+\lambda(ba^r)+\lambda(ba^{r+n}))=e_1+\sum_{j=1,\,j \text{ even}}^{n-1} q(0,\frac{jr\pi}{n},j)\,.\\
  Q'_r&=\frac{1}{4}(1+\lambda(a^n)-\lambda(ba^r)-\lambda(ba^{r+n}))=e_2+\sum_{j=1,\,j \text{ even}}^{n-1} q(0,\frac{jr\pi}{n}+\pi,j)\,.
\end{align*}
For $r$ even:
\begin{align*}
  P_r&=\frac{1}{4}(1-\lambda(a^n)-\lambda(ba^r)+\lambda(ba^{r+n}))=e_3+\sum_{j=1,\,j \text{ odd}}^{n-1} q(0,\frac{jr\pi}{n} +\pi,j)\,,\\
  P'_r&=\frac{1}{4}(1-\lambda(a^n)+\lambda(ba^r)-\lambda(ba^{r+n}))=e_4+\sum_{j=1,\,j \text{ odd}}^{n-1} q(0,\frac{jr\pi}{n},j)\,.
\end{align*}
For $r$ odd:
\begin{align*}
  P_r&=\frac{1}{4}(1-\lambda(a^n)-\lambda(ba^r)+\lambda(ba^{r+n}))=e_4+\sum_{j=1,\,j \text{ odd}}^{n-1} q(0,\frac{jr\pi}{n}+\pi,j)\,,\\
  P'_r&=\frac{1}{4}(1-\lambda(a^n)+\lambda(ba^r)-\lambda(ba^{r+n}))=e_3+\sum_{j=1,\,j \text{ odd}}^{n-1} q(0,\frac{jr\pi}{n},j)\,.
\end{align*}
Applying~\ref{omega}, yields that:
$$\Delta(Q_r)=Q_r\otimes Q_r+Q'_r\otimes Q'_r+(V\otimes V^*)(P_r\otimes P_r)(V^*\otimes V) + (V\otimes V^*)(P'_r\otimes P'_r)(V^*\otimes V)\,,$$
where $V=(e_3-ie_4-ip_1+p_2)$, and $(V\otimes V^*)$ is the unique
term of $\Omega$ which acts non trivially. Note that the following
relations hold:
\begin{equation*}
  Ve_3V^*=e_3, \qquad Ve_4V^*=e_4, \qquad P_r+P'_r=e_3+e_4\,.
\end{equation*}
Therefore, it is sufficient to conjugate $q(0,\frac{jr\pi}{n},j)$ with
$\pm ip_1+p_2$. As desired, we get:
$$\Delta(Q_r)=Q_r\otimes Q_r+Q'_r\otimes Q'_r+S(\tilde{P}_r)\otimes \tilde{P}_r + S(\tilde{P'}_r)\otimes \tilde{P'}_r\,,$$
with, for $r$ even:
\begin{displaymath}
  \tilde{P}_r =e_3+\sum_{j=1,\,j \text{ odd}}^{n-1} q(-\frac{jr\pi}{n},\pi,j)
  \quad \text{ and }\quad
  \tilde{P'}_r =e_4+\sum_{j=1,\,j \text{ odd}}^{n-1} q(\frac{jr\pi}{n},0,j)\,,
\end{displaymath}
and for $r$ odd:
\begin{displaymath}
  \tilde{P}_r =e_4+\sum_{j=1,\,j \text{ odd}}^{n-1} q(-\frac{jr\pi}{n},\pi,j)
  \quad \text{ and }\quad
  \tilde{P'}_r =e_3+\sum_{j=1,\,j \text{ odd}}^{n-1} q(\frac{jr\pi}{n},0,j)\,.
\end{displaymath}
From~\ref{form2}, we get: $S(q(\alpha,\beta,j))=q(-\alpha,\beta,j)$.

The calculation of $\Delta(Q_r)$ in the case $n$ even is similar with
\begin{displaymath}
  \tilde{P}_r =\sum_{j=1,\,j \text{ odd}}^{n-1} q(-\frac{jr\pi}{n},\pi,j) \quad \text{ and}\quad \tilde{P'}_r =\sum_{j=1,\,j \text{ odd}}^{n-1} q(\frac{jr\pi}{n},0,j)\,,
\end{displaymath}
and with, for $r$ even:
\begin{displaymath}
  Q_r=e_1+e_4+\sum_{j=1,\,j \text{ even}}^{n-1} q(0,\frac{jr\pi}{n},j)\quad \text{ and}\quad Q'_r=e_2+e_3+\sum_{j=1,\,j \text{ even}}^{n-1} q(0,\frac{jr\pi+\pi}{n},j)\,,
\end{displaymath}
and for $r$ odd:
\begin{displaymath}
  Q_r=e_1+e_3+\sum_{j=1,\,j \text{ even}}^{n-1} q(0,\frac{jr\pi}{n},j)\quad \text{ and}\quad Q'_r=e_2+e_4+\sum_{j=1,\,j \text{ even}}^{n-1} q(0,\frac{jr\pi+\pi}{n},j)\,.
 \end{displaymath}

\subsection{Coproducts of Jones projections for \sacigs of dimension $2n$}

\subsubsection{Standard coproduct expressions}
\label{copG}

We need to calculate the twisted coproduct of certain projections. We
start with their standard coproducts in $\C[D_{2n}]$, using the
results of~\ref{groupe}.
\begin{proposition}
  The standard coproduct $\Delta_s$ of the Kac algebra of the dihedral
  group $D_{2n}$ satisfies:
 \begin{align*}
   \Delta_s(e_1+e_2)=&\,(e_1+e_2)\otimes(e_1+e_2)+(e_3+e_4)\otimes(e_3+e_4)\\
   &+\sum_{j=1}^{n-1} e^j_{1,1}\otimes e^{j}_{2,2}+e^{j}_{2,2}\otimes e^j_{1,1}\\
   \Delta_s(e_1+e_3)=&\,(e_1+e_3)\otimes(e_1+e_3)+(e_2+e_4)\otimes(e_2+e_4)\\
   & +\frac{1}{2}\sum_{j=1}^{n-1}( e^j_{1,1}+ e^{n-j}_{2,2})\otimes (e^j_{2,2}+e^{n-j}_{1,1})
   +( e^j_{1,2}- e^{n-j}_{2,1})\otimes (e^j_{2,1}- e^{n-j}_{1,2})\\
   \Delta_s(e_1+e_4)=&\,(e_1+e_4)\otimes(e_1+e_4)+(e_2+e_3)\otimes(e_2+e_3)\\
   &+\frac{1}{2}\sum_{j=1}^{n-1}(e^{n-j}_{1,1}+e^j_{2,2})\otimes (e^j_{1,1}+e^{n-j}_{2,2})+( e^j_{1,2}+ e^{n-j}_{2,1})\otimes (e^j_{2,1}+ e^{n-j}_{1,2})\\
   \Delta_s(e_3+e_4)=&\,(e_1+e_2)\otimes(e_3+e_4)+(e_3+e_4)\otimes(e_1+e_2)\\
   &+\sum_{j=1}^{n-1} e^{n-j}_{1,1}\otimes e^{j}_{1,1}+e^{n-j}_{2,2}\otimes e^j_{2,2}\\
   \Delta_s(e_1+e_2+e_3+e_4)=&\,(e_1+e_2+e_3+e_4)\otimes(e_1+e_2+e_3+e_4)\\
   &+\sum_{j=1}^{n-1} (e^{n-j}_{1,1}+e^j_{2,2})\otimes (e^j_{1,1}+e^{n-j}_{2,2})\\
   \Delta_s(e^j_{1,1})=&\,e^j_{1,1}\otimes(e_1+e_2)+e^{n-j}_{2,2}\otimes(e_3+e_4)+(e_1+e_2)\otimes e^j_{1,1}+(e_3+e_4)\otimes e^{n-j}_{2,2} &\\
    &+\sum_{j'<j} e^{j-j'}_{1,1}\otimes e^{j'}_{1,1} +e^{n-(j-j')}_{2,2}\otimes e^{n-j'}_{2,2}\\
    &+\sum_{j'>j}e^{j'-j}_{2,2}\otimes e^{j'}_{1,1}+e^{n-(j'-j)}_{1,1}\otimes e^{n-j'}_{2,2}
  \end{align*}
\end{proposition}
\begin{proof}
  The projections $e_1+e_2$ and of $e_1+e_2+e_3+e_4$ are the Jones
  projections of the subgroups generated respectively by $\lambda(a)$
  and $\lambda(a^2)$. Their standard coproducts can be derived
  from~\ref{groupe}.

  From $e_1+e_4 =\frac{1}{2n}
  \sum_{k=0}^{n-1}(\lambda(a^{2k})+\lambda(ba^{2k}))$, and using the
  formulas of~\ref{form2} we get:
  \begin{align*}
    \Delta_s(e_1+e_4)=&\,\frac{1}{2n} \sum_{k=0}^{n-1}\lambda(a^{2k})\otimes \lambda(a^{2k})+\lambda(ba^{2k})\otimes \lambda(ba^{2k})\\
    =&\,\frac{1}{2}\Delta_s(e_1+e_2+e_3+e_4)
    \\&+\frac{1}{2n} \sum_{k=0}^{n-1}\lambda(ba^{2k})\otimes \left((e_1+e_4)-(e_2+e_3)+\sum_{j=1}^{n-1}\mathrm{e}^{-2ijk\pi/n}e^j_{1,2}+\mathrm{e}^{2ijk\pi/n}e^j_{2,1}\right)\\
    =&\,\frac{1}{2}\Delta_s(e_1+e_2+e_3+e_4)+\frac{1}{2}\left((e_1+e_4)-(e_2+e_3)\right)\otimes\left((e_1+e_4)-(e_2+e_3)\right)\\
    &+\sum_{j=1}^{n-1}\left(\frac{1}{2n} \sum_{k=0}^{n-1}\mathrm{e}^{-2ijk\pi/n}\lambda(ba^{2k})\otimes
    e^j_{1,2}+\frac{1}{2n} \sum_{k=0}^{n-1}\mathrm{e}^{2ijk\pi/n}\lambda(ba^{2k})\otimes e^j_{2,1}\right)\\
    =&\,(e_1+e_4)\otimes(e_1+e_4)+(e_2+e_3)\otimes(e_2+e_3)\\ &+\frac{1}{2}\sum_{j=1}^{n-1}(e^{n-j}_{1,1}+e^j_{2,2})\otimes (e^j_{1,1}+e^{n-j}_{2,2})+( e^j_{2,1}+ e^{n-j}_{1,2})\otimes e^j_{1,2}
    +( e^j_{1,2}+ e^{n-j}_{2,1})\otimes e^j_{2,1}\\
    =&\,(e_1+e_4)\otimes(e_1+e_4)+(e_2+e_3)\otimes(e_2+e_3)\\
    &+\frac{1}{2}\sum_{j=1}^{n-1}(e^{n-j}_{1,1}+e^j_{2,2})\otimes (e^j_{1,1}+e^{n-j}_{2,2})+( e^j_{1,2}+ e^{n-j}_{2,1})\otimes (e^j_{2,1}+ e^{n-j}_{1,2})
  \end{align*}
  The other coproduct expressions can be calculated similarly.
\end{proof}

\subsubsection{Twisted coproduct expressions}
\label{cop}

Using computer exploration as described in~\ref{mupad.K2=Dn}, and with
some factorization efforts, one gets:\\
For $n=3$:
\begin{align*}
  \Delta(e_1+e_2)=&\,(e_1+e_2)\otimes(e_1+e_2)+(e_3+e_4)\otimes(e_3+e_4)\\
  &+ r^1_{1,1}\otimes r^1_{1,1}+r^1_{2,2}\otimes r^1_{2,2}+e^2_{1,1}\otimes e^{2}_{2,2}+e^{2}_{2,2}\otimes e^2_{1,1}\,,\\
  \Delta(e_1+e_3)=&\,(e_1+e_3)\otimes(e_1+e_3)+(e_2+e_4)\otimes(e_2+e_4)\\
  &+\frac{1}{2}\left(( e^{1}_{2,2}+ r^2_{1,1})\otimes (e^{1}_{1,1}+ r^2_{1,1})+( e^1_{1,1}+ r^{2}_{2,2})\otimes (e^{1}_{2,2}+ r^2_{2,2})\right)\\
  &+\frac{1}{2}\left(( e^1_{1,2}+r^{2}_{2,1})\otimes (e^1_{2,1}+r^{2}_{2,1})+( e^1_{2,1}+r^{2}_{1,2})\otimes (e^1_{1,2}+r^{2}_{1,2})\right)\,,\\
  \Delta(e_1+e_4)=&\,(e_1+e_4)\otimes(e_1+e_4)+(e_2+e_3)\otimes(e_2+e_3)\\
  &+\frac{1}{2}\left(( e^{1}_{1,1}+ r^2_{1,1})\otimes (e^{1}_{2,2}+ r^2_{1,1})+( e^1_{2,2}+ r^{2}_{2,2})\otimes (e^{1}_{1,1}+ r^2_{2,2})\right)\\
  &+\frac{1}{2}\left(( e^1_{1,2}-r^{2}_{1,2})\otimes (e^1_{2,1}-r^{2}_{1,2})+( e^1_{2,1}-r^{2}_{2,1})\otimes (e^1_{1,2}-r^{2}_{2,1})\right)\,,\\
  \Delta(e_1+e_2+&e_3+e_4)=(e_1+e_2+e_3+e_4)\otimes(e_1+e_2+e_3+e_4)\\&+(r^1_{1,1}+e^{2}_{2,2})\otimes(r^1_{1,1}+e^{2}_{1,1})+(r^1_{2,2}+e^{2}_{1,1})\otimes(r^1_{2,2}+e^{2}_{2,2})\,.
\end{align*}
For $n=4$:
\begin{align*}
  \Delta(e_1+e_2)=&\,(e_1+e_2)\otimes(e_1+e_2)+(e_3+e_4)\otimes(e_3+e_4)
  + r^1_{1,1}\otimes r^1_{1,1}\\&+r^1_{2,2}\otimes r^1_{2,2}
  +e^2_{1,1}\otimes e^{2}_{2,2}+e^{2}_{2,2}\otimes e^2_{1,1}+ r^3_{1,1}\otimes r^3_{1,1}+r^3_{2,2}\otimes r^3_{2,2}\\
  \Delta(e_1+e_3)=&\,(e_1+e_3)\otimes(e_1+e_3)+(e_2+e_4)\otimes(e_2+e_4)\,,\\
  &+\frac{1}{2}\left( ( e^{1}_{1,1}+ e^{3}_{1,1})\otimes (e^{1}_{2,2}+ e^{3}_{2,2})+( e^1_{2,2}+ e^{3}_{2,2})\otimes (e^{1}_{1,1}+ e^{3}_{1,1})\right)\\
  &+\frac{1}{2}\left(( e^2_{1,2}-e^{3}_{1,2})\otimes (e^1_{2,1}-e^{3}_{2,1})+( e^1_{2,1}-e^{3}_{2,1})\otimes (e^1_{1,2}-e^{3}_{1,2})\right)\\
  &+ r^2_{1,1} \otimes r^2_{1,1} +r^2_{2,2} \otimes r^2_{2,2}\,,\\
  \Delta(e_1+e_4)=&\,(e_1+e_4)\otimes(e_1+e_4)+(e_2+e_3)\otimes(e_2+e_3)\\
  &+\frac{1}{2}\left(( e^{1}_{1,1}+ e^{3}_{2,2})\otimes (e^{1}_{2,2}+ e^{3}_{1,1})+( e^1_{2,2}+ e^{3}_{1,1})\otimes (e^{1}_{1,1}+ e^{3}_{2,2})\right)\\
  &+\frac{1}{2}\left(( e^1_{1,2}+e^{3}_{2,1})\otimes (e^1_{2,1}+e^{3}_{1,2})+( e^1_{2,1}+e^{3}_{1,2})\otimes (e^1_{1,2}+e^{3}_{2,1})\right)\\
  &+  p^2_{1,1} \otimes p^2_{1,1} +p^2_{2,2} \otimes p^2_{2,2}\,,\\
  \Delta(e_1+e_2+&e_3+e_4)=(e_1+e_2+e_3+e_4)\otimes(e_1+e_2+e_3+e_4)\\&+(r^1_{1,1}+r^{3}_{2,2})\otimes(r^1_{1,1}+r^{3}_{2,2})+(r^1_{2,2}+r^{3}_{1,1})\otimes(r^1_{2,2}+r^{3}_{1,1})\\&+(e^2_{1,1}+e^{2}_{2,2})\otimes(e^2_{1,1}+e^{2}_{2,2})\,.
\end{align*}

Using the conjugation rules of Remark~\ref{omega}, the formulas for
the standard coproduct, and the expressions of the twisted coproduct
for $n=3$ and $n=4$, and with the help of the computer as described in
\ref{mupad.K2=Dn}, we get:
\begin{align*}
\Delta(e_1+e_2)=&\,(e_1+e_2)\otimes(e_1+e_2)+(e_3+e_4)\otimes(e_3+e_4)\\
&+ \sum_{j=1,\,j \text{ odd}}^{n-1}r^j_{1,1}\otimes r^j_{1,1}+r^j_{2,2}\otimes r^j_{2,2}
+\sum_{j=1,\,j \text{ even}}^{n-1}e^j_{1,1}\otimes e^{j}_{2,2}+e^{j}_{2,2}\otimes e^j_{1,1}\,.
\end{align*}
For $n$ odd:
\begin{align*}
  \Delta(e_1+e_3)=&\,(e_1+e_3)\otimes(e_1+e_3)+(e_2+e_4)\otimes(e_2+e_4)\\
  &+\frac{1}{2} \sum_{j=1,\,j \text{ odd}}^{n-1}\left(( e^{j}_{2,2}+ r^{n-j}_{1,1})\otimes (e^{j}_{1,1}+ r^{n-j}_{1,1})+( e^j_{1,1}+ r^{n-j}_{2,2})\otimes (e^{j}_{2,2}+ r^{n-j}_{2,2})\right)\\
  &+\frac{1}{2}\sum_{j=1,\,j \text{ odd}}^{n-1}\left(( e^j_{1,2}+r^{n-j}_{2,1})\otimes (e^j_{2,1}+r^{n-j}_{2,1})+( e^j_{2,1}+r^{n-j}_{1,2})\otimes (e^j_{1,2}+r^{n-j}_{1,2})\right)\,,\\
  \Delta(e_1+e_4)=&\,(e_1+e_4)\otimes(e_1+e_4)+(e_2+e_3)\otimes(e_2+e_3)\\
  &+\frac{1}{2} \sum_{j=1,\,j \text{ odd}}^{n-1}\left(( e^{j}_{1,1}+ r^{n-j}_{1,1})\otimes (e^{j}_{2,2}+ r^{n-j}_{1,1})+( e^j_{2,2}+ r^{n-j}_{2,2})\otimes (e^{j}_{1,1}+ r^{n-j}_{2,2})\right)\\
  &+\frac{1}{2}\sum_{j=1,\,j \text{ odd}}^{n-1}\left(( e^j_{1,2}-r^{n-j}_{1,2})\otimes (e^j_{2,1}-r^{n-j}_{1,2})+( e^j_{2,1}-r^{n-j}_{2,1})\otimes (e^j_{1,2}-r^{n-j}_{2,1})\right)\,,
\end{align*}
\begin{align*}
  \Delta(e_1+e_2+&e_3+e_4)=(e_1+e_2+e_3+e_4)\otimes(e_1+e_2+e_3+e_4)\\
  &+\sum_{j=1,\,j \text{ odd}}^{n-1}\left((r^j_{1,1}+e^{n-j}_{2,2})\otimes(r^j_{1,1}+e^{n-j}_{1,1})+(r^j_{2,2}+e^{n-j}_{1,1})\otimes(r^j_{2,2}+e^{n-j}_{2,2})\right)\,,\\
  \Delta(e_3+e_4)=&\,(e_1+e_2)\otimes(e_3+e_4)+(e_3+e_4)\otimes(e_1+e_2)\\
  &+ \sum_{j=1,\,j \text{ even}}^{n-1} r^{n-j}_{1,1}\otimes e^{j}_{1,1}+ r^{n-j}_{2,2}\otimes e^{j}_{2,2}
  +  e^{j}_{1,1}\otimes r^{n-j}_{2,2}+ e^{j}_{2,2}\otimes r^{n-j}_{1,1}\,,\\
  \\
  \Delta(e^2_{1,1})=&\,e^2_{1,1}\otimes(e_1+e_2)+r^{n-2}_{1,1}\otimes(e_3+e_4)+(e_1+e_2)\otimes e^2_{1,1}+(e_3+e_4)\otimes r^{n-2}_{2,2}\\
  &+r^1_{2,2}\otimes r^{1}_{1,1}+ e^{n-1}_{2,2}\otimes e^{n-1}_{2,2}\\
  &+\sum_{j>2,\,j \text{ even}}^{n-1}e^{j-2}_{2,2}\otimes e^{j}_{1,1}+r^{n-j+2}_{2,2}\otimes r^{n-j}_{2,2}\\
  &+\sum_{j>2,\,j \text{ odd}}^{n-1}r^{j-2}_{1,1}\otimes r^{j}_{1,1}+e^{n-j+2}_{1,1}\otimes e^{n-j}_{2,2}\,.
\end{align*}
For $n$ even:
\begin{align*}
  \Delta(e_1+e_3)=&\,(e_1+e_3)\otimes(e_1+e_3)+(e_2+e_4)\otimes(e_2+e_4)\\
  &+\frac{1}{2} \sum_{j=1,\,j\text{ odd}}^{n-1}\left(( r^{j}_{2,2}+ r^{n-j}_{1,1})\otimes (r^{j}_{2,2}+ r^{n-j}_{1,1})+(r^j_{2,1}-r^{n-j}_{1,2})\otimes( r^j_{2,1}-r^{n-j}_{1,2})\right)\\
  &+\frac{1}{2}\sum_{j=1,\,j \text{ even}}^{n-1}\left(( e^{j}_{1,1}+ e^{n-j}_{2,2})\otimes (e^{j}_{2,2}+ e^{n-j}_{1,1})+( e^j_{1,2}-e^{n-j}_{2,1})\otimes (e^j_{2,1}-e^{n-j}_{1,2})\right)\,,
  \\
  \Delta(e_1+e_4)=&\,(e_1+e_4)\otimes(e_1+e_4)+(e_2+e_3)\otimes(e_2+e_3)\\
  &+\frac{1}{2} \sum_{j=1,\,j  \text{ odd}}^{n-1}\left(( r^{j}_{1,1}+ r^{n-j}_{2,2})\otimes (r^{j}_{1,1}+ r^{n-j}_{2,2})+(r^j_{2,1}+r^{n-j}_{1,2})\otimes( r^j_{2,1}+r^{n-j}_{1,2})\right)\\
  &+\frac{1}{2}\sum_{j=1,\, j \text{ even}}^{n-1}\left(( e^{n-j}_{1,1}+ e^{j}_{2,2})\otimes (e^{j}_{1,1}+ e^{n-j}_{2,2})+( e^j_{1,2}+e^{n-j}_{2,1})\otimes (e^j_{2,1}+e^{n-j}_{1,2})\right)\,,
  \\
  \Delta(e_3+e_4)=&\,(e_1+e_2)\otimes(e_3+e_4)+(e_3+e_4)\otimes(e_1+e_2)\\
  &+ \sum_{j=1,\,j \text{ odd}}^{n-1} r^{j}_{1,1}\otimes r^{n-j}_{2,2}+ r^{j}_{2,2}\otimes r^{n-j}_{1,1}\\
  &+ \sum_{j=1,\,j \text{ even}}^{n-1} e^{n-j}_{1,1}\otimes e^j_{1,1}+ e^{n-j}_{2,2}\otimes e^j_{2,2}\,\\
  \Delta(e_1+e_2+&e_3+e_4)=(e_1+e_2+e_3+e_4)\otimes(e_1+e_2+e_3+e_4)\\
  &+\sum_{j=1,\,j \text{ odd}}^{n-1}(r^j_{1,1}+r^{n-j}_{2,2})\otimes(r^j_{1,1}+r^{n-j}_{2,2})\\
  &+\sum_{j=1,\,j \text{ even}}^{n-1}(e^j_{1,1}+e^{n-j}_{2,2})\otimes(e^j_{2,2}+e^{n-j}_{1,1})\,,
\end{align*}
\begin{align*}
  \Delta(e^2_{1,1})= &e^2_{1,1}\otimes(e_1+e_2)+e^{n-2}_{2,2}\otimes(e_3+e_4)+(e_1+e_2)\otimes e^2_{1,1}+(e_3+e_4)\otimes e^{n-2}_{2,2}\\
  &+r^1_{2,2}\otimes r^{1}_{1,1}+ r^{n-1}_{1,1}\otimes r^{n-1}_{2,2}\\
  &+\sum_{j>2,\, j \text{ even}}^{n-1}e^{j-2}_{2,2}\otimes e^{j}_{1,1}+e^{n-j+2}_{1,1}\otimes e^{n-j}_{2,2}\\
  &+\sum_{j>2,\, j \text{ odd}}^{n-1}r^{j-2}_{1,1}\otimes r^{j}_{1,1}+r^{n-j+2}_{2,2}\otimes r^{n-j}_{2,2}\,.
\end{align*}

\subsection{Coproducts of $v$ and $w$ of $K_3$ in $\KD(2m)$}
\label{subsection.copv}

In this section, we calculate the coproducts of the unitary elements
$v$ and $w$ of $K_3$ in $\KD(2m)$ needed
in~\ref{B4m}. %

\begin{prop}
  With the notations of~\ref{B4m}, the unitary elements $v$ and $w$ of
  $K_3$ expand as follow in the group basis:
  \begin{align*}
    v&=\lambda(ba)B_0-\frac12B_1[\mathrm{i}(\lambda(a)-\lambda(a^{-1}))+\lambda(ba)+\lambda(ba^{-1})]
    \\
    w&=\lambda(ba^{-1})B_0+\frac12B_1[\mathrm{i}(\lambda(a)-\lambda(a^{-1}))-\lambda(ba)-\lambda(ba^{-1})]
  \end{align*}
  The coproduct of  $v$ satisfies : $\Delta(v)=v \otimes B_0v + w\otimes B_1 v$.\\
  Furthermore: $\Delta(w)=w \otimes B_0w + v\otimes B_1 w$.
\end{prop}
\begin{proof}
  A straightforward calculation gives the expressions in the group
  basis.

  We now check the expression of $\Delta(v)$ by proving that
  untwisting it yields back the standard coproduct of $v$.  Namely,
  setting $D_{i,j}=(B_i\otimes B_j)(v \otimes B_0v + w\otimes B_1 v)$
  for $i=1,2$ and $j=1,2$, we check the equalities:
  $\Omega^{*}D_{i,j}\Omega=(B_i\otimes B_j)\Delta_s(v)$, for $i=1,2$
  and $j=1,2$.

  Note first that $B_1$ is the projection on odd factors and that
  $$\Delta(B_0)=B_0\otimes B_0+B_1\otimes B_1\quad \text{et} \quad \Delta(B_1)=B_0\otimes B_1+B_1\otimes B_0$$
  since $\lambda(a^n))$ belongs to the intrinsic group. Using
  Remark~\ref{remark.omega}, it follows that:
  $$\Delta_s(vB_0)=\lambda(ba)B_0\otimes \lambda(ba)B_0+\lambda(ba)B_1\otimes \lambda(ba)B_1=\Omega^{*}D_{0,0}\Omega+\Omega^{*}D_{1,1}\Omega$$

  The action of $\Omega$ on $x\otimes y$ depends only on the parity of
  the factors containing $x$ and $y$, and in particular is independent
  of $n$ (even). It is therefore sufficient to use the following
  formulas which have been obtained by computer in $\KD(4)$.
  \begin{itemize}
  \item For tensors in $\KD(2m)^{2j+1} \otimes \KD(2m)^{2k}$:
    \begin{align*}
      \Omega^{*}(r_{1,2}\otimes e_{1,2})\Omega&=\frac{\mathrm{i}}{2}(e_{1,1}\otimes e_{1,1}-e_{2,2}\otimes e_{2,2})-\frac12(e_{1,2}\otimes e_{1,2}+e_{2,1}\otimes e_{2,1})\\
      \Omega^{*}(r_{1,2}\otimes e_{2,1})\Omega&=\frac{\mathrm{i}}{2}(e_{1,1}\otimes e_{2,2}-e_{2,2}\otimes e_{1,1})-\frac12(e_{1,2}\otimes e_{2,1}+e_{2,1}\otimes e_{1,2})\\
      \Omega^{*}(r_{2,1}\otimes e_{1,2})\Omega&=\frac{\mathrm{i}}{2}(e_{2,2}\otimes e_{1,1}-e_{1,1}\otimes e_{2,2})-\frac12(e_{2,1}\otimes e_{1,2}+e_{1,2}\otimes e_{2,1})\\
      \Omega^{*}(r_{2,1}\otimes e_{2,1})\Omega&=\frac{\mathrm{i}}{2}(e_{2,2}\otimes e_{2,2}-e_{1,1}\otimes e_{1,1})-\frac12(e_{1,2}\otimes e_{1,2}+e_{2,1}\otimes e_{2,1})
    \end{align*}
  \item For tensors in $\KD(2m)^{2j+1} \otimes (\C e_1\oplus \C e_2\oplus \C e_3\oplus \C e_4)$:  
    \begin{align*}
      \Omega^{*}[r_{1,2}\otimes ((e_1+e_3)-(e_2+e_4))]\Omega&=\frac{\mathrm{i}}{2}[(e_{1,1}-e_{2,2})\otimes ((e_1+e_2)-(e_3+e_4))]\\
      &-\frac12[(e_{1,2}+e_{2,1})\otimes ((e_1+e_3)-(e_2+e_4))]\\
      \Omega^{*}[r_{2,1}\otimes ((e_1+e_3)-(e_2+e_4))]\Omega&=\frac{\mathrm{i}}{2}[(e_{2,2}-e_{1,1})\otimes ((e_1+e_2)-(e_3+e_4))]\\
      &-\frac12[(e_{1,2}+e_{2,1})\otimes ((e_1+e_3)-(e_2+e_4))]
    \end{align*}
  \item For tensors in $\KD(2m)^{2k} \otimes \KD(2m)^{2j+1}$:
    \begin{align*}
      \Omega^{*}(e_{1,2}\otimes r_{1,2})\Omega&=\frac{\mathrm{i}}{2}(e_{2,2}\otimes e_{1,1}-e_{1,1}\otimes e_{2,2})-\frac12(e_{1,2}\otimes e_{2,1}+e_{2,1}\otimes e_{1,2})\\
      \Omega^{*}(e_{2,1}\otimes r_{1,2})\Omega&=\frac{\mathrm{i}}{2}(e_{1,1}\otimes e_{1,1}-e_{2,2}\otimes e_{2,2})-\frac12(e_{1,2}\otimes e_{1,2}+e_{2,1}\otimes e_{2,1})\\
      \Omega^{*}(e_{1,2}\otimes r_{2,1})\Omega&=\frac{\mathrm{i}}{2}(e_{2,2}\otimes e_{2,2}-e_{1,1}\otimes e_{1,1})-\frac12(e_{1,2}\otimes e_{1,2}+e_{2,1}\otimes e_{2,1})\\
      \Omega^{*}(e_{2,1}\otimes r_{2,1})\Omega&=\frac{\mathrm{i}}{2}(e_{1,1}\otimes e_{2,2}-e_{2,2}\otimes e_{1,1})-\frac12(e_{2,1}\otimes e_{1,2}+e_{1,2}\otimes e_{2,1})
    \end{align*}
   \item For tensors in $(\C e_1\oplus \C e_2\oplus \C e_3\oplus \C e_4)
 \otimes \KD(2m)^{2j+1}$:
      \begin{align*}
      \Omega^{*}[((e_1+e_3)-(e_2+e_4))\otimes r_{1,2}]\Omega&=\frac{\mathrm{i}}{2}[((e_1+e_2)-(e_3+e_4))\otimes(e_{1,1}-e_{2,2}) ]\\
      &-\frac12[((e_1+e_3)-(e_2+e_4))\otimes(e_{1,2}+e_{2,1}) ]\\
      \Omega^{*}[((e_1+e_3)-(e_2+e_4))\otimes r_{2,1}]\Omega&=\frac{\mathrm{i}}{2}[((e_1+e_2)-(e_3+e_4))\otimes(e_{2,2}-e_{1,1}) ]\\
      &-\frac12[((e_1+e_3)-(e_2+e_4))\otimes(e_{1,2}+e_{2,1}) ]\\
    \end{align*}
  \end{itemize}
  It follows that
  $\Delta_s(vB_1)=\Omega^{*}D_{0,1}\Omega+\Omega^{*}D_{1,0}\Omega$. Let
  us detail, for example, the calculation of
  $\Omega^{*}D_{1,0}\Omega$:
  \begin{align*}
    \Omega^{*}(B_1v\otimes B_0v)\Omega
    &=\sum_{j=0}^{m-1}\Omega^{*}\left((\epsilon^{-(2j+1)}r_{1,2}^{2j+1}+\epsilon^{2j+1}r_{2,1}^{2j+1})\otimes ((e_1+e_3)-(e_2+e_4))\right)\Omega\\
    &+\sum_{j=0}^{m-1}\sum_{k=1}^{m-1} \Omega^{*}\left((\epsilon^{-(2j+1)}r_{1,2}^{2j+1}+\epsilon^{2j+1}r_{2,1}^{2j+1})\otimes (\epsilon^{-2k}e_{1,2}^{2k} +\epsilon^{2k}e_{2,1}^{2k})\right)\Omega
  \end{align*}
  \def\phantomdoublesum{\phantom{-\frac12\sum_{j=0}^{m-1}\sum_{k=1}^{m-1}}}
  \begin{align*}
    =&-\frac{\mathrm{i}}{2}\sum_{j=0}^{m-1}\left((\epsilon^{2j+1}e^{2j+1}_{1,1}+\epsilon^{-(2j+1)}e^{2j+1}_{2,2})-(\epsilon^{-(2j+1)}e^{2j+1}_{1,1}+\epsilon^{2j+1}e^{2j+1}_{2,2})\right)\otimes ((e_1+e_2)-(e_3+e_4)) \\
    &-\frac12\sum_{j=0}^{m-1}\left((\epsilon^{-(2j+1)}e_{1,2}^{2j+1}+\epsilon^{2j+1}e_{2,1}^{2j+1})+(\epsilon^{2j+1}e_{1,2}^{2j+1}+\epsilon^{-(2j+1)}e_{2,1}^{2j+1})\right)\otimes ((e_1+e_3)-(e_2+e_4))\\
    &+\frac{\mathrm{i}}{2}\sum_{j=0}^{m-1}\sum_{k=1}^{m-1} \epsilon^{-(2j+1)}e_{1,1}^{2j+1}\otimes \epsilon^{-2k}e_{1,1}^{2k}+\epsilon^{-(2j+1)}e_{1,1}^{2j+1}\otimes \epsilon^{2k}e_{2,2}^{2k} \\
    &\phantomdoublesum+\epsilon^{2j+1}e_{2,2}^{2j+1}\otimes \epsilon^{-2k}e_{1,1}^{2k}+\epsilon^{2j+1}e_{2,2}^{2j+1}\otimes \epsilon^{2k}e_{2,2}^{2k}\\
    &\phantomdoublesum- \epsilon^{-(2j+1)}e_{2,2}^{2j+1}\otimes \epsilon^{-2k}e_{2,2}^{2k}-\epsilon^{-(2j+1)}e_{2,2}^{2j+1}\otimes \epsilon^{2k}e_{1,1}^{2k}\\
    &\phantomdoublesum-\epsilon^{2j+1}e_{1,1}^{2j+1}\otimes \epsilon^{-2k}e_{2,2}^{2k}-\epsilon^{2j+1}e_{1,1}^{2j+1}\otimes \epsilon^{2k}e_{1,1}^{2k}\\
    &-\frac12\sum_{j=0}^{m-1}\sum_{k=1}^{m-1}
    \epsilon^{-(2j+1)}e_{1,2}^{2j+1}\otimes \epsilon^{-2k}e_{1,2}^{2k}+\epsilon^{-(2j+1)}e_{2,1}^{2j+1}\otimes \epsilon^{-2k}e_{2,1}^{2k}\\
    &\phantomdoublesum+\epsilon^{-(2j+1)}e_{1,2}^{2j+1}\otimes \epsilon^{2k}e_{2,1}^{2k}+\epsilon^{-(2j+1)}e_{2,1}^{2j+1}\otimes \epsilon^{2k}e_{1,2}^{2k}\\
    &\phantomdoublesum+\epsilon^{2j+1}e_{2,1}^{2j+1}\otimes\epsilon^{-2k}e_{1,2}^{2k}+\epsilon^{2j+1}e_{1,2}^{2j+1}\otimes\epsilon^{-2k}e_{2,1}^{2k}\\
    &\phantomdoublesum+\epsilon^{2j+1}e_{1,2}^{2j+1}\otimes\epsilon^{2k}e_{1,2}^{2k}+\epsilon^{2j+1}e_{2,1}^{2j+1}\otimes\epsilon^{2k}e_{2,1}^{2k}\\
    =&-\frac{\mathrm{i}}{2}\sum_{j=0}^{m-1}(\epsilon^{2j+1}e^{2j+1}_{1,1}+\epsilon^{-(2j+1)}e^{2j+1}_{2,2})\otimes \left((e_1+e_2)-(e_3+e_4))
    +\sum_{k=1}^{m-1} \epsilon^{2k}e_{1,1}^{2k}+\epsilon^{-2k}e_{2,2}^{2k}\right)\\
    &+\frac{\mathrm{i}}{2}\sum_{j=0}^{m-1}(\epsilon^{-(2j+1)}e^{2j+1}_{1,1}+\epsilon^{2j+1}e^{2j+1}_{2,2})
    \otimes \left((e_1+e_2)-(e_3+e_4))+\sum_{k=1}^{m-1}\epsilon^{-2k}e_{1,1}^{2k}+\epsilon^{2k}e_{2,2}^{2k}\right) \\
    &-\frac12\sum_{j=0}^{m-1}(\epsilon^{-(2j+1)}e_{1,2}^{2j+1}+\epsilon^{2j+1}e_{2,1}^{2j+1})\otimes \left((e_1+e_3)-(e_2+e_4)+\sum_{k=1}^{m-1} \epsilon^{-2k}e_{1,2}^{2k}+\epsilon^{2k}e_{2,1}^{2k}\right)\\
    &-\frac12\sum_{j=0}^{m-1}(\epsilon^{2j+1}e_{1,2}^{2j+1}+\epsilon^{-(2j+1)}e_{2,1}^{2j+1})\otimes \left((e_1+e_3)-(e_2+e_4)+\sum_{k=1}^{m-1} \epsilon^{2k}e_{1,2}^{2k}+\epsilon^{-2k}e_{2,1}^{2k}\right)\\
    =&-\frac{\mathrm{i}}{2}\left(\lambda(a)B_1\otimes \lambda(a)B_0-\lambda(a^{-1})B_1\otimes \lambda(a^{-1})B_0\right)\\
    &-\frac12\left(\lambda(ba)B_1\otimes \lambda(ba)B_0+\lambda(ba^{-1})B_1\otimes \lambda(ba^{-1})B_0\right)\,.
  \end{align*}
  The calculation of the coproduct of $w$ is analogous.
\end{proof}

\subsection{Structure of $K_{33}$ in $K_3$ of $\KD(4m')$}
\label{K33}

The Jones projection of the \sacig $K_{33}$ of $K_3$ in $\KD(4m')$ is
$p_{33}=e_1+e_3+r^m_{1,1}$. Furthermore,
\begin{align*}
p_{33}&=e_1+e_3+r^m_{1,1}\\
&=\frac{1}{2n} \sum_{k=0}^{n-1}\lambda(a^{2k})+\lambda(ba^{2k+1})+\frac{1}{2n}\sum_{k=0}^{n-1}(-1)^k\lambda(a^{2k})+\frac{1}{2n}\sum_{k=0}^{n-1}(-1)^k\lambda(ba^{2k+1})\\
&=\frac{1}{n} \sum_{k=0}^{m-1}\lambda(a^{4k})+\lambda(ba^{4k+1})
\end{align*}
Therefore $p_{33}$ is the Jones projection of the algebra $A$ of the
subgroup $D_{2m'}=\langle a^4, ba\rangle$ in $\C[D_{2n}]$. Using e.g.
Proposition~\ref{prop.resumep}, its standard coproduct can be
expressed in terms of the matrix units of $A$. This subalgebra admits
four central projections, associated to the factors of dimension $1$:
$\tilde{e}_1=p_{33}$, $\tilde{e}_2$, $\tilde{e}_3$, and $\tilde{e}_4$,
with
\begin{align*}
  \tilde{e}_2&=\frac{1}{2m} \sum_{k=0}^{m-1}(\lambda(a^{4k})-\lambda(ba^{4k+1}))\\
  &=e_2+e_4+\frac{1}{2m}\sum_{j=1}^{n-1}\sum_{k=0}^{m-1}(\epsilon_m^{2jk} e^j_{1,1}+\epsilon_m^{-2jk} e^j_{2,2})-(\epsilon_m^{-2jk} \epsilon_n^{-j}e^j_{1,2}+\epsilon_m^{2jk}\epsilon_n^{j} e^j_{2,1})\\
  &=e_2+e_4+r^m_{2,2}
\end{align*}
\begin{align*}
  \tilde{e}_3&=\frac{1}{2m} \sum_{k=0}^{m-1}(-1)^k(\lambda(a^{4k})-\lambda(ba^{4k+1}))\\
  &=\frac{1}{2m}\sum_{j=1}^{n-1}\sum_{k=0}^{m-1}(-1)^k(\epsilon_m^{2jk} e^j_{1,1}+\epsilon_m^{-2jk} e^j_{2,2})-(-1)^k(\epsilon_m^{-2jk} \epsilon_n^{-j}e^j_{1,2}+\epsilon_m^{2jk}\epsilon_n^{j} e^j_{2,1})\\
  &=\frac{1}{2m}\sum_{j=1}^{n-1}\sum_{k=0}^{m-1}(\epsilon_m^{2(m'+j)k} e^j_{1,1}+\epsilon_m^{-2(m'+j)k} e^j_{2,2})-(\epsilon_m^{-2(m'+j)k} \epsilon_n^{-j}e^j_{1,2}+\epsilon_m^{2(m'+j)k}\epsilon_n^{j} e^j_{2,1})\\
  &=\frac{1}{2}(e^{m'}_{1,1}+e^{m'}_{2,2}+\mathrm{e} ^{3i\pi/4}e^{m'}_{1,2}+\mathrm{e} ^{-3i\pi/4}e^{m'}_{2,1})+\frac{1}{2}(e^{3m'}_{1,1}+e^{3m'}_{2,2}+\mathrm{e} ^{i\pi/4}e^{3m'}_{1,2}+\mathrm{e} ^{-i\pi/4}e^{3m'}_{2,1})\\
  &=q(0,-3\pi/4,m')+q(0,-\pi/4,3m')
\end{align*}
\begin{align*}
  \tilde{e}_4&=\frac{1}{2m} \sum_{k=0}^{m-1}(-1)^k(\lambda(a^k)+\lambda(ba^{4k+1}))\\
  &=\frac{1}{2m}\sum_{j=1}^{n-1}\sum_{k=0}^{m-1}(-1)^k(\epsilon_m^{2jk} e^j_{1,1}+\epsilon_m^{-2jk} e^j_{2,2})+(-1)^k(\epsilon_m^{-2jk} \epsilon_n^{-j}e^j_{1,2}+\epsilon_m^{2jk}\epsilon_n^{j} e^j_{2,1})\\
  &=\frac{1}{2m}\sum_{j=1}^{n-1}\sum_{k=0}^{m-1}(\epsilon_m^{2(m'+j)k} e^j_{1,1}+\epsilon_m^{-2(m'+j)k} e^j_{2,2})+(\epsilon_m^{-2(m'+j)k} \epsilon_n^{-j}e^j_{1,2}+\epsilon_m^{2(m'+j)k}\epsilon_n^{j} e^j_{2,1})\\
  &=\frac{1}{2}(e^{m'}_{1,1}+e^{m'}_{2,2}+\mathrm{e} ^{-i\pi/4}e^{m'}_{1,2}+\mathrm{e} ^{i\pi/4}e^{m'}_{2,1})+\frac{1}{2}(e^{3m'}_{1,1}+e^{3m'}_{2,2}+\mathrm{e} ^{-3i\pi/4}e^{3m'}_{1,2}+\mathrm{e} ^{3i\pi/4}e^{3m'}_{2,1})\\
  &=q(0,\pi/4,m')+q(0,3\pi/4,3m')
\end{align*}
as well as $m'-1$ blocks of dimension $2$ whose matrix units are
derived from the partial isometries $\tilde{e}^h_{1,2}$ with $h=1,
\dots, m'-1$:
\begin{align*}
  \tilde{e}^h_{1,2}&=\frac{1}{m} \sum_{k=0}^{m-1}\mathrm{e}^{ihk\pi/m'}\lambda(ba^{4k+1})\\
  &=\frac{1}{m}\sum_{j=1}^{n-1} \sum_{k=0}^{m-1}\epsilon^{2hk}_{m}(\epsilon_m^{-2jk} \epsilon_n^{-j}e^j_{1,2}+\epsilon_m^{2jk}\epsilon_n^{j} e^j_{2,1})\\
  &=\frac{1}{m}\sum_{j=1}^{n-1} \sum_{k=0}^{m-1}(\epsilon_m^{2(h-j)k} \epsilon_n^{-j}e^j_{1,2}+\epsilon_m^{2(h+j)k}\epsilon_n^{j} e^j_{2,1})\\
  &=\epsilon_n^{-h}e^h_{1,2}+\epsilon_n^{-(m+h)}e^{m+h}_{1,2}
  +\epsilon_n^{(m-h)}e^{m-h}_{2,1}+\epsilon_n^{(n-h)}e^{n-h}_{2,1}\\
  &=\epsilon_n^{-h}(e^h_{1,2}-e^{n-h}_{2,1})+\epsilon_n^{-(m+h)}(e^{m+h}_{1,2}
  -e^{n-(m+h)}_{2,1})\\
  &=\epsilon_n^{-h}\left((e^h_{1,2}-e^{n-h}_{2,1})-i(e^{m+h}_{1,2}
  -e^{n-(m+h)}_{2,1})\right)\,.
\end{align*}
Upon deformation by $\Omega$, the right legs $\tilde{e}_1$,
$\tilde{e}_2$, $\tilde{e}_3$, $\tilde{e}_4$, and $\tilde{e}^h_{i,j}$
for $h$ even of $\Delta_s(p_{33})$ are left unchanged; the remaining
rights legs $\tilde{e}^h_{i,j}$ for $h$ odd become
$\epsilon_n^{-h}[(r^h_{1,2}-r^{n-h}_{2,1})-i(r^{m+h}_{1,2}-r^{n-(m+h)}_{2,1})]$. Therefore,
$K_{33}$ admits exactly four blocs of dimension $1$, as is the case
for $A$.

\subsection{Calculations for Theorem~\ref{self-dual}}
\label{psicoproduit}

We complete the demonstration of the self-duality of $\KD(n)$ for $n$
odd (Theorem~\ref{self-dual}) by calculating explicitly
$\psi(\lambda(a))$ and $\psi(\lambda(b))$ and checking that their
coproducts are preserved by $\psi$. The calculations are
straightforward albeit lengthy; regrettably, we could not delegate
them to the computer in the general case.

\subsubsection{Preservation of the coproduct of $\lambda(b)$}
\begin{lemma}
$$\psi(\lambda(b))=4n\widehat{e_4} \qquad \text{and} \qquad \Delta(\psi(\lambda(b)))=(\psi\otimes\psi)(\Delta(\lambda(b))\,.$$
\end{lemma}
\begin{proof}
  From~\ref{lambda} and~\ref{q}, one can write:
  \begin{align*}
    \lambda(b)&=e_1-e_2-e_3+e_4+\sum_{j=1}^{n-1}e^j_{1,2}+e^j_{2,1}\\
    &=e_1-e_2-e_3+e_4+\sum_{k=1,\, k \text{ odd}}^{n-1}-(r^k_{1,2}+r^k_{2,1})+e^{n-k}_{1,2}+e^{n-k}_{2,1}\,.
  \end{align*}
  It follows that
  \begin{align*}
    \psi(\lambda(b))=&\,\chi_1-\chi_{a^n}-\chi_{ba^n}+\chi_{b}+\sum_{k=1,\, k \text{ odd}}^{n-1}-(E^k_{1,2}+E^k_{2,1})+E^{n-k}_{1,2}+E^{n-k}_{2,1}\\
    =&\,\chi_1-\chi_{a^n}-\chi_{ba^n}+\chi_{b}\\
    &+\sum_{k=1,\, k \text{ odd}}^{n-1}\chi_{ba^{n-k}}+\chi_{ba^{n+k}}
    -\chi_{ba^k}-\chi_{ba^{2n-k}}-\chi_{a^k}-\chi_{a^{2n-k}}+\chi_{a^{n+k}}+\chi_{a^{n-k}}\\
    =&\sum_{j=0}^{2n-1} (-1)^j(\chi_{a^j}+\chi_{ba^j})=4n\widehat{e_4}\,.
  \end{align*}
  Then,
  \begin{align*}
    &(\psi\otimes\psi)(\Delta(\lambda(b))=16n^2\widehat{e_4}\otimes \widehat{e_4}\\
    =&\sum_{j=0}^{2n-1}\sum_{k=0}^{2n-1} (-1)^{k+j}(\chi_{a^j}+\chi_{ba^j})\otimes  (\chi_{a^k}+\chi_{ba^k})\\
    =&\sum_{j=0}^{2n-1}\sum_{k=0}^{2n-1} (-1)^{k+j}\left(\chi_{a^j}\otimes  (\chi_{a^{-j}a^{j+k}}+\chi_{a^{-j}ba^{k-j}})+ \chi_{ba^j}\otimes  (\chi_{(a^{-j}b)ba^{k+j}}+\chi_{(a^{-j}b)a^{k-j}})\right)\,;
  \end{align*}
  and since $j+k$ and $j-k$ have the same parity:
  \begin{align*}
    (\psi\otimes\psi)(\Delta(\lambda(b))&=\sum_{j=0}^{2n-1}\sum_{s\in
      D_{2n}} (-1)^{j}\chi_{s}\otimes
    (\chi_{s^{-1}a^{j}}+\chi_{s^{-1}ba^{j}})=\Delta(4n\widehat{e_4}) =
    \Delta(\psi(\lambda(b)))\,.\qedhere
  \end{align*}
\end{proof}

\subsubsection{Preservation of the coproduct of $\lambda(a)$}

\begin{lemma}
  \begin{displaymath}
    \psi(\lambda(a))=n(2\widehat{e^{n-1}_{1,1}}+ \widehat{e^1_{2,2}}- (\widehat{e^1_{1,1}}+\widehat{e^{n-1}_{1,2}}+\widehat{e^{n-1}_{2,1}}))\quad \text{and} \quad \Delta(\psi(\lambda(a)))=(\psi\otimes\psi)(\Delta(\lambda(a))\,.
  \end{displaymath}
\end{lemma}

We prove this lemma in the following subsections. In order to split
the calculations, using~\ref{lambda} and
$\epsilon=\mathrm{e}^{\mathrm{i}\pi/n}$, we write: $\lambda(a)=U+V$,
with
$$U=e_1+e_2+\sum_{k=1,\, k \text{ even}}^{n-1} \epsilon^k e^k_{1,1}+ \epsilon^{-k} e^k_{2,2}\quad \text {and } \quad V=-e_3-e_4+\sum_{k=1,\,k \text{ odd}}^{n-1} \epsilon^k e^k_{1,1}+ \epsilon^{-k} e^k_{2,2}\,.$$.

Hence, since all the terms of $U-(e_1+e_2)$ (resp. $V+(e_3+e_4)$) are
in $M_2(\C)$ factors of same parity, we may use
Remark~\ref{remark.omega} to calculate the coproducts of $U$ and $V$.

\subsubsection{Expression of $\psi(U)$ and $\psi(V)$}\label{lambdaa}
  \begin{lem}
    $$\psi(U)=2n\widehat{e^{n-1}_{1,1}},\qquad\psi(V)=n (\widehat{e^1_{2,2}}-\widehat{e^1_{1,1}}-\widehat{e^{n-1}_{1,2}}-\widehat{e^{n-1}_{2,1}})\,.$$
  \end{lem}
  \begin{proof}
    On the one hand,
    \begin{align*}
      \psi(U)&=\chi_1+\chi_{a^n}+\sum_{k=1,\, k \text{ odd}}^{n-1}\epsilon^{n-k} E^{n-k}_{1,1}+ \epsilon^{-(n-k)} E^{n-k}_{2,2}\\
      &= \chi_1+\chi_{a^n}+\sum_{j=1,\, j \text{ even}}^{n-1}\epsilon^j (\chi_{a^j}+\chi_{a^{n+j}})+ \epsilon^{-j} (\chi_{a^{n-j}}+\chi_{a^{2n-j}})\\
      &=\sum_{j=0}^{2n-1}(-1)^j\epsilon^j\chi_{a^j}
      =\sum_{j=0}^{2n-1}\epsilon^{-(n-1)j}\chi_{a^j}=2n\widehat{e^{n-1}_{1,1}}\,.
    \end{align*}
    On the other hand,
    \begin{align*}
      (V+V^*)&=2(-e_3-e_4)+\sum_{k=1,\, k \text{ odd}}^{n-1}
      (\epsilon^k+\epsilon^{-k}) (r^k_{1,1}+  r^k_{2,2})\\
      (V-V^*)&=\rm{i}\sum_{k=1,\, k \text{ odd}}^{n-1}  (\epsilon^k-\epsilon^{-k})(r^k_{2,1}-r^k_{1,2})\,.
    \end{align*}
    It follows that:
    \begin{align*}
      \psi(V+V^*)&=2(-\chi_{ba^n}-\chi_{b})+\sum_{k=1,\, k \text{ odd}}^{n-1}(\epsilon^k + \epsilon^{-k}) (\chi_{ba^{n+k}}+\chi_{ba^{k}}+\chi_{ba^{2n-k}}+\chi_{ba^{n-k}})\\
      &=- \sum_{j=0}^{2n-1}(\epsilon^{(n-1)j} + \epsilon^{-(n-1)j}) \chi_{ba^{j}}
      =-2n (\widehat{e^{n-1}_{1,2}}+\widehat{e^{n-1}_{2,1}})\,,\\
      \psi(V-V^*)&=\sum_{k=1,\, k \text{ odd}}^{n-1}(\epsilon^k - \epsilon^{-k}) (\chi_{a^{k}}+\chi_{a^{n-k}}-\chi_{a^{n+k}}-\chi_{a^{2n-k}})\\
      &=\sum_{j=0}^{2n-1}(\epsilon^j - \epsilon^{-j}) \chi_{a^{j}}
      =2n (\widehat{e^1_{2,2}}-\widehat{e^1_{1,1}})\,,
    \end{align*}
    which gives the desired expression for $\psi(V)$. The expressions
    for $\psi(V+V^*)$ and $\psi(V-V^*)$ will be reused later on.
  \end{proof}
  \subsubsection{Coproducts of $\psi(U)$ and $\psi(V)$}\label{UV}
  \begin{lem}
    $$\Delta(\psi(U)) =4n^2(\widehat{e^{n-1}_{1,1}}\otimes \widehat{e^{n-1}_{1,1}}+\widehat{e^{n-1}_{1,2}}\otimes \widehat{e^{n-1}_{2,1}})\,.$$

    \begin{multline*}
      \Delta(\psi(V))=2n^2(\widehat{e^1_{2,2}} \otimes\widehat{e^1_{2,2}}
      +\widehat{e^1_{2,1}} \otimes \widehat{e^1_{1,2}}
      -\widehat{e^1_{1,1}} \otimes\widehat{e^1_{1,1}}
      -\widehat{e^1_{1,2}} \otimes \widehat{e^1_{2,1}}\\
      -\widehat{e^{n-1}_{2,2}} \otimes \widehat{e^{n-1}_{2,1}}
      -\widehat{e^{n-1}_{2,1}} \otimes \widehat{e^{n-1}_{1,1}}
      -\widehat{e^{n-1}_{1,1}} \otimes \widehat{e^{n-1}_{1,2}}
      -\widehat{e^{n-1}_{1,2}} \otimes \widehat{e^{n-1}_{2,2}})
    \end{multline*}
  \end{lem}
  \begin{proof}
    \begin{align*}
      \Delta(\psi(U))&=\sum_{k=0}^{2n-1}\sum_{j=0}^{2n-1}\epsilon^{-(n-1)j}(\chi_{a^k}\otimes\chi_{a^{j-k}}+\chi_{a^k b}\otimes\chi_{ba^{j-k}})\\
      &=\sum_{k=0}^{2n-1}\sum_{j=0}^{2n-1}\epsilon^{-(n-1)(k+j)}(\chi_{a^k}\otimes\chi_{a^{j}}+\chi_{a^kb}\otimes\chi_{ba^{j}})\\
      &=4n^2(\widehat{e^{n-1}_{1,1}}\otimes \widehat{e^{n-1}_{1,1}}+\widehat{e^{n-1}_{1,2}}\otimes \widehat{e^{n-1}_{2,1}})\,.
    \end{align*}
    Similar calculations give:
    \begin{align*}
      \Delta(2n \widehat{e^1_{2,2}}) &=  4n^2(\widehat{e^1_{2,2}} \otimes\widehat{e^1_{2,2}}
      +\widehat{e^1_{2,1}} \otimes \widehat{e^1_{1,2}})\,,\\
      \Delta(2n \widehat{e^{n-1}_{1,2}}) &=
      4n^2(\widehat{e^{n-1}_{2,2}} \otimes \widehat{e^{n-1}_{2,1}}
      +\widehat{e^{n-1}_{1,2}} \otimes \widehat{e^{n-1}_{1,1}})\,.
    \end{align*}
    The announced formula follows using the equalities
    $\widehat{e^j_{2,2}}=\widehat{e^j_{1,1}}^*$ and
    $\widehat{e^j_{2,1}}=\widehat{e^j_{1,2}}^*$.
  \end{proof}

\subsubsection{Image of $\Delta(U)$ by $\psi\otimes \psi$}
  We now turn to the preservation by $\psi$ of the coproducts of $U$
  and $V$, and therefore of $\lambda(a)=U+V$.  Since
  $\lambda(a^{n+1})=U-V$, we can calculate the coproducts of
  $U=1/2(\lambda(a)+\lambda(a^{n+1}))$ and
  $V=1/2(\lambda(a)-\lambda(a^{n+1}))$ using the expression for
  $\Omega$ for $n$ odd given in~\ref{omega}.

  Let us start with $U$:
  \begin{align*}
    &\Delta(U)=\Omega(U\otimes U+V\otimes V)\Omega^*\\
    &=(e_1+e_2+q_1+q_2)\otimes(e_1+e_2+q_1+q_2)(U\otimes U) (e_1+e_2+q_1+q_2)\otimes(e_1+e_2+q_1+q_2) \\
    &+(e_3-ie_4-ip_1+p_2)\otimes(e_3+ie_4+ip_1+p_2)(V\otimes V) (e_3+ie_4+ip_1+p_2)\otimes (e_3-ie_4-ip_1+p_2)\\
    &=U\otimes U + V'\otimes V'^*\,,
  \end{align*}
  where $V'=(-e_3-e_4+\sum_{k=1,\, k \text{ odd}}^{n-1} \epsilon^k
  r^k_{2,2}+ \epsilon^{-k} r^k_{1,1})$. The image of $V'$ by $\psi$ is:
  \begin{displaymath}
    \psi(V')=-\chi_{ba^n}-\chi_{b}+\sum_{k=1,\, k \text{ odd}}^{n-1}\epsilon^k E^k_{2,2}+ \epsilon^{-k} E^k_{1,1}
    =-\sum_{j=0}^{2n-1}\epsilon^{(n-1)j}  \chi_{ba^{j}}
    =-2n\widehat{e^{n-1}_{1,2}}\,.
  \end{displaymath}
  It follows that $\psi(V'^*)=-2n\widehat{e^{n-1}_{2,1}}$, and we get:
  \begin{displaymath}
    (\psi\otimes \psi)(\Delta(U))=4n^2(\widehat{e^{n-1}_{1,1}}\otimes \widehat{e^{n-1}_{1,1}}+\widehat{e^{n-1}_{1,2}}\otimes \widehat{e^{n-1}_{2,1}})\,.
  \end{displaymath}
  Using Lemma~\ref{UV}, we conclude that $\psi$ preserves the
  coproduct of $U$.

\subsubsection{Image of $\Delta(V)$  by $\psi\otimes \psi$}

  We now calculate the image of $\Delta(V)$ by $\psi\otimes \psi$:
  \begin{align*}
    \Delta(V)&=\Omega(U\otimes V+V\otimes U)\Omega^* \,,\\
    &=U_1\otimes V_1+U_2\otimes V_2+i(U_3\otimes V_3-U_4\otimes V_4)\\
    &+V_1\otimes U_1+V_2\otimes U_2-i(V_3\otimes U_3-V_4\otimes U_4)\,,
  \end{align*}
  where:
  \begin{align*}
    U_1&=(e_1+q_1)U(e_1+q_1)=e_1+\sum_{k=1,\, k \text{ even}}^{n-1}
    \frac{1}{2}(\epsilon^k + \epsilon^{-k})p^k_{1,1}\,,\\
    V_1&=(e_3+e_4+p_1+p_2)V(e_3+e_4+p_1+p_2)=V\,,\\
    U_2&= (e_2+q_2)U (e_2+q_2)=e_2+\sum_{k=1,\, k \text{ even}}^{n-1}
    \frac{1}{2}(\epsilon^k +
    \epsilon^{-k})p^k_{2,2}\,,\\
    V_2&=(e_3-e_4-p_1+p_2)V(e_3-e_4-p_1+p_2)=V^*\,,\\
  \end{align*}
  \begin{align*}
    U_3&=(e_2+q_2)U(e_1+q_1)=\sum_{k=1,\, k \text{ even}}^{n-1} \frac{1}{4}(\epsilon^k - \epsilon^{-k})(e^k_{1,1}+e^k_{1,2}-e^k_{2,1}-e^k_{2,2})\,,\\
    V_3&=(e_3-e_4-p_1+p_2)V(e_3+e_4+p_1+p_2)=-e_3+e_4-\sum_{k=1,\, k \text{ odd}}^{n-1} \epsilon^{-k} e^{k}_{1,2}+ \epsilon^{k} e^k_{2,1}\,\\
    U_4&=(e_1+q_1)U(e_2+q_2)=\sum_{k=1,\, k \text{ even}}^{n-1}
    \frac{1}{4}(\epsilon^k -
    \epsilon^{-k})(e^k_{1,1}-e^k_{1,2}+e^k_{2,1}-e^k_{2,2})\,,\\
    V_4&=(e_3+e_4+p_1+p_2)V(e_3-e_4-p_1+p_2)=-e_3+e_4-\sum_{k=1,\, k \text{ odd}}^{n-1} \epsilon^k e^{k}_{1,2}+ \epsilon^{-k} e^k_{2,1}\,.
  \end{align*}
  To split the calculation into more manageable chunks, we write:
  \begin{align*}
    2[U_1\otimes V_1&+U_2\otimes V_2+i(U_3\otimes V_3-U_4\otimes V_4)]=\\
    &(U_1+U_2)\otimes (V_1+V_2)+(U_1-U_2)\otimes
    (V_1-V_2)\\
    &+\rm{i}[(U_3-U_4)\otimes (V_3+V_4)+(U_3+U_4)\otimes
    (V_3-V_4)]\,.
  \end{align*}

  \begin{lem}
    \begin{align*}
      \psi(U_1+U_2)&=        n(\widehat{e^{n-1}_{1,1}}+\widehat{e^{n-1}_{2,2}})\,,&
      \psi(U_1-U_2)&=        n(\widehat{e^{1}_{1,1}}  +\widehat{e^{1}_{2,2}})\,,\\
      \psi(U_3+U_4)&=        n(\widehat{e^{n-1}_{1,1}}-\widehat{e^{n-1}_{2,2}})\,,&
      \psi(U_3-U_4)&= {\rm i}n(\widehat{e^{1}_{1,2}}  -\widehat{e^{1}_{2,1}})\,,\\
      \psi(V_1+V_2)&=      -2n(\widehat{e^{n-1}_{1,2}}-\widehat{e^{n-1}_{2,1}})\,,&
      \psi(V_1-V_2)&=       2n(\widehat{e^1_{2,2}}    -\widehat{e^1_{1,1}})\,,\\
      \psi(V_3+V_4)&=       2n(\widehat{e^1_{1,2}}    +\widehat{e^1_{2,1}})\,,&
      \psi(V_3-V_4)&=2{\rm i}n(\widehat{e^{n-1}_{1,2}}-\widehat{e^{n-1}_{2,1}})\,.
    \end{align*}
  \end{lem}
  \begin{proof}
    \begin{align*}
      \psi(U_1+U_2)&=\psi(e_1+e_2+\sum_{k=1,\, k \text{ even}}^{n-1} \frac{1}{2}(\epsilon^k + \epsilon^{-k})(e^k_{1,1}+e^k_{2,2}))\\
      &=\chi_1+\chi_{a^n}+\sum_{k=1,\, k \text{ even}}^{n-1} \frac{1}{2}(\epsilon^k + \epsilon^{-k})(\chi_{a^{n-k}}+\chi_{a^{2n-k}}+\chi_{a^k}+\chi_{a^{n+k}})=n(\widehat{e^{n-1}_{1,1}}+\widehat{e^{n-1}_{2,2}})\,,
      \\
      \psi(U_1-U_2)&=\psi(e_1-e_2+\sum_{k=1,\, k \text{ even}}^{n-1} \frac{1}{2}(\epsilon^k + \epsilon^{-k})(e^k_{1,2}+e^k_{2,1}))\\
      &=\chi_1-\chi_{a^n}+\sum_{k=1,\, k \text{ even}}^{n-1}
      \frac{1}{2}(\epsilon^k +
      \epsilon^{-k})(-\chi_{a^{n-k}}+\chi_{a^{2n-k}}+\chi_{a^k}-\chi_{a^{n+k}})=n(\widehat{e^{1}_{1,1}}+\widehat{e^{1}_{2,2}})\,,
    \end{align*}

    \begin{align*}\psi(U_3+U_4)&=\psi(\sum_{k=1,\, k \text{ even}}^{n-1} \frac{1}{2}(\epsilon^k - \epsilon^{-k})(e^k_{1,1}-e^k_{2,2}))\\
      &=\sum_{k=1,\, k \text{ even}}^{n-1} \frac{1}{2}(\epsilon^k - \epsilon^{-k})(-\chi_{a^{n-k}}-\chi_{a^{2n-k}}+\chi_{a^k}+\chi_{a^{n+k}})=n(\widehat{e^{n-1}_{1,1}}-\widehat{e^{n-1}_{2,2}})\\
      \psi(U_3-U_4)&=\psi(\sum_{k=1,\, k \text{ even}}^{n-1} \frac{1}{2}(\epsilon^k - \epsilon^{-k})(e^k_{1,2}-e^k_{2,1}))\,,\\
      &={\rm i}\sum_{k=1,\, k \text{ even}}^{n-1} \frac{1}{2}(\epsilon^k - \epsilon^{-k})(\chi_{ba^k}-\chi_{ba^{2n-k}}-\chi_{ba^{n+k}}+\chi_{ba^{n-k}})={\rm i} n(\widehat{e^{1}_{1,2}}-\widehat{e^{1}_{2,1}})\,.
    \end{align*}
    The expressions of $\psi(V_1+V_2)$ and $\psi(V_1-V_2)$ follow from
    Lemma~\ref{lambdaa}. Similarly, one calculates:
     \begin{align*}
      \psi(V_3+V_4)&=\psi(2(-e_3-e_4)+\sum_{k=1,\, k \text{ odd}}^{n-1} (\epsilon^k+\epsilon^{-k}) (r^k_{1,2}+  r^k_{2,1}))\\
      &=-2\chi_{ba^n}+2\chi_b+\sum_{k=1,\, k \text{ odd}}^{n-1} (\epsilon^k+\epsilon^{-k}) (\chi_{ba^k}+\chi_{ba^{2n-k}}-\chi_{ba^{n+k}}-\chi_{ba^{n-k}})\\
      &=2n(\widehat{e^1_{1,2}}+\widehat{e^1_{2,1}})\,,\end{align*}
\begin{align*}
      \psi(V_3-V_4)&=\psi(\rm{i}\sum_{k=1,\, k \text{ odd}}^{n-1}  (\epsilon^k-\epsilon^{-k})(r^k_{1,1}-r^k_{2,2}))\\
      &={\rm i}\sum_{k=1,\, k \text{ odd}}^{n-1}  (\epsilon^k-\epsilon^{-k})(\chi_{ba^{n+k}}+\chi_{ba^{k}}-\chi_{ba^{2n-k}}-\chi_{ba^{n-k}}) \\
      &= 2{\rm i} n (\widehat{e^{n-1}_{1,2}}-\widehat{e^{n-1}_{2,1}})\,,
    \end{align*}
    which concludes the proof of this lemma.
  \end{proof}

  We can at last check that the coproduct of $V$ is preserved:
  \begin{align*}
    n^{-2}(\psi\otimes \psi)(\Delta(V))&=
    -(\widehat{e^{n-1}_{1,1}}+\widehat{e^{n-1}_{2,2}})\otimes(\widehat{e^{n-1}_{1,2}}+\widehat{e^{n-1}_{2,1}})
    +(\widehat{e^{1}_{1,1}}+\widehat{e^{1}_{2,2}})\otimes (\widehat{e^1_{2,2}}-\widehat{e^1_{1,1}})\\
    &-(\widehat{e^{n-1}_{1,2}}+\widehat{e^{n-1}_{2,1}})\otimes(\widehat{e^{n-1}_{1,1}}+\widehat{e^{n-1}_{2,2}})
    +(\widehat{e^1_{2,2}}-\widehat{e^1_{1,1}})\otimes (\widehat{e^{1}_{1,1}}+\widehat{e^{1}_{2,2}})
    \\
    &-(\widehat{e^{1}_{1,2}}-\widehat{e^{1}_{2,1}})\otimes (\widehat{e^1_{1,2}}+\widehat{e^1_{2,1}})
    -(\widehat{e^{n-1}_{1,1}}-\widehat{e^{n-1}_{2,2}})\otimes (\widehat{e^{n-1}_{1,2}}-\widehat{e^{n-1}_{2,1}})\\
    &+(\widehat{e^1_{1,2}}+\widehat{e^1_{2,1}})\otimes(\widehat{e^{1}_{1,2}}-\widehat{e^{1}_{2,1}})+(\widehat{e^{n-1}_{1,2}}-\widehat{e^{n-1}_{2,1}})\otimes(\widehat{e^{n-1}_{1,1}}-\widehat{e^{n-1}_{2,2}})\\
    &=n^{-2}\Delta(\psi(V))\,.\qedhere
  \end{align*}

\newpage
\section{Computer exploration with \mupadcombinat{}}
\label{section.computerExploration}

Most of the research we report on in this paper has been driven by
computer exploration. In this section, we quickly describe the tools
we designed, implemented, and used, present typical computations, and
discuss some exploration strategies. To this end, we use the
construction of all \sacigs{} of $\KD(3)$ as running example. We
recommend to start by skimming through the demonstration, in order to
get a rough idea of what computations are, or not, achievable.

\subsection{Software context}

Our work is based on \mupadcombinat~\cite{MuPAD-Combinat}, an
open-source algebraic combinatorics package for the computer algebra
system \mupad~\cite{MuPAD.96}. Among other things, it provides a
high-level framework for implementing Hopf algebras and the like. All
the extensions we wrote for this research are publicly available from
the developers repository (see \url{http://mupad-combinat.sf.net/});
in fact, the two authors used that mean to share the code between
them. With it, new finite dimensional Kac algebras obtained by
deformation of group algebras may be implemented concisely (the full
implementation of the Kac-Paljutkin algebra takes about 50 lines of
code, including comments). Most of the code is fairly generic and
already integrated in the \mupadcombinat{} core, and was beneficial to
several unrelated research projects. The remaining specific code is
provided as a separate worksheet. Feel free to contact the second
author for help.

Since then, the \mupadcombinat project was migrated to the completely
open-source platform \texttt{Sage}~\cite{Sage,Sage-Combinat}. Specific
extensions, like this one, are migrated according to their usefulness
for new research projects.

\subsection{Setup}

The first step is to start a new \mupad{} session, and to setup the stage
for the computations (think of it as the preliminaries section of a
research paper which defines short hand notations). We load the
\mupadcombinat{} package by issuing:
\begin{Mexin}
package("Combinat")
\end{Mexin}
Next we load a worksheet which contains code and short-hand notations
specific to that specific research project. For the user convenience,
a short help will be displayed.
\begin{Mexin}
read("experimental/2005-09-08-David.mu"):
\end{Mexin}
\begin{Mexout}
//////////////////////////////////////////////////////////////////////
Loading worksheet: Twisted Kac algebras
Cf.  p. 715 of '2-cocycles and twisting of Kac algebras'

Version: $Id: KD.tex 433 2010-12-06 15:00:19Z nthiery $
To update to the latest version, go to the MuPAD-Combinat directory and type:
        svn update -d
Content:
  G := DihedralGroup(4)                 -- the dihedral group

  TwistedDihedralGroupAlgebra:
    KD4 := TwistedDihedralGroupAlgebra(4):
    KD4 := KD(4):                       -- shortcut
    KD4::G = KD4::group                 -- KD4 expressed on group elements
    KD4::G([3,1])                       -- a^3 b
    KD4::M = KD4::matrix                -- KD4 expressed as block diagonal matrices
    KD4::G::tensorSquare                -- the tensor product KD4::G # KD4::G
    KD4::M::tensorSquare                -- the tensor product KD4::M # KD4::M

    KD4::coeffRing                      -- the coefficient field
    KD4::coeffRing::primitiveUnitRoot(4)-- the complex value I

    KD4::M(x), KD4::G(x)                -- conversions between bases
    KD4::e(1), KD4::e(2,2,1)            -- matrix units
    KD4::p(2,2,j), KD4::r(2,2,j)        -- some projections of the j-th block
    KD4::p1, KD4::p2, KD4::q1, KD4::q1  -- some projections
    KD4::G::Omega                       -- Omega in the group basis
    KD4::M::tensorSquare(KD4::G::Omega) -- Omega in the matrix basis
    KD4::M::coproductAsMatrix(e(1))     -- the coproduct of e(1) as a matrix

    // Short hands, e.g. to write e(2,2,1) instead of KD4::e(2,2,1)
    export(KD4, Alias, e, p1, p2, q1, q2):

  TwistedQuaternionGroupAlgebra(N)
    KQ4 := TwistedDihedralGroupAlgebra(4):
    KQ4 := KD(4):                     -- shortcut
    Same usage as for KD(N)

  Isomorphism KD(2N) <-> KQ(2N):
    KQ4::G(KD4::G([1,0])):            -- The image of a of KD4 in KQ4
    KD4::G(KQ4::G([0,1])):            -- The image of b of KQ4 in KD4

  Computing with coideal subalgebras:
    algebraClosure([a,b,c])           -- A basis of the subalgebra generated by a,b,c
    coidealClosure([a,b,c])           -- A basis of the coideal generated by a,b,c
    coidealAndAlgebraClosure([a,b,c]) -- A basis of the coideal subalgebra ...
    echelonForm([a,b,c], Reduced)

    SkewTensorProduct(A, B) -- Skew tensor product of A and B (A being the dual of B)
    coidealDual([ p ])      -- Basis of the dual of the left coideal generated by p

A sample computation:
  M := KQ(4):
  Fbasis := coidealAndAlgebraClosure([M::e(1) + M::e(2)]):
  F := Dom::SubFreeModule(Fbasis,
                          [Cat::FiniteDimensionalHopfAlgebraWithBasis(M::coeffRing)]):
  Fdual := Dom::DualOfFreeModule(F):
  G := Fdual::intrinsicGroup():
  G::list()

//////////////////////////////////////////////////////////////////////
\end{Mexout}
Mind that this worksheet is experimental; for further help one needs
to dig into the code.  On the other hand, all features that are
integrated into \mupadcombinat{} or \mupad{} are documented within the
usual \mupad{} help system.

\subsection{Computing with elements}

Let us define $\KD(3)$:
\begin{Mexin}
KD3 := KD(3):
\end{Mexin}
and shortcuts to its generators:
\begin{Mexin}
[aa,bb] := KD3::group::algebraGenerators::list()
\end{Mexin}
\begin{Mexout}
                         [B(a), B(b)]
\end{Mexout}
Now we can use \mupad{} as a pocket calculator:
\begin{Mexin}
bb^2
\end{Mexin}
\begin{Mexout}
                             B(1)
\end{Mexout}
The point is that \mupad{} knows that \texttt{bb} lies in $\KD(3)$
(more precisely, the object \texttt{bb} is in the domain \texttt{KD3::group}\footnote{In
  \mupad{} parlance, concrete classes are called \emph{domains}}), and therefore applies the corresponding
computation rules (usual object oriented programming paradigm).  Here
are some further simple computations:
\begin{Mexin}
aa^2, aa^6, bb*aa
\end{Mexin}
\begin{Mexout}
                        2            5
                     B(a ), B(1), B(a  b)
\end{Mexout}
and a more complicated one:
\begin{Mexin}
(1 - aa^3)*(1 + aa^3) + 1/2*bb*aa^3
\end{Mexin}
\begin{Mexout}
                                 3
                          1/2 B(a  b)
\end{Mexout}
Note that all computations above are done in the group algebra. Namely,
\texttt{KD3::group} (or \texttt{KD3::G}) models the concrete algebra
$\KD(3)$ with its elements expanded on the group basis. However,
$\KD(3)$ can also be represented as a block-matrix algebra, with matrix
units as basis, and it is often more convenient or efficient to do the
computations there.  This basis is modeled by the domain
\texttt{KD3::matrix} (or \texttt{KD3::M} for short), and the change of
basis is done in the natural way:
\begin{Mexin}
KD3::M(aa + 2*bb)
\end{Mexin}
\begin{Mexoutsmall}
                +-                                                           -+
                |  3,  0,  0, 0,    0,         0,           0,          0     |
                |                                                             |
                |  0, -1,  0, 0,    0,         0,           0,          0     |
                |                                                             |
                |  0,  0, -3, 0,    0,         0,           0,          0     |
                |                                                             |
                |  0,  0,  0, 1,    0,         0,           0,          0     |
                |                                                             |
                |  0,  0,  0, 0, epsilon,      2,           0,          0     |
                |                                                             |
                |  0,  0,  0, 0,    2,    1 - epsilon,      0,          0     |
                |                                                             |
                |  0,  0,  0, 0,    0,         0,      epsilon - 1,     2     |
                |                                                             |
                |  0,  0,  0, 0,    0,         0,           2,      -epsilon  |
                +-                                                           -+
\end{Mexoutsmall}

Some comments are in order:
\begin{itemize}
\item An element of \texttt{KD::M} is displayed as a single large
  matrix; however, the four $1\times 1$ blocks and the three
  $2\times2$ blocks inside are well visible in the example above.
\item So far, we have not specified the ground field. It must be of
  characteristic zero, and contain some roots of unity to define
  $\Omega$ (see~\ref{omega}) and the left regular representation
  (see~\ref{lambda}). In theory one can just take $\C$, but in
  practice one needs a computable field. By default, an appropriate
  algebraic extension of $\Q$ is automatically constructed:
\begin{Mexin}
KD3::coeffRing
\end{Mexin}
\vspace{-1ex}
\begin{Mexout}
                        Q(II, epsilon)
\end{Mexout}
  where $\texttt{II}^4=1$, and $\texttt{epsilon}^6=1$.
\item The basis change is implemented by specifying the images of
  \texttt{a} and \texttt{b} and stating that it is an algebra
  morphism. The inverse basis change is deduced automatically by
  matrix inversion. Appropriate caching is done to avoid computation
  overhead. This is completely transparent to the user, and mostly
  transparent for the developer (encapsulation principle).
\item We show here the \mupad{} output in the text interface. In the
  graphical interface things look better; in particular,
  \texttt{epsilon} could be typeset as $\epsilon$.
\end{itemize}

So far, we have only played with the algebra structure of $\KD(3)$
which is just the usual group algebra structure. Let us compute some
coproducts, starting with some group like elements (note: tensor
products are denoted by the symbol \#):
\begin{Mexin}
coproduct(aa^3), coproduct(bb)
\end{Mexin}
\begin{Mexout}
                           3       3
                        B(a ) # B(a ), B(b) # B(b)
\end{Mexout}
Here is the coproduct of $a$:
\begin{Mexin}
coproduct(aa)
\end{Mexin}
\begin{Mexout}
             4         4      /   II        \    4         5
     1/16 B(a  b) # B(a  b) + | - -- - 1/16 | B(a  b) # B(a ) +
                              \    8        /

            ... one hundred lines sniped out ...

                        2
        -1/16 B(a) # B(a ) + 7/16 B(a) # B(a)
\end{Mexout}
The implementation of the coproduct follows closely Vainerman's
definition~\cite{Vainerman.1998} by deformation of the usual coproduct. In
particular, it goes through the definition of $\omega$ and $\Omega$:
\begin{Mexin}
KD3::G::Omega
\end{Mexin}
\begin{Mexout}
           3         3      /   II       \    3         3
    1/8 B(a  b) # B(a  b) + | - -- - 1/8 | B(a  b) # B(a ) +
                            \    4       /

            ... ten lines sniped out ...

    / II       \           3
    | -- - 1/8 | B(b) # B(a ) + 1/8 B(b) # B(1) + 1/8 B(b) # B(b)
    \  4       /
\end{Mexout}
and of the twisted coproduct:
\begin{Mexin}
expose(KD3::G::coproductBasis)
\end{Mexin}
\begin{Mexout}
     proc(x : DihedralGroup(6)) : KD3::G::tensorSquare
       name KD3::G::coproductBasis;
       option remember;
     begin
       dom::Omega * dom::tensorSquare(K::coproductBasis(x)) * dom::OmegaStar
     end_proc
\end{Mexout}
This function just computes the image of a basis element, and the
actual coproduct is obtained by linearity. Thanks to the \texttt{option
remember}, the computation is done only once. \texttt{dom} is a place
holder for the current domain (here $\KD(n)$), and \texttt{K} denotes
the original Kac algebra (here $\C[D_6]$). The code is generic, and
can be used to twist any Kac algebra by an appropriate cocycle.
\texttt{KD3::M::tensorSquare} and \texttt{KD3::G::tensorSquare} model
$\KD(3)\otimes \KD(3)$ respectively in the group and the matrix basis.
The changes of basis between the two are defined automatically, and
registered as implicit conversions. Those conversions are used
transparently to compute coproducts in the matrix basis.

\subsection{Computing with \sacigs}
\label{subsection.demo.sacig}

A foremost feature that we used for exploration was the ability to
compute properties of \sacigs generated by various elements, in the
hope to find Jones projections. Let us first define a shortcut for the
matrix units:
\begin{Mexin}
e := KD3::e:
\end{Mexin}
Now, $e_i$ and $e_{i,j}^k$ are given respectively by \texttt{e(i)} and
\texttt{e(i,j,k)}. We compute the \sacig $K_1$ generated by the
projection $e_1+e_2+e_3+e_4$:
\begin{Mexin}
K1basis := coidealAndAlgebraClosure([e(1)+e(2)+e(3)+e(4)])
\end{Mexin}
\begin{Mexoutsmall}
-- +-                        -+  +-                        -+  +-                        -+ --
|  |  1, 0, 0, 0, 0, 0, 0, 0  |  |  0, 0, 0, 0, 0, 0, 0, 0  |  |  0, 0, 0, 0, 0, 0, 0, 0  |  |
|  |                          |  |                          |  |                          |  |
|  |  0, 1, 0, 0, 0, 0, 0, 0  |  |  0, 0, 0, 0, 0, 0, 0, 0  |  |  0, 0, 0, 0, 0, 0, 0, 0  |  |
|  |                          |  |                          |  |                          |  |
|  |  0, 0, 0, 0, 0, 0, 0, 0  |  |  0, 0, 1, 0, 0, 0, 0, 0  |  |  0, 0, 0, 0, 0, 0, 0, 0  |  |
|  |                          |  |                          |  |                          |  |
|  |  0, 0, 0, 0, 0, 0, 0, 0  |  |  0, 0, 0, 1, 0, 0, 0, 0  |  |  0, 0, 0, 0, 0, 0, 0, 0  |  |
|  |                          |, |                          |, |                          |  |
|  |  0, 0, 0, 0, 0, 0, 0, 0  |  |  0, 0, 0, 0, 0, 0, 0, 0  |  |  0, 0, 0, 0, 0, 0, 0, 0  |  |
|  |                          |  |                          |  |                          |  |
|  |  0, 0, 0, 0, 0, 0, 0, 0  |  |  0, 0, 0, 0, 0, 0, 0, 0  |  |  0, 0, 0, 0, 0, 0, 0, 0  |  |
|  |                          |  |                          |  |                          |  |
|  |  0, 0, 0, 0, 0, 0, 0, 0  |  |  0, 0, 0, 0, 0, 0, 0, 0  |  |  0, 0, 0, 0, 0, 0, 0, 0  |  |
|  |                          |  |                          |  |                          |  |
|  |  0, 0, 0, 0, 0, 0, 0, 0  |  |  0, 0, 0, 0, 0, 0, 0, 0  |  |  0, 0, 0, 0, 0, 0, 0, 1  |  |
-- +-                        -+  +-                        -+  +-                        -+ --
\end{Mexoutsmall}
The result is a basis of $K_1$ in echelon form. Its dimension is
consistent with the trace of $e_1+e_2+e_3+e_4$:

\begin{Mexin}
1 / (e(1)+e(2)+e(3)+e(4))::traceNormalized()
\end{Mexin}
\begin{Mexout}
                                3
\end{Mexout}
It follows that $e_1+e_2+e_3+e_4$ is the Jones projection $p_{K_1}$
(see Remark~\ref{resumep}).
\clearpage
To give a flavor of the implementation work, here is the code for
computing algebra and \sacig closures. It is defined generically for
any domain implementing the appropriate operations (\texttt{dom} is a
place holder for the current domain):
\begin{Mexin}
    algebraClosure :=
    proc(generators: Type::ListOf(dom)) : Type::ListOf(dom)
    begin
        userinfo(3, "Computing the (non unital!) algebra closure");
        dom::moduleClosure(generators,
                           [proc(x: dom) : Type::SequenceOf(dom)
                                local generator;
                            begin
                                x * generator $ generator in generators
                            end_proc])
    end_proc;

    coidealAndAlgebraClosure :=
    proc(generators: Type::ListOf(dom),
         direction=Left: Type::Enumeration(Left,Right))
        : Type::ListOf(dom)
    begin
        userinfo(3, "Computing the coideal and algebra closure");
        // Proposition: the algebra closure of a coideal is again a coideal!
        dom::algebraClosure(dom::coidealClosure(generators, direction));
    end_proc;
\end{Mexin}
In short: thanks to the underlying computer science work in the design
and implementation of the platform, the algorithms may be written in a
reasonnably \emph{expressive} and \emph{mathematically meaningful}
way.

\subsubsection{The \sacig $K_2$ and subalgebras of functions on a group}

Consider now the \sacig $K_2$ generated by the projection $e_1+e_2$
of trace $1/6$:
\begin{Mexin}
K2basis := coidealAndAlgebraClosure([e(1)+e(2)]):
nops(K2basis)
\end{Mexin}
\begin{Mexout}
                               6
\end{Mexout}
It has the desired dimension, so that $e_1+e_2$ is the Jones
projection of $K_2$.

Consider now the coproduct of $e_1+e_2$:
\begin{Mexin}
c := (e(1)+e(2))::coproduct()
\end{Mexin}
\begin{Mexout}
e(1, 1, 2) # e(2, 2, 2) + e(2, 2, 2) # e(1, 1, 2) +

   1/2 e(2, 2, 1) # e(2, 2, 1) + 1/2 e(1, 1, 1) # e(2, 2, 1) +

   1/2 e(2, 2, 1) # e(1, 1, 1) + 1/2 e(1, 1, 1) # e(1, 1, 1) +

   -1/2 e(1, 2, 1) # e(1, 2, 1) + 1/2 e(2, 1, 1) # e(1, 2, 1) +

   1/2 e(1, 2, 1) # e(2, 1, 1) + -1/2 e(2, 1, 1) # e(2, 1, 1) + e(4) # e(4) +

   e(3) # e(4) + e(4) # e(3) + e(3) # e(3) + e(2) # e(2) + e(1) # e(2) +

   e(2) # e(1) + e(1) # e(1)
\end{Mexout}
It turns out to be symmetric:
\begin{Mexin}
  c - c::mapsupport(revert)
\end{Mexin}
\begin{Mexout}
                             0
\end{Mexout}
and therefore $K_2$ is a Kac subalgebra, the properties of which we
now investigate. To this end, we define the subspace spanned by this
basis, and claim to \mupad{} that it is indeed a Kac subalgebra:
\begin{Mexin}
K2 := Dom::SubFreeModule(K2basis,
        [Cat::FiniteDimensionalHopfAlgebraWithBasis(KD3::coeffRing)]):
\end{Mexin}
The computation rules inside \texttt{K2} are then inherited from those
of \texttt{KD3::M}. We first ask whether $K_2$ is commutative or
cocommutative:
\begin{Mexin}
K2::isCommutative(), K2::isCocommutative()
\end{Mexin}
\begin{Mexout}
                        TRUE, FALSE
\end{Mexout}
This tells us that $K_2$ is the dual of the algebra $\C[G]$ of some
non commutative group $G$. To find $G$, we first define the dual of
$K_2$:
\begin{Mexin}
K2dual := K2::Dual():
\end{Mexin}
As expected, there are six group like elements in it:
\begin{Mexin}
K2dual::groupLikeElements()
\end{Mexin}
\begin{Mexout}
           _           _             _           _          _
      [-II B([6, 5]) + B([5, 5]), II B([6, 5]) + B([5, 5]), B([7, 7]),

         _          _          _
         B([8, 8]), B([1, 1]), B([3, 3])]
\end{Mexout}
They are expressed in the dual basis of the row reduced echelon basis for
\texttt{K2}; this representation is not very much useful. On the other
hand, we may instead consider all of them together as the instrisic
group:
\begin{Mexin}
intrinsicGroup := K2dual::intrinsicGroup():
\end{Mexin}
and test that it is isomorphic to the dihedral group $D_3$:
\begin{Mexin}
D3 := DihedralGroup(3):
nops(D3::groupEmbeddings(intrinsicGroup()))
\end{Mexin}
\begin{Mexout}
                              6
\end{Mexout}
The algorithmic behind this last step is currently simplistic. It
could not deal with large groups like other specialized software like
\gap{} could, but is sufficient for our purpose. On the other
hand, the computation of the group-like elements themselves is rather
efficient: it is done by computing the rank one central projections in
the dual algebra. More generally, we can compute the full
representation theory of finite dimensional algebras. For example:
\begin{Mexin}
K2dual::isSemiSimple()
\end{Mexin}
\begin{Mexout}
                             TRUE
\end{Mexout}
\begin{Mexin}
K2dual::simpleModulesDimensions()
\end{Mexin}
\begin{Mexout}
                           [2, 1, 1]
\end{Mexout}

\subsubsection{\Sacigs in $K_2$}
\label{subsubsection.mupad.sacig.K2}

We now show how to construct the Jones projections of the \sacigs of
$K_2=L^\infty(D_3)$, which are in correspondence with the subgroups of
$D_3$ (see~\ref{groupe} and~\ref{section.K2}). Here we consider as
example a subgroup $Z_2$ of order $2$ of $D_3$.  Take the second
generator of the intrinsic group of $\widehat{K_2}$, and write it as an
element of $\widehat{K_2}$:
\begin{Mexin}
c := intrinsicGroup([2]);
c := c::lift()
\end{Mexin}
\begin{Mexout}
                              [2]

                       _           _
                   -II B([6, 5]) + B([5, 5])
\end{Mexout}
\pagebreak[3]
Here is the subgroup it generates:
\begin{Mexin}
Z2 := K2dual::multiplicativeClosure([c])
\end{Mexin}
\begin{Mexout}
             [B([5, 5]) + -II B([6, 5]), B([1, 1])]
\end{Mexout}

The corresponding \sacig $I$ consists of the functions on $D_3$ which are
constant on right cosets for $Z_2$; it is of dimension $[D_3:Z_2]=3$.
The Jones projection is given by the formula $\sum_{g\in Z_2} \widehat
g$ (see~\ref{groupe}):
\begin{Mexin}
pI := _plus( g::groupLikeToIdempotentOfDual() $ g in Z2 ):
\end{Mexin}
The result is actually given in $K_2$; we lift it to an element of
$\KD(3)$:
\begin{Mexin}
pI := pI::toSupModule()
\end{Mexin}
\begin{Mexoutsmall}
                       +-                             -+
                       |  1, 0, 0, 0,  0,    0,  0, 0  |
                       |                               |
                       |  0, 1, 0, 0,  0,    0,  0, 0  |
                       |                               |
                       |  0, 0, 0, 0,  0,    0,  0, 0  |
                       |                               |
                       |  0, 0, 0, 0,  0,    0,  0, 0  |
                       |                               |
                       |                     II        |
                       |  0, 0, 0, 0, 1/2, - --, 0, 0  |
                       |                      2        |
                       |                               |
                       |               II              |
                       |  0, 0, 0, 0,  --,  1/2, 0, 0  |
                       |                2              |
                       |                               |
                       |  0, 0, 0, 0,  0,    0,  0, 0  |
                       |                               |
                       |  0, 0, 0, 0,  0,    0,  0, 0  |
                       +-                             -+
\end{Mexoutsmall}
This gives us $p_{K_{21}}=e_1+e_2+r_{1,1}^1$, as in~\ref{KD3section}.

As a double check, here is the basis of the coideal $K_{21}$ it generates:
\begin{Mexin}
coidealAndAlgebraClosure([pI])
\end{Mexin}
\begin{Mexoutsmall}
-- +-                        -+  +-                        -+  +-                             -+ --
|  |  1, 0, 0, 0, 0, 0, 0, 0  |  |  0, 0, 0, 0, 0, 0, 0, 0  |  |  0, 0, 0, 0,  0,  0,   0,  0  |  |
|  |                          |  |                          |  |                               |  |
|  |  0, 1, 0, 0, 0, 0, 0, 0  |  |  0, 0, 0, 0, 0, 0, 0, 0  |  |  0, 0, 0, 0,  0,  0,   0,  0  |  |
|  |                          |  |                          |  |                               |  |
|  |  0, 0, 0, 0, 0, 0, 0, 0  |  |  0, 0, 1, 0, 0, 0, 0, 0  |  |  0, 0, 0, 0,  0,  0,   0,  0  |  |
|  |                          |  |                          |  |                               |  |
|  |  0, 0, 0, 0, 0, 0, 0, 0  |  |  0, 0, 0, 1, 0, 0, 0, 0  |  |  0, 0, 0, 0,  0,  0,   0,  0  |  |
|  |                          |, |                          |, |                               |  |
|  |  0, 0, 0, 0, 1, 0, 0, 0  |  |  0, 0, 0, 0, 0, 0, 0, 0  |  |  0, 0, 0, 0, II, -1,   0,  0  |  |
|  |                          |  |                          |  |                               |  |
|  |  0, 0, 0, 0, 0, 1, 0, 0  |  |  0, 0, 0, 0, 0, 0, 0, 0  |  |  0, 0, 0, 0,  1, II,   0,  0  |  |
|  |                          |  |                          |  |                               |  |
|  |  0, 0, 0, 0, 0, 0, 1, 0  |  |  0, 0, 0, 0, 0, 0, 0, 0  |  |  0, 0, 0, 0,  0,  0, 2 II, 0  |  |
|  |                          |  |                          |  |                               |  |
|  |  0, 0, 0, 0, 0, 0, 0, 0  |  |  0, 0, 0, 0, 0, 0, 0, 1  |  |  0, 0, 0, 0,  0,  0,   0,  0  |  |
-- +-                        -+  +-                        -+  +-                             -+ --
\end{Mexoutsmall}
which is consistent with the trace of the projection:
\begin{Mexin}
1 / pI::traceNormalized()
\end{Mexin}
\begin{Mexout}
                               3
\end{Mexout}
Doing this for all subgroups of $D_3$ yields all \sacigs of $\KD(3)$ of
dimension $2$ and $3$.

\subsubsection{Deeper study of $K_2$, and generalization to $\KD(n)$}\label{mupad.K2=Dn}

Studying $K_2$ this way had the advantage to be completely automatic.
However to generalize the results to any $\KD(n)$, we need to have a
closer look at the structure, and do some things manually. Here, we
are lucky enough that the row reduced basis comes out quite nicely
(this is seldom the case):
\pagebreak[4]
\begin{Mexin}
K2basis := coidealAndAlgebraClosure([e(1)+e(2)])
\end{Mexin}
\begin{Mexoutsmall}
-- +-                        -+  +-                        -+  +-                        -+
|  |  1, 0, 0, 0, 0, 0, 0, 0  |  |  0, 0, 0, 0, 0, 0, 0, 0  |  |  0, 0, 0, 0, 0, 0, 0, 0  |
|  |                          |  |                          |  |                          |
|  |  0, 1, 0, 0, 0, 0, 0, 0  |  |  0, 0, 0, 0, 0, 0, 0, 0  |  |  0, 0, 0, 0, 0, 0, 0, 0  |
|  |                          |  |                          |  |                          |
|  |  0, 0, 0, 0, 0, 0, 0, 0  |  |  0, 0, 1, 0, 0, 0, 0, 0  |  |  0, 0, 0, 0, 0, 0, 0, 0  |
|  |                          |  |                          |  |                          |
|  |  0, 0, 0, 0, 0, 0, 0, 0  |  |  0, 0, 0, 1, 0, 0, 0, 0  |  |  0, 0, 0, 0, 0, 0, 0, 0  |
|  |                          |, |                          |, |                          |,
|  |  0, 0, 0, 0, 0, 0, 0, 0  |  |  0, 0, 0, 0, 0, 0, 0, 0  |  |  0, 0, 0, 0, 0,-1, 0, 0  |
|  |                          |  |                          |  |                          |
|  |  0, 0, 0, 0, 0, 0, 0, 0  |  |  0, 0, 0, 0, 0, 0, 0, 0  |  |  0, 0, 0, 0, 1, 0, 0, 0  |
|  |                          |  |                          |  |                          |
|  |  0, 0, 0, 0, 0, 0, 0, 0  |  |  0, 0, 0, 0, 0, 0, 0, 0  |  |  0, 0, 0, 0, 0, 0, 0, 0  |
|  |                          |  |                          |  |                          |
|  |  0, 0, 0, 0, 0, 0, 0, 0  |  |  0, 0, 0, 0, 0, 0, 0, 0  |  |  0, 0, 0, 0, 0, 0, 0, 0  |
-- +-                        -+  +-                        -+  +-                        -+

   +-                        -+  +-                        -+  +-                        -+ --
   |  0, 0, 0, 0, 0, 0, 0, 0  |  |  0, 0, 0, 0, 0, 0, 0, 0  |  |  0, 0, 0, 0, 0, 0, 0, 0  |  |
   |                          |  |                          |  |                          |  |
   |  0, 0, 0, 0, 0, 0, 0, 0  |  |  0, 0, 0, 0, 0, 0, 0, 0  |  |  0, 0, 0, 0, 0, 0, 0, 0  |  |
   |                          |  |                          |  |                          |  |
   |  0, 0, 0, 0, 0, 0, 0, 0  |  |  0, 0, 0, 0, 0, 0, 0, 0  |  |  0, 0, 0, 0, 0, 0, 0, 0  |  |
   |                          |  |                          |  |                          |  |
   |  0, 0, 0, 0, 0, 0, 0, 0  |  |  0, 0, 0, 0, 0, 0, 0, 0  |  |  0, 0, 0, 0, 0, 0, 0, 0  |  |
   |                          |, |                          |, |                          |  |
   |  0, 0, 0, 0, 1, 0, 0, 0  |  |  0, 0, 0, 0, 0, 0, 0, 0  |  |  0, 0, 0, 0, 0, 0, 0, 0  |  |
   |                          |  |                          |  |                          |  |
   |  0, 0, 0, 0, 0, 1, 0, 0  |  |  0, 0, 0, 0, 0, 0, 0, 0  |  |  0, 0, 0, 0, 0, 0, 0, 0  |  |
   |                          |  |                          |  |                          |  |
   |  0, 0, 0, 0, 0, 0, 0, 0  |  |  0, 0, 0, 0, 0, 0, 1, 0  |  |  0, 0, 0, 0, 0, 0, 0, 0  |  |
   |                          |  |                          |  |                          |  |
   |  0, 0, 0, 0, 0, 0, 0, 0  |  |  0, 0, 0, 0, 0, 0, 0, 0  |  |  0, 0, 0, 0, 0, 0, 0, 1  |  |
   +-                        -+  +-                        -+  +-                        -+ --
\end{Mexoutsmall}
In particular, we can read off this basis the complete algebra
structure of $K_2$: it is a commutative algebra whose minimal projections are easy to find.

The expression of the coproduct of $e_1+e_2$ did not look very good.
As usual in computer exploration, an essential issue is to find
the right view where the output is readable and exploitable by a
\emph{human}; customizable outputs are therefore at a premium. For
example the first author's favorite view for a tensor element (which
tends to be huge) is as a matrix $M=(m_{a,b})_{a,b}$ whose rows and
columns are indexed by the matrix units of $\KD(m)$ (here, for $m=3$:
$(e_1,\dots,e_4,e^1_{1,1},e^1_{2,2},e^1_{1,2},e^1_{2,1})$), and where
$m_{a,b}$ is the coefficient of $a \otimes b$.
With this view, the coproduct of $e_1+e_2$ now also comes out nicely
(see also~\ref{cop}):
\begin{Mexin}
KD3::M::tensorElementToMatrix((e(1)+e(2))::coproduct())
\end{Mexin}
\begin{Mexoutsmall}
                +-                                              -+
                |  1, 1, 0, 0,  0,   0,    0,    0,  0, 0, 0, 0  |
                |                                                |
                |  1, 1, 0, 0,  0,   0,    0,    0,  0, 0, 0, 0  |
                |                                                |
                |  0, 0, 1, 1,  0,   0,    0,    0,  0, 0, 0, 0  |
                |                                                |
                |  0, 0, 1, 1,  0,   0,    0,    0,  0, 0, 0, 0  |
                |                                                |
                |  0, 0, 0, 0, 1/2, 1/2,   0,    0,  0, 0, 0, 0  |
                |                                                |
                |  0, 0, 0, 0, 1/2, 1/2,   0,    0,  0, 0, 0, 0  |
                |                                                |
                |  0, 0, 0, 0,  0,   0,  -1/2,  1/2, 0, 0, 0, 0  |
                |                                                |
                |  0, 0, 0, 0,  0,   0,   1/2, -1/2, 0, 0, 0, 0  |
                |                                                |
                |  0, 0, 0, 0,  0,   0,    0,    0,  0, 1, 0, 0  |
                |                                                |
                |  0, 0, 0, 0,  0,   0,    0,    0,  1, 0, 0, 0  |
                |                                                |
                |  0, 0, 0, 0,  0,   0,    0,    0,  0, 0, 0, 0  |
                |                                                |
                |  0, 0, 0, 0,  0,   0,    0,    0,  0, 0, 0, 0  |
                +-                                              -+
\end{Mexoutsmall}

For example, it becomes obvious that this coproduct is symmetric,
implying that $K_2$ is a Kac subalgebra (see~\ref{irredprof2}); as it
is commutative, we have $K_2\equiv L^\infty(G)$ for some group $G$
(see~\ref{groupe}). Note that a basis for $K_2$ can also be read off
the rows.  Using duality, the elements of $G$ are given by the rank
$1$ central projections.  With a bit more work, one can deduce from
the expression of $\Delta(e_1+e_2)$ all the pairs $(g,g^{-1})$ of
inverse elements in $G$.  Looking at some other coproducts reveals the
complete group structure of $G$ (see~\ref{e1e2}).

To generalize this we look at $n=5$:
\begin{Mexin}
K := KD(5):
K::M::tensorElementToMatrix((K::e(1)+K::e(2))::coproduct())
\end{Mexin}
\begin{Mexoutsmall}
        +-                                                                                -+
        |  1, 1, 0, 0,  0,   0,    0,    0,  0, 0, 0, 0,  0,   0,    0,    0,  0, 0, 0, 0  |
        |                                                                                  |
        |  1, 1, 0, 0,  0,   0,    0,    0,  0, 0, 0, 0,  0,   0,    0,    0,  0, 0, 0, 0  |
        |                                                                                  |
        |  0, 0, 1, 1,  0,   0,    0,    0,  0, 0, 0, 0,  0,   0,    0,    0,  0, 0, 0, 0  |
        |                                                                                  |
        |  0, 0, 1, 1,  0,   0,    0,    0,  0, 0, 0, 0,  0,   0,    0,    0,  0, 0, 0, 0  |
        |                                                                                  |
        |  0, 0, 0, 0, 1/2, 1/2,   0,    0,  0, 0, 0, 0,  0,   0,    0,    0,  0, 0, 0, 0  |
        |                                                                                  |
        |  0, 0, 0, 0, 1/2, 1/2,   0,    0,  0, 0, 0, 0,  0,   0,    0,    0,  0, 0, 0, 0  |
        |                                                                                  |
        |  0, 0, 0, 0,  0,   0,  -1/2,  1/2, 0, 0, 0, 0,  0,   0,    0,    0,  0, 0, 0, 0  |
        |                                                                                  |
        |  0, 0, 0, 0,  0,   0,   1/2, -1/2, 0, 0, 0, 0,  0,   0,    0,    0,  0, 0, 0, 0  |
        |                                                                                  |
        |  0, 0, 0, 0,  0,   0,    0,    0,  0, 1, 0, 0,  0,   0,    0,    0,  0, 0, 0, 0  |
        |                                                                                  |
        |  0, 0, 0, 0,  0,   0,    0,    0,  1, 0, 0, 0,  0,   0,    0,    0,  0, 0, 0, 0  |
        |                                                                                  |
        |  0, 0, 0, 0,  0,   0,    0,    0,  0, 0, 0, 0,  0,   0,    0,    0,  0, 0, 0, 0  |
        |                                                                                  |
        |  0, 0, 0, 0,  0,   0,    0,    0,  0, 0, 0, 0,  0,   0,    0,    0,  0, 0, 0, 0  |
        |                                                                                  |
        |  0, 0, 0, 0,  0,   0,    0,    0,  0, 0, 0, 0, 1/2, 1/2,   0,    0,  0, 0, 0, 0  |
        |                                                                                  |
        |  0, 0, 0, 0,  0,   0,    0,    0,  0, 0, 0, 0, 1/2, 1/2,   0,    0,  0, 0, 0, 0  |
        |                                                                                  |
        |  0, 0, 0, 0,  0,   0,    0,    0,  0, 0, 0, 0,  0,   0,  -1/2,  1/2, 0, 0, 0, 0  |
        |                                                                                  |
        |  0, 0, 0, 0,  0,   0,    0,    0,  0, 0, 0, 0,  0,   0,   1/2, -1/2, 0, 0, 0, 0  |
        |                                                                                  |
        |  0, 0, 0, 0,  0,   0,    0,    0,  0, 0, 0, 0,  0,   0,    0,    0,  0, 1, 0, 0  |
        |                                                                                  |
        |  0, 0, 0, 0,  0,   0,    0,    0,  0, 0, 0, 0,  0,   0,    0,    0,  1, 0, 0, 0  |
        |                                                                                  |
        |  0, 0, 0, 0,  0,   0,    0,    0,  0, 0, 0, 0,  0,   0,    0,    0,  0, 0, 0, 0  |
        |                                                                                  |
        |  0, 0, 0, 0,  0,   0,    0,    0,  0, 0, 0, 0,  0,   0,    0,    0,  0, 0, 0, 0  |
        +-                                                                                -+
\end{Mexoutsmall}
This suggests the general formulas of~\ref{cop} which we
can double check for other values of $n$ before starting to prove
them:
\begin{Mexin}
for n from 1 to 7 do
    r := (\KD(n))::r:
    e := (\KD(n))::e:
    print(n, iszero(coproduct(e(1)+e(2)) -
                    (  ( e(1)+e(2) ) # ( e(1)+e(2) )
                     + ( e(3)+e(4) ) # ( e(3)+e(4) )
                     + _plus(  r(1,1,j) # r(1,1,j) + r(2,2,j) # r(2,2,j)
                             $ j in select([ $1..n-1], testtype, Type::Odd) )
                     + _plus(  e(1,1,j) # e(2,2,j) + e(2,2,j) # e(1,1,j)
                             $ j in select([ $1..n-1], testtype, Type::Even) )
                    )));
end_for:
\end{Mexin}
\begin{Mexout}
                                    1, TRUE
                                    2, TRUE
                                      ...
                                    7, TRUE
\end{Mexout}
With some patience and perseverance, the other required coproducts can
be reverse engineered the same way for all $n$.

\subsection{Computing Kac isomorphisms and applications}

We now demonstrate the use of algorithm~\ref{algo.isomorphism} to
compute automorphisms, self-duality, and isomorphisms, with the search
of coideals as motivation.

\subsubsection{The \sacigs $K_3$ and $K_4$, and automorphisms of $\KD(3)$}
\label{subsubsection.automorphismsK3}

Computing the coproducts of $e_1+e_3$ and $e_1+e_4$ as above shows
that the corresponding \sacigs $K_3$ and $K_4$ are not Kac
subalgebras. However, they look similar, and it is natural to ask
whether there exists an automorphism of $\KD(3)$ which would exchange
them. To this end, we compute the automorphism group of $\KD(3)$
(see~\ref{algo.isomorphism}):
\begin{Mexin}
automorphismGroup := KD3::G::algebraEmbeddings(KD3::G):
\end{Mexin}
A few minutes and a tea later, we obtain four automorphisms:
\begin{Mexin}
for phi in automorphismGroup do
    fprint(Unquoted,0, _concat("-" $ 78)):
    print(hold(phi)(a) = phi(KD3::G::algebraGenerators[a]));
    print(hold(phi)(b) = phi(KD3::G::algebraGenerators[b]));
end:
\end{Mexin}
\begin{Mexout}
------------------------------------------------------------------------------
                             5            4           2
             phi(a) = 1/2 B(a ) + -1/2 B(a ) + 1/2 B(a ) + 1/2 B(a)

                                            3
                                phi(b) = B(a  b)
------------------------------------------------------------------------------
                             5           4            2
             phi(a) = 1/2 B(a ) + 1/2 B(a ) + -1/2 B(a ) + 1/2 B(a)

                                            3
                                phi(b) = B(a  b)
------------------------------------------------------------------------------
                                             5
                                 phi(a) = B(a )

                                 phi(b) = B(b)
------------------------------------------------------------------------------
                                 phi(a) = B(a)

                                 phi(b) = B(b)
\end{Mexout}
Half of them are obviously induced by automorphisms of the group which
fix $H$. The other half are obtained from $\Theta'$:
\begin{Mexin}
ThetaPrime := automorphismGroup[1]:
\end{Mexin}
which is an involution:
\begin{Mexin}
ThetaPrime(ThetaPrime(KD3::G::algebraGenerators[a])),
ThetaPrime(ThetaPrime(KD3::G::algebraGenerators[b]))
\end{Mexin}
\begin{Mexout}
                                 B(a), B(b)
\end{Mexout}
The generalization of the formula for $\Theta'$ to any $n$ is
straightforward
(see Proposition~\ref{proposition.automorphism.KD.ThetaPrime}).

Going back to our original problem, we see that $\Theta'$ exchanges
$e_1+e_3$ and $e_1+e_4$:
\begin{Mexin}
KD3::M(ThetaPrime(KD3::G( KD3::e(1) + KD3::e(3))))
\end{Mexin}
\begin{Mexoutsmall}
                                   +-                        -+
                                   |  1, 0, 0, 0, 0, 0, 0, 0  |
                                   |                          |
                                   |  0, 0, 0, 0, 0, 0, 0, 0  |
                                   |                          |
                                   |  0, 0, 0, 0, 0, 0, 0, 0  |
                                   |                          |
                                   |  0, 0, 0, 1, 0, 0, 0, 0  |
                                   |                          |
                                   |  0, 0, 0, 0, 0, 0, 0, 0  |
                                   |                          |
                                   |  0, 0, 0, 0, 0, 0, 0, 0  |
                                   |                          |
                                   |  0, 0, 0, 0, 0, 0, 0, 0  |
                                   |                          |
                                   |  0, 0, 0, 0, 0, 0, 0, 0  |
                                   +-                        -+
\end{Mexoutsmall}
Therefore $K_3$ and $K_4$ are isomorphic (see~\ref{iso-n-impair}).

\subsubsection{Self-duality of $\KD(3)$}
\label{subsubsection.demo.self-dual}

In the previous subsubsections, we have determined all \sacigs of
$\KD(3)$, except those of dimension $4$ which we now investigate.
By the classification of small dimension Kac algebras, we knew that
$\KD(3)$ was self-dual, and therefore that there existed exactly three
\sacigs of dimension $4$. In this subsubsection, we demonstrate how we
found an explicit isomorphism between $\KD(2m+1)$ and its dual
(see Theorem~\ref{self-dual}). Then, in the next subsubsection, we derive the
explicit construction of the Jones projection of the \sacigs of
dimension $4$ in $\KD(3)$, and in general in $\KD(2m+1)$.

First, we compute all the Kac isomorphisms from $\KD(3)$ to its dual.
It occurs that the computation is much quicker in the dual of the
matrix basis \texttt{KD3::M::Dual()} rather than of the group basis
\texttt{KD3::G::Dual()} (the expressions of the group like elements
and of $\Omega$ are sparser there):
\begin{Mexin}
isomorphisms := KD3::G::algebraEmbeddings(KD3::M::Dual()):
\end{Mexin}
Due to the non trivial automorphism group of $\KD(3)$, there are four
of them:
\begin{Mexin}
for phi in isomorphisms do
    fprint(Unquoted,0, _concat("-" $ 78)):
    print(hold(phi)(a) = phi(KD3::G::algebraGenerators[a]));
    print(hold(phi)(b) = phi(KD3::G::algebraGenerators[b]));
end:
\end{Mexin}
\begin{Mexout}
------------------------------------------------------------------------------
              _              _               _               _          _
phi(a) = -1/2 e(1,1,1) + 1/2 e(2,2,1) + -1/2 e(1,2,2) + -1/2 e(2,1,2) + e(2,2,2)
                                            _
                                   phi(b) = e(3)
------------------------------------------------------------------------------
             _               _          _               _               _
phi(a) = 1/2 e(1,1,1) + -1/2 e(2,2,1) + e(1,1,2) + -1/2 e(1,2,2) + -1/2 e(2,1,2)
                                            _
                                   phi(b) = e(3)
------------------------------------------------------------------------------
              _              _          _               _               _
phi(a) = -1/2 e(1,1,1) + 1/2 e(2,2,1) + e(1,1,2) + -1/2 e(1,2,2) + -1/2 e(2,1,2)
                                            _
                                   phi(b) = e(4)
------------------------------------------------------------------------------
             _               _               _               _          _
phi(a) = 1/2 e(1,1,1) + -1/2 e(2,2,1) + -1/2 e(1,2,2) + -1/2 e(2,1,2) + e(2,2,2)
                                            _
                                   phi(b) = e(4)
\end{Mexout}
Running the same computation for $n=5$ was sufficient to pick one of
them and guess the general formulas of
Theorem~\ref{self-dual}.

Now, we can construct this isomorphism $\Phi$ directly with:
\begin{Mexin}
psi := KD3::toDualIsomorphism:
\end{Mexin}
and check that it is indeed an isomorphism with:
\begin{Mexin}
KD3::G::isHopfAlgebraMorphism(psi)
KD3::G::kernelOfModuleMorphism(psi)
\end{Mexin}
\begin{Mexout}
                                     TRUE
                                      []
\end{Mexout}
We actually checked this up to $n=21$ (the computation for $n=21$ took
two days on a $\unit[2]{GHz}$ PC and $\unit[1.6]{Gb}$ of memory):
\begin{Mexin}
for n from 3 to 21 step 2 do
    K := \KD(n):
    psi := K::toDualIsomorphism:
    print(n, K::G::isHopfAlgebraMorphism(psi), K::G::kernelOfModuleMorphism(psi))
end_for:
\end{Mexin}

Here is finally how we guessed the formulas for $\psi$:
\begin{Mexin}
MdualToGdual := KD3::G::dualOfModuleMorphism(KD3::M):
Phi := MdualToGdual @ psi @ KD3::G:
Phi(KD3::e(2,1,1))
\end{Mexin}
{
\scriptsize
\begin{Mexout}
/ II       \ _        /       II \ _  2     / II       \ _  4     / II       \ _  5
| -- - 1/2 | B(a b) + | 1/2 - -- | B(a b) + | -- + 1/2 | B(a b) - | -- + 1/2 | B(a b)
\  2       /          \        2 /          \  2       /          \  2       /
\end{Mexout}
}

\subsubsection{\Sacigs of dimension $4$ and the antiisomorphism $\delta$}
\label{subsubsection.demo.delta}

We now use the isomorphism of the previous section to construct
explicitly the coideals of dimension $4$ of $\KD(3)$.  Consider for
example the coideal $K_{21}$ of dimension $3$ whose Jones projection
we computed in~\ref{subsubsection.mupad.sacig.K2}. We construct a
basis for $\delta(K_{21})$ in $\widehat{\KD(3)}$ by looking for the
commutant of $p_{K_{21}}$ in $\widehat{\KD(3)} \subset
\widehat{\KD(3)} \rtimes \KD(3)$.

\begin{Mexin}
e := KD3::e: r := KD3::r:
pK21 := e(1) + e(2) + r(1,1,1):
deltaK21 := coidealDual([pK21])
\end{Mexin}
\vspace{-2ex}
\begin{Mexout}
 _     _         _            _               _            _
[e(1), e(2), -II e(1, 1, 1) + e(1, 2, 1), -II e(2, 1, 1) + e(2, 2, 1)]
\end{Mexout}

Using the inverse of the isomorphism $\Phi$:
\begin{Mexin}
psi := KD3::toDualIsomorphism:
psiInv := KD3::G::inverseOfModuleMorphism(psi):
\end{Mexin}
we identify $\delta(K_{21})$ as a coideal $J_{20}$ of $\KD(3)$:

\begin{Mexin}
J20 := map(deltaK21, psiInv)
\end{Mexin}
\begin{Mexout}
--          3   /   II       \    5      II    5           4      II    4
|  B(1), B(a ), | - -- - 1/4 | B(a  b) + -- B(a ) + 1/4 B(a  b) + -- B(a ) +
--              \    2       /            4                        4

   /       II \    2        II    2                    II
   | 1/4 - -- | B(a  b) + - -- B(a ) + -1/4 B(a b) + - -- B(a),
   \        2 /              4                          4

   / II       \    5             5      II    4             4
   | -- + 1/2 | B(a  b) + 1/4 B(a ) + - -- B(a  b) + 1/4 B(a ) +
   \  4       /                          4

   /       II \    2              2    II                    --
   | 1/2 - -- | B(a  b) + -1/4 B(a ) + -- B(a b) + -1/4 B(a)  |
   \        4 /                         4                    --
\end{Mexout}

Here, we are lucky enough that the obtained basis for $J_{20}$ is orthonormal:
\begin{Mexin}
matrix(4,4, (i,j) -> scalar(J20[i], J20[j]))
\end{Mexin}
\begin{Mexoutsmall}
                                     +-            -+
                                     |  1, 0, 0, 0  |
                                     |              |
                                     |  0, 1, 0, 0  |
                                     |              |
                                     |  0, 0, 1, 0  |
                                     |              |
                                     |  0, 0, 0, 1  |
                                     +-            -+
\end{Mexoutsmall}
Therefore, we can compute the Jones projection $p_{J_{20}}$ right away by
inverting its coproduct formula (otherwise, we should have
orthogonalized it first by Gramm-Schmidt, and use a variant of the
formula involving the norms):
\begin{Mexin}
pJ20 := KD3::G::fiberOfModuleMorphism
    (KD3::G::coproduct,
     (1/nops(J20)) * _plus((b::involution())::antipode() # b
                           $ b in J20)) [1]
\end{Mexin}%
\begin{Mexout}
                       5             3           2
                1/4 B(a  b) + 1/4 B(a ) + 1/4 B(a  b) + 1/4 B(1)
\end{Mexout}
Now the construction of $p_{J_{20}}$ (and of $p_{J_0}$ and $p_{J_2}$)
is obvious, and can be right away generalized to obtain the \sacigs of
dimension $4$ in $\KD(2n+1)$ (see~\ref{abelien},~\ref{sacig4impair},
and~\ref{KD3section}). Beware however that the correspondence we used
between coideals of dimension $d$ of $\KD(3)$ and those of dimension
$\dim \KD(3) / d$ is only defined up to an automorphism of $\KD(3)$.

\subsubsection{The isomorphism between $\KD(2n)$ and $\KQ(2n)$}
\label{subsubsection.demo.KDKQ}
\label{mupad.plonge}

The lattice of \sacigs of $\KD(3)$ being complete, we now show how parts of the
lattice for the larger algebras can be build up from that of the
smaller ones (see~\ref{KD6}). Let us start by defining $\KD(6)$:
\begin{Mexin}
KD6 := KD(6):
\end{Mexin}
We know that $\KD(3)$ embeds as $K_4$ in $\KD(6)$. Let us show how to
obtain the \sacigs of $K_4$ from those of $\KD(3)$.
We have to be a bit careful as, by default, $\KD(3)$ and $\KD(6)$ are
not defined over the same ground field (unless stated otherwise, the
ground field for $\KD(n)$ is $\Q(i, \epsilon)$ where $\epsilon$ is a
$2n$-th root of unity). Here, we force $\KD(3)$ to use the same field
as $\KD(6)$.
\begin{Mexin}
KD3 := KD(3, KD6::coeffRing):
\end{Mexin}
To ease the notations, we define shortcuts for the algebra generators
of $\KD(3)$ and $\KD(6)$:
\begin{Mexin}
[a3, b3] := KD3::G::algebraGenerators::list():
[a6, b6] := KD6::G::algebraGenerators::list():
\end{Mexin}

Now we can define the embedding of $\KD(3)$ as $K_4$ in $\KD(6)$ (see~\ref{plonge}):
\begin{Mexin}
KD3ToKD6 := KD3::G::algebraMorphism(table(a = a6^2, b = b6 )):
\end{Mexin}
and use it as follows:
\begin{Mexin}
KD3ToKD6(1 + 2*a3^3 + 3 * b3)
\end{Mexin}
\begin{Mexout}
                                 6
                            2 B(a ) + 3 B(b) + B(1)
\end{Mexout}

We now take the Jones projection of the coideal $J_{20}$ in $\KD(3)$
(see Appendix~\ref{subsubsection.demo.delta}), and lift it into
$\KD(6)$:
\begin{Mexin}
pJ20 := 1/4 * ( 1 + a3^3 + a3^2 * b3 + a3^5 * b3 ):
KD3ToKD6(pJ20)
\end{Mexin}
\begin{Mexout}
                                 4             6           10
               1/4 B(1) + 1/4 B(a  b) + 1/4 B(a ) + 1/4 B(a   b)
\end{Mexout}
The result is a posteriori obvious, but the point is that this
construction is completely automatic.

In a similar vein, one can obtain all the \sacigs of $K_3$, as it is
$\KQ(3)$ via the isomorphism between $\KD(2n)$ and $\KQ(2n)$ (see
Theorem~\ref{theorem.isomorphism.KD.KQ}). This isomorphism was first
suspected by comparing the representation theory of the algebras,
their dual, and the properties of the simple \sacigs. The systematic
computation of the isomorphisms for $n\leq 3$ suggested the general
formulas which then were checked for $n\leq 10$:
\begin{Mexin}
KQ6 := KQ(6):
\end{Mexin}
\begin{Mexin}
isomorphisms := KD6::G::algebraEmbeddings(KQ6::G):
for phi in isomorphisms do
    fprint(Unquoted,0, _concat("-" $ 78)):
    print(hold(phi)(a) = phi(a6));
    print(hold(phi)(b) = phi(b6))
end:
\end{Mexin}
\begin{Mexout}
                           2              4              5           7
phi(a) = 3/4 B(a) + 1/4 B(a  b) + -1/4 B(a  b) + -1/4 B(a ) + 1/4 B(a ) +

           8             10             11
   -1/4 B(a  b) + 1/4 B(a   b) + 1/4 B(a  )

                     II    3           3      II    9           9
          phi(b) = - -- B(a ) + 1/2 B(a  b) + -- B(a ) + 1/2 B(a  b)
                      2                        2
\end{Mexout}
The isomorphism we chose in Theorem~\ref{theorem.isomorphism.KD.KQ} is
now built into the system:
\begin{Mexin}
KQ6::G(a6)
\end{Mexin}
\begin{Mexout}
                  2              4              5           7
3/4 B(a) + 1/4 B(a  b) + -1/4 B(a  b) + -1/4 B(a ) + 1/4 B(a ) +

           8             10             11
   -1/4 B(a  b) + 1/4 B(a   b) + 1/4 B(a  )
\end{Mexout}
\begin{Mexin}
KQ6::G(b6)
\end{Mexin}
\begin{Mexout}
                II    3           3      II    9           9
              - -- B(a ) + 1/2 B(a  b) + -- B(a ) + 1/2 B(a  b)
                 2                        2
\end{Mexout}

\begin{Mexin}
KQ3 := KQ(3, KQ6::coeffRing):
\end{Mexin}
\begin{Mexin}
[aq3, bq3] := KQ3::G::algebraGenerators::list():
[aq6, bq6] := KQ6::G::algebraGenerators::list():
KQ3ToKQ6 := KQ3::G::algebraMorphism(table(a = aq6^2,
                                          b = bq6 )):
\end{Mexin}

We can now conclude by demonstrating the lifting of the
$4$-dimensional coideal $I=\delta(K_{21})$ of $\KQ(3)$ to a coideal $J_5$
of $\KD(6)$ (see~\ref{K3inKD6}). Here is the Jones projection of
$\delta(K_{21})$ (see~\ref{KQ3section}):
\begin{Mexin}
pI := KQ3::e(1) + KQ3::q(0,2*PI/3,2):
\end{Mexin}
and its image in $\KD(6)$:
\begin{Mexin}
pIInKD6 := KQ3ToKQ6(KQ3::G(pI))
\end{Mexin}
\begin{Mexout}
                                             6           7
              1/4 B(1) + 1/4 B(a b) + 1/4 B(a ) + 1/4 B(a  b)
\end{Mexout}
Again, the result is trivial, and the expression in the matrix basis
of~\ref{K3inKD6} can be obtained with:
\begin{Mexin}
KQ6::M(pIInKD6)
\end{Mexin}
As expected, this projection generates a \sacig{} of dimension $4$:
\begin{Mexin}
nops(coidealAndAlgebraClosure([pIInKD6])), 1/pIInKD6::traceNormalized()
\end{Mexin}
\begin{Mexout}
                                  4, 4
\end{Mexout}

\subsection{Further directions}

Putting everything together, it is completely automatic to test
whether a projection is the Jones projection of a \sacig, and to
construct new Jones projections from previous ones by various
techniques: embeddings, self-duality, intersection and (completed)
union, etc. In the small examples we considered, this was sufficient
to actually construct semi-automatically all \sacigs. Furthermore,
given the Jones projection of two \sacigs $A\subset B$, it would be
straightforward to compute the Bratelli diagram of inclusion (this is
not yet implemented). The ultimate goal would be to compute
automatically all the Jones projections, and therefore the full
lattice of \sacigs, but we do not know yet how to achieve this.

\newpage

\bibliographystyle{alpha}
\bibliography{KD}

\end{document}